\documentclass[10pt]{article}

\usepackage[dvips]{graphicx}
\usepackage{cite}
\usepackage{amsmath,amssymb,subfigure,epsfig,bbm,stmaryrd,MnSymbol,mathrsfs,url}
\usepackage{pstricks}
\usepackage{subfigure}
\graphicspath{{../figures/}} \DeclareGraphicsExtensions{.eps}
\usepackage{amstext}
\usepackage{array,xcolor}
\usepackage{color,soul}
\usepackage{hyperref}
\DeclareMathOperator{\sign}{sign}
\DeclareMathAlphabet{\mathpzc}{OT1}{pzc}{m}{it}
\DeclareFontFamily{OT1}{pzc}{}
\DeclareFontShape{OT1}{pzc}{m}{it}{<-> s * [1.2] pzcmi7t}{}
\DeclareMathAlphabet{\mathpzc}{OT1}{pzc}{m}{it}
\DeclareMathAlphabet{\mathbbit}{U}{bbm}{m}{sl}

\usepackage{algorithm,mdframed}
\usepackage{algorithmic}

\usepackage[percent]{overpic}
\usepackage{amsmath,mathtools}
\usepackage[all,cmtip]{xy}
\usepackage{chemfig}

\newcounter{example}

\newtheorem{theorem}{Theorem}[section]
\newtheorem{lemma}[theorem]{Lemma}
\newtheorem{proposition}[theorem]{Proposition}

\newtheorem{definition}[theorem]{Definition}

\newenvironment{proof}[1][\it Proof:]{\begin{trivlist}
\item[\hskip \labelsep {\bfseries #1}]}{\qed \end{trivlist}}

\newcommand{\qed}{\nobreak \ifvmode \relax \else
      \ifdim\lastskip<1.5em \hskip-\lastskip
      \hskip1.5em plus0em minus0.5em \fi \nobreak
      \vrule height0.75em width0.5em depth0.25em\fi}

\newcommand{\supp}{\mbox{supp}}
\newcommand{\rank}{\mbox{rank}}

\newcommand{\bdelta}{\boldsymbol{\delta}}

\newcommand{\balpha}{\boldsymbol{\alpha}}

\newcommand{\Ip}{\mathcal{I}_\oplus}
\newcommand{\In}{\mathcal{I}_\ominus}

\newcommand{\Kp}{\mathcal{K}_\oplus}
\newcommand{\Kn}{\mathcal{K}_\ominus}
\newcommand{\Np}{\mathcal{N}_\oplus}
\newcommand{\Nn}{\mathcal{N}_\ominus}

\newcommand{\Beta}{{\boldsymbol{\beta}}}

\newcommand{\Blambda}{{\boldsymbol{\Lambda}}}
\newcommand{\Bz}{{\boldsymbol{z}}}
\newcommand{\By}{{\boldsymbol{y}}}
\newcommand{\Bh}{{\boldsymbol{h}}}

\newcommand*\hexbrace[2]{%
  \overset{#2}{\overbrace{\rule{#1}{0pt}}}}

\usepackage{mathtools,amssymb,graphicx}
\newcommand*{\seq}{\mathrel{\text{\raisebox{.25ex}{\rotatebox[origin=c]{180}{$\cong$}}}}}

\newcommand*{\nseq}{\mathrel{\text{\raisebox{.25ex}{\rotatebox[origin=c]{180}{$\ncong$}}}}}

\usepackage[T1]{fontenc}
\usepackage[utf8]{inputenc}
\usepackage[font=small,labelfont=bf,tableposition=top]{caption}

\usepackage{easybmat}
\usepackage{multirow,bigdelim}

\usepackage{makeidx}
\makeindex

\usepackage{multirow,bigdelim}
\usepackage{blkarray}

\begin{document}
\title{\vspace{-2cm}\bf{Convex Cardinal Shape Composition}}   

\author{Alireza Aghasi and Justin Romberg\thanks{School of Electrical and Computer Engineering, Georgia Institute of Technology, Atlanta, GA. Emails: {\tt aaghasi@ece.gatech.edu} and {\tt jrom@ece.gatech.edu}.}}

\date{}    

\maketitle

\begin{abstract}
We propose a new shape-based modeling technique for applications in imaging problems. Given a collection of shape priors (a shape dictionary), we define our problem as choosing the right dictionary elements and geometrically composing them through basic set operations to characterize desired regions in an image. This is a combinatorial problem solving which requires an exhaustive search among a large number of possibilities. We propose a convex relaxation to the problem to make it computationally tractable. We take some major steps towards the analysis of the proposed convex program and characterizing its minimizers. Applications vary from shape-based characterization, object tracking, optical character recognition, and shape recovery in occlusion, to other disciplines such as the geometric packing problem.

\end{abstract}


\section{Introduction}
\label{sec:intro}
In many machine vision applications, the main objective is to identify and characterize regions of interest in an image. For instance, in image segmentation, the main goal is to label and aggregate pixels that are visually or statistically related. Specifically, in the binary case \cite{chan2001active, pal1993review}, a given image $D$ is partitioned into disjoint regions $\Sigma$ and $D\setminus \Sigma$, so that ``similar'' pixels lie on the same segments. In object tracking, objects of interest are identified and traced in consecutive frames of a video stream. Optical character recognition (OCR) is another example, where the characters inside an image need to be identified and correctly labeled to model a natural perception.


It often happens that the geometry of interest is  composed of simpler building-blocks. Moreover, if we have a fixed collection of simple prototype shapes, being composed of few prototypes may be a way of regularizing the problem.

Some interesting imaging problems can be viewed as characterization problems, where the target geometry can be approximated as a composition of known prototype shapes. For instance, in OCR, a word may be considered as the union of distinct or overlapping characters. Characterizing the inclusion in the image using alphabetic characters as the prototype shapes may allow us to identify the constituting letters. Another example is tracking multiple known objects across video frames, where they may overlap or become occluded by simple optical obstacles. In both examples, applying basic set operations among a number of prototype shapes allows us to express the target geometry.


In the past few decades considerable effort has been devoted to variational models and level set type methods linked to the image segmentation problem \cite{chan2003variational, chan2001active, osher2001level}. Specifically in the context of using shape priors to perform the segmentation, the earliest work is reported by Leventon \emph{et al.} \cite{leventon2000statistical}. The main idea is to setup the problem as an active contour evolution using a shape model prior. The parametric shape model is built by applying principal component analysis (PCA) on the signed distance maps of the training shapes. The work was later improved by performing a direct optimization within the subspace spanned by the principal components \cite{tsai2001curve, tsai2003shape}. Other frameworks such as
direct application of the shape constraints to the zero level of the embedding function \cite{chen2002using}, and simultaneously imposing shape priors to multiple objects in the image \cite{cremers2006multiphase} were considered to improve the segmentation efficiency. However, despite their widespread use, the majority of the level set-based techniques are cast as non-convex optimizations, where global solutions are often inaccessible. While the use of convex models for the image segmentation has been considered in some of the more recent works (e.g., see \cite{chan2006algorithms, brown2012completely, chambolle2012convex}), incorporating prior shape information into a convex formulation has remained a challenge \cite{song2013optimal}. This paper will present a convex method to incorporate the geometry of prior shapes into the segmentation framework.

In the context of binary image segmentation, consider a fixed collection of prototype shapes (a shape dictionary), and a given composition rule (such as the union operation). The segmentation problem of interest is determining a partitioner $\Sigma$ that is representable by the composition of some elements in the dictionary. To  promote geometrically simpler outcomes, the number of constituting elements in $\Sigma$ may be limited to a fixed level.

Despite its intuitive nature, the aforementioned problem is combinatorial, and its solution requires an exhaustive search among a large number of possibilities. Focusing on a composition rule that allows both component inclusion and exclusion (through the set union and set difference operations), in this paper we will study computationally tractable techniques that follow similar performance patterns as the combinatorial problem. Our main strategy is to provide a convex relaxation to the original problem.
-eps-converted-to.pdf
\begin{figure}[t!]%
\centering
\subfigure[][]{\includegraphics[width=84mm]{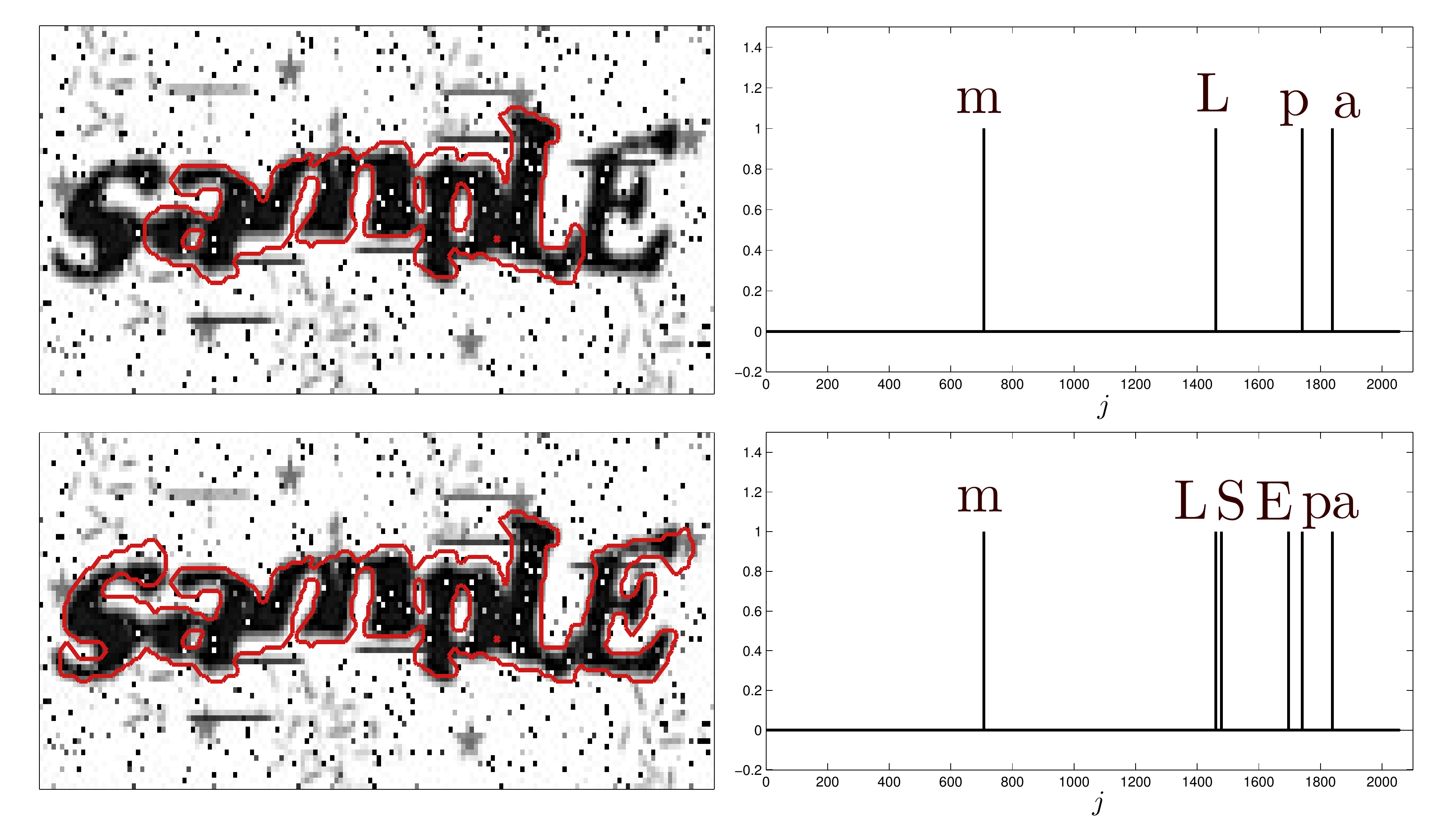}}
\subfigure[][]{\includegraphics[width=41mm]{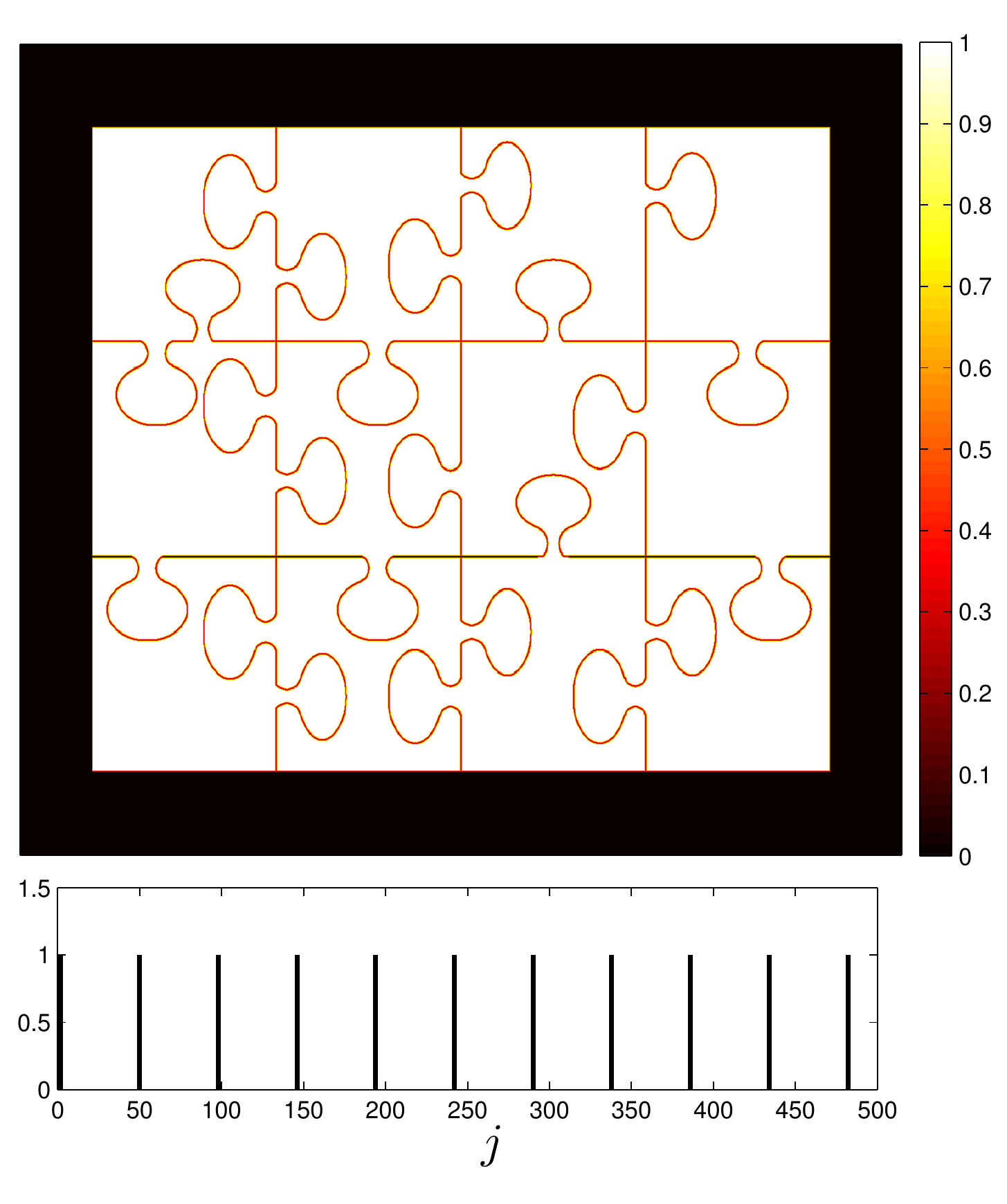}}
\caption{(a) An image segmentation problem with a dictionary of characters as the shape priors; along with the segmentation, the constituting characters are also identified; (b) Solving a puzzle by applying the proposed convex shape composition algorithm}%
\label{fig1intro}%
\end{figure}

Before presenting the technical details of the segmentation model and the proposed convex relaxation, we would like to motivate the reader by presenting some related results. Figure \ref{fig1intro}(a) presents a challenging OCR problem where aside from the clutter present in the image, the characters are rotated and overlapped. Using lower and upper-case letters distributed throughout the imaging domain as the prototype shapes, we have shown the outcome of the segmentation when the number of constituting elements are limited to 4 (the top panel) and 6 (the bottom panel). The stem plots next to each figure report the active elements in the segmentation, from which the underlying letters are identified. Among the most favorable features of our convex model are the simple formulation and the convenience in tuning the free parameter related to the number of active shapes (which often sweeps integer values).

Figure \ref{fig1intro}(b) illustrates another binary segmentation problem, where $\Sigma$ is simply a rectangular puzzle pad. The prototype shapes used in the dictionary are rotated and displaced elements of a puzzle. When the number of constituting elements is limited to 12, the optimal partitioner corresponds to the puzzle solution, which matches the outcome of our algorithm. Details for both examples are available in Section \ref{sec5}.

After this quick overview, we will spend the remainder of this section to a more technical presentation of the problem, specifically focusing on the binary image segmentation task as a variational problem \cite{chan2003variational}.

%
%
%
%
%
%
%
%

\subsection{Shape Composition and Binary Image Segmentation}
For an imaging domain $D\subset\mathbb{R}^d$ with pixel values $u(x)$, $x\in D$, a binary segmentation corresponds to partitioning $D$ into two disjoint regions $\Sigma$ and $D\setminus \Sigma$, where each region encompasses similar pixels. A well-known variational model is determining $\Sigma$ (and accordingly $D\setminus \Sigma$) via the minimization
\begin{equation}\label{eq1}
\Sigma^*=\operatorname*{arg\,min}_{\Sigma}\;\;\gamma(\Sigma)+\int_{\Sigma}\Pi_{in}(x)\;\mbox{d}x + \int_{D\setminus \Sigma}\Pi_{ex}(x)\;\mbox{d}x,
\end{equation}
where $\gamma(\Sigma)$ is a regularization term promoting a desired structure, and $\Pi_{in}(.)\geq 0$ and $\Pi_{ex}(.) \geq 0$ are some image-dependent inhomogeneity measures.

A widely-used measure is the one proposed by Chan and Vese \cite{chan2001active}, which takes $\Pi_{in}(x)=(u(x)-\tilde u_{in})^2$ and $\Pi_{ex}(x)=(u(x)-\tilde u_{ex})^2$, for scalar values $\tilde u_{in}$ and $\tilde u_{ex}$. These scalars are chosen to be representatives of the average pixel value within each region. The resulting partitions $\Sigma^*$ and $D\setminus \Sigma^*$ are expected to aggregate pixels that take values around $\tilde u_{in}$ and $\tilde u_{ex}$, respectively. The Chan-Vese model is in fact a particular case of the Mumford-Shah functional, used as a criterion of optimality for segmenting an image \cite{mumford1989optimal}.

In a more general but closely related framework \cite{cremers2007review}, the inhomogeneity measures are modeled as the log-likelihood function
\[\Pi_r(x) = - \log  \mathpzc{p}\big(\upsilon(x)|\;\theta_r\big) \qquad r:in,ex,
\]
where $\upsilon(x)$ is the similarity feature of interest (such as intensity, texture, etc) and $\theta_r$ parameterizes the probability density function $\mathpzc{p}(.)$. Regardless of the approach taken, throughout the paper we keep the notation general and only assume
\[\forall x\in D: \qquad \Pi_r(x)\geq 0, \qquad r:in,ex.
\]

When $\Pi_{in}$ and $\Pi_{ex}$ are fixed and known a priori,
\[\int_{D\setminus \Sigma}\Pi_{ex}(x)\;\mbox{d}x = \int_{D}\Pi_{ex}(x)\;\mbox{d}x - \int_{ \Sigma}\Pi_{ex}(x)\;\mbox{d}x,
\]
and since the first term on the right-hand side is constant with respect to $\Sigma$, an equivalent formulation of (\ref{eq1}) is
\begin{equation}\label{eq2}
\Sigma^*=\operatorname*{arg\,min}_{\Sigma}\;\;\gamma(\Sigma)+\int_{\Sigma}\big(\Pi_{in}(x) - \Pi_{ex}(x)\big)\;\mbox{d}x.
\end{equation}

In the majority of the techniques developed to address the segmentation problem (\ref{eq1}), the algorithms alternate between minimizing the cost and updating the inhomogeneity measures $\Pi_{in}$ and $\Pi_{ex}$ \cite{chan2001active, cremers2007review}. We will take $\Pi_{in}$ and $\Pi_{ex}$ to be fixed and known, as our focus will be on solving the variational problem (\ref{eq1}). More details about the choice of the inhomogeneity measures will be provided in the simulations section.

When the regularization term is included in the cost, additional structure inherited by the nature of $\gamma$ affects $\Sigma^*$. Well-known regularizations used in the context of shape recovery often promote boundary smoothness and compactness on the resulting partitioner \cite{chan2001active, osher2003geometric}.

As mentioned earlier, we consider a different way of regularizing the segmentation problem, where $\Sigma$ inherits its geometric features from a set of prototype shapes. We consider $\Sigma$ to be the result of a set-algebraic ``composition rule'' on a number of prototype shapes.

Given a dictionary of shapes $\mathfrak{D}=\{\mathcal{S}_1, \mathcal{S}_2, \cdots ,\mathcal{S}_{n_s}\}$, a shape composition rule, $\mathpzc{R}$, is a set algebraic expression (formed by basic operations, such as the union, intersection and set difference), which imposes a certain model for the objects sought. A natural choice for $\mathpzc{R}$, also used in \cite{aghasi2013sparse}, is
\begin{equation}\label{eq3}
\mathpzc{R}_{\;\Ip,\In}\triangleq\big(\bigcup_{j\in \Ip}\mathcal{S}_j\big) \big\backslash \big(\bigcup_{j\in \In}\mathcal{S}_j\big),
\end{equation}
where a combination of suitable elements in the dictionary through the union and relative complement reconstructs $\Sigma$. We require $\Ip$ and $\In$ not to be redundant, which intuitively means that removing any index from either set will change the resulting geometry (in Section \ref{secdef} we will provide a precise definition of non-redundancy).

With the dictionary in place, determining the non-redundant index sets $\Ip$ and $\In$ is the main objective of our segmentation problem, i.e.,
\begin{equation}\label{eq4}
\hspace{-2.5cm}\mbox{\scriptsize(SC)\normalsize}\qquad\qquad\quad \{ {\Ip}^*, {\In}^*\}= \operatorname*{arg\,min}_{\Ip, \In} \int_{\mathpzc{R}_{\;\Ip,\In}} \big(\Pi_{in}(x) - \Pi_{ex}(x)\big)\;\mbox{d}x.
\end{equation}
For $n_s$ elements in the dictionary, the number of possible regions that (\ref{eq3}) can produce counts to\footnote{To count the possibilities, we consider the total number of ways to deposit the shapes into 3 bins and exclude from that the cases that the bin associated with $\Ip$ is empty}
\[\sum_{k_\oplus+k_\ominus\leq n_s}\dbinom{n_s}{ (n_s - k_\oplus - k_\ominus),k_\oplus, k_\ominus} - \sum_{k_\ominus=1}^ {n_s}\dbinom{n_s}{k_\ominus} = 3^{n_s}-2^{n_s}+1.
\]
Addressing (\ref{eq4}) would therefore require a search among an exponentially large number of possibilities. We will henceforth refer to problem (\ref{eq4}) as the basic shape composition (SC) problem.

To promote simpler compositions and avoid shape redundancy, we will limit the number of shapes in the reconstruction. We regularize the SC problem by restricting the representation to at most $s$ shapes
\begin{align}
\label{eq5}
\hspace{-.3cm}\mbox{\scriptsize(Cardinal-SC)\normalsize}\qquad\min_{\Ip, \In} \int_{\mathpzc{R}_{\;\Ip,\In}} \big(\Pi_{in}(x) - \Pi_{ex}(x)\big)\;\mbox{d}x\qquad \mbox{s.t.}:\quad |\Ip|+|\In|\leq s.
\end{align}
Solving (\ref{eq5}) is the central focus of this paper and will be referred to as the cardinal shape composition (Cardinal-SC) problem. Addressing the Cardinal-SC problem also requires an exhaustive search among a large number of possibilities, specifically
\[1+\sum_{k=1}^s\dbinom{n_s}{k}(2^k-1),
\]
cases\footnote{To count the possibilities, we select $k$ shapes from the $n_s$ shapes and count the number of ways to deposit them into two bins, excluding the empty $\Ip$ case. The total number of possibilities requires a sum over $k=1,\cdots,s$ (plus the no-shape selection case)}. This does not give rise to a computationally tractable process, especially for dense and over-complete dictionaries. We thus need to explore close alternatives to the Cardinal-SC problem that can be addressed in a computationally efficient manner.

As we discuss in the next section, basic set operations among the shapes can be modeled by superimposing the corresponding characteristic functions. We exploit this property along with ideas from convex analysis to derive a convex proxy to the Cardinal-SC problem as
\begin{equation}\label{eq5x}
\min_{\balpha} \int_{D}\max\Big(\big(\Pi_{in}(x) - \Pi_{ex}(x)\big)\mathpzc{L}_{\boldsymbol{\alpha}}(x), \big(\Pi_{in}(x) - \Pi_{ex}(x)\big)^-   \Big)\;\mbox{d}x \quad s.t. \quad  \|\balpha\|_1\leq \tau.
\end{equation}
Here $\mathpzc{L}_{\boldsymbol{\alpha}}(x)$ is a linear combination of the characteristic functions associated with the dictionary elements;
\begin{equation}\label{eq9xyz}
\mathpzc{L}_{\boldsymbol{\alpha}}(x)\triangleq\sum_{j=1}^{n_s} \alpha_j \chi_{\mathcal{S}_j}(x), \qquad \mbox{where}\qquad \chi_\mathcal{S}(x)\triangleq\left\{
\begin{array}{cr}
1 & x\in \mathcal{S}\\ 0 & x\notin \mathcal{S}
\end{array}.
\right.
\end{equation}

Roughly speaking, in relating the minimizer of (\ref{eq5x}) to the optimal index sets associated with the Cardinal-SC problem, the active $\alpha_j$ values identify the active shapes in the composition and their sign determines the index set ($\Ip$ or $\In$) they belong to. Inspired by ideas from sparse recovery, the $\ell_1$ constraint is used to control the number of active shapes in the final representation.

The main objective of this paper is to provide a more detailed
presentation of (\ref{eq5x}) as a proxy to the Cardinal-SC problem. Our presentation is equipped with some level of analysis over basic setups, where the relaxation performs a similar job as the Cardinal-SC problem. Moreover, we will present generic sufficient conditions under which the proposed proxy perfectly recovers a target composition. Some of the techniques employed to analyze the problem are adopted from the compressive sensing and structured recovery literature \cite{foucart2013mathematical}. However, the overall analysis theme, to the best of our knowledge, is the first of its kind and can be used as a benchmark for future extensions of the problem. The examples in the simulation section serve as the complementary justification to our analysis.

\subsection{Notation and Organization}
Our presentation mainly relies on multidimensional calculus. We use bold characters to denote vectors and matrices. Considering a matrix $\boldsymbol{A}$ and the index sets $\Gamma_1$, and $\Gamma_2$, we use $\boldsymbol{A}_{\Gamma_1,:}$ to denote the matrix obtained by restricting the rows of $\boldsymbol{A}$ to $\Gamma_1$. Similarly, $\boldsymbol{A}_{:,\Gamma_2}$ denotes the restriction of $\boldsymbol{A}$ to the columns specified by $\Gamma_2$, and $\boldsymbol{A}_{\Gamma_1,\Gamma_2}$ is the submatrix with the rows and columns restricted to $\Gamma_1$ and $\Gamma_2$, respectively. For a vector $\boldsymbol{a}=[a_1,\cdots,a_n]^T$, we use $\boldsymbol{a}_i$ to denote the $i$-th element of $\boldsymbol{a}$, i.e., $\boldsymbol{a}_i=a_i$. A quick reference table containing the notations mainly used throughout the paper is included in the Appendix.

The remainder of this paper is organized as follows. In Section \ref{sec2} we present a procedure that leads us to the convex relaxation of the Cardinal-SC problem. Section \ref{sec3} is devoted to some basic definitions and preliminary results on the performance of the proposed convex surrogate. More extensive technical details and convex analysis tools are presented in Section \ref{sec4}. The main theme in this section is the possibility of recovering a target composition through the proposed convex proxy. Section \ref{sec5} presents some experiments and simulations to support the ideas developed in the paper. Section \ref{sec6} ultimately provides the concluding remarks and discusses the possible future directions. To help with the flow, the majority of the technical proofs are moved to Section \ref{sec:proof}.

\section{A Convex Relaxation to the Shape Composition Problem}\label{sec2}
This section discusses the rationale for using (\ref{eq5x}) as a convex proxy to the Cardinal-SC problem.

\subsection{Shape Composition by Superimposing Characteristic Functions} \label{sec-pseudo} In \cite{aghasi2011parametric, aghasi2013geometric, aghasi2013sparse}, Aghasi \emph{et al.} introduced the notion of a \emph{pseudo-logical} property. This property states that basic arithmetic operations among compactly-supported Lipschitz functions (called \emph{knolls} in \cite{aghasi2013sparse}) can approximately model certain set operations among the function supports. Due to particular differentiability requirements, the pseudo-logical property in \cite{aghasi2011parametric, aghasi2013geometric, aghasi2013sparse} was only discussed in the context of Lipschitz functions and remained an approximation. In this paper, we show that composing shapes can be \emph{exactly} characterized by superimposing characteristic functions.

Consider two given shapes $\mathcal{S}_1$ and $\mathcal{S}_2$ and the corresponding characteristic functions $\chi_{\mathcal{S}_1}$ and $\chi_{\mathcal{S}_2}$ as defined in (\ref{eq9xyz}). Basic set operations ($\cup$, $\setminus$ and $\cap$) on $\mathcal{S}_1$ and $\mathcal{S}_2$ can be related to basic arithmetic operation on the corresponding characteristic functions via the following three facts (illustrated in Figure \ref{fig1}):
\begin{equation}\label{eq6}
\supp^+ (\alpha_1\chi_{\mathcal{S}_1}+\alpha_2\chi_{\mathcal{S}_2})= \mathcal{S}_1\cup \mathcal{S}_2 \qquad (\alpha_1>0, \alpha_2>0),
\end{equation}
\begin{equation}
\label{eq7}
\supp^+ (\alpha_1\chi_{\mathcal{S}_1}-\alpha_2\chi_{\mathcal{S}_2})= \mathcal{S}_1\setminus \mathcal{S}_2  \qquad (\alpha_2>\alpha_1>0),
\end{equation}
and
\begin{equation}
\label{eq8}
\supp^+ (\alpha_1\chi_{\mathcal{S}_1} \alpha_2\chi_{\mathcal{S}_2})= \mathcal{S}_1\cap \mathcal{S}_2  \qquad (\alpha_1\alpha_2>0).
\end{equation}
Here $\supp^+(.)$ denotes the positive support of the corresponding functions, i.e., for a given function $\chi(x)$
\begin{equation}\label{eqsupp}
\supp^+(\chi)\triangleq\supp(\chi^+),
\end{equation}
where $\chi^+$ returns the $\chi$ values when $\chi>0$ and zero otherwise.

\begin{figure}%
\centering
\subfigure[][]{\includegraphics[width=44mm]{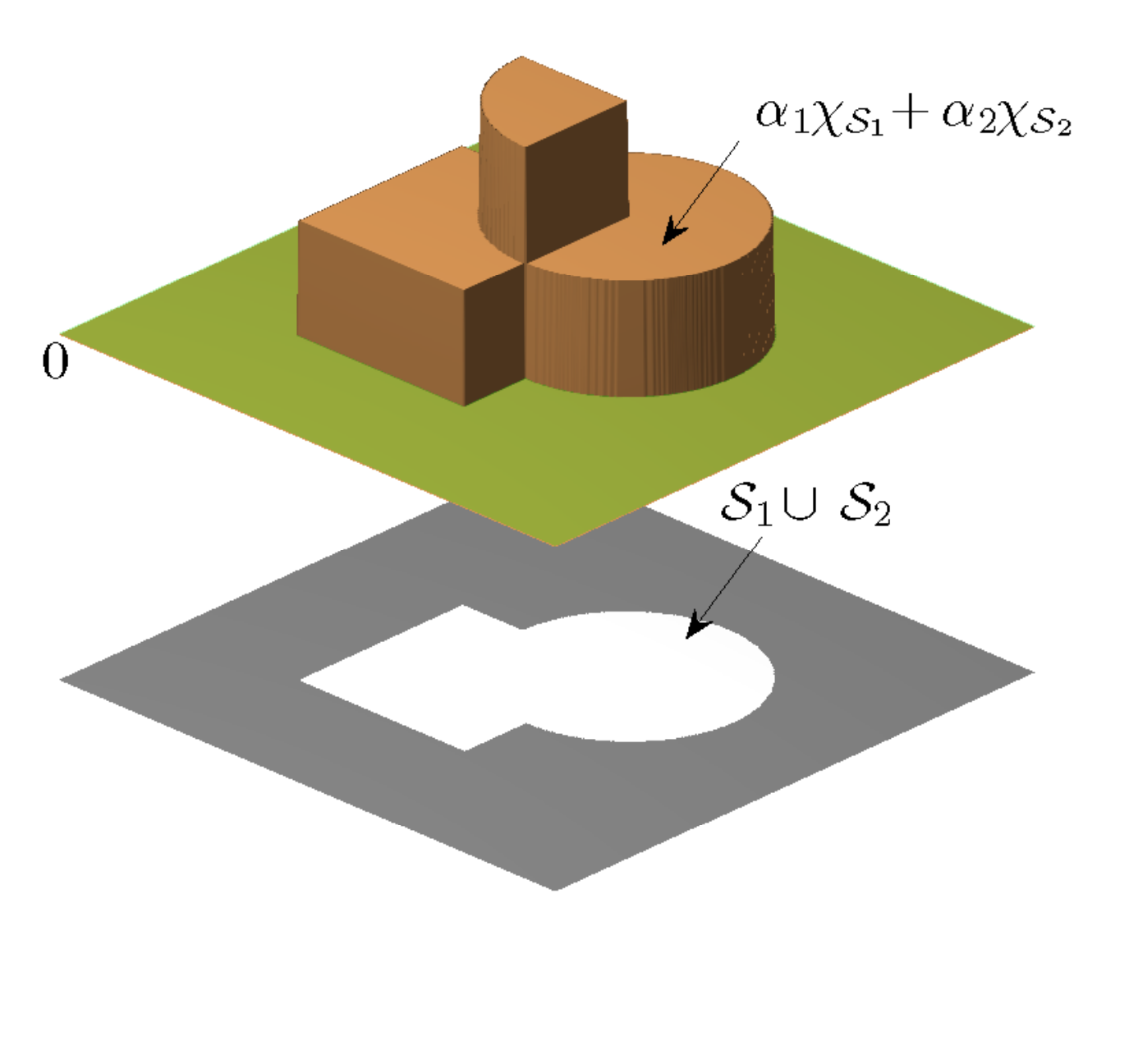}}
\subfigure[][]{\includegraphics[width=44mm]{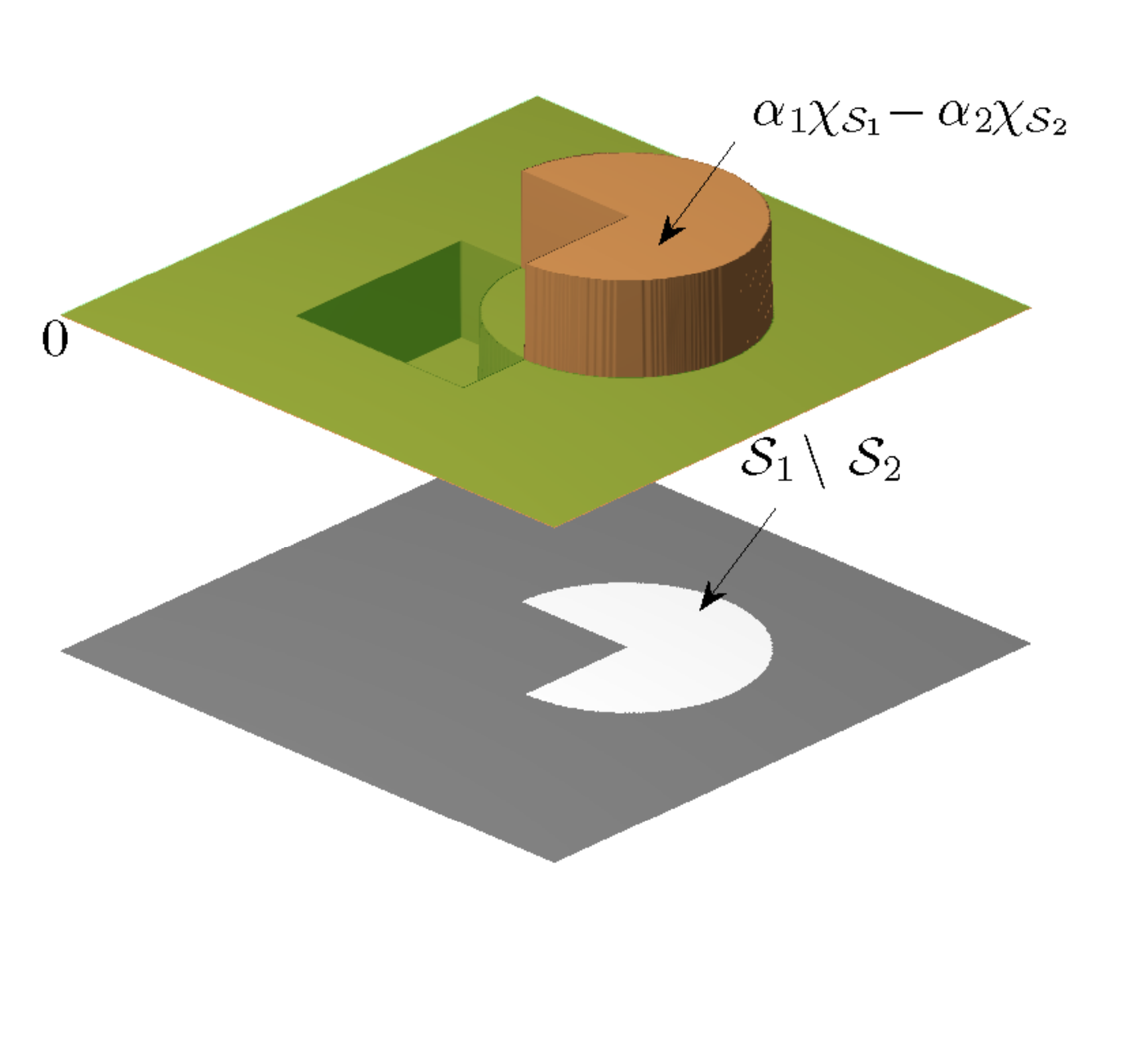}}
\subfigure[][]{\includegraphics[width=44mm]{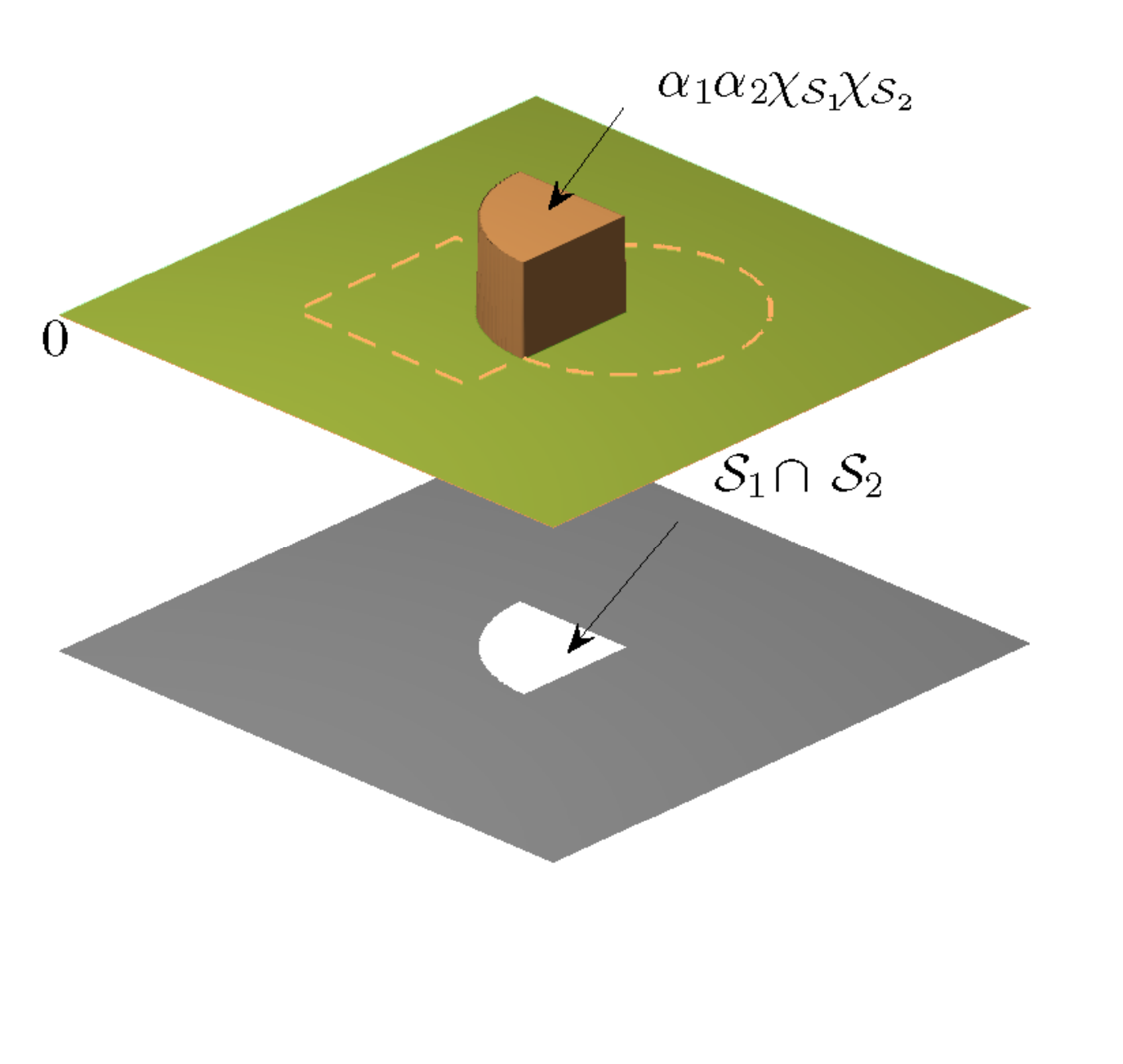}}
\caption{Representation of shape compositions via basic arithmetic operations on the corresponding characteristic functions: (a) union of two shapes, (b) difference of two shapes, (c) intersection of two shapes}%
\label{fig1}%
\end{figure}

Based on the preceding facts, a linear combination of characteristic functions
\begin{equation*}
\mathpzc{L}_{\boldsymbol{\alpha}}(x) = \sum_{j=1}^{n_s} \alpha_j \chi_{\mathcal{S}_j}(x),
\end{equation*}
could be a representative of $\mathpzc{R}_{\;\Ip,\In}$ for suitably tuned $\alpha_j$ values. This is codified with the following result which is proved in Section \ref{sec:proof}:
\begin{proposition} \label{th1}
Given a collection of shapes $\{\mathcal{S}_j\}_{j=1}^{n_s}$ in $D$, for any composition $\Sigma = (\underset{j\in \Ip}{\bigcup}\mathcal{S}_j) \backslash (\underset{j\in \In}{\bigcup}\mathcal{S}_j)$, there exist scalars $\alpha_j$ such that  $\mathpzc{L}_{\boldsymbol{\alpha}}(x)=\hspace{-.3cm}\underset{j\in \Ip \cup \In}{\sum} \alpha_j \chi_{\mathcal{S}_j}(x)$ satisfies\vspace{-.3cm}
\begin{equation}\left\{
     \begin{array}{lr}
       \mathpzc{L}_{\boldsymbol{\alpha}}(x)\geq 1 &  x \in \Sigma\\
       \mathpzc{L}_{\boldsymbol{\alpha}}(x)\leq 0 &  x \notin \Sigma
     \end{array}.   \right.
     \label{eq9a}
\end{equation}
\end{proposition}
A corollary of Proposition \ref{th1} is the existence of $ \mathpzc{L}_{\boldsymbol{\alpha}}(x)$ such that $\supp^+\big(\mathpzc{L}_{\boldsymbol{\alpha}}(x)\big)=\Sigma$. The separation between the values of $\mathpzc{L}_{\boldsymbol{\alpha}}(x)$ inside and outside $\Sigma$ makes the representation stable.

Using a function to characterize a region in a domain is a well-established technique in the imaging community, mainly discussed in the context of the level set method \cite{osher2003geometric}. Considering a level set function $\phi(x)$, which characterizes the boundaries of a region $\Sigma$ by its zero-level set (and takes positive values inside $\Sigma$), the image-dependent objective in (\ref{eq2}) may be written as
\begin{equation}\label{eq22}
\int_{D}\big(\Pi_{in}(x)-\Pi_{ex}(x)\big)H\big(\phi(x)\big)\;\mbox{d}x ,
\end{equation}
where $H(\phi) = \boldsymbol{1}_{(0,\infty)}(\phi)$, is the Heaviside function. Analogously, the smoothing and compactness regularizers can be cast in terms of $\phi$ to produce an overall functional merely in terms of $\phi$ (see examples in \cite{chan2001active, aubert2006mathematical}). This way of modeling allows casting the segmentation problem as a variational problem in terms of $\phi$.

Regardless of the variational approach taken to carry out the minimization, based on Proposition \ref{th1} and the functional representation (\ref{eq22}), we suggest the following surrogate to the basic SC problem:
\begin{equation}\label{eq9x}
\min_{\balpha} \int_{D} \big(\Pi_{in}(x) - \Pi_{ex}(x)\big)H\big(\mathpzc{L}_{\boldsymbol{\alpha}}(x)\big)\;\mbox{d}x\qquad \mbox{s.t.}\quad \mathpzc{L}_{\boldsymbol{\alpha}}(x) = \sum_{j=1}^{n_s} \alpha_j \chi_{\mathcal{S}_j}(x),
\end{equation}
where the $\alpha_j$ coefficients are used to identify the shape composition elements.

Regarding this model, few remarks are in order. First, the proposed surrogate requires searching for the minimizer in a subspace spanned by the shape characteristics, while the basic SC problem requires searching among a certain class of objects that follow the composition pattern (\ref{eq3}). Based on Proposition \ref{th1}, we are however certain that all possible shape compositions have a representation in the subspace associated with the surrogate model. In other words, the optimization domain associated with the surrogate model entirely covers the corresponding optimization domain associated with the basic SC problem.

Second, to characterize a region, our modeling relies on the positive support of the function $\mathpzc{L}_{\boldsymbol{\alpha}}(x)$, in spite of the level set method which concerns a function's zero-level set to identify the region boundaries. The distinction is made more clear by noting that the zero-level set of $\mathpzc{L}_{\boldsymbol{\alpha}}(x)$ does not necessarily identify a region boundaries and can easily be a set with nonzero Lebesgue measure (e.g., see the green regions in Figure \ref{fig1}).

We would like to note that the ultimate goal in our problem is identifying suitable elements of a shape dictionary, composing which according to (\ref{eq3}) yield an optimal segmentation. The level set discussions are mainly brought here to provide the proper transition between our framework and the classic segmentation techniques.

%

\subsection{The Convex Formulation}

The objective function in (\ref{eq22}) is non-convex. In \cite{chan2006algorithms} this problem is relaxed into a convex form by replacing $H(\phi)$ with a bounded version of $\phi$,
\begin{align}\label{eq23}
\min_{0\leq \phi\leq 1} \int_{D}\big(\Pi_{in}(x)- \Pi_{ex}(x)\big)\phi(x)\;\mbox{d}x.
\end{align}
With this setup, the optimal $\phi$ function takes a value of 1 where $\Pi_{in}(x)<\Pi_{ex}(x)$ and vanishes where $\Pi_{in}(x)>\Pi_{ex}(x)$. For any $v\in(0,1)$, the $v$-level set of the optimal $\phi(.)$ function identifies the optimal shape \cite{chan2006algorithms}.

To bias the segmentation outcomes with the composition rule $\mathpzc{R}_{\;\Ip,\In}$, similar to the previous section, one may think of replacing $\phi$ with $\mathpzc{L}_{\boldsymbol{\alpha}}(x)$ in (\ref{eq23}) and cast the optimization in terms of $\boldsymbol{\alpha}$. However, identifying a shape through the $v$-level set of $\mathpzc{L}_{\boldsymbol{\alpha}}(x)$ for any $v\in(0,1)$, and still exploring the composition features (\ref{eq6}) and (\ref{eq7}) require $\mathpzc{L}_{\boldsymbol{\alpha}}(x)$ to take values beyond $[0,1]$ (see equation (\ref{eq9a})). In other words, we need to determine a reasonably tight convex formulation that pushes $\mathpzc{L}_{\boldsymbol{\alpha}}(x)$ towards $[1,\infty)$ anywhere that $\Pi_{in}(x)<\Pi_{ex}(x)$ and towards $(-\infty,0]$ where $\Pi_{in}(x)>\Pi_{ex}(x)$.

To proceed with deriving a convex formulation, we rewrite the objective in (\ref{eq9x}) as
\begin{align}\label{eq22x}\nonumber
\int_{D}\big(\Pi_{in}(x)-\Pi_{ex}(x)\big)H\big(\mathpzc{L}_{\boldsymbol{\alpha}}(x)\big)\;\mbox{d}x  =& \int_{D}\big(\Pi_{in}(x)-\Pi_{ex}(x)\big)^+H\big(\mathpzc{L}_{\boldsymbol{\alpha}}(x)\big)\;\mbox{d}x \\& - \int_{D}\big(\Pi_{ex}(x)-\Pi_{in}(x)\big)^+H\big(\mathpzc{L}_{\boldsymbol{\alpha}}(x)\big)\;\mbox{d}x.
\end{align}
For a fixed $x$, only one of the integrands on the right hand side of (\ref{eq22x}) can be nonzero. Based on the sign of $(\Pi_{in}-\Pi_{ex})$ we can consider a certain convex relaxation to each integrand. These relaxations are $(\Pi_{in} - \Pi_{ex})\max(\mathpzc{L}_{\boldsymbol{\alpha}},0)$ and $(\Pi_{in} - \Pi_{ex})\min(\mathpzc{L}_{\boldsymbol{\alpha}},1)$ when $(\Pi_{in} - \Pi_{ex})$ is respectively positive and negative (Figure \ref{fig0}). An alternative relaxation to (\ref{eq9x}) is therefore
\begin{equation*}
\int_{D}\big(\Pi_{in}(x) - \Pi_{ex}(x)\big)^+\max\big(\mathpzc{L}_{\boldsymbol{\alpha}}(x),0\big)\;\mbox{d}x - \int_{D}\big(\Pi_{ex}(x)-\Pi_{in}(x)\big)^+\min\big(\mathpzc{L}_{\boldsymbol{\alpha}}(x),1\big)\;\mbox{d}x.
\end{equation*}
This objective can be written in a more concise form as
\begin{equation}\label{eq22xxx}
\int_{D}\max\Big(\big(\Pi_{in}(x) - \Pi_{ex}(x)\big)\mathpzc{L}_{\boldsymbol{\alpha}}(x), \big(\Pi_{in}(x) - \Pi_{ex}(x)\big)^-   \Big)\;\mbox{d}x,
\end{equation}
where $\Pi^-$ returns the $\Pi$ values when $\Pi<0$ and zero otherwise.

\begin{figure}%
\centering
\subfigure[][]{\includegraphics[width=66mm]{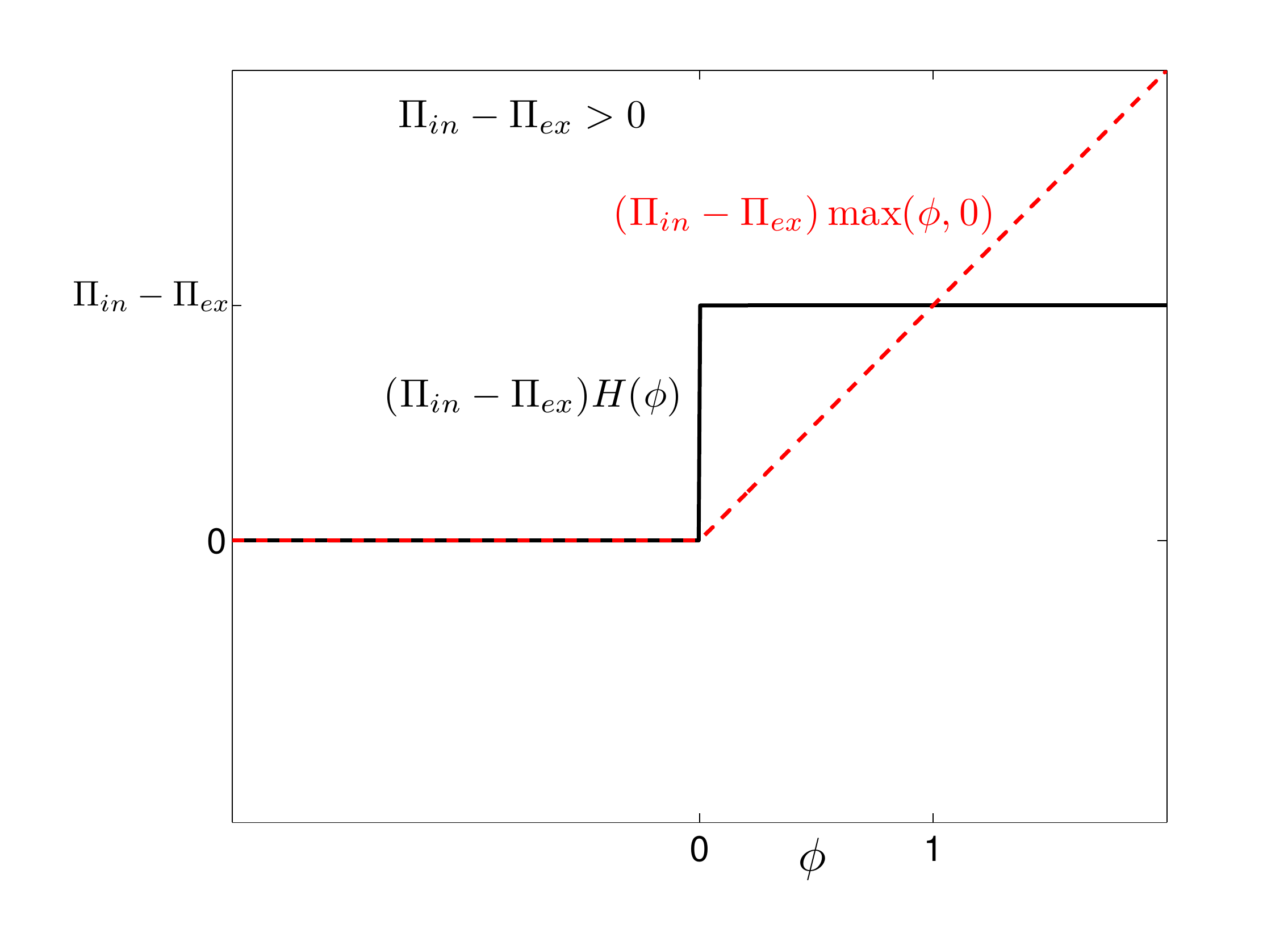}}\hspace{-.5cm}
\subfigure[][]{\includegraphics[width=66.5mm]{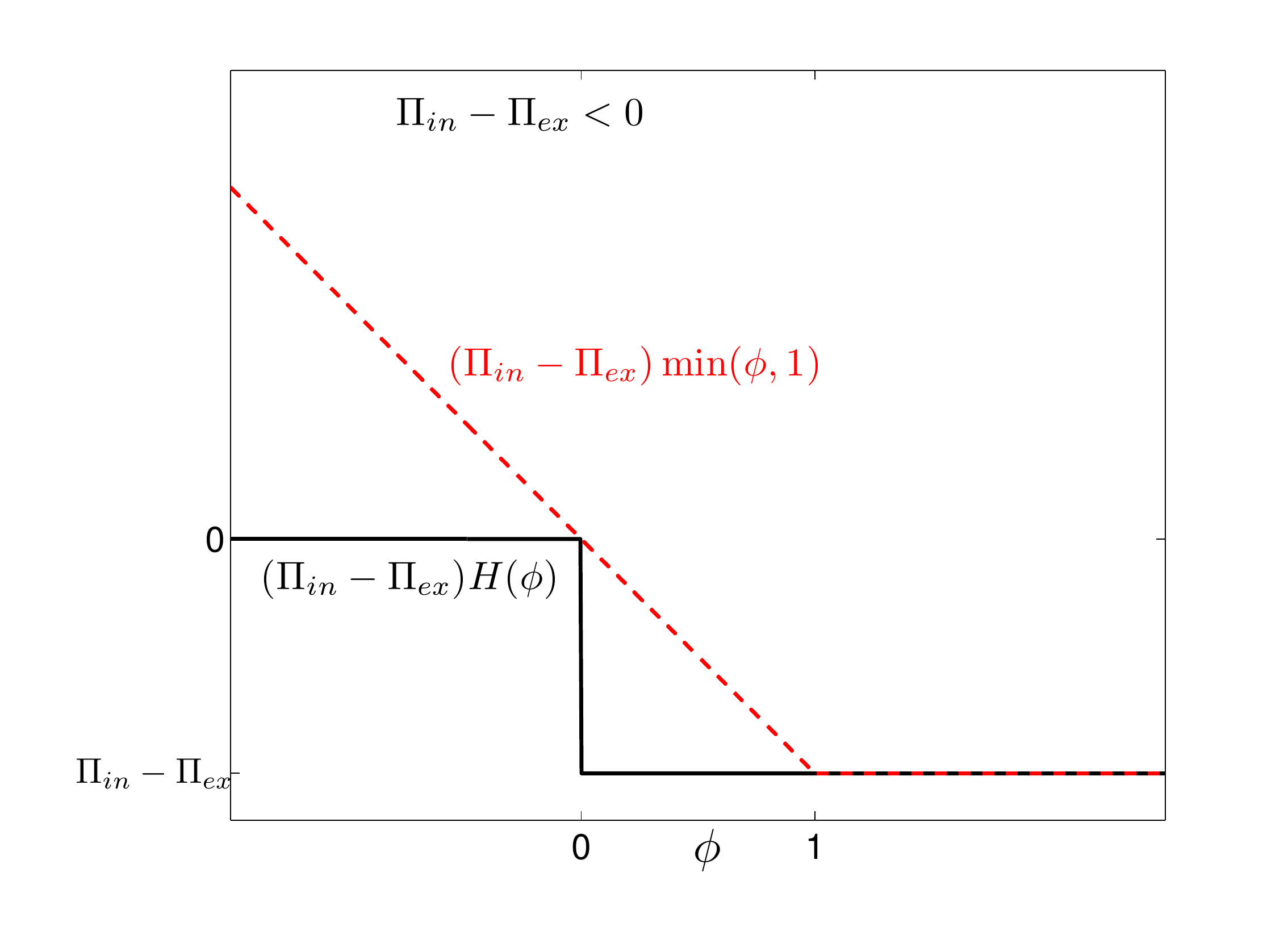}}
\caption{The proposed convexification based on the sign of $\Pi_{in}-\Pi_{ex}$}\label{fig0}
\end{figure}

Minimization of (\ref{eq22xxx}) is our proposed convex proxy to the SC problem. To impose the cardinality constraint on the number of active shapes, inspired by the ongoing trend in compressive sensing \cite{candes2006robust, foucart2013mathematical, candes2013simple}, the cardinality of $\boldsymbol{\alpha}$ may be controlled by imposing an $\ell_1$-constraint on $\boldsymbol{\alpha}$. Ultimately,
we suggest the following minimization as a convex proxy to the Cardinal-SC problem, and refer to it as the sparse-convex shape composition (Sparse-CSC) problem:
\begin{equation}\label{eq22y}
\hspace{-.03cm}\mbox{\scriptsize(Sparse-CSC)\normalsize}\;\; \min_{\|\boldsymbol{\alpha}\|_1\leq \tau} \int_{D}\max\Big(\big(\Pi_{in}(x) - \Pi_{ex}(x)\big)\mathpzc{L}_{\boldsymbol{\alpha}}(x), \big(\Pi_{in}(x) - \Pi_{ex}(x)\big)^-   \Big)\;\mbox{d}x.
\end{equation}
The Sparse-CSC problem may also appear in a regularized form as
\begin{equation}\label{eq22yy}
\min_{\balpha} \int_{D}\max\Big(\big(\Pi_{in}(x) - \Pi_{ex}(x)\big)\mathpzc{L}_{\boldsymbol{\alpha}}(x), \big(\Pi_{in}(x) - \Pi_{ex}(x)\big)^-   \Big)\;\mbox{d}x + \lambda \|\balpha\|_1.
\end{equation}
The parameters $\tau$ and $\lambda$ are free parameters that control the level of sparsity. In this paper we specifically focus on the constrained form (\ref{eq22y}), as tuning $\tau$ is often performed in a more controllable way.

From an implementation standpoint, either one of the convex programs (\ref{eq22y}) or (\ref{eq22yy}) are what we need to address, in order to perform a segmentation based on the sparse composition of prototype shapes.

\section{Preliminary Notions and Results}\label{sec3}
We start this section by providing a more formal definition of the fundamental concepts such as shape, disjointness and the non-redundancy of a composition. In Section \ref{disjoint_sec} we demonstrate the equivalent performance of the Cardinal-SC and the Sparse-CSC in the case of disjoint dictionary elements. As a second extreme case, in Section \ref{sec:loc} we discuss the similar performance of the two problems when we make an extreme assumption about the inhomogeneity measures. These results bring a more intuitive sense about the problem without getting involved in the complexity of the model.


\subsection{Basic Concepts}\label{secdef}
For a detailed study of the shape composition problem, we need to provide solid statements about the fundamental concepts. We take the initial step by presenting our definition of a shape.

\begin{definition} Given the imaging domain $D\subset \mathbb{R}^d$, we call $\mathcal{S}\subset D$ a \emph{shape} if
\begin{itemize}
\item[(I)] $\mathcal{S}$ is a closed set (hence, the union of finitely many closed sets) in $D$.
\item[(II)] $\mathcal{S}$ is not a null set, i.e.,
\[
\int_\mathcal{S} \mbox{d}\mu(x)>0,
\]
where $\mu$ is the standard Lebesgue measure\footnote{In this paper we use the simpler notation $\int_\mathcal{S}\mbox{d}x$ to denote a volumetric integration.} on $\mathbb{R}^d$. To express this property, throughout the paper we equivalently use the notation $\mbox{int}(\mathcal{S})\neq\emptyset$, where $\mbox{int}(.)$ denotes the set interior.
\end{itemize}
\end{definition}
As a concrete example, in case of $d=2$, a finite collection of closed discs can be viewed as a single shape even if they are disjoint. However, a line segment cannot be classified as a shape in $\mathbb{R}^2$, since it violates the second property.

For two given sets $S_1$ and $S_2$ in $\mathbb{R}^d$, we will often make use of the following notations:
\begin{itemize}
\item $S_1\seq S_2$, when $\mbox{int}(S_1) = \mbox{int}(S_2)$,
\item $S_1\nseq S_2$, when $\mbox{int}(S_1) \neq \mbox{int}(S_2)$.
\end{itemize}
These notations assist us to establish precise arguments regarding the shape composition problem. Pragmatically, they are the means to compare regions without taking into account the boundary points (and the null sets). We hereby introduce the notion of disjoint shapes which plays a major role in our subsequent discussions.

\begin{definition}
Two shapes $\mathcal{S}_1$ and $\mathcal{S}_2$ are called disjoint if
\[\mathcal{S}_1 \cap \mathcal{S}_2\seq \emptyset.
\]
\end{definition}
In other words, the intersection of two disjoint shapes can at most be a null set.

The non-redundancy of a composition is another concept that we alluded to earlier, but requires a more precise exposition.

\begin{definition}\label{nonred:def}
Given a set of shapes $\{\mathcal{S}_j\}_{j\in \Ip\cup \In}$, a composition $\mathpzc{R}_{\;\Ip,\In} = ({\bigcup}_{j\in \Ip}\mathcal{S}_j) \backslash ({\bigcup}_{j\in \In}\mathcal{S}_j)$ is called non-redundant, if excluding any shape from the composition results in a measure change, i.e.,
\[\forall j'\in \Ip: \quad \mathpzc{R}_{\;\Ip\setminus \{j'\},\In} \nseq \mathpzc{R}_{\;\Ip,\In},
\]
and
\[\forall j'\in \In: \quad \mathpzc{R}_{\;\Ip,\In\setminus \{j'\}} \nseq \mathpzc{R}_{\;\Ip,\In}.
\]
\end{definition}

\subsection{Disjoint Dictionary Elements}\label{disjoint_sec}
Here we consider a simple scenario where we can easily establish that the solutions to the Sparse-CSC and the Cardinal-SC are identical. We consider the case where the dictionary elements are disjoint (viz., every two shapes in the dictionary are disjoint). The analysis of this problem will provide insight into the more general case of overlapping shapes.

In the case of disjoint dictionary elements (DDE), the non-redundancy of $\mathpzc{R}_{\;\Ip,\In}$ requires $\In=\emptyset$ (see Definition \ref{nonred:def}). Consequently, the composition rule is simplified to $\mathpzc{R}_{\;\Ip}=\bigcup_{j\in\Ip}\mathcal{S}_j$ and the Cardinal-SC problem becomes
\begin{equation}
\label{eq25}
\Ip^*= \operatorname*{arg\,min}_{\Ip} \int_{\mathpzc{R}_{\Ip}} \big(\Pi_{in}(x) - \Pi_{ex}(x)\big)\;\mbox{d}x \qquad \mbox{s.t.}:\quad |\Ip|\leq s.
\end{equation}
Under the DDE conditions, $\Ip^*$ can be uniquely determined by making mild assumptions about the average inhomogeneity measures over the shapes. For this purpose, given a dictionary of shapes $\{\mathcal{S}_j\}_{j=1}^{n_s}$ we define the following quantities for every dictionary element:
\begin{equation*}
P_j \triangleq \int_{\mathcal{S}_j}  \big(\Pi_{in}(x) - \Pi_{ex}(x)\big)^+ \; \mbox{d}x\geq 0, \qquad Q_j \triangleq \int_{\mathcal{S}_j}  \big(\Pi_{ex}(x) - \Pi_{in}(x)\big)^+ \; \mbox{d}x\geq 0.
\end{equation*}
Clearly, when the dictionary elements are disjoint, for any $\Ip \subset \{1,\cdots,n_s\}$
\begin{equation*}
\int_{\mathpzc{R}_{\Ip}} \big(\Pi_{in}(x) - \Pi_{ex}(x)\big)\;\mbox{d}x= \sum\limits_{j\in \Ip} (P_j-Q_j).
\end{equation*}
Based on this observation, we sort the shapes upon their $(P_j-Q_j)$ values. That is, considering an index order $j_1,j_2,\cdots, j_{n_s}$ such that
\begin{equation}\label{eq27a}
P_{j_1}-Q_{j_1}\leq P_{j_2}-Q_{j_2}\leq \cdots\leq P_{j_m}-Q_{j_m}<0\leq P_{j_{m+1}}-Q_{j_{m+1}} \leq \cdots\leq P_{j_{n_s}}-Q_{j_{n_s}}.
\end{equation}
Determining $\Ip^*$ is now straightforward:
\begin{proposition}\label{prop:disjoint}
Given a set of disjoint shapes $\{\mathcal{S}_j\}_{j=1}^{n_s}$ and the image dependent measures $\Pi_{in}(x)$ and $\Pi_{ex}(x)$, consider the general progression (\ref{eq27a}) for the shape indices $1,\cdots, n_s$. If the $s$-th inequality in (\ref{eq27a}) is strict and $s\leq m$, the solution to the Cardinal-SC problem (\ref{eq25}) is unique and
\begin{equation}\label{eq27b}
\Ip^* = \{j_1,j_2,\cdots , j_s\}.
\end{equation}
\end{proposition}
We will skip the proof for Proposition \ref{prop:disjoint} as it is trivial. We, however, note that the condition $s\leq m$ makes the sparsity constraint in (\ref{eq25}) active. For $s>m$ the Cardinal-SC problem (\ref{eq25}) is equivalent to the corresponding SC problem, where the cardinality constraint is disregarded.

\begin{theorem}
Suppose the assumptions of Proposition \ref{prop:disjoint} hold. Then, for $\tau=s$, the minimizer to the Sparse-CSC program (denoted as $\balpha^*$) is unique and $\Ip^* = \{j:\balpha^*_j = 1\}$.
\end{theorem}

\begin{proof}
We use the disjoint property among the shapes (and accordingly the non-overlapping property among their characteristic functions) to simplify the underlying convex cost as follows:
\begin{align*}\nonumber
&\int_{D}\big(\Pi_{in}(x) - \Pi_{ex}(x)\big)^+\max\big(\sum_{j=1}^{n_s} \alpha_j \chi_{\mathcal{S}_j}(x),0\big)\;\mbox{d}x \\ \nonumber
 =&\int_{D}\big(\Pi_{in}(x) - \Pi_{ex}(x)\big)^+\Big(\sum_{j=1}^{n_s} \chi_{\mathcal{S}_j}(x)  \max( \alpha_j ,0)\Big)\;\mbox{d}x\\ \nonumber
 =&\int_{D}\Big(\sum_{j=1}^{n_s}  \big(\Pi_{in}(x) - \Pi_{ex}(x)\big)^+ \chi_{\mathcal{S}_j}(x)  \max( \alpha_j ,0)\Big)\;\mbox{d}x
 \\ \nonumber
 =&\;\sum_{j=1}^{n_s} \max( \alpha_j ,0) \int_{D}  \big(\Pi_{in}(x) - \Pi_{ex}(x)\big)^+ \chi_{\mathcal{S}_j}(x)\; \mbox{d}x
 \\
 =&\;\sum_{j=1}^{n_s} P_j\max( \alpha_j ,0).
\end{align*}
In a similar fashion
\[\int_{D}\big(\Pi_{ex}(x) - \Pi_{in}(x)\big)^+\min\big(\sum_{j=1}^{n_s} \alpha_j \chi_{\mathcal{S}_j}(x),1\big)\;\mbox{d}x = \sum_{j=1}^{n_s} Q_j\min( \alpha_j ,1),
\]
and therefore in the DDE case
\[
 \int_{D}\max\Big(\big(\Pi_{in}(x) - \Pi_{ex}(x)\big)\mathpzc{L}_{\boldsymbol{\alpha}}(x), \big(\Pi_{in}(x) - \Pi_{ex}(x)\big)^-   \Big)\;\mbox{d}x = \sum_{j=1}^{n_s} F_j(\alpha_j)
\]
where
\begin{equation*}
F_j(\alpha)\triangleq P_j\max( \alpha ,0)  - Q_j\min( \alpha ,1).
\end{equation*}
In other words, under the DDE condition the convex objective reduces to the sum of separable costs in terms of $\balpha$ components. As depicted in Figure \ref{fig5}, the underlying costs take their minima at 0 or 1, depending on the sign of $P_j-Q_j$. Moreover,
\begin{equation*}
\min_{\alpha} F_j(\alpha)=\left\{
     \begin{array}{cc}
       P_j-Q_j & 0\leq P_j\leq Q_j\\
       0 & 0\leq Q_j\leq P_j
     \end{array}.
   \right.
\end{equation*}
To minimize $\sum_{j=1}^{n_s} F_j(\alpha_j)$ subject to $\|\balpha\|_1\leq s$, for $s\leq m$, we need to assign unit values to the coefficients $\alpha_{j_1},\alpha_{j_2},\cdots, \alpha_{j_s}$ and set the others to zero, otherwise, we are not minimizing the objective sum to the maximum extent. Comparing this result with (\ref{eq27b}) reveals that $\Ip^* = \{j:\balpha^*_j = 1\}$.
\end{proof}

The preceding results confirm the possibility of using Sparse-CSC as a proxy to the Cardinal-SC problem for the proposed setup. Aside from the DDE condition, the mild conditions stated in Proposition \ref{prop:disjoint}
guarantee the uniqueness of the solution for the Cardinal-SC problem and the convex proxy, and allow us to show their equivalence.

\begin{figure}%
\centering
\subfigure[][]{\includegraphics[width=60mm]{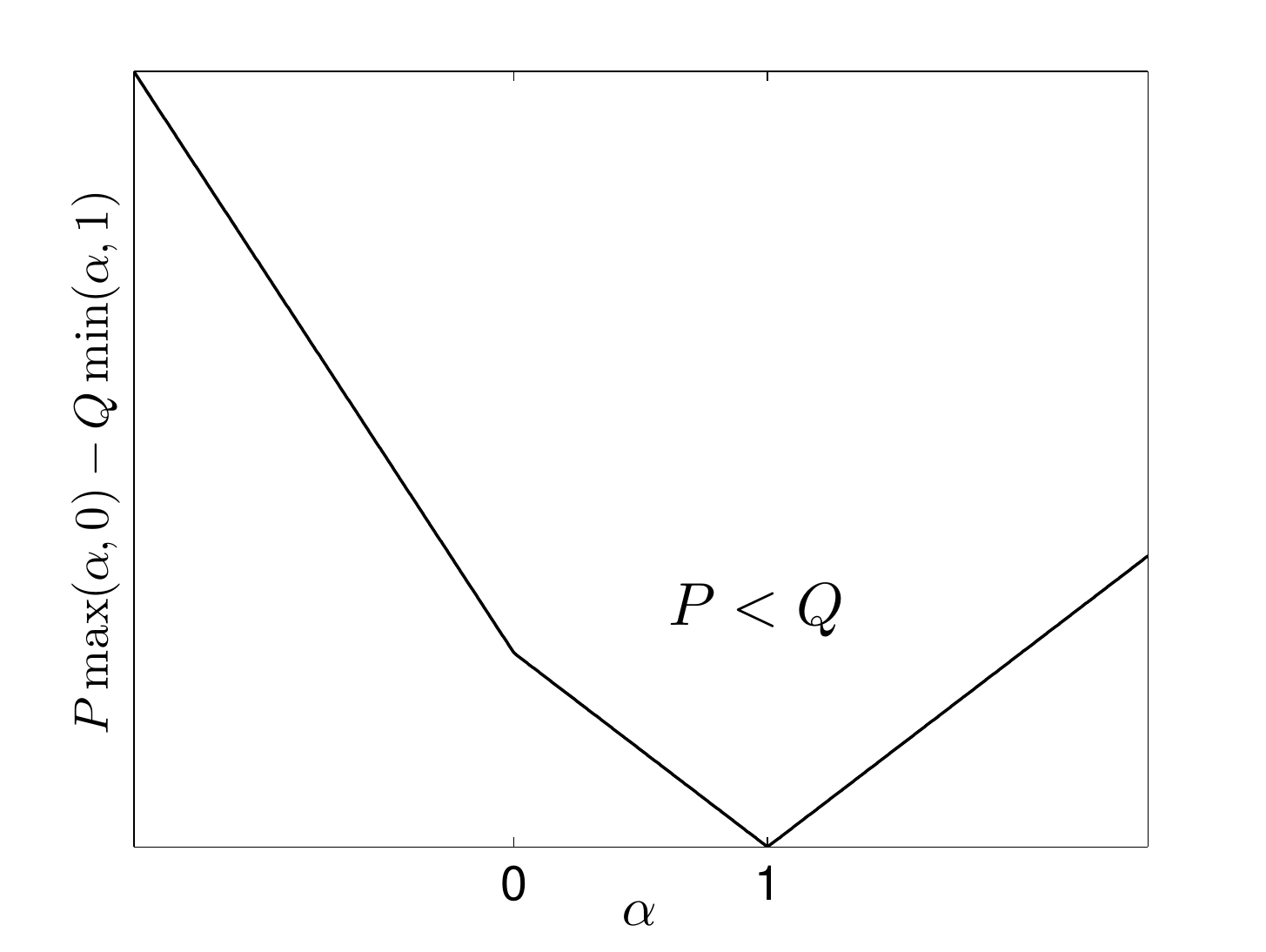}}
\subfigure[][]{\includegraphics[width=60mm]{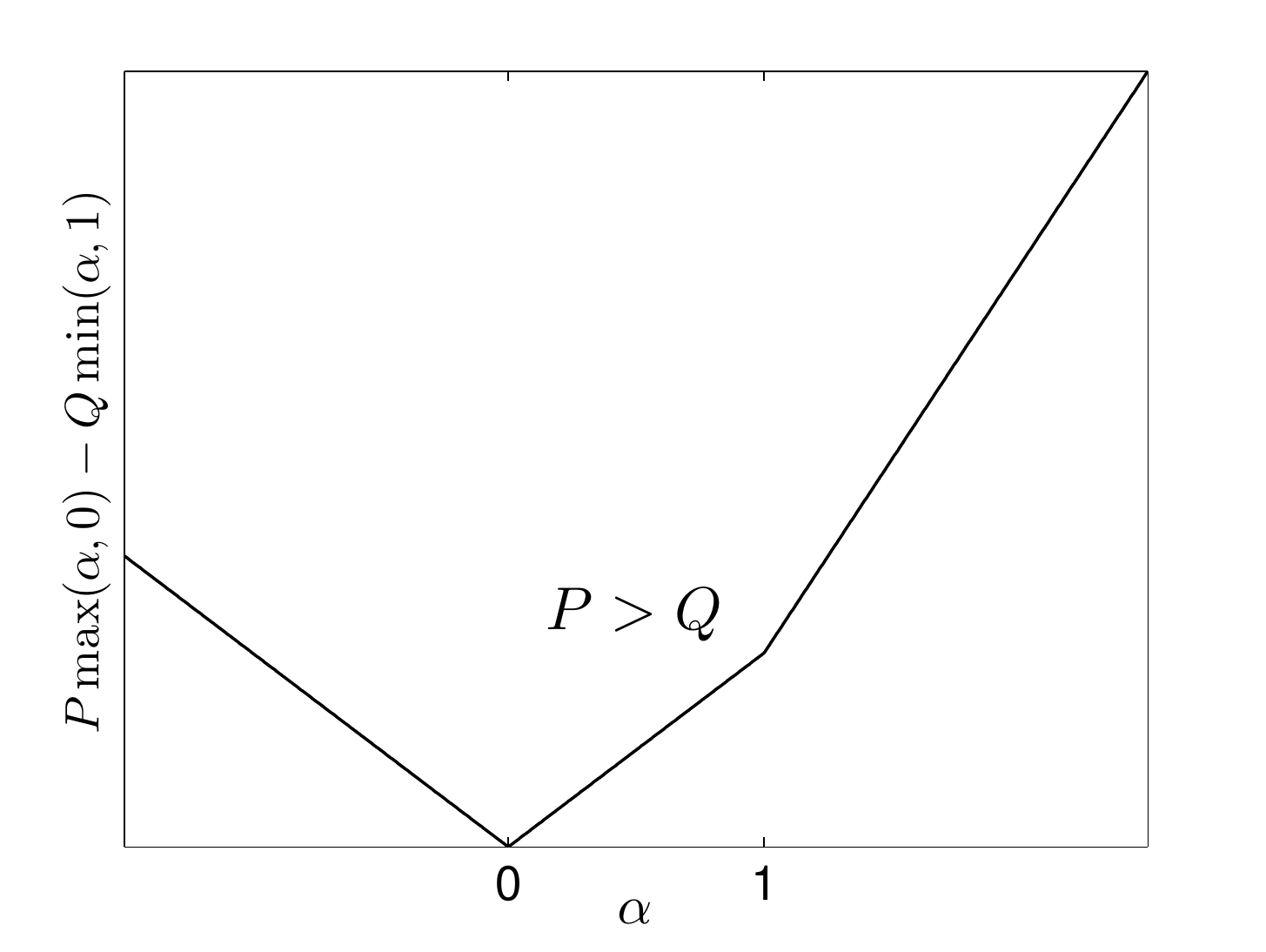}}
\caption{The plot of $F(\alpha) = P\max( \alpha ,0)  - Q\min( \alpha ,1)$, which takes its minimum at 1 when $P<Q$ and takes 0 as the minimizer when $P>Q$.}\label{fig5}
\end{figure}

\subsection{Lucid Object}\label{sec:loc}
In the previous section we considered an extreme case of non-overlapping dictionary elements. In this section we consider another extreme case, which reveals new facts about the similar performance of the SC problem and the convex formulation.

\begin{definition}
For a given region $\Sigma\subset D$, the lucid object condition (LOC) holds if
\[\left\{\begin{array}{lc}
\Pi_{in}(x)<\Pi_{ex}(x) & x\in \Sigma\\
\Pi_{in}(x)>\Pi_{ex}(x) & x\in D\setminus\Sigma
\end{array}.
\right.
\]
\end{definition}
As an example, consider an image where the pixel values are bounded as
\begin{equation*}
\left\{\begin{array}{lc}
0<u(x)<\frac{1}{2} &x\in\Sigma\\
\frac{1}{2}<u(x)<1 &x\in D\setminus \Sigma
\end{array},
\right.
\end{equation*}
and the Chan-Vese inhomogeneity measures $\Pi_{in}(x)=(u(x)-1/4)^2$ and $\Pi_{ex}(x)=(u(x)-3/4)^2$ are considered. It is straightforward to verify that in such setup, the LOC holds for $\Sigma$.

When the LOC holds for $\Sigma$ and there exists a composition of the dictionary elements linked to $\Sigma$, the outcomes of the SC problems and the proposed convex proxy maintain some general properties that will be detailed.
\begin{proposition}\label{th3}
Consider $\Sigma\subset D$ and a dictionary of shapes $\{\mathcal{S}_j\}_{j=1}^{n_s}$. If the LOC holds for $\Sigma$ and there exists a unique non-redundant representation $\mathpzc{R}_{\;\Ip^\Sigma,\In^\Sigma}$ such that $\Sigma = \mbox{cl}(\mathpzc{R}_{\;\Ip^\Sigma,\In^\Sigma})$, then
\begin{equation*}
\{\Ip^\Sigma,\In^\Sigma\} = \operatorname*{arg\,min}_{\Ip, \In} \int_{\mathpzc{R}_{\;\Ip,\In}} \big(\Pi_{in}(x) - \Pi_{ex}(x)\big)\;\mbox{d}x.
\end{equation*}
\end{proposition}

An immediate interpretation of Proposition \ref{th3} is, under the conditions stated, the SC problem identifies the constituting dictionary elements associated with $\Sigma$. This result is readily generalizable to the case that there are multiple non-redundant representation $\mathpzc{R}_{\;\Ip,\In}$ linked to $\Sigma$, and $\mathpzc{R}_{\;\Ip^\Sigma,\In^\Sigma}$ is the unique representation with the fewest number of elements. In this case
\[\{\Ip^\Sigma,\In^\Sigma\} = \operatorname*{arg\,min}_{\Ip, \In} \int_{\mathpzc{R}_{\;\Ip,\In}} \big(\Pi_{in}(x) - \Pi_{ex}(x)\big)\;\mbox{d}x\quad s.t. \quad |\Ip|+|\In|\leq |\Ip^{\Sigma}|+|\In^{\Sigma}|.
\]
We now draw our attention to the Sparse-CSC and its performance under the LOC.
\begin{proposition}\label{prop4}
Consider $\Sigma\subset D$ and a dictionary of shapes $\{\mathcal{S}_j\}_{j=1}^{n_s}$. If the LOC holds for $\Sigma$, then for any $\balpha\in\mathbb{R}^{n_s}$ the convex cost (\ref{eq22xxx}) is lower bounded as
\begin{equation*}
\int_{D}\max\Big(\big(\Pi_{in}(x) - \Pi_{ex}(x)\big)\mathpzc{L}_{\boldsymbol{\alpha}}(x), \big(\Pi_{in}(x) - \Pi_{ex}(x)\big)^-   \Big)\;\mbox{d}x\geq -\int_\Sigma \Pi_{ex}(x)\;\mbox{d}x.
\end{equation*}
\end{proposition}
As a matter of fact, the lower bound in Proposition \ref{prop4} can be attained and the corresponding minimizer is related to $\Sigma$ as follows:
\begin{theorem}\label{th5}
Suppose the LOC holds for $\Sigma$ and there exists at least one composition of the dictionary elements for which $\Sigma = \mbox{cl}(\mathpzc{R}_{\;\Ip,\In})$. Then, $\balpha^*$ minimizes the cost
\begin{equation}\label{eq31}
\int_{D}\max\Big(\big(\Pi_{in}(x) - \Pi_{ex}(x)\big)\mathpzc{L}_{\boldsymbol{\alpha}}(x), \big(\Pi_{in}(x) - \Pi_{ex}(x)\big)^-   \Big)\;\mbox{d}x
\end{equation}
if and only if
\begin{equation}\label{eq32}
\left\{\begin{array}{lc}
\mathpzc{L}_{\boldsymbol{\alpha^*}}(x)\geq 1&x\in\mbox{int}(\Sigma)\\
\mathpzc{L}_{\boldsymbol{\alpha^*}}(x)\leq 0&x\in \mbox{int}(D\setminus \Sigma)\\
\end{array}.
\right.
\end{equation}
\end{theorem}
From Theorem \ref{th5} we can immediately see that $\balpha^*$ is not unique (e.g., if $\balpha^*$ is a minimizer, $k\balpha^*$ is also a minimizer for any $k>1$). However, for any minimizer $\balpha^*$ we have
\[\supp^+(\mathpzc{L}_{\balpha^*}(x)) = \Sigma,
\]
which means that by minimizing the convex cost (\ref{eq22xxx}) we can identify the underlying object $\Sigma$.

The identification of $\Sigma$ is the common outcome of the SC problem and the proposed convex proxy. However, the link that Proposition \ref{th1} establishes between an arbitrary composition and a vector satisfying (\ref{eq32}), makes us hopeful about identifying the constituting elements of $\Sigma$ by inspecting the entries of $\balpha^*$. In this context,
an $\ell_1$ restriction on $\balpha$ is capable of making $\balpha^*$ unique and more conveniently linked to a specific composition. Moreover, such restriction allows us to control the number of shape elements appearing in the representation of $\Sigma$.

The remainder of the paper is devoted to the analysis of the Sparse-CSC under more general conditions. The main theme is the possibility of using the proposed convex proxy as a tool to extract the constituting elements of a composition along with the segmented image. While a general framework is considered throughout the analysis, we will later revisit the LOC as a particular problem to which the proposed tools can be readily applied.

\section{Extended Tools and Convex Analysis}\label{sec4}

Both the information embedded within the image, as well as the overlapping of the dictionary elements play key roles in characterizing the minimizers of the Sparse-CSC program. In this section we will have a deeper study of the convex problem, and provide some analysis tools to characterize the possible outcomes of the problem.

In Section \ref{DSDsec} we introduce a process that extracts the disjoint components of overlapping shapes, as a simplifying tool for the analysis of the Sparse-CSC problem. In Section \ref{Sec:Rel} we establish a bijective relationship between the representation of a composition in the shape domain and the $\balpha$-domain. To construct a bijection, we relate the two representations via a linear program and discuss the uniqueness conditions for the proposed problem. In Section \ref{sec-conv} we derive sufficient conditions under which a target vector $\balpha$ minimizes the Sparse-CSC problem. The derivation is performed in a general setup, where there are no restrictive assumptions about the overlapping of the dictionary elements and the spatial variability of the inhomogeneity measures. The developed tools are employed in Section \ref{sec:accrecovery} to derive sufficient conditions under which the Sparse-CSC program extracts the constituting elements of a target composition.

\subsection{Disjoint Shape Decomposition} \label{DSDsec}
As detailed in Section \ref{disjoint_sec}, for a dictionary with disjoint shapes it is quite straightforward to infer the solution to the Cardinal-SC problem and relate it to the solution of the corresponding convex proxy. To pave the analysis path for the case of overlapping shapes, in this section we propose a procedure to decompose overlapping shapes into non-overlapping ones and exploit that to approach the general shape composition problem.

Given $n$ overlapping shapes $\mathcal{S}_1,\mathcal{S}_2,\cdots,\mathcal{S}_n$, consider an $n$-dimensional vector $\boldsymbol{J}\in\{0,1\}^n\setminus \{0\}^n$. A set $\Omega$ with nonempty interior is called a \emph{shapelet} if
\begin{equation}\label{eq10}
\Omega = \bigcap_{j=1}^n\Theta_{\boldsymbol{J}_j}(\mathcal{S}_j),
\end{equation}
where $\boldsymbol{J}_j$ denotes the $j$-th element of $\boldsymbol{J}$, and
\begin{equation}\label{eq10xyz}
\Theta_m(\mathcal{S}) =
\left\{
\begin{array}{lr}
       \mathcal{S}&m=1\\
       \mbox{cl}(\mathcal{S}^c)&m=0
     \end{array}.
   \right.
\end{equation}
Here, $\mbox{cl}(.)$ denotes the closure of the set. The appearance of $\mbox{cl}(.)$ guarantees that the resulting shapelets maintain the properties of a shape. We will call $\boldsymbol{J}$ the constructor vector associated with $\Omega$. As some of the intersections in the form of (\ref{eq10}) might be null sets, the number of shapelets, denoted as $n_\Omega$, is at most $2^n-1$.

We index the shapelets as $\Omega_i$ and the corresponding constructor vectors as $\boldsymbol{J}^{(i)}$ for $i=1,2,\cdots,n_\Omega$. Furthermore, for every shape $\mathcal{S}_j$ we define the index set
\begin{equation}\label{eq11}
\mathcal{I}_j \triangleq \{i:\boldsymbol{J}^{(i)}_j =1\},
\end{equation}
which corresponds to the shapelets inside $\mathcal{S}_j$. We henceforth use the terminology \emph{disjoint shape decomposition} (DSD) for the process of generating the $\Omega$ shapelets from a collection of given shapes $\mathcal{S}_1, \mathcal{S}_2, \cdots, \mathcal{S}_n$. We succinctly write
\begin{equation}\label{eq12}
\{\Omega_i,\boldsymbol{J}^{(i)}\}_{i=1}^{n_\Omega} = \mbox{DSD}\Big( \{\mathcal{S}_j\}_{j=1}^{n}\Big).
\end{equation}

Furthermore, we may form a binary matrix $\boldsymbol{B}\in \{0,1\}^{n_\Omega\times n }$ by stacking up the transposed constructor vectors $\boldsymbol{J}^{(i)}$ from the DSD process (\ref{eq12}); more precisely
\begin{equation*}
\boldsymbol{B}_{i,j}=\boldsymbol{J}_j^{(i)}.
\end{equation*}
When $\boldsymbol{B}$ is resulted from the DSD process over a specific set of shapes, we call it the corresponding \emph{bearing matrix}, and when its construction involves all the elements of the dictionary,  we refer to it as the \emph{dictionary bearing matrix}.

\begin{figure}[t]
    \centering
    \includegraphics[width=55mm]{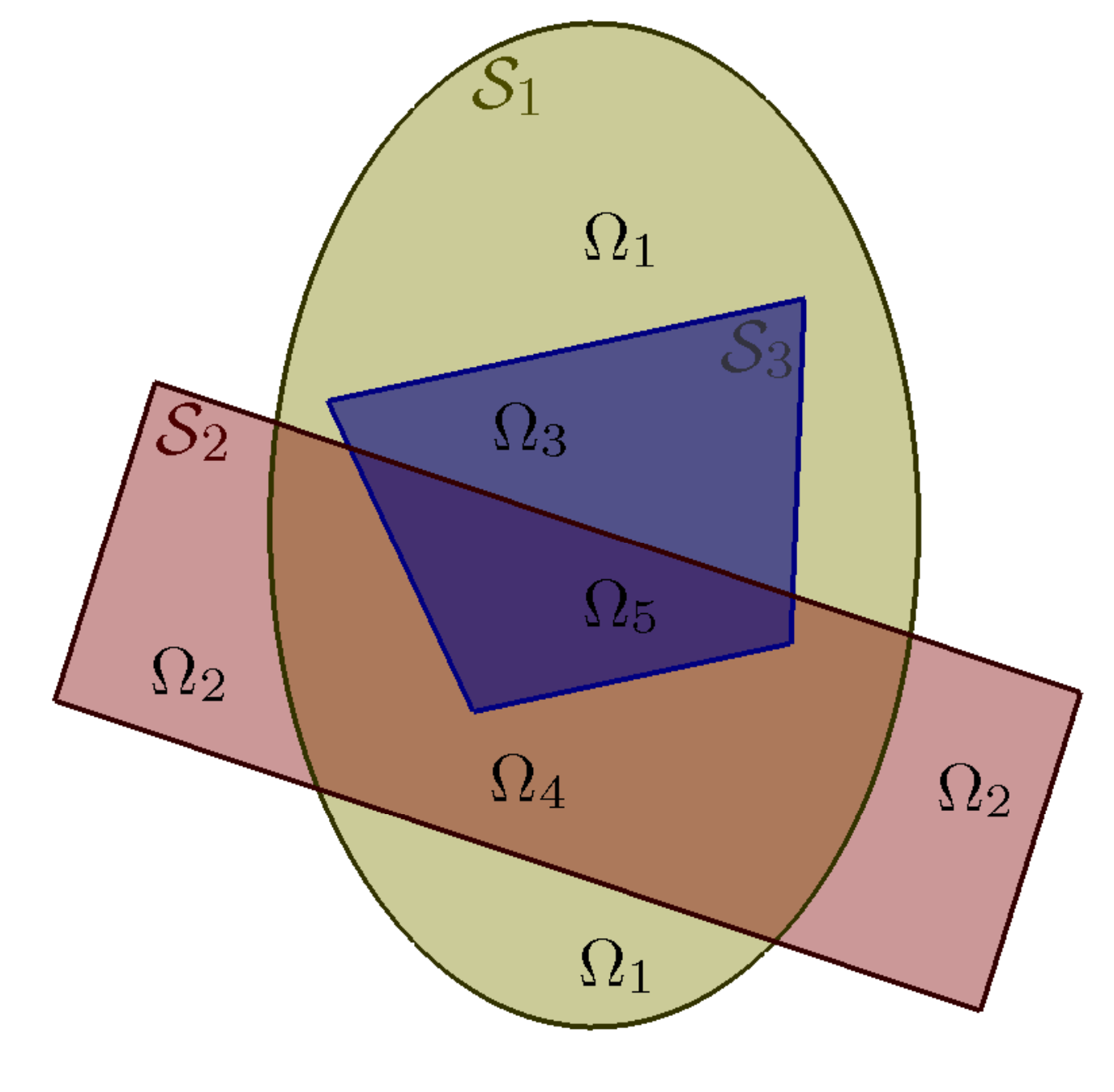}\quad
    \begin{tabular}[b]{c}\vspace{1cm}
    \begin{tabular}[b]{c|c|c}
      $i$ & $\boldsymbol{J}^{(i)}$ & $\Omega_i = \bigcap_{j=1}^3 \Theta_{\boldsymbol{J}_j^{(i)}}(\mathcal{S}_j)$\\ \hline
      \vspace{-.3cm}& & \\[.02cm]
      $1$ & $[1,0,0]^T$ & $\mathcal{S}_1 \cap \bar{\mathcal{S}_2^c} \cap \bar{\mathcal{S}_3^c}$\\
      $2$ & $[0,1,0]^T$ & $\bar{\mathcal{S}_1^c} \cap \mathcal{S}_2 \cap \bar{\mathcal{S}_3^c}$\\
      $3$ & $[1,0,1]^T$ & $\mathcal{S}_1 \cap \bar{\mathcal{S}_2^c} \cap \mathcal{S}_3$\\
      $4$ & $[1,1,0]^T$ & $\mathcal{S}_1 \cap \mathcal{S}_2 \cap \bar{\mathcal{S}_3^c}$\\
      $5$ & $[1,1,1]^T$ & \hspace{.01 cm} $\mathcal{S}_1 \cap \mathcal{S}_2 \cap \mathcal{S}_3$
    \end{tabular}\\[-.7cm]
     \vspace{.5cm}$\boldsymbol{B}=\begin{pmatrix} 1 & 0 & 0\\ 0 & 1 & 0\\ 1 & 0 & 1\\ 1 & 1 & 0\\ 1 & 1 & 1 \end{pmatrix}$
    \end{tabular}
    \captionlistentry[table]{A table beside a figure}
    \caption{An example of disjoint shape decomposition, where the shapes $\mathcal{S}_1$, $\mathcal{S}_2$ and $\mathcal{S}_3$ are decomposed into the shapelets $\Omega_1,\cdots,\Omega_5$. The constructor vectors are listed in the table and the corresponding bearing matrix is presented as $\boldsymbol{B}$. To highlight the pattern we succinctly used the notation $\bar{O}$ to represent $\mbox{cl}(O)$. For this example $\mathcal{I}_1=\{1,3,4,5\}$, $\mathcal{I}_2=\{2,4,5\}$ and $\mathcal{I}_3=\{3,5\}$}\label{fig2}
  \end{figure}

\begin{proposition}\label{pr2}
Given $n$ (overlapping) shapes $\mathcal{S}_1, \mathcal{S}_2,\cdots,\mathcal{S}_n$, for the corresponding shapelets constructed through the DSD process the following properties hold:
\begin{itemize}
\item[(a)] $\Omega_{i_1}\cap\Omega_{i_2} \seq \emptyset,\qquad\qquad$ (for $i_1,i_2\in\{1,2,\cdots,n_\Omega\}, \quad i_1\neq i_2$)
\item[(b)]$\bigcup_{i=1}^{n_\Omega} \Omega_i = \bigcup_{j=1}^{n}\mathcal{S}_j$,
\item[(c)]$\mathcal{S}_j = \bigcup_{i\in \mathcal{I}_j} \Omega_i$.
\end{itemize}
\end{proposition}

Putting this into words, the DSD process performs a partitioning on each shape (ignoring the measure-zero boundaries). In the simplest case of two overlapping shapes $\mathcal{S}_1$ and $\mathcal{S}_2$, the three possible shapelets are $\Omega_1=\mbox{cl}(\mathcal{S}_1^c)\cap\mathcal{S}_2$, $\Omega_2=\mathcal{S}_1\cap\mathcal{S}_2$ and $\Omega_3=\mathcal{S}_1\cap\mbox{cl}(\mathcal{S}_2^c)$ for which one can easily verify the properties (a)-(c). Figure \ref{fig2} provides another example for $n=3$ where $n_\Omega=5<2^3-1$.

\begin{theorem}\label{th2x}
For a given set of shapes $\{\mathcal{S}_j\}_{j\in \Ip\cup \In}$, consider the shapelets resulted by applying the DSD process. If the composition $\mathpzc{R}_{\;\Ip,\In} = ({\bigcup}_{j\in \Ip}\mathcal{S}_j) \backslash ({\bigcup}_{j\in \In}\mathcal{S}_j)$ is non-redundant, then
\begin{itemize}
\item[(a)] for each $j\in \Ip$, there exists a (unique) shapelet $\Omega \subset \mathcal{S}_j$ such that \vspace{-.1cm}
\[\forall j'\in (\Ip\cup \In) \; \mbox{and}\; j'\neq j : \quad \Omega \cap \mathcal{S}_{j'}\seq \emptyset;
\]
\item[(b)]\vspace{-.1cm} for each $j\in \In$, there exist some shapelets $\Omega \subset \mathcal{S}_j \cap (\underset{\hat j\in \Ip}{\bigcup}\mathcal{S}_{\hat j})$ such that \vspace{-.3cm}
\[\forall j'\in \In \; \mbox{and}\; j'\neq j : \quad \Omega \cap \mathcal{S}_{j'}\seq \emptyset.
\]
\end{itemize}
\end{theorem}
As an illustrative example, Figure \ref{fig3} shows the resulting shapelets of the DSD process on four shapes $\mathcal{S}_1,\cdots,\mathcal{S}_4$. The composition $\mathpzc{R}_{\;\{1,2\},\{3,4\}} $ is a non-redundant representation. It can be readily observed that part (a) of Theorem \ref{th2x} applies to $\Omega_1\subset \mathcal{S}_1$, and $\Omega_3\subset \mathcal{S}_2$. Likewise, part (b) applies to $\Omega_5\subset \mathcal{S}_3$, and to the multiple shapelets $\Omega_{8}$, $\Omega_{9}$ and $\Omega_{11}$ which are in $\mathcal{S}_4$.

\subsection{A Bijective Relationship Between the Shape Composition and the Combination of Characteristics} \label{Sec:Rel} Proposition \ref{th1} establishes a relationship between $\mathpzc{R}_{\;\Ip,\In}$ and $\mathpzc{L}_{\boldsymbol{\alpha}}(x)$. To avoid an exhaustive search associated with the shape composition problem, our proposed strategy is determining the $\alpha_j$ coefficients as a proxy to selecting the index sets $\Ip$ and $\In$. For this purpose we need to provide a procedure that relates the two representations. More specifically:
\begin{itemize}
\item[(1)] Given $\boldsymbol{\alpha}$ and knowing that $\mathpzc{L}_{\boldsymbol{\alpha}}(x)$ corresponds to a non-redundant composition $\mathpzc{R}_{\;\Ip,\In}$, how can we determine $\Ip$ and $\In$?
\item[(2)] Given $\{\mathcal{S}_j\}_{j=1}^{n_s}$ and the index sets $\Ip$ and $\In$, how can we determine an ${\boldsymbol{\alpha}}$ vector that is a representative for $\mathpzc{R}_{\;\Ip,\In}$?
\end{itemize}
To address the first question, for a given $\boldsymbol{\alpha}$ we define \begin{equation}\label{eq13}
\left\{
     \begin{array}{l}
       \xi^+(\balpha) \triangleq \big\{j: j\in\{1,2,\cdots, n_s\}, \alpha_j>0\big\}\\
       \xi^-(\balpha) \triangleq \big\{j: j\in\{1,2,\cdots ,n_s\}, \alpha_j<0\big\}
     \end{array}.
     \right.
\end{equation}
Knowing that $\mathpzc{L}_{\boldsymbol{\alpha}}(x)$ corresponds to an $\mathpzc{R}_{\;\Ip,\In}$, the natural choices for $\Ip$ and $\In$ are $\xi^+(\balpha)$ and $\xi^-(\balpha)$, respectively.

To address the second question, from Proposition \ref{th1} we are always certain about the existence of an $\balpha$ vector such that
\begin{equation}\label{eq14}
\supp^+\big(\mathpzc{L}_{\boldsymbol{\alpha}}(x) \big) = \mathpzc{R}_{\;\Ip,\In}.
\end{equation}
In fact, as inferred from the proof of Proposition \ref{th1}, there is always an infinite number of possible $\balpha$ vectors that satisfy (\ref{eq14}). As a requirement for our subsequent analysis, we will introduce a procedure that limits the corresponding $\balpha$ vectors to a certain set, potentially a singleton (a single element set). For this purpose, we will pose the process of determining the representative $\balpha$ as a linear program and will discuss the uniqueness conditions for the proposed problem. This process is closely linked to the proposed convex framework and specifically designed to facilitate the analysis.

Given $\{\mathcal{S}_j\}_{j\in \Ip \cup \In}$, we perform a DSD to generate a set of disjoint shapelets, i.e.,
\begin{equation}\label{eq15}
\{\Omega_i,\boldsymbol{J}^{(i)}\}_{i=1}^{n_\Omega} = \mbox{DSD}\Big( \{\mathcal{S}_j\}_{j\in \Ip \cup \In}\Big).
\end{equation}
As the $\Omega_i$ shapelets are disjoint, and $\mathcal{S}_j = \bigcup_{i\in \mathcal{I}_j} \Omega_i$  for each $j\in \Ip \cup \In$, we can write
\begin{align*}
\mathpzc{R}_{\;\Ip,\In} = \big(\bigcup_{j\in \Ip}\mathcal{S}_j\big) \big\backslash \big(\!\bigcup_{j\in \In}\mathcal{S}_j\big) = \big(   \bigcup_{i\in\!\!\! \underset{j\in\Ip}{\bigcup}\!\!\! \mathcal{I}_j}\!\!\Omega_i\big) \big\backslash \big(   \bigcup_{i\in\!\!\! \underset{j\in\In}{\bigcup}\!\!\! \mathcal{I}_j}\!\!\Omega_i\big)
=   \bigcup_{i\in (\!\!\underset{j\in\Ip}{\bigcup}\!\!\! \mathcal{I}_j)\setminus  (\!\!\underset{j\in\In}{\bigcup}\!\!\! \mathcal{I}_j) } \!\!\!\!\Omega_i.
\end{align*}
To simplify the future formulations, we will make use of the following notations:
\begin{equation*}
\Np \triangleq \bigcup_{j\in \Ip}\mathcal{I}_j \qquad \mbox{and} \qquad \Nn \triangleq \bigcup_{j\in \In}\mathcal{I}_j.
\end{equation*}
The sets $\Np$ and $\Nn$ contain the indices of the shapelets that are in $\bigcup_{j\in \Ip}\mathcal{S}_j$ and $\bigcup_{j\in \In}\mathcal{S}_j$, respectively. Clearly, the two sets are not necessarily disjoint.

In constructing an $\balpha$ vector, we need to make sure that for every $i\in \Np\setminus \Nn$ we have
\begin{equation}\label{eq16}
\mathpzc{L}_{\boldsymbol{\alpha}}(x)>0, \quad x\in \Omega_i
\end{equation}
and for every $i\in\Nn$:
\begin{equation}\label{eq17}
\mathpzc{L}_{\boldsymbol{\alpha}}(x)\leq 0, \quad x\in \Omega_i.
\end{equation}
Conditions (\ref{eq16}) and (\ref{eq17}) basically assure that (\ref{eq14}) is satisfied. We proceed by defining
\begin{equation}\label{eq18}
\left\{
\begin{array}{l}
       \Kp^{(i)} \triangleq \{j:j\in\Ip, \boldsymbol{J}_j^{(i)}=1\}\\[.2cm]
       \Kn^{(i)} \triangleq \{j:j\in\In, \boldsymbol{J}_j^{(i)}=1\}
     \end{array}.
   \right.
\end{equation}
The set $\Kp^{(i)}$ (correspondingly $\Kn^{(i)}$) contains the index of the shapes $\{\mathcal{S}_j\}_{j\in \Ip}$ (correspondingly $\{\mathcal{S}_j\}_{j\in \In}$) that overlap with $\Omega_i$.

\begin{figure}[t]
    \centering
    \centering \includegraphics[width=52mm]{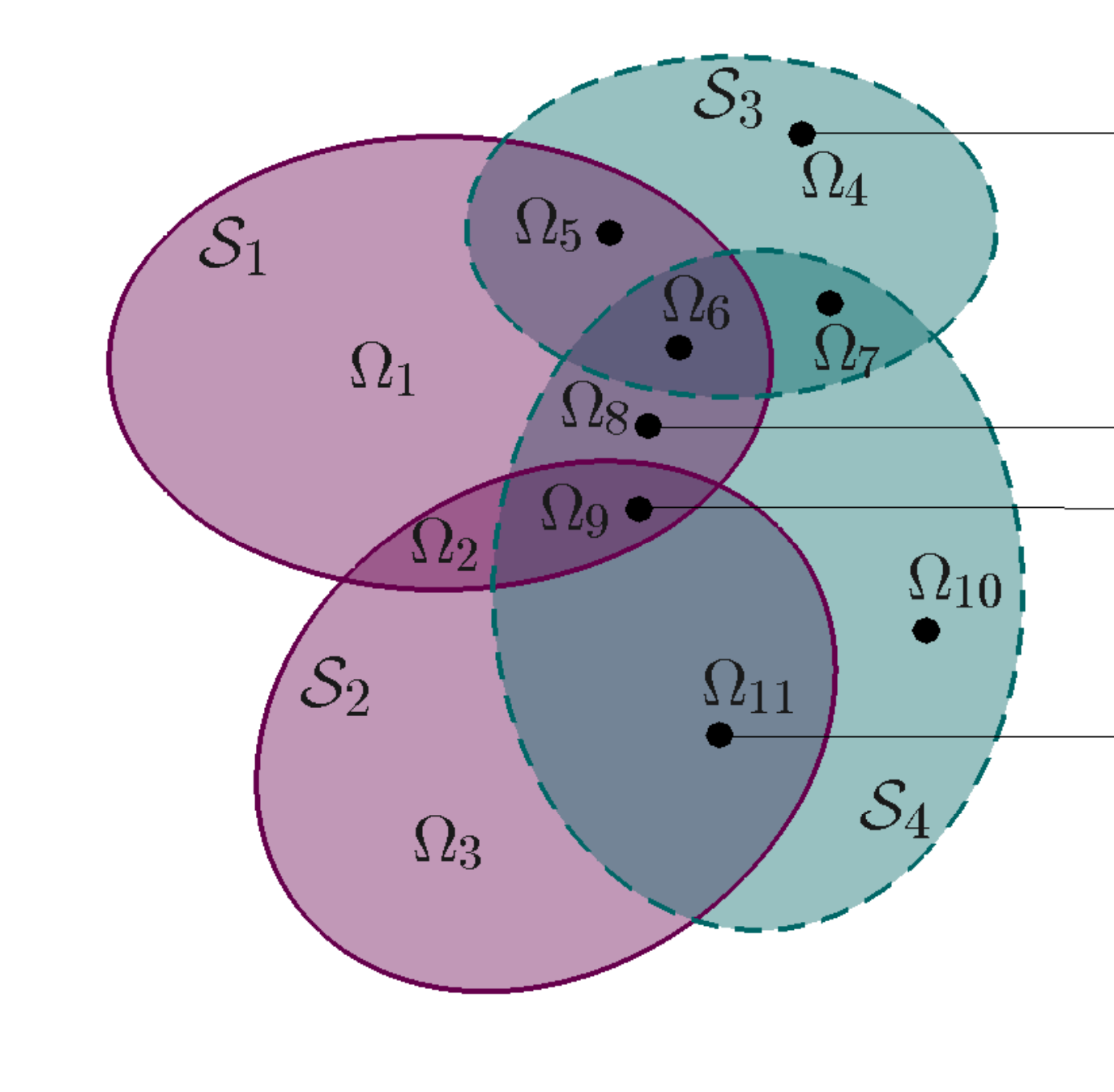}\setlength{\tabcolsep}{2pt}
    \begin{tabular}[b]{ccc}
      $\boldsymbol{J}^{(4)}=[0,0,1,0]^T$& \quad$\therefore$&\quad $\alpha_3\leq 0$\vspace{.25cm}\\ \vspace{.25cm}\vdots & &\vdots \\
       $\boldsymbol{J}^{(8)}=[1,0,0,1]^T$& \quad$\therefore$&\quad $\alpha_4\leq -1$\\
       $\boldsymbol{J}^{(9)}=[1,1,0,1]^T$& \quad$\therefore$& \quad  $\alpha_4\leq -2$\vspace{.15cm}\\\vspace{.15cm}\vdots & &\vdots \\
       $\boldsymbol{J}^{(11)}=[0,1,0,1]^T$& \quad$\therefore$& \quad $\alpha_4\leq -1$\vspace{.15cm}\\\vspace{.75cm}
    \end{tabular}
    \captionlistentry[table]{A table beside a figure}
    \caption{Finding the $\balpha$ vector corresponding to $(\mathcal{S}_1\cup\mathcal{S}_2) \setminus (\mathcal{S}_3\cup \mathcal{S}_4)$: based on the DSD applied, we have $\Np = \mathcal{I}_1\cup \mathcal{I}_2 = \{1,2,3,5,6,8,9,11\}$ and $\Nn=\mathcal{I}_3\cup \mathcal{I}_4 = \{4,5,6,7,8,9,10,11\}$. After setting $\alpha_1$ and $\alpha_2$ to 1, we need to choose $\alpha_3$ and $\alpha_4$ in a way that in all shapelets marked with a black circle (indexed by $\Nn$) the value of $\mathpzc{L}_{\balpha}(x)$ is non-positive. This requirement gives rise to eight linear inequalities in terms of $\alpha_3$ and $\alpha_4$. The linear constraints can be deduced by referring to the constructor vector of each shapelet. For instance, for $\Omega_{11}$ we have $\Kp^{(11)}=\{2\}$ and $\Kn^{(11)}=\{4\}$ and based on (\ref{eq20}) the corresponding inequality is $\alpha_4\leq -1$.}\label{fig3}
  \end{figure}

If for $j\in\Ip$ we set $\alpha_j=1$, then (\ref{eq16}) is automatically satisfied. To assure that (\ref{eq17}) holds, we need to impose
\begin{equation}\label{eq19}
\forall i\in \Nn \qquad:\qquad\sum_{j\in \Kp^{(i)}} \alpha_j + \sum_{j\in \Kn^{(i)}} \alpha_j\leq 0.
\end{equation}
Since $\alpha_j=1$ for $j\in\Ip$, (\ref{eq19}) is equivalent to
\begin{equation}\label{eq20}
\forall i\in \Nn \qquad:\qquad \sum_{j\in \Kn^{(i)}} \alpha_j\leq -| \Kp^{(i)}|,
\end{equation}
where $|\;.\;|$ represents the set cardinality. The linear inequality constraints (\ref{eq20}) define a convex set for the permissible $\alpha_j$ coefficients when $j\in\In$. These constraints are met when the $\alpha_j$ values are sufficiently negative for $j\in\In$. To limit the number of possibilities we consider solving the linear program
\begin{equation}\label{eq21}
\left\{
\begin{array}{lll}
       \max & \underset{j\in \In}{\sum}\alpha_j &\\
       s.t.&\underset{j\in \Kn^{(i)}}{\sum} \alpha_j\leq -| \Kp^{(i)}|&\quad\quad \forall i\in \Nn
     \end{array},
   \right.
\end{equation}
which is capable of having a unique solution.

Summarily, given $\{\mathcal{S}_j\}_{j\in \Ip\cup\In}$, $\Ip$ and $\In$, determining the corresponding $\balpha$ is performed through the following \emph{linkage process}:
\emph{
\begin{itemize}
\item[(I)] for $j\in\Ip$, set $\alpha_j=1$;
\item[(II)] for $j\in\In$, the $\alpha_j$ values are selected to be a maximizer to the linear program (\ref{eq21}).
\end{itemize}}

\noindent We again make a reference to Figure \ref{fig3} as an illustrative example.

Corresponding to any vector $\balpha^\dagger$ obtained through the linkage process, we consider a vector $\Beta^\dagger\in \mathbb{R}^{n_\Omega}$ as
\begin{equation*}
\beta_i^\dagger  = \sum_{j\in\Kp^{(i)}\cup\Kn^{(i)}}\alpha_j^\dagger, \quad i=1,2,\cdots, n_\Omega.
\end{equation*}
The quantities $\beta_i^\dagger$ are simply the value of $\mathcal{L}_{\balpha^\dagger}(x)$ over $\mbox{int}(\Omega_i)$ (Figure \ref{fig6}). Denoting the bearing matrix associated with (\ref{eq15}) as
\[ \boldsymbol{B}^\mathpzc{R} = \big[\boldsymbol{J}^{(i)}_j\big],
\]
the shape and shapelet coefficients may simply be related via
\begin{equation}\label{eq21x}
\Beta^\dagger = \boldsymbol{B}^\mathpzc{R} \balpha^\dagger.
\end{equation}
The following result states that when $\mathpzc{R}_{\;\Ip,\In}$ is non-redundant, regardless of the maximizer's uniqueness in (\ref{eq21}),  there are some general properties that hold for $\balpha^\dagger$ and $\Beta^\dagger$.

\begin{figure}[t]
    \centering\vspace{-1cm}
    \centering \includegraphics[width=68mm]{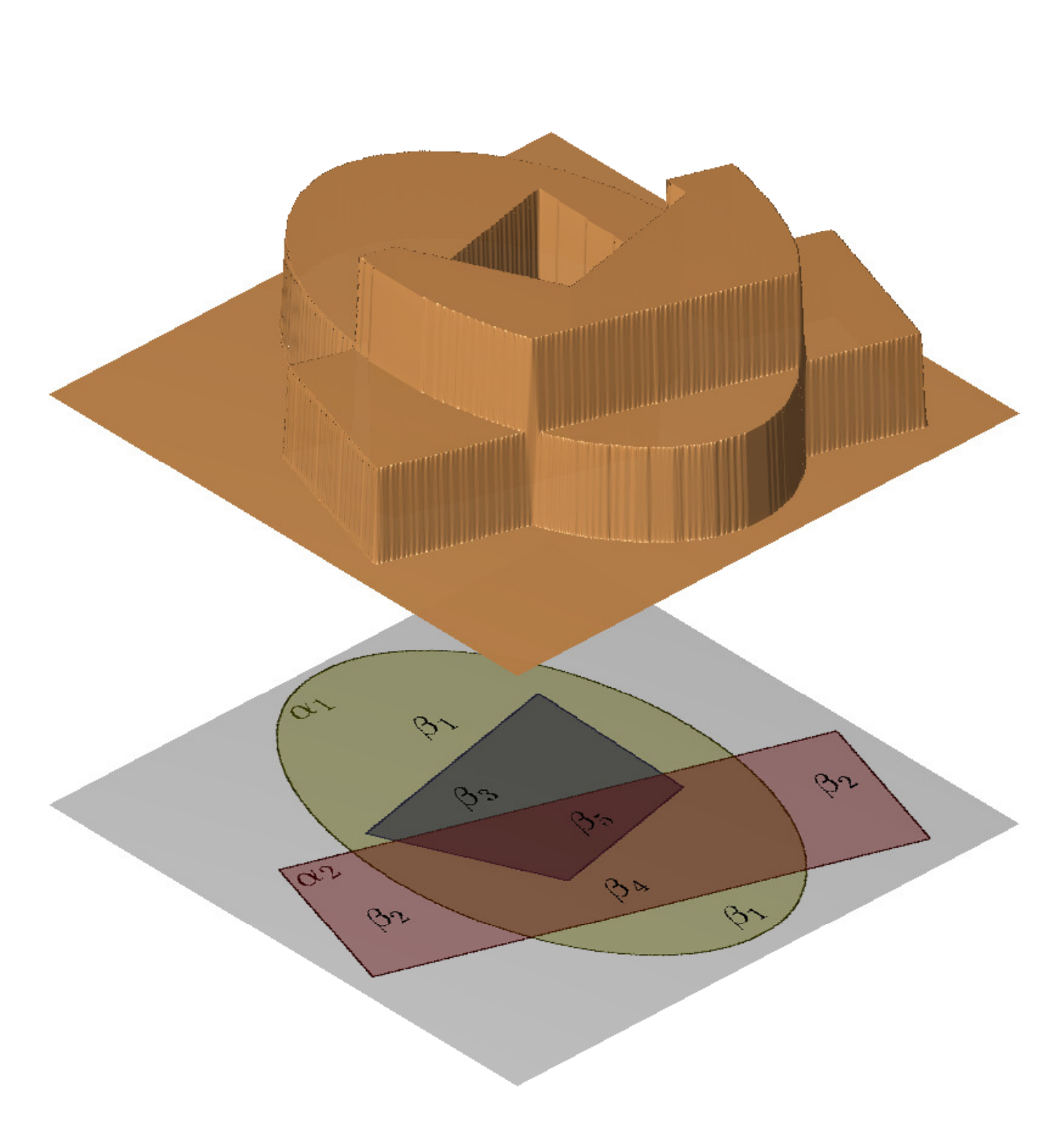}\qquad\quad
    \begin{tabular}[b]{c}
      $\boldsymbol{\beta}^\dagger  = \left( \begin{array}{ccc}
1 & 0 & 0 \\
0 & 1 & 0 \\
1 & 0 & 1 \\
1 & 1 & 0 \\
1 & 1 & 1 \end{array} \right) \begin{pmatrix} 1\\ 1\\ -2\end{pmatrix}$\\ \\
       \\ \\ \\
    \end{tabular}
    \captionlistentry[table]{A table beside a figure}
    \caption{The function $\mathpzc{L}_{\balpha}(x)$ as a representative of the non-redundant composition $\mathpzc{R}_{\{1,2\},\{3\}}$, for the shapes presented in Figure \ref{fig2}. Stepping through the linkage process for this composition simply yields $\balpha^\dagger = (1,1,-2)^T$ and accordingly $\Beta^\dagger = (1,1,-1,2,0)^T$ }\label{fig6}
  \end{figure}

\begin{proposition}\label{th2xx}
Given a set of shapes $\{\mathcal{S}_j\}_{j\in\Ip\cup\In}$ and the corresponding index sets $\Ip$ and $\In$, suppose the composition $\mathpzc{R}_{\;\Ip,\In} = ({\bigcup}_{j\in \Ip}\mathcal{S}_j) \backslash ({\bigcup}_{j\in \In}\mathcal{S}_j)$ is non-redundant, $\balpha^\dagger$ is a vector obtained through the aforementioned linkage process and $\Beta^\dagger  =\boldsymbol{B}^\mathpzc{R} \balpha^\dagger$ is the corresponding shapelet coefficient vector. The following properties hold:
\begin{itemize}
\item[(a)] for each $j\in\In$, $\alpha_j^\dagger\leq -1$;
\item[(b)] for $j\in\Ip$, there exists a unique $i\in\mathcal{I}_j$ such that $\beta_i^\dagger = 1$ (unit-valued shapelet);
\item[(c)] for $j\in\In$, there exist $i\in\mathcal{I}_j$ such that $\beta_i^\dagger = 0$ (null-valued shapelets);
\item[(d)] the mapping (\ref{eq21x}) relating $\balpha^\dagger$ to $\Beta^\dagger$ is injective.
\end{itemize}
\end{proposition}

When the maximizer of (\ref{eq21}) is unique, the linkage process allows relating a given composition $\mathpzc{R}_{\;\Ip,\In}$ to a well-defined $\balpha_\mathpzc{R}$ vector. In this case we may simply represent the relationship by
\begin{equation*}
\balpha_\mathpzc{R} \triangleq \mathcal{A}\big(\{\mathcal{S}_j\}_{j\in\Ip\cup\In};\Ip,\In\big).
\end{equation*}
By construction, if $\balpha_\mathpzc{R} = \mathcal{A}\big(\{\mathcal{S}_j\}_{j\in\Ip\cup\In};\Ip,\In\big)$ then $\xi^+(\balpha_\mathpzc{R})=\Ip$ and $\xi^-(\balpha_\mathpzc{R})=\In$.

In general there is a possibility that stepping through
the linkage process does not yield a well defined $\balpha_\mathpzc{R}$, due to the non-uniqueness issue. An example of this case is shown in Figure \ref{fig4}(a) to construct which we have performed a careful shape selection and alignment.

\begin{figure}[]
\centering\subfigure[][]{\includegraphics[width=60mm]{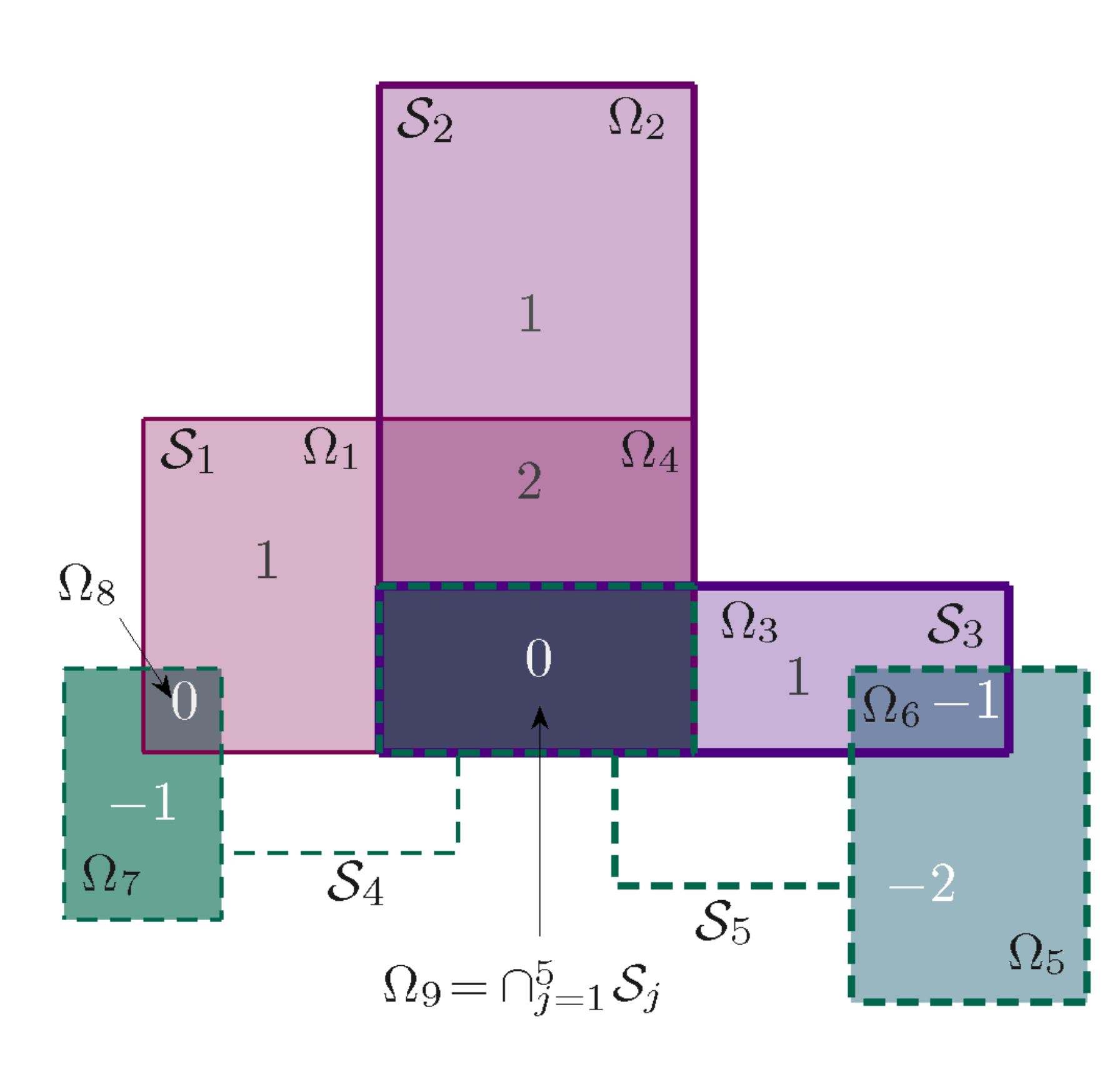}}\hspace{1cm}
\subfigure[][]{\begin{overpic}[width=0.4\textwidth,tics=10]{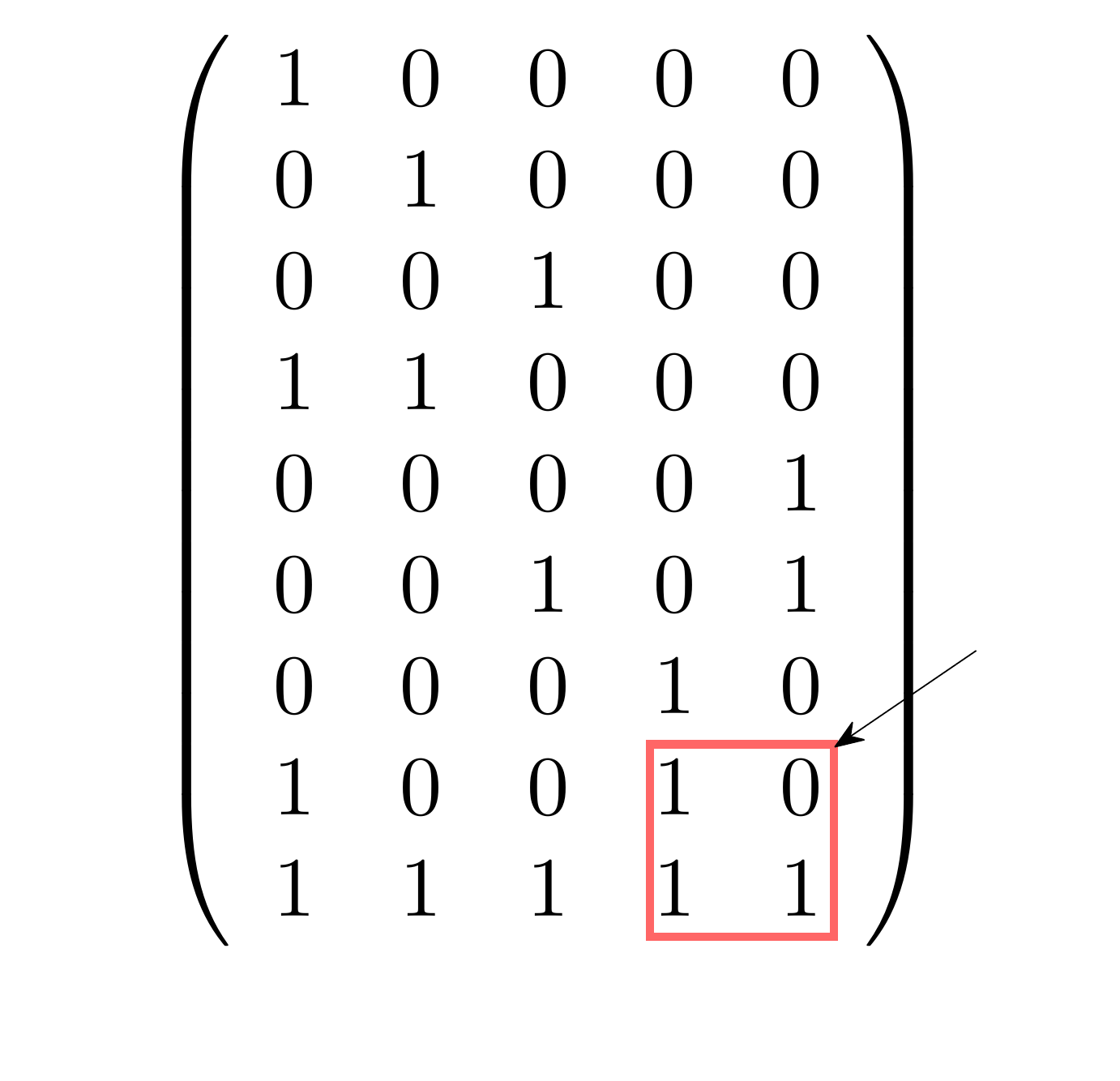}
 \put (90,35) {\large$\boldsymbol{\Delta}_\mathpzc{R}$} \put (-10,50) {\large$\boldsymbol{B}^\mathpzc{R} = $} \put (8,88) {\small$\Omega_1 $} \put (8,79) {\small$\Omega_2 $} \put (10,69) {\small$\vdots $}  \put (10,34) {\small$\vdots $}    \put (8,24) {\small$\Omega_8 $} \put (8,15) {\small$\Omega_9 $}   \put (24,97) {\small$\mathcal{S}_1 $} \put (47,97) {\small$\cdots $} \put (68,97) {\small$\mathcal{S}_5 $}
\end{overpic} }
\caption{ (a) A combination of five rectangular shapes $\mathcal{S}_1,\mathcal{S}_2,\cdots,\mathcal{S}_5$, carefully overlapped. The shapes $\mathcal{S}_4$ and $\mathcal{S}_5$ are each composed of two components connected by the dash-lines. For the composition $\mathpzc{R}_{\{1,2,3\},\{4,5\}}$, the linkage process does not produce a unique $\balpha$-vector. Some possible  outcomes of the linkage process are $\balpha_1 = (1,1,1,-1,-2)^T$, $\balpha_2 = (1,1,1,-2,-1)^T$ and $\balpha_3 = (1,1,1,-1.5,-1.5)^T$. The value of $\mathpzc{L}_{\balpha_1}(x)$ over each shapelet is displayed in the figure. (b) The bearing matrix corresponding to the composition depicted in panel (a). Considering $\balpha_1$ to be an outcome of linkage process, the resulting null-valued shapelets are $\Omega_8$ and $\Omega_9$. Although the discriminant matrix is full-rank, equation (\ref{eqxx3}) does not return a strictly positive solution}\label{fig4}
\end{figure}
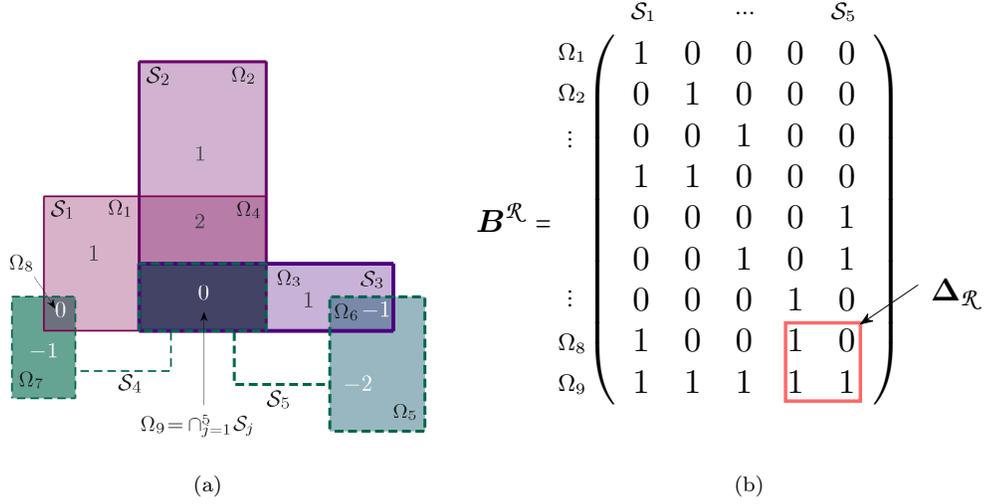

A natural question would be: what class of non-redundant compositions supports the formation of $\mathcal{A}$ (or simply when (\ref{eq21}) produces a unique solution)? In the following we will discuss that since the existence of the null-valued shapelets is an intrinsic property of any non-redundant composition (Proposition \ref{th2xx}), under simple assumptions about the structure of the bearing matrix restricted to such shapelets, the uniqueness may be guaranteed.

For the sake of convenience in the derivation, we consider a matrix representation of (\ref{eq21}). It is straightforward to verify that for a non-redundant composition $\mathpzc{R}_{\;\Ip,\In}$, the bearing matrix is structured as
\vspace{-.2cm}
\begin{equation}\label{eqbr}
\begin{array}{c@{}c}
\hspace{1.cm}\hexbrace{.9 cm}{\Ip}\;\;\;\;\hexbrace{.9cm}{\In}\\[-1.5em]
\boldsymbol{B}^\mathpzc{R} =\begin{bmatrix}
 \boldsymbol{B}^{(1,1)} & \boldsymbol{0}\\[.2cm] \boldsymbol{B}^{(2,1)} & \boldsymbol{B}^{(2,2)}
 \end{bmatrix}
&\hspace{-.15cm}
\begin{array}{l}
  \\ \rdelim\}{1}{3mm}[${}^{\Np\setminus \Nn}$] \\[.2cm]   \rdelim\}{1}{2mm}[${}^{\Nn}$] \\ \\
\end{array} \\[-1ex]
\end{array}\qquad,
\end{equation}
where $\boldsymbol{B}^{(1,1)}$, $\boldsymbol{B}^{(2,1)}$ and $\boldsymbol{B}^{(2,2)}$ are binary matrices. The linear program (\ref{eq21}) may then be cast as
\begin{equation}\label{eqxx1}
\left\{
\begin{array}{lc}
       \underset{\balpha_{\In}}{\max} & \boldsymbol{1}^T\balpha_{\In} \\
       s.t.& \boldsymbol{B}^{(2,2)} \balpha_{\In} \preceq -\boldsymbol{B}^{(2,1)}\boldsymbol{1}
     \end{array}.
   \right.
\end{equation}
For an outcome of the linkage process, the unit-valued shapelets may be indicated by
\begin{align*}
\Gamma_1^\mathpzc{R} \triangleq  \{i: \boldsymbol{B}^{(1,1)}_{i,:} \boldsymbol{1} = 1 \}.
\end{align*}
Likewise, if $\balpha_{\In}^\dagger$ is a maximizer of (\ref{eqxx1}), we can indicate the resulting null-valued shapelets by
\begin{align*}
\Gamma_0^\mathpzc{R} \triangleq  \{i: \boldsymbol{B}^{(2,2)}_{i,:} \balpha_{\In}^\dagger  +\boldsymbol{B}^{(2,1)}_{i,:}\boldsymbol{1}=0\},
\end{align*}
and correspondingly define a \emph{discriminant matrix} as
\begin{equation}\label{dismat}
\boldsymbol{\Delta}_\mathpzc{R} \triangleq \boldsymbol{B}^{(2,2)}_{\Gamma_0^\mathpzc{R},:}.
\end{equation}
An argument of duality, presented in the proof of Theorem \ref{thbasic}, reveals that there always exists a vector $\boldsymbol{w}\succeq \boldsymbol{0}$ such that
\begin{equation}\label{eqxx3}
\boldsymbol{\Delta}_\mathpzc{R}^T\boldsymbol{w} = \boldsymbol{1}.
\end{equation}
It will be discussed in the sequel that if the discriminant matrix associated with a non-redundant representation maintains certain properties, the linear program (\ref{eqxx1}) attains a unique maximizer. We specifically focus on the non-redundant compositions that the number of null-valued shapelets does not exceed $|\In|$.

\begin{definition}\label{linkable}
A non-redundant composition $\mathpzc{R}_{\;\Ip,\In}$ is basic, if for a maximizer of (\ref{eqxx1}), $\rank(\boldsymbol{\Delta}_\mathpzc{R}) = |\In|=|\Gamma_0^\mathpzc{R}|$ and the (essentially non-negative) solution of (\ref{eqxx3}) is strictly positive.
\end{definition}

\begin{theorem}\label{thbasic}
If a non-redundant composition is basic, then the maximizer of (\ref{eqxx1}) is unique. Conversely, if for a non-redundant composition $|\Gamma_0^\mathpzc{R}| = |\In|$ and the maximizer of (\ref{eqxx1}) is unique, the composition is basic.
\end{theorem}


Simply, applying the linkage process to a basic  composition yields a unique and well-defined representation of the composition in the $\balpha$-domain. As exemplified earlier, for the composition depicted in Figure \ref{fig4}(a), the linkage process fails to provide a unique representation. By analyzing the underlying discriminant matrix demonstrated in Figure \ref{fig4}(b), we easily observe that the composition fails to pass the required criteria of being basic.

As a matter of fact, the uniqueness of the maximizer in (\ref{eqxx1}) can be warranted for a broader class of compositions. Moreover, if $\balpha^\dagger$ is an outcome of the linkage process for a given composition  $\mathpzc{R}_{\;\Ip,\In}$, by construction $\xi^+(\balpha^\dagger)=\Ip$ and $\xi^-(\balpha^\dagger)=\In$, regardless of the outcome's uniqueness.  In other words, if an algorithm uses $\balpha$ as a proxy to identify the index sets $\Ip$ and $\In$, once it recovers any outcomes of the linkage process, the identification is successful. Focusing on the basic compositions allows us to markedly simplify the analysis by exploiting some of their favorable properties. However, the main ideas developed in this paper are extendible to broader classes of compositions.

For a basic composition, the conditions stated in Proposition \ref{th2xx} and Theorem \ref{thbasic} allow us to switch between different representations. Schematically, we can close the following chain
\begin{equation*}
\xymatrix{\{\Ip,\In\} \ar @/^/[r]^-{\mathcal{A}}& \balpha_\mathpzc{R} \ar@/^/[l]^-{\xi^+\cup\;\xi^-}  \ar @/^/[r]^-{}& \Beta_\mathpzc{R} \ar@/^/[l]^-{}}
\end{equation*}
where switching back and forth between different representations is conveniently possible.



\subsection{Sufficient Conditions for Unique Optimality of Sparse-CSC}\label{sec-conv}
In analyzing the Sparse-CSC program a key problem is deriving sufficient conditions under which a target vector $\balpha$ minimizes the underlying convex objective. This section is devoted to addressing the aforementioned problem in a general setup, where there are no specific restrictions on the overlapping of the dictionary elements and the spatial variability of the inhomogeneity measures. We will show that a given vector $\balpha$ is a solution to the Sparse-CSC problem if a certain system of constrained equations can be solved, where the coefficient matrix is related to the dictionary matrix and the right-hand side depends on the image data.

To maintain brevity we will use the notation
\begin{equation}\label{econv}
G(\balpha) \triangleq \int_{D}\max\Big(\big(\Pi_{in}(x) - \Pi_{ex}(x)\big)\mathpzc{L}_{\boldsymbol{\alpha}}(x), \big(\Pi_{in}(x) - \Pi_{ex}(x)\big)^-   \Big)\;\mbox{d}x,
\end{equation}
and delineate the goal as deriving sufficient conditions under which a designated vector $\balpha^*\in\mathbb{R}^{n_s}$ is the \emph{unique} minimizer to the convex program
\begin{equation}\label{ee-2}
\min_{\balpha} \;\;G(\balpha)\quad s.t. \quad \|\boldsymbol{\alpha}\|_1\leq \tau.
\end{equation}
We are certainly interested in the case that $\balpha^*$ is sparse.

As appeared in Section \ref{disjoint_sec}, the analysis of the problem becomes significantly easier when the dictionary elements are disjoint. This is certainly not the case in a general setup, however, through the DSD process we can always generate a \emph{super-dictionary} of the shapelets where every two elements are disjoint:
\[
\{\Omega_i,\boldsymbol{J}^{(i)}\}_{i=1}^{n_\Omega} = \mbox{DSD}\Big( \{\mathcal{S}_j\}_{j=1}^{n_s}\Big).
\]
Accordingly, as illustrated in Figure \ref{fig6}, one may rephrase $\mathpzc{L}_{\balpha}(x)$ in terms of the disjoint shapelets as
\begin{equation}
\mathpzc{L}_{\balpha}(x) = \sum_{j=1}^{n_s}\alpha_j\chi_{\mathcal{S}_j}(x)=\sum_{i=1}^{n_\Omega}\beta_i\chi_{\Omega_i}(x),  \label{eq41}
\end{equation}
where
\[
\beta_i=\sum_{j=1}^{n_s} \boldsymbol{J}^{(i)}_j\alpha_j.
\]
Consequently, the corresponding vectors can be linked through the dictionary bearing matrix
\begin{equation}\label{eq42}
\Beta = \boldsymbol{B}\balpha,
\end{equation}
where as discussed in Section \ref{DSDsec}, $\boldsymbol{B}_{i,j}=\boldsymbol{J}^{(i)}_j$.

By plugging (\ref{eq41}) into the convex cost (\ref{econv}) and taking a similar path as Section \ref{disjoint_sec}, we can rewrite $G(\balpha)$ in terms of $\Beta$ as
\begin{align}\nonumber
G(\balpha) &=
\sum_{i=1}^{n_\Omega}p_i\max(\beta_i,0)-q_i\min(\beta_i,1)\\ & \triangleq \mathcal{G}(\Beta) \label{eq43}
\end{align}
where
\begin{equation}\label{eqpiqi}
p_i \triangleq \int_{\Omega_i}  \big(\Pi_{in}(x) - \Pi_{ex}(x)\big)^+ \; \mbox{d}x,\qquad q_i \triangleq \int_{\Omega_i}  \big(\Pi_{ex}(x) - \Pi_{in}(x)\big)^+ \; \mbox{d}x.
\end{equation}
In other words, at the expense of the additional linear constraints (\ref{eq42}), the non-separable cost function $G(\balpha)$ can be rewritten as a separable objective $\mathcal{G}(\Beta)$.

Aside from a separable representation of the convex objective, we suggest an additional simplifying step which privileges the possibility of modeling (\ref{ee-2}) as a convex program with linear constraints. The following result plays the major role towards this purpose.

\begin{proposition}\label{te1}
Consider a target vector $\balpha^*\in \mathbb{R}^{n_s}$, supported on $\Gamma$ ($\alpha_j^*\neq 0$ for $j\in \Gamma$ and $\balpha_{\Gamma^c}^* = \boldsymbol{0}$). Let  $\boldsymbol{c}\in\mathbb{R}^{n_s}$ be a vector such that $c_j = \sign(\alpha_j^*)$ for $j\in\Gamma$ and $\|\boldsymbol{c}_{\Gamma^c}\|_\infty\leq 1$. For $\tau = \|\balpha^*\|_1$, $\balpha^*$ is the unique minimizer of the convex program
\begin{align*}
\min_{\balpha} \;\;G(\balpha)\quad s.t. \quad \|\boldsymbol{\alpha}\|_1\leq \tau,
\end{align*}
if the following conditions hold:\\
(I) $\balpha^*$ is the unique minimizer of the convex program
\begin{align}\label{ee-1}
\min_{\balpha} \;\;G(\balpha)\quad s.t. \quad \boldsymbol{c}^T\balpha = \tau,
\end{align}
(II) the convex set $\mathcal{C} = \{\balpha:G(\balpha)\leq G(\balpha^*)\}$ has a nonempty interior and
\begin{align*}
\exists \boldsymbol{\delta}\in T_\mathcal{C}(\balpha^*) \quad s.t. \quad  \boldsymbol{c}^T \boldsymbol{\delta} > 0,
\end{align*}
 where $T_\mathcal{C}(\balpha^*)$ denotes the tangent cone of $\mathcal{C}$ at $\balpha^*$.
\end{proposition}
\begin{figure}[t]
    \centering
    \centering \includegraphics[width=85mm]{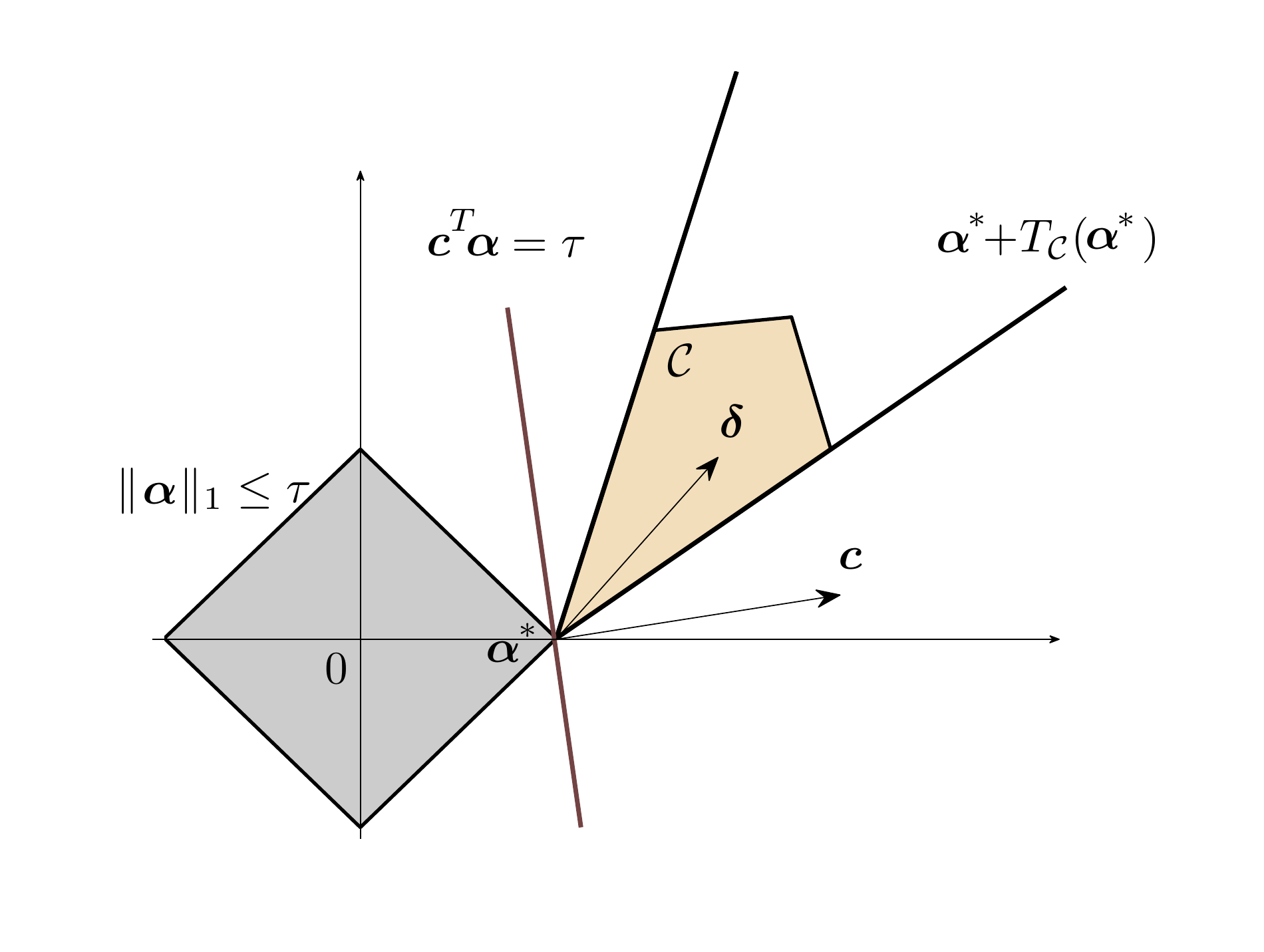}
    \caption{A schematic of the sufficient conditions indicated in Proposition \ref{te1}}\label{fig7.5}
    \end{figure}

Basically, Proposition \ref{te1} states that instead of inspecting the conditions under which $\mathcal{C}$ touches the $\ell_1$-ball at $\balpha^*$, one may look into the problem of $\mathcal{C}$ touching a separating hyperplane between the two sets, at $\balpha^*$. Condition (II) of the proposition simply assures that the two sets are standing on opposite sides of the proposed hyperplane (as illustrated in Figure \ref{fig7.5}).

Supposing that one could verify the tangent cone property (II) for a target vector $\balpha^*$, property (I) allows us to focus on the convex program (\ref{ee-1}) which enjoys a simple linear constraint. This constraint may be combined with (\ref{eq42}) through
\begin{align}\label{ee-4}
\begin{bmatrix}
\boldsymbol{B}\\ \boldsymbol{c}^T
\end{bmatrix}\balpha = \begin{bmatrix}
 \Beta\\ \tau
 \end{bmatrix},
\end{align}
which compactly presents the governing linear constraints on the $(\balpha,\Beta)$ pair.

When $n_\Omega\gg n_s$, we will have a left kernel for the left-hand side matrix in (\ref{ee-4}). Accordingly, we can build a matrix $\boldsymbol{N}\in\mathbb{R}^{(n_\Omega +1)\times n_N}$, columns of which form a spanning set for $Null([\boldsymbol{B}^T , \boldsymbol{c}])$. Applying $\boldsymbol{N}^T$ to both sides of (\ref{ee-4}) annihilates $\balpha$ and leaves us with a linear set of constraints merely in terms of $\Beta$. More specifically,
\begin{align}\label{ee-5}
\boldsymbol{N}^T\begin{bmatrix}
 \Beta\\ \tau
 \end{bmatrix} = \boldsymbol{N}^T\begin{bmatrix}
\boldsymbol{B}\\ \boldsymbol{c}^T
\end{bmatrix}\balpha = \boldsymbol{0}.
\end{align}
Making use of (\ref{ee-5}) along with (\ref{eq43}), allow us to address the convex program (\ref{ee-1}) in the $\Beta$-domain via
\begin{align}\label{ee-6}
\min_{\Beta} \;\;\mathcal{G}(\Beta)\qquad s.t. \qquad  \boldsymbol{N}^T\begin{bmatrix}
 \Beta\\ \tau
 \end{bmatrix}  = \boldsymbol{0}.
\end{align}
Unlike (\ref{ee-2}), this convex program enjoys a  linear set of constraints and a separable objective, which significantly simplify the analysis. In Theorem \ref{th7} we will derive a sufficient set of conditions under which the minimizer of (\ref{ee-6}) is unique. These conditions are directly expressed in terms of $\boldsymbol{B}$ and $\boldsymbol{c}$, and the appearance of $\boldsymbol{N}$ in the formulation is mainly for the purpose of clarifying the rationale behind the variable change. Before presenting Theorem \ref{th7} we need to proceed with introducing some new notions.

Consider a potential minimizer $\balpha^*$ and correspondingly $\Beta^* = \boldsymbol{B}\balpha^*$. We assume that all the $\beta_i^*$ values lie outside the interval $(0,1)$. This assumption is consistent with the arguments presented in Section \ref{Sec:Rel}, where the $\balpha$ vector assigned to a given composition $\mathpzc{R}_{\;\Ip,\In}$ maintains such property. Based on the value of each entry in $\Beta^*$, we may partition the set $\{1,2,\cdots,n_\Omega\}$ into four disjoint index sets denoted and labeled as follows:
\begin{align}\hspace{-1.4cm}\label{ee-7}
\left\{\begin{array}{l}
\Gamma_{0^-} \triangleq \Big\{i:\beta_i<0,i\in\{1,2,\cdots,n_\Omega\}\Big\}\\
\Gamma_{1^+} \triangleq \Big\{i:\beta_i>1,i\in\{1,2,\cdots,n_\Omega\}\Big\}
\end{array}
\right. ,\qquad \mbox{\emph{(the support set)}}
\end{align}
and
\begin{align}\label{ee-8}
\left\{\begin{array}{l}
\Gamma_0 \triangleq \Big\{i:\beta_i=0,i\in\{1,2,\cdots,n_\Omega\}\Big\}\\
\Gamma_1 \triangleq \Big\{i:\beta_i=1,i\in\{1,2,\cdots,n_\Omega\}\Big\}
\end{array}
\right., \qquad \mbox{\emph{(the off-support set)}}.
\end{align}
We also introduce the bounding vectors $\boldsymbol{l}$ and $\boldsymbol{u} \in \mathbb{R}^{|\Gamma_0\cup\Gamma_1|}$ with entries
\begin{equation}\label{ee-9}
\boldsymbol{l}_{m(\ell)} =  \left\{\begin{array}{ll}
q_\ell-p_\ell & \ell\in \Gamma_0\\
-p_\ell &\ell \in \Gamma_1
\end{array}
\right. \qquad \mbox{and}\qquad \boldsymbol{u}_{m(\ell)} =  \left\{\begin{array}{ll}
q_\ell & \ell\in \Gamma_0\\
q_\ell - p_\ell &\ell \in \Gamma_1
\end{array}
\right. ,
\end{equation}
where $m(.): \Gamma_0\cup\Gamma_1 \mapsto \{1,2,\cdots ,|\Gamma_0\cup\Gamma_1|\}$ is a simple bijective map that allows filling in the bounding vector entries in an order.

Finally, we define a so called \emph{LOC violation} vector $\boldsymbol{e}\in\mathbb{R}^{n_s}$, evaluated through
\begin{equation}\label{ee-10}
\boldsymbol{e}_j \triangleq\!\! \sum_{\ell\in\mathcal{I}_j \cap \Gamma_{1^+}}\!\!\!\!p_\ell -\!\!\sum_{\ell\in\mathcal{I}_j \cap \Gamma_{0^-}}\!\!\! q_\ell, \qquad\qquad  j=1,2,\cdots,n_s.
\end{equation}
The LOC violation vector in some sense weights the LOC disruption over each shape. Basically, when the LOC holds and $\balpha^*$ is a vector satisfying (\ref{eq32}), $\boldsymbol{e} = \boldsymbol{0}$.\footnote{It is straightforward to see that $\boldsymbol{e}$ also vanishes when $\balpha^*$ is a minimizer in the case of disjoint dictionary elements (described in Section
\ref{disjoint_sec}), basically because $\Gamma_{1^+} = \Gamma_{0^-}=\emptyset$.}

\begin{theorem}\label{th7}
Given the dictionary bearing matrix $\boldsymbol{B}\in \{0,1\}^{n_\Omega\times n_s}$, consider a target vector $\balpha^*\in\mathbb{R}^{n_s}$ that is feasible for (\ref{ee-1}) and correspondingly $\Beta^* = \boldsymbol{B}\balpha^*$, such that all entries of $\Beta^*$ lie outside the interval $(0,1)$. Further, consider $\boldsymbol{c}\in\mathbb{R}^{n_s}$ constructed as stated in Proposition \ref{te1}. If the matrix $\begin{bmatrix}(\boldsymbol{B}_{\Gamma_0\cup\Gamma_1,:})^T\;,\; \boldsymbol{c} \end{bmatrix}$ has full row rank and there exist $\boldsymbol{\eta}\in\mathbb{R}^{|\Gamma_0\cup\Gamma_1|}$ and $\eta_c\in\mathbb{R}$ that satisfy
\begin{align}\label{ee-11}
\begin{bmatrix}\big(\boldsymbol{B}_{\Gamma_0\cup\Gamma_1,:}\big)^T\;,\; \boldsymbol{c} \end{bmatrix}\begin{bmatrix}\boldsymbol{\eta}\\\eta_c\end{bmatrix}
 =  \boldsymbol{e}
\end{align}
and \footnote{The rows of $\boldsymbol{B}_{\Gamma_0\cup\Gamma_1,:}$ also need to be arranged according to the index map $m(.)$ used in (\ref{ee-9}).}
\begin{align}\label{ee-12}
\boldsymbol{l}\prec\boldsymbol{\eta}\prec \boldsymbol{u},
\end{align}
then, $\balpha^*$ is the unique minimizer of the convex program (\ref{ee-1}).
\end{theorem}


Altogether, in order to show that $\balpha^*$ is the unique minimizer of the $\ell_1$-constrained problem (\ref{ee-2}), we need to show that the tangent cone property of Proposition \ref{te1}, and the uniqueness requirements of Theorem \ref{th7} can be established simultaneously. To make a more intuitive sense of the material presented, we conclude this section with a simple toy example.

\begin{example}\label{examp}
Consider the three shapes $\mathcal{S}_1$, $\mathcal{S}_2$ and $\mathcal{S}_3$ depicted in Figure \ref{figex}. The resulting shapelets are shown and denoted by $\Omega_1, \cdots,\Omega_7$. For simplicity, the inhomogeneity measures are somehow that $p_1 = 1-q_1 = 0$, and $q_i=1-p_i=0$ for $i=2,3,\cdots,7$. In other words, the LOC holds for $\Omega_1 =\mbox{cl}(\mathcal{S}_1\setminus (\mathcal{S}_2\cup\mathcal{S}_3))$.

\begin{figure}[t]
    \centering
    \centering \includegraphics[width=52mm]{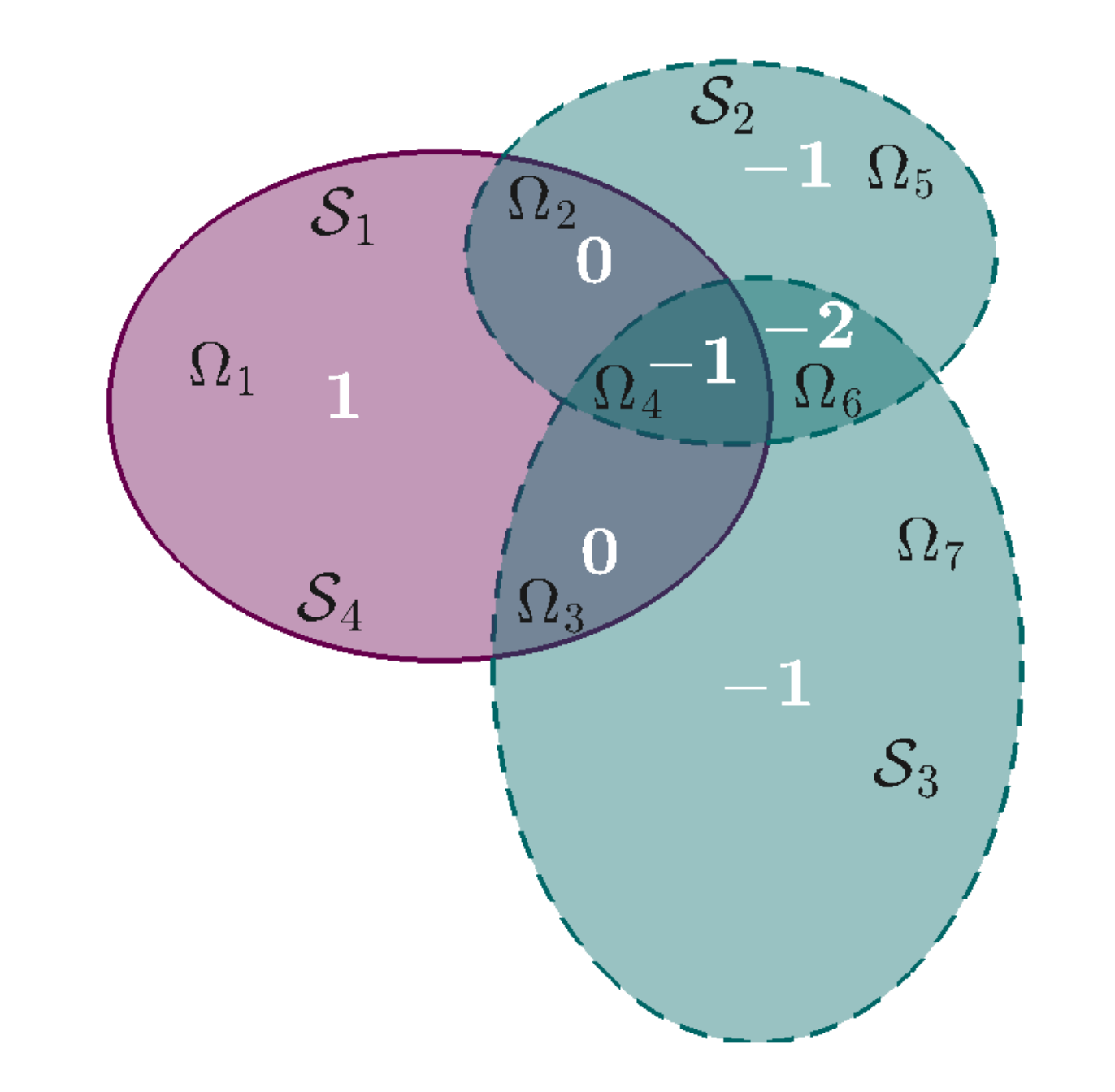}
    \begin{tabular}[b]{c}
      $\boldsymbol{B} = \left( \begin{array}{cccc}
1 & 0 & 0&1 \\
1 & 1 & 0& 0 \\
1 & 0 & 1& 0\\
1 & 1 & 1 & 0\\
0 & 1 & 0& 0\\
0 & 1 & 1 & 0 \\
0 & 0 & 1&0
\end{array} \right)$\\
       \\ \\ \\
    \end{tabular}
    \captionlistentry[table]{A table beside a figure}
    \caption{The shapes used in Example 4.1 and the corresponding bearing matrix. The values shown in light color are the entries of $\Beta_1^* = \boldsymbol{B}\boldsymbol{\balpha_1^*}$ over each shapelet when $\balpha_1^*=[1,-1,-1, 0]^T$}\label{figex}
  \end{figure}

Suppose we add a forth shape to the dictionary and carefully select it to be $\Omega_1$, i.e., $\mathcal{S}_4 = \mbox{cl}(\Omega_1)$. Focusing on the $\ell_1$-constrained problem (\ref{ee-2}) for a fixed $\tau=3$, we can easily observe that $\balpha_1^*=[1,-1,-1, 0]^T$ and $\balpha_2^*=[0,0,0,1]^T$ are two possible minimizers, as they are both feasible and satisfy condition (\ref{eq32}) of Theorem \ref{th5}. As will be detailed in the sequel, we show that there is no way to establish the conditions of Theorem \ref{th7} for $\balpha_1^*$.

According to the figure, for the given value of $\balpha_1^*$, $\Gamma_0 = \{2,3\}$ and $\Gamma_1 = \{1\}$. Also, $\boldsymbol{c}=[1,-1,-1,c_4]^T$ where $|c_4|\leq 1$ and therefore equation (\ref{ee-11}) appears as
\begin{equation}\label{ee-13}
\begin{bmatrix}
1&1&1&1\\ 0& 1& 0& -1\\ 0 & 0& 1& -1\\ 1 & 0 & 0 & c_4
\end{bmatrix}\begin{bmatrix}\boldsymbol{\eta}\\ \eta_c\end{bmatrix}=\boldsymbol{0}.
\end{equation}
On the other hand for the bounding vectors we have
\begin{equation}\label{ee-14}
\boldsymbol{l}=\begin{bmatrix}-p_1\\ q_2-p_2\\ q_3 -p_3\end{bmatrix} = \begin{bmatrix}0\\ -1\\ -1 \end{bmatrix} \qquad \mbox{and} \qquad \boldsymbol{u}=\begin{bmatrix}1\\ 0\\ 0\end{bmatrix}.
\end{equation}
We can see that for any $c_4\in[-1,1]$, equation (\ref{ee-13}) has the unique trivial solution $\boldsymbol{\eta}=\boldsymbol{0}$, $\eta_c = 0$, and given the bounding vectors specified in (\ref{ee-14}), there is no way to establish condition (\ref{ee-12}).
\end{example}

\subsection{Accurate Recovery of the Shape Elements}\label{sec:accrecovery}
The tools developed in Sections \ref{Sec:Rel} and \ref{sec-conv} can be employed to characterize the minimizers of Sparse-CSC, and explore the possibility of identifying the constituting elements of a target composition. Since considerable attention has already been drawn to the derivation of the proposed convex proxy and the related analysis tools, we will consider a setup pertinent to the preceding example that allows integrating the techniques in a concise, yet insightful way. The ideas developed here may be further extended to derive more general results.

Consider a dictionary of shape elements $\{\mathcal{S}_j\}_{j=1}^{n_s}$ and an image $D$, where the LOC holds for $\Sigma\subset D$. We assume that there exists a basic non-redundant composition $\mathpzc{R}_{\Ip,\In}$ such that $\Ip,\In\subset\{1,\cdots,n_s\}$ and $\Sigma = \mbox{cl}(\mathpzc{R}_{\Ip,\In})$. Concerning the basic composition, the outcome of the linkage process is represented by
\[\balpha_\mathpzc{R} = \mathcal{A}\big(\{\mathcal{S}_j\}_{j\in\Ip\cup\In};\Ip,\In\big).
\]
To avoid indexing complications, we simply assume that $\Ip = 1,\cdots, n_\oplus$ and $\In = n_\oplus+1, \cdots,n_\oplus+n_\ominus$. We use the terminology of \emph{exterior} for the off-target dictionary elements, index by $\{n_\oplus+n_\ominus+1,\cdots,n_s\}$.

We will assert that when $n_s = n_\oplus+n_\ominus$, for $\tau = \big\|\mathcal{A}\big(\{\mathcal{S}_j\}_{j\in\Ip\cup\In};\Ip,\In\big)\big\|_1$, the unique outcome of the Sparse-CSC is $\balpha^* = \balpha_\mathpzc{R}$. Moreover, for $n_s>n_\oplus+n_\ominus$, if the exterior shapes are placed in an ``unstructured'' way and maintain a restricted level of overlap with $\{\mathcal{S}_j\}_{j\in\Ip\cup\In}$, the Sparse-CSC program still attains a unique minimizer $\balpha^*$, such that $\balpha^*_{\Ip\cup\In}=\balpha_\mathpzc{R}$ and $\balpha^*_{(\Ip\cup\In)^c}=\boldsymbol{0}$. In a sense, under such conditions, the exterior shapes cannot distract the convex proxy from identifying $\Ip$ and $\In$.

Consider $\{\Omega^\mathpzc{R}_\ell\}_{\ell=1}^{n_{\Omega^\mathpzc{R}}}$ and $\boldsymbol{B}^\mathpzc{R} \in \{0,1\}^{n_{\Omega^\mathpzc{R}}\times (n_\oplus+n_\ominus)}$ to be the shapelets and the bearing matrix associated with the DSD process
\begin{equation*}
\{\Omega^\mathpzc{R}_\ell,\boldsymbol{B}^\mathpzc{R}_{\ell,:}\}_{\ell=1}^{n_{\Omega^\mathpzc{R}}} = \mbox{DSD}\Big( \{\mathcal{S}_j\}_{j\in\Ip\cup\In}\Big).
\end{equation*}
When a shape is added to the collection $\{\mathcal{S}_j\}_{j\in\Ip\cup\In}$, depending on its overlap with the present elements, the outcome of the DSD process may change to finer partitions. Strictly speaking, when $n_s > n_\oplus+n_\ominus$ and
\begin{equation*}
\{\Omega_i,\boldsymbol{B}_{i,:}\}_{i=1}^{n_{\Omega}} = \mbox{DSD}\Big( \{\mathcal{S}_j\}_{j=1}^{n_s}\Big),
\end{equation*}
there exist index sets $\mathcal{J}_\ell$, such that
\begin{equation}\label{eqmesh}
\Omega^\mathpzc{R}_\ell = \bigcup_{i\in \mathcal{J}_\ell}\Omega_i, \qquad \ell = 1,2,\cdots, n_{\Omega^\mathpzc{R}}.
\end{equation}
To avoid confusion, we will refer to the super-shapelets $\{\Omega_i\}_{i=1}^{n_\Omega}$ as $cells$.

The DSD process generates disjoint cells from the elements in the dictionary. For the sake of discussion in this section, it might be more convenient to look at a reversed process. Here a collection of disjoint closed sets, yet referred to as cells, is fixed and each shape in the dictionary is constructed by making a union over a number of them.

Although the fixed cells are not generated by any specific DSD process, we will still denote them by $\{\Omega_i\}_{i=1}^{n_\Omega}$ (and note that (\ref{eqmesh}) holds). This setup allows us to more conveniently track the process of adding a new element without updating the shapelet architecture every time. The only update made on the dictionary matrix would be adding a new column, while the number of rows always remains to be $n_\Omega$ (see Figure \ref{figcell}).

\begin{figure}[t]
\hspace{-.6cm}\begin{overpic}[width=0.4\textwidth,tics=10]{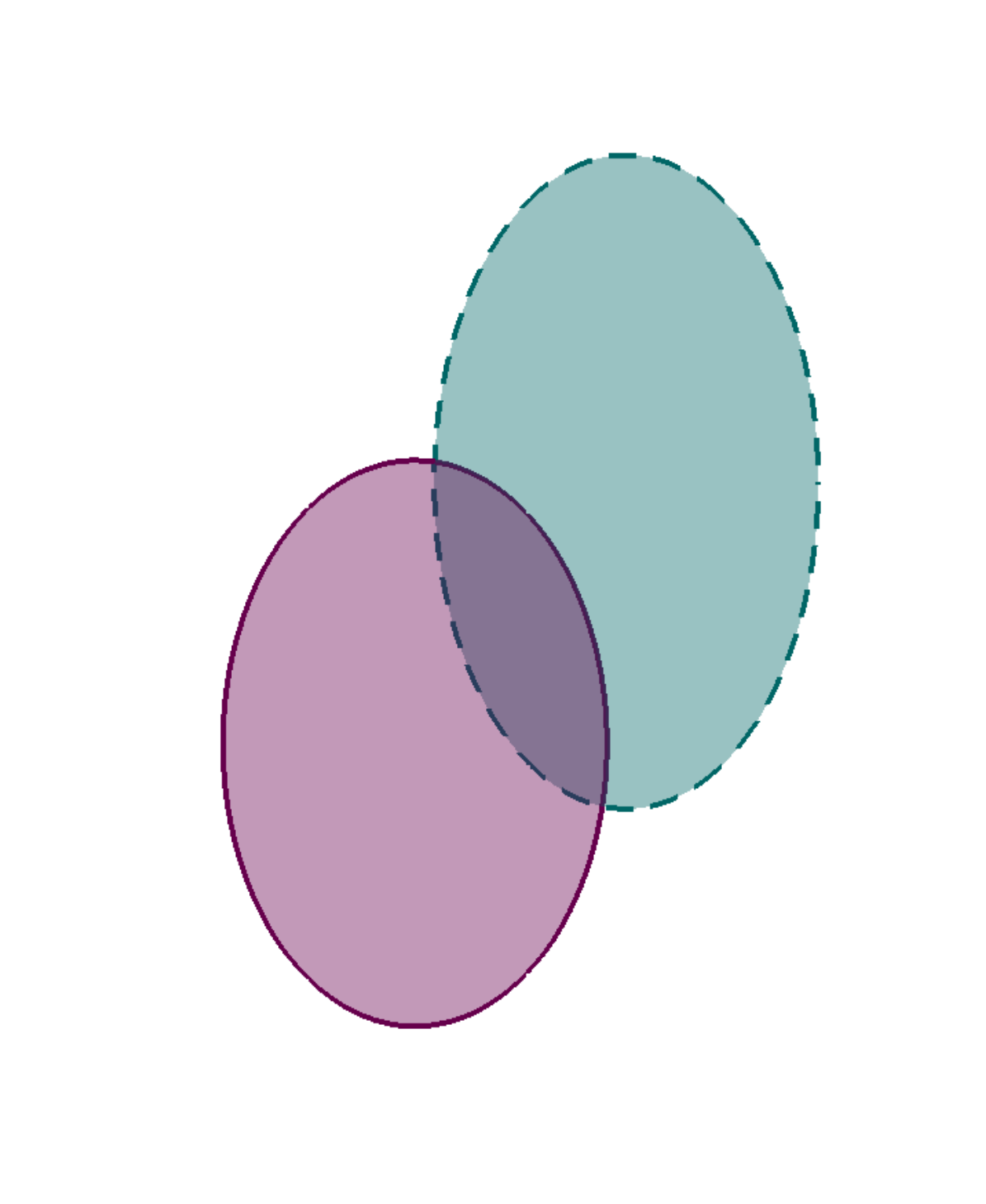}
 \put (20,35) {$\mathcal{S}_1$} \put (31,19) {$\Omega^\mathpzc{R}_1$} \put (47,80) {$\Omega^\mathpzc{R}_2$}  \put (40,45) {$\Omega^\mathpzc{R}_3$} \put (60,55) {$\mathcal{S}_2$}
\end{overpic}
\hspace{-1.5cm}
    \begin{tabular}[b]{c}
      $\boldsymbol{B}^{\mathpzc{R}} = \left( \begin{array}{cccc}
1 & 0 \\
0 & 1  \\
1 & 1
\end{array} \right)$\\
       \\ \\ \\
    \end{tabular}\hspace{-.3cm}
 \includegraphics[width=50mm]{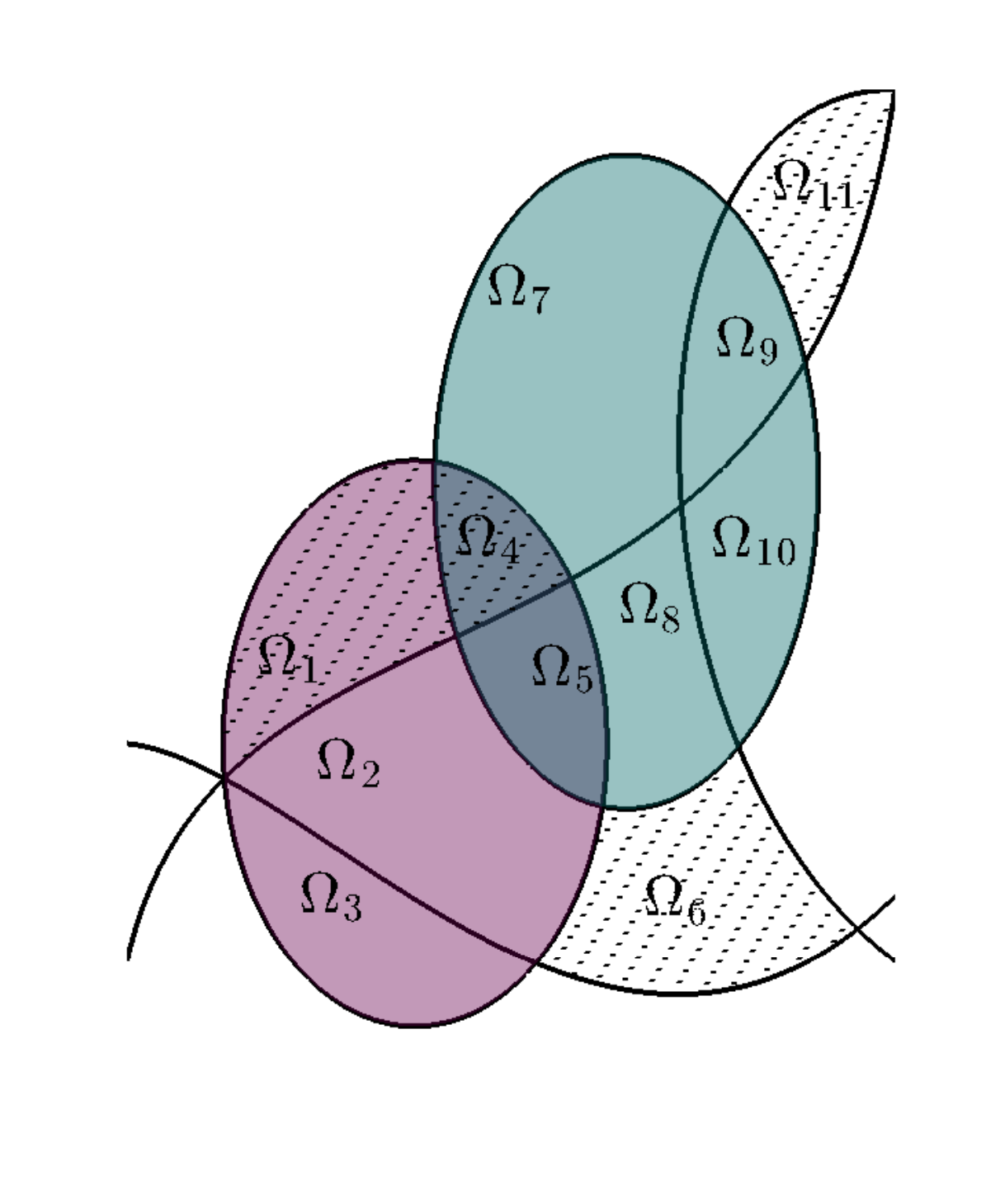}\hspace{-.7cm}
    \begin{tabular}[b]{c}
      $\boldsymbol{B} = \left( \begin{array}{cccc}
1 & 0 & 1\\
1 & 0 & 0 \\
1 & 0 & 0\\
1 & 1 & 1 \\
1 & 1 & 0\\
0 & 0 & 1 \\
0 & 1 & 0 \\
0 & 1 & 0 \\
0 & 1 & 0 \\
0 & 1 & 0 \\
0 & 0 & 1
\end{array} \right)$\\
       \\ \\ \\
    \end{tabular}\vspace{-.6cm}
    \captionlistentry[table]{A table beside a figure}
    \caption{Left: the bearing matrix associated with the overlapped shapes $\mathcal{S}_1$ and $\mathcal{S}_2$. Right: a collection of fixed cells indicated by $\Omega_1,\cdots,\Omega_{11}$, which can be used as the building blocks for $\mathcal{S}_1$ and $\mathcal{S}_2$. The shaded cells form a third shape $\mathcal{S}_3=\bigcup_{i=1,4,6,11}\Omega_i$. The overall dictionary matrix $\boldsymbol{B}$ is presented}\label{figcell}
  \end{figure}
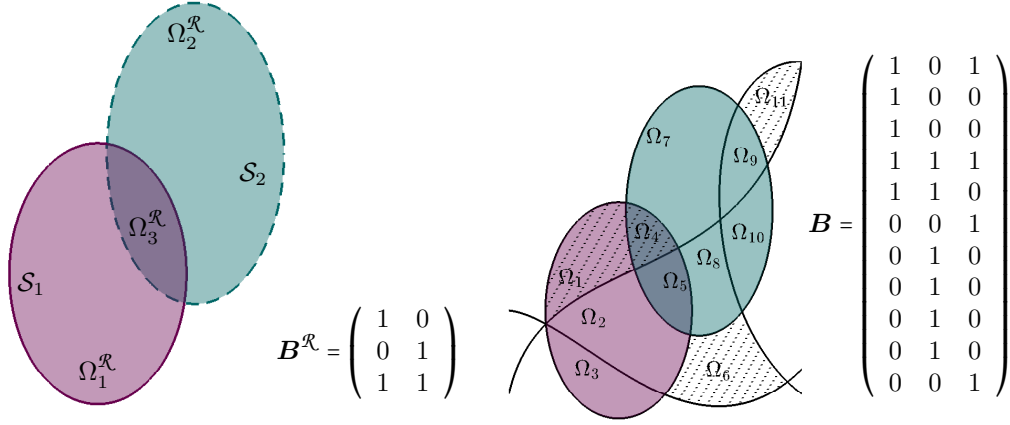

While we do not consider any specific geometry for the cells, at the finest partitioning level, the cells may simply be taken as the image pixels, and each element of the dictionary to be a collection of pixels. We assume each cell contributes in the construction of at least one shape (hence, the final dictionary matrix does not contain zero rows).

%
%

As we will discuss, the null and unit-valued shapelets play a key role in recovering a basic composition through the Sparse-CSC. Denoting the unit and null-valued shapelets by $\Gamma_1^\mathpzc{R}$ and $\Gamma_0^\mathpzc{R}$, a key property of basic compositions is presented as follows.
\begin{proposition} \label{lemcons} Let $\boldsymbol{c}\in \mathbb{R}^{n_\oplus +n_\ominus}$, where $\boldsymbol{c}_{\Ip} = \boldsymbol{1}$ and $\boldsymbol{c}_{\In} = -\boldsymbol{1}$. Considering the basic non-redundant composition $\mathpzc{R}_{\Ip,\In}$, the linear system
\begin{equation}\label{eqbw}
\big(\boldsymbol{B}_{\Gamma_0^\mathpzc{R}\cup\Gamma_1^\mathpzc{R},:}^\mathpzc{R}\big)^T\boldsymbol{w}= -\boldsymbol{c}
\end{equation}
has a unique solution, which satisfies $-(1+n_\ominus)\boldsymbol{1} \preceq \boldsymbol{w}_{\Gamma_1^\mathpzc{R}}\preceq -\boldsymbol{1}$ and $\boldsymbol{0}\prec \boldsymbol{}\boldsymbol{w}_{\Gamma_0^\mathpzc{R}}\preceq \boldsymbol{1}$.
\end{proposition}

As may be noticed from the proof of Proposition \ref{lemcons}, a basic composition $\mathpzc{R}_{\Ip,\In}$ has exactly $n_\oplus$ unit-valued (and $n_\ominus$ null-valued) shapelets. Solving (\ref{eqbw}) for $\boldsymbol{w}$ assigns negative quantities to the unit-valued (and positive quantities to the null-valued) shapelets. We will refer to the entries of $\boldsymbol{w}$ as the \emph{bearing constants}. The purpose of considering the specific equation (\ref{eqbw}), and the connection between $\boldsymbol{w}$ and the shape identification problem will be revealed later in this section.

Now, consider a vector $\balpha^*\in\mathbb{R}^{n_s}$ such that
\begin{equation}\label{eqalph}
\balpha^*_{\Ip\cup\In}=\balpha_\mathpzc{R}\qquad \mbox{and}\qquad  \balpha^*_{(\Ip\cup\In)^c}=\boldsymbol{0}.
\end{equation}
Focusing on Proposition \ref{te1} and Theorem \ref{th7}, we proceed to derive  sufficient conditions for $\balpha^*$ to be the unique minimizer of the convex program
\begin{equation}\label{eqcon2}
\min_{\balpha} \;\;G(\balpha) \qquad s.t. \qquad \|\balpha\|_1\leq \big\|\mathcal{A}\big(\{\mathcal{S}_j\}_{j\in\Ip\cup\In};\Ip,\In\big)\big\|_1,
\end{equation}
where $G(.)$ is our proposed convex objective (\ref{econv}). We begin by showing that when the LOC holds for $\Sigma = \mbox{cl}(\mathpzc{R}_{\Ip,\In})$, the tangent cone property of Proposition \ref{te1} can be conveniently established.
\begin{lemma} \label{lemloc}
Given $\balpha^*$ as stated in (\ref{eqalph}), consider a vector $\boldsymbol{c}\in \mathbb{R}^{n_s}$ such that $\boldsymbol{c}_{\Ip\cup\In} = \sign(\balpha^*_{\Ip\cup\In})$ and $\|\boldsymbol{c}_{(\Ip\cup\In)^c}\|_\infty \leq 1$. There exists a point $\hat\balpha$ inside the convex set $\mathcal{C}=\{\balpha:G(\balpha)\leq G(\balpha^*)\}$ such that $\boldsymbol{c}^T(\hat\balpha - \balpha^*)>0$.
\end{lemma}

Since the lemma warrants the tangent cone property, based on Proposition \ref{te1} we would only need to focus on the uniqueness conditions for the minimizer of (\ref{ee-1}).

Suppose that $\Beta^*\in \mathbb{R}^{n_\Omega}$ contains the values of $\mathcal{L}_{\balpha^*}(x)$ over the cells $\{\Omega_i\}_{i=1}^{n_\Omega}$. Let $T$ denote the index set associated with the cells that do not overlap with the constituting elements of $\Sigma$, i.e.,
\begin{equation*}
T = \Big\{i:\Omega_i \subset \Big(\bigcup_{j=1}^{n_s} \mathcal{S}_j\Big) \big\backslash \Big(\bigcup_{j=1}^{n_\oplus+n_\ominus} \mathcal{S}_j\Big)  ,i\in\{1,2,\cdots,n_\Omega\}\Big\}.
\end{equation*}
By looking at the image of $\balpha^*$ in the $\Beta$-domain, we expect $\Beta^*$ to match the values of $\mathcal{L}_{\balpha_\mathpzc{R}}(x)$ over the cells within $\bigcup_{j=1}^{n_\oplus+n_\ominus} \mathcal{S}_j$ and to vanish over the remaining cells. More specifically,
\begin{equation*}
 \Beta^*_i = \left \{\begin{array}{lc}
 \Beta_\ell^\mathpzc{R} & i\in \mathcal{J}_\ell , \quad \ell \in \Gamma_0^\mathpzc{R}, \Gamma_1^\mathpzc{R}, \Gamma_{0^-}^\mathpzc{R}, \Gamma_{1^+}^\mathpzc{R}\\
 0 & i\in T
\end{array}.
\right.
\end{equation*}
Similar to (\ref{ee-8}), the off-support index sets associated with $\Beta^*$ may be indicated by
\begin{equation*}
\Gamma_0 = T\cup \bigcup_{\ell\in \Gamma_0^\mathpzc{R}} \mathcal{J}_\ell \qquad \mbox{and}\qquad \Gamma_1 =  \bigcup_{\ell\in \Gamma_1^\mathpzc{R}} \mathcal{J}_\ell.
\end{equation*}
Using Theorem \ref{th7}, we will show that the shape identification condition is closely related to the structure of the dictionary matrix, and the overlap between the exterior shapes and the composition elements. To present the result in a more propitious and intuitive way, we proceed by introducing the related matrix blocks and an overlapping measure.

Let us assume the cell index assignment is performed in a way that $\Gamma_0\cup\Gamma_1 = \{1,2,\cdots, |\Gamma_0\cup\Gamma_1|\}$. Such assumption would avoid index mapping complications. Now, consider the binary matrix $\boldsymbol{B}'\in\{0,1\}^{(n_\oplus+n_\ominus)\times|\Gamma_0\cup\Gamma_1|}$ constructed as
\begin{equation*}
\boldsymbol{B}_{j,i}'= 1_{\{\Omega_i\subset\mathcal{S}_j\}} = \left\{\begin{array}{lc}1& \Omega_i\subset\mathcal{S}_j\\0 & \Omega_i\nsubset\mathcal{S}_j \end{array} \right.,  \qquad j=1,\cdots, n_\oplus+n_\ominus, \quad i\in \Gamma_0\cup\Gamma_1.
\end{equation*}
The columns of $\boldsymbol{B}'$ are clearly zero over $T$. One can verify that the remaining columns are multiple replications of the columns of $(\boldsymbol{B}_{\Gamma_0^\mathpzc{R}\cup\Gamma_1^\mathpzc{R},:}^\mathpzc{R})^T$, which yields \vspace{-.25cm}
\begin{equation*}
\rank(\boldsymbol{B}') = \rank(\boldsymbol{B}_{\Gamma_0^\mathpzc{R}\cup\Gamma_1^\mathpzc{R},:}^\mathpzc{R}) = n_\oplus + n_\ominus.
\end{equation*}
We may follow a similar pattern for the exterior shapes to construct a matrix $\boldsymbol{B}''$ as
\begin{equation*}
\boldsymbol{B}_{j-n_\oplus-n_\ominus,i}''= 1_{\{\Omega_i\subset\mathcal{S}_j\}}, \qquad j=n_\oplus+n_\ominus+1, \cdots,  n_s, \quad i\in \Gamma_0\cup\Gamma_1.
\end{equation*}
As a measure of overlap between the exterior shapes and the null and unit-valued shapelets, we define the quantities $\gamma_{\ell,j}\in[0,1]$ as
\begin{equation}\label{gammalj}
 \gamma_{\ell,j} \triangleq \frac{1}{|\mathcal{J}_\ell|}    \sum_{i\in\mathcal{J}_\ell}   1_{\{\Omega_i\subset\mathcal{S}_j\}},\qquad  \ell\in \Gamma_0^\mathpzc{R}\cup\Gamma_1^\mathpzc{R} , \;\; j= n_\oplus + n_\ominus + 1, \cdots, n_s.
\end{equation}
These quantities are simply the relative number of cells that are common between the shapelet $\Omega^\mathpzc{R}_\ell$ and an exterior shape $\mathcal{S}_j$.

\begin{theorem}\label{thglast}
Following the preceding setup, suppose $\Sigma = \mbox{cl}(\mathpzc{R}_{\;\Ip,\In})$ satisfies the LOC and $\mathpzc{R}_{\;\Ip,\In}$ is a basic composition. If the cellular standing of the exterior dictionary elements is in a way that \vspace{-.2cm}
\begin{align}\label{eqrowrank}\begin{bmatrix} \boldsymbol{B}'\\ \boldsymbol{B}''\end{bmatrix}\vspace{-.2cm}
\end{align}
has full row rank and
\begin{align}\label{eqcoh}
\big|  \sum_{\ell\in \Gamma_0^\mathpzc{R}\cup \Gamma_1^\mathpzc{R}} \gamma_{\ell,j}\boldsymbol{w}_\ell\big |<1\qquad \forall j\in\{ n_\oplus+n_\ominus+1, \cdots, n_s \},
\end{align}
then the unique minimizer of the convex program (\ref{eqcon2}) is $\balpha^*$, which satisfies $\balpha^*_{\Ip\cup\In}=\balpha_\mathpzc{R}$ and $\balpha^*_{(\Ip\cup\In)^c}=\boldsymbol{0}$.
\end{theorem}

The magnitude of the quantity
\begin{equation}\label{eqcoherence}
C_j \triangleq \big|  \sum_{\ell\in \Gamma_0^\mathpzc{R}\cup \Gamma_1^\mathpzc{R}} \gamma_{\ell,j}\boldsymbol{w}_\ell\big |, \qquad j\in (\Ip\cup\In)^c,
\end{equation}
is somehow related to the way an exterior shape overlaps with the elements of the target composition. The critical shapelets that contribute in the value of $C_j$ are only the null and unit valued shapelets associated with $\mathpzc{R}_{\Ip,\In}$. When the exterior shapes stay away from the elements of $\mathpzc{R}_{\Ip,\In}$, the result of Theorem \ref{thglast} is quite consistent with what we expect to observe as the outcome of the convex program (\ref{eqcon2}). We can somehow describe $C_j$ as the ``geometric coherence'' between an exterior shape $\mathcal{S}_j$ and $\mathpzc{R}_{\Ip,\In}$.

An intuitive interpretation of Theorem \ref{thglast} is upon having a dense grid of fixed cells ($n_\Omega \sim |\Gamma_0\cup\Gamma_1|\gg n_s$), when the exterior shapes are posed in an unstructured way (to satisfy the row-rank constraint) and maintain a sufficiently small level of overlap with the elements of $\mathpzc{R}_{\Ip,\In}$, we can expect to identify the composition elements through the proposed convex formulation. As a matter of fact, it can be shown that the exterior shapes completely disjoint with $\bigcup_{j=1}^{n_\oplus+n_\ominus}\mathcal{S}_j$ are even exempt from the row-rank condition of the matrix presented in (\ref{eqrowrank}).

In view of $ \boldsymbol{w}_{\Gamma_0^\mathpzc{R}} \preceq \boldsymbol{1}\preceq |\boldsymbol{w}_{\Gamma_1^\mathpzc{R}}|$, overlapping of an exterior shape with the null-valued shapelets maintains a lower risk of violating (\ref{eqcoh}) than overlapping with the unit-valued shapelets. Moreover, since the bearing constants take opposite signs over the null and unit-valued shapelets, an exterior shape that overlaps with both class of shapelets yet has the potential to maintain a small geometric coherence. While we have specifically focused on the case that the LOC holds for $\Sigma$, a line of argument similar to the proof of Theorem \ref{thglast} would allow us to talk about the possibility of identifying the composition elements when the entries of the LOC violation vector (see equation (\ref{ee-10})) are sufficiently small.

An interesting problem arises when the exterior shapes are constructed through a random selection of the fixed cells. Bringing randomness into the problem would allow us to derive more qualitative results regarding the performance of the convex formulation in different setups, as well as relating the minimizers of convex proxy to the outcomes of the Cardinal-SC problem under certain regimes. While this type of analysis is very well-established in the compressive sensing community \cite{candes2006robust, foucart2013mathematical}, we would leave that as a future work due to the present load of the paper.

Finally, an essential piece of information that plays a key role in the Sparse-CSC formulation is the value of $\tau$. Based on the suggested formulation in (\ref{eqcon2}), setting $\tau$ to $\tau^* = \|\mathcal{A}(\{\mathcal{S}_j\}_{j\in\Ip\cup\In};\Ip,\In)\|_1$ grants the possibility of recovering the constituting elements of a composition. Generally, the value of $\tau^*$ is not known a priori, however, the fact that for many basic compositions the entries of $\balpha_\mathpzc{R}$ are simply integer quantities, limits the selection of $\tau^*$ among integer possibilities. In the simplest case, when $\In = \emptyset$, the value of $\tau^*$ is  $|\Ip|$, the number of elements in the composition. In the next section we will present some related simulations which demonstrate that trying successive integer values for $\tau^*$ may allow us to control the number of active elements in the recovered composition.

\section{Simulation Results}\label{sec5}

In this section we test the efficiency of the proposed scheme in some basic examples. For the current simulations we have used the CVX Matlab toolbox \cite{cvx}, and have tried to maintain a moderate size for the shape dictionary\footnote{Examples of the code are available at: \href{http://users.ece.gatech.edu/aaghasi3/software.html}{http://users.ece.gatech.edu/aaghasi3/software.html}}. Exploring more extensive simulations (emphasizing on the optimization tools to address the Sparse-CSC problem) remains a future work.


The inhomogeneity measure used throughout the simulations is the basic Chan-Vese, where $\Pi_{in/ ex}(x)=(u(x)-\tilde u_{in/ ex})^2$. For images that are close to being binary, the values of $\tilde u_{in}$ and $\tilde u_{ex}$ are fixed and simply set to the minimum and maximum intensity values in the image. For the more noisy images (later specified in the context), $\tilde u_{in}$ and $\tilde u_{ex}$ are naively selected to be the 15\% and 85\% quantiles of the image histogram. In other words, the promising results presented are obtained with the least effort on optimizing the texture measures, and future developments in this area would further strengthen the technique.

\subsection{Basic Image Segmentation}
As the first experiment, we consider an image segmentation problem where very little information about the image content is available. In this case the elements of the shape dictionary are simply selected to be square blocks of the same size distributed throughout the imaging domain.

A test image of size $120\times 100$ pixels is shown in Figure \ref{figsegmentation}(a). Each element of the shape dictionary is a square of size $15\times 15$ pixels, centered at a certain location in the domain. The shape centroids are taken to be on a uniform $40\times 30$ grid over the imaging domain (i.e., $3\times 3\frac{1}{3}$ pixels spacing in the vertical and horizontal directions), which together result in a dictionary of size $n_s = 1200$ shape elements.

\begin{figure*}
\centering
\begin{tabular}{ccc}
&\hspace{-.5cm}\includegraphics[width=47mm]{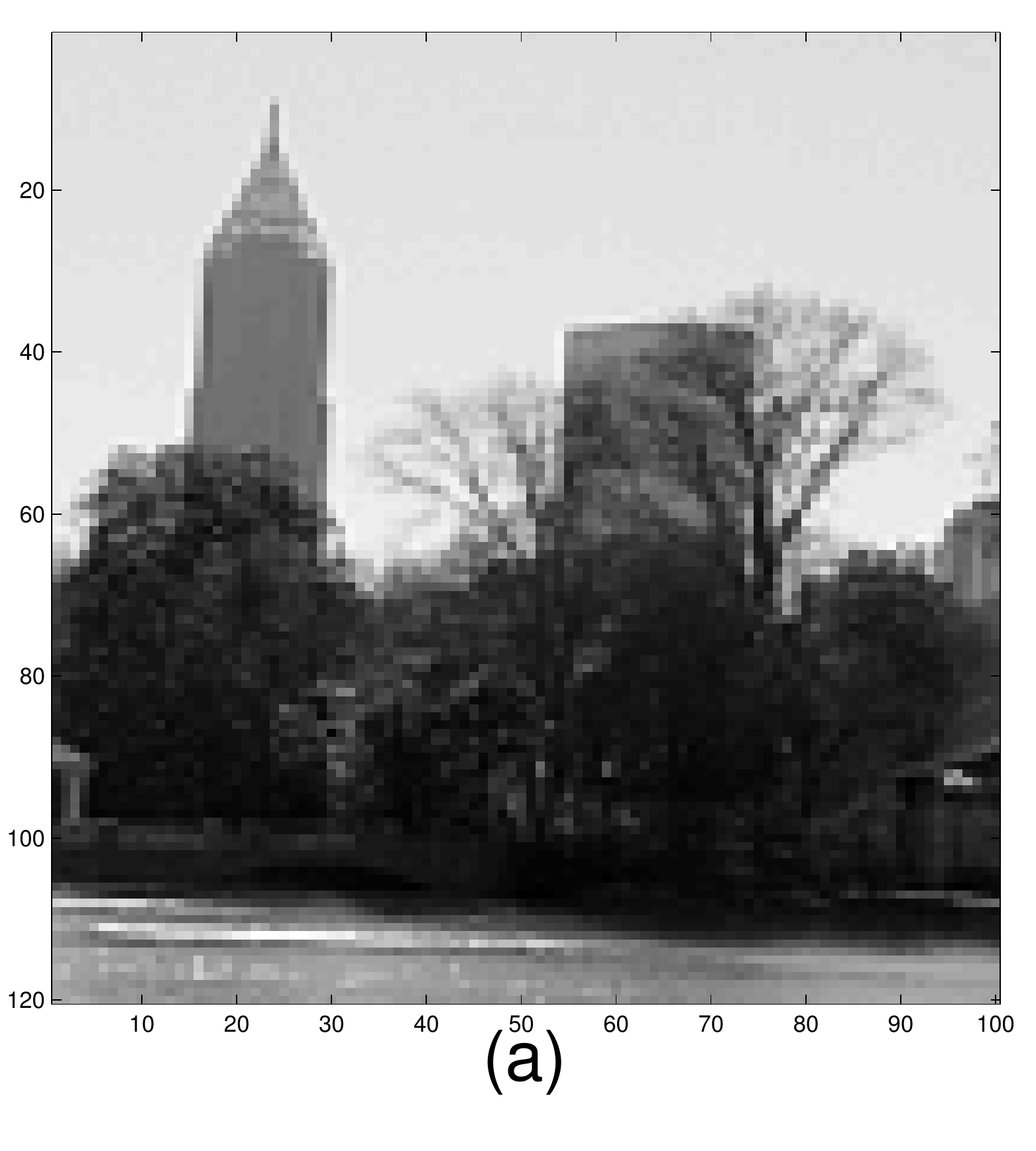}& \\[-.2cm]
\includegraphics[width=42mm]{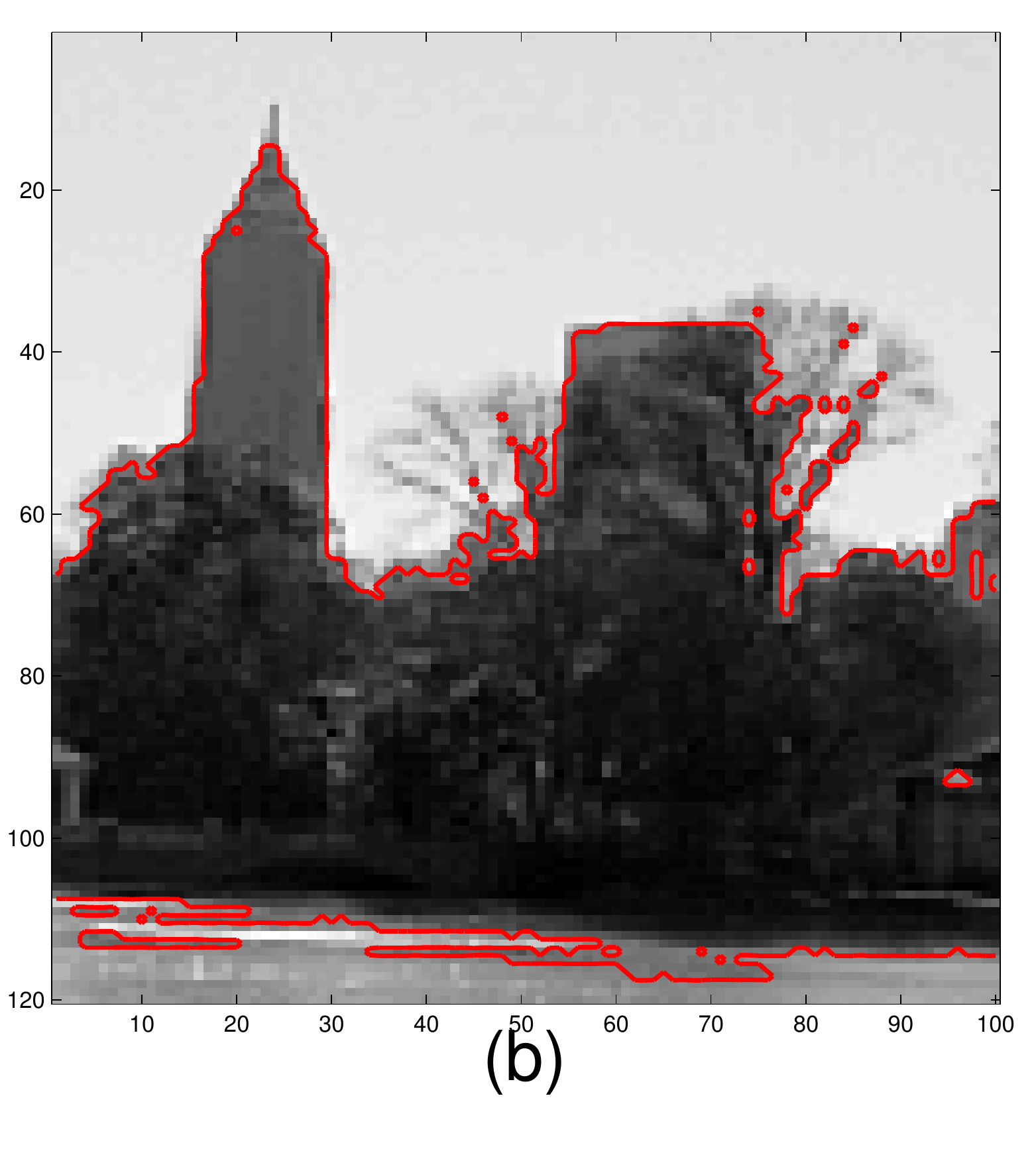}&\hspace{-.5cm}\includegraphics[width=42mm]{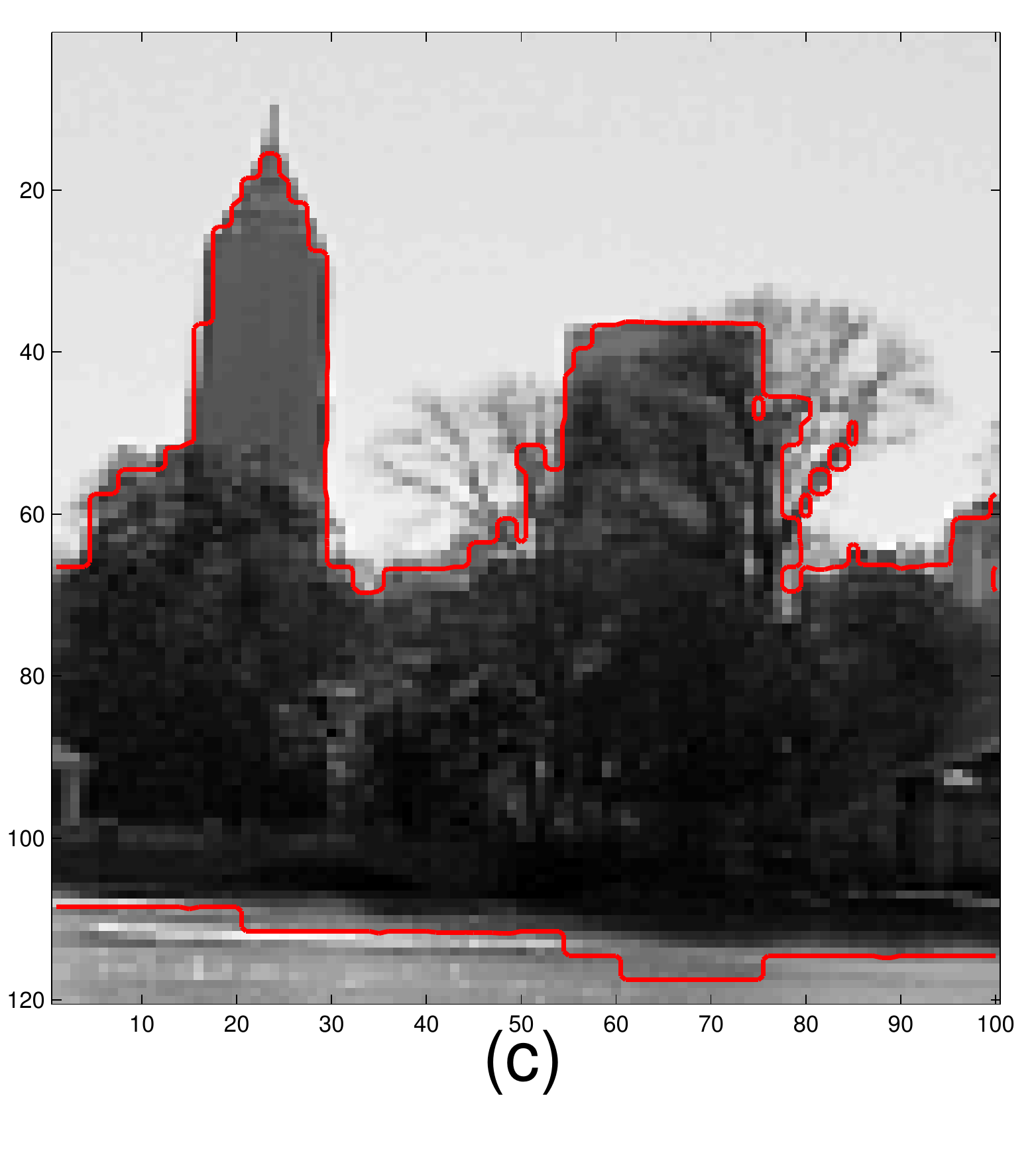}&\hspace{-.5cm}\includegraphics[width=42mm]{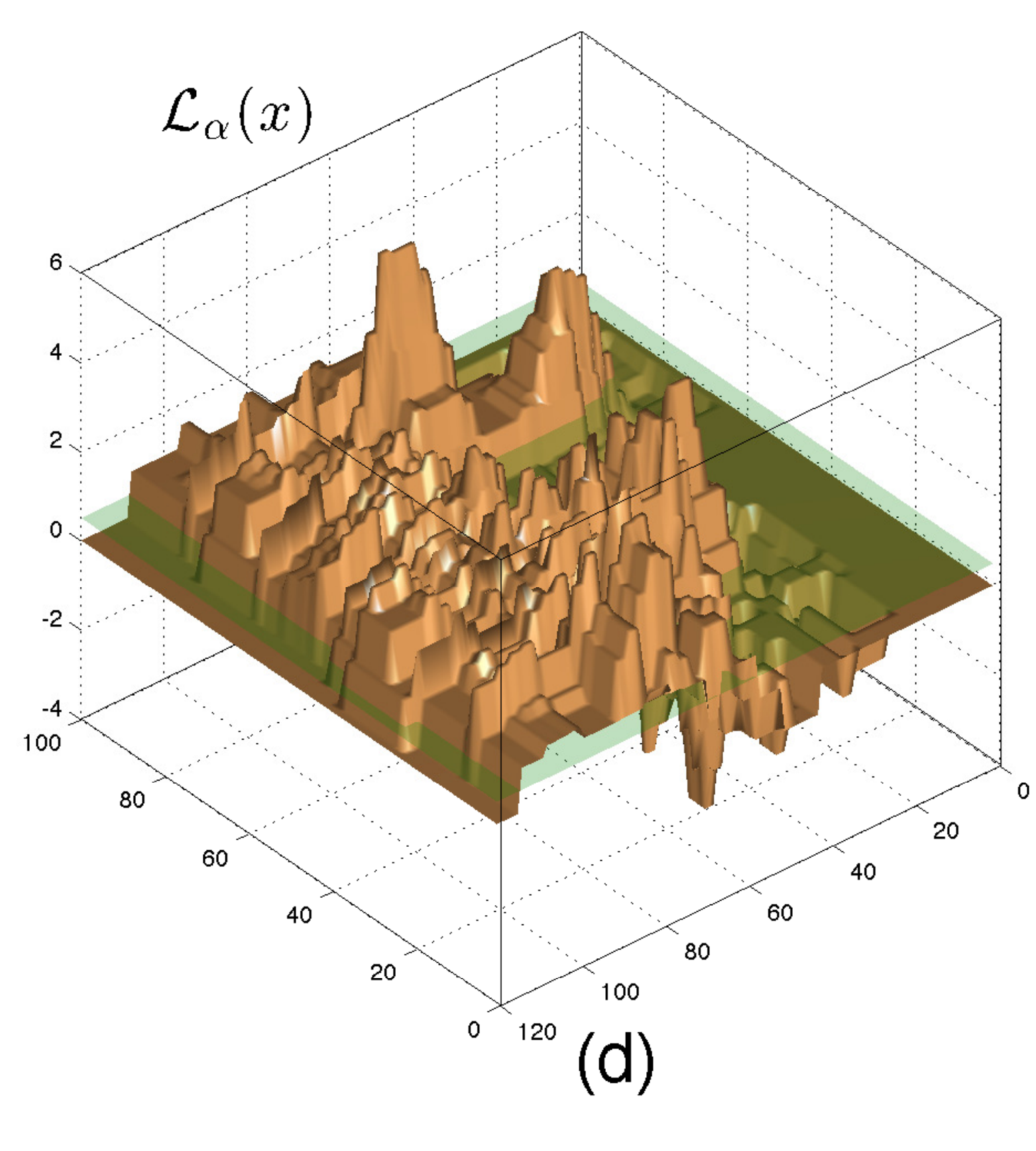} \\[-.2cm]

\includegraphics[width=42mm]{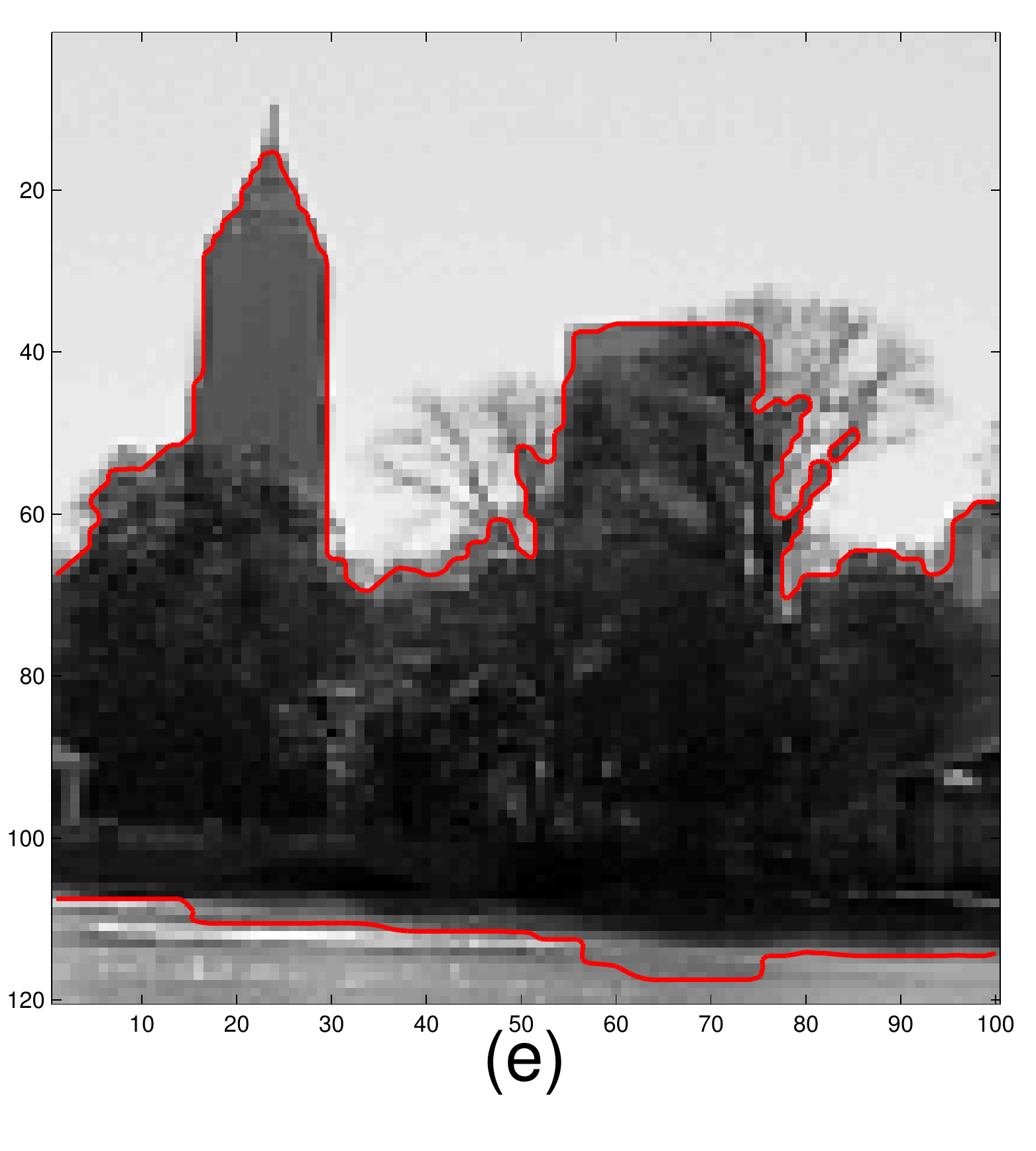}&\hspace{-.5cm}\includegraphics[width=42mm]{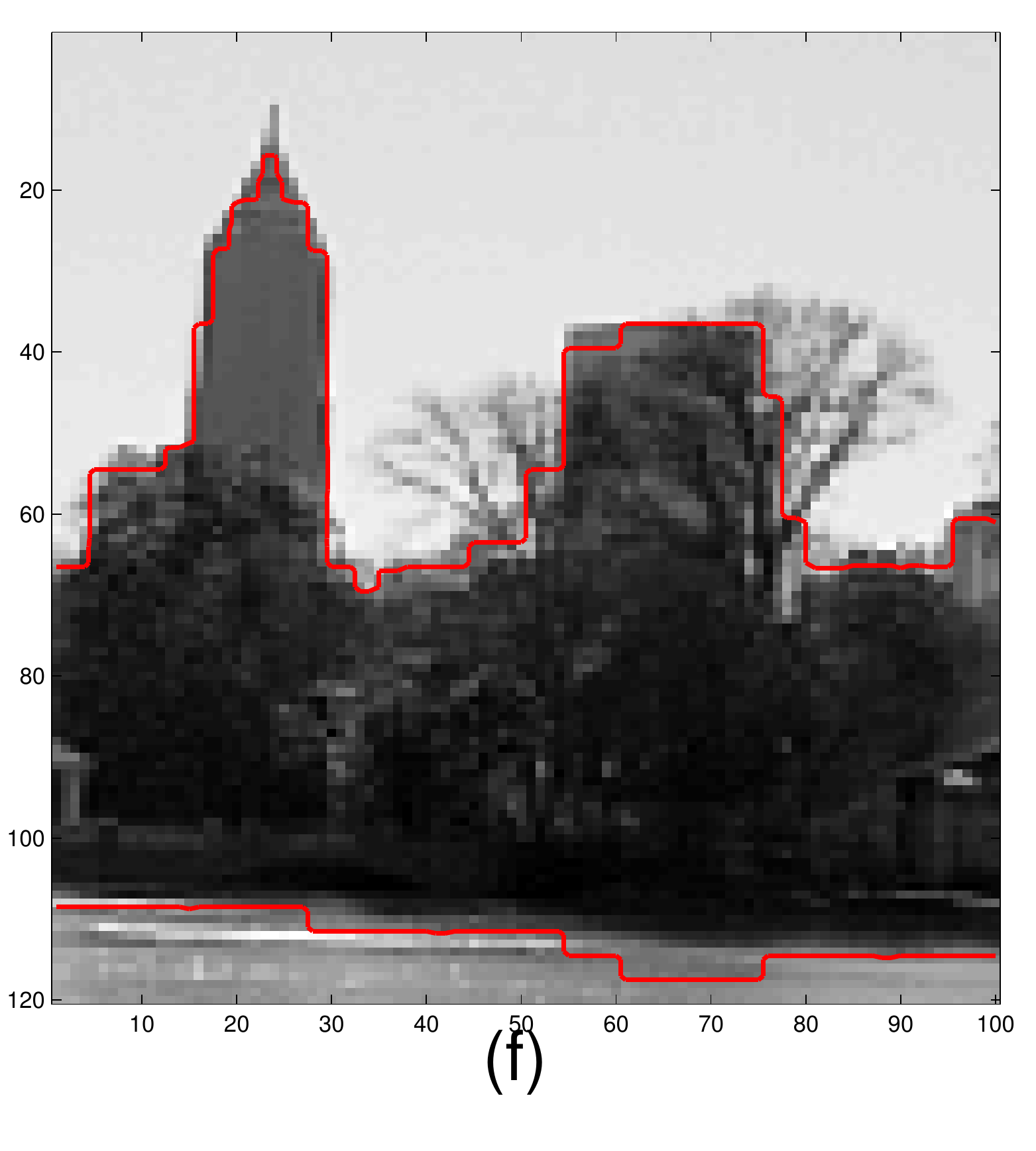}&\hspace{-.5cm}\includegraphics[width=42mm]{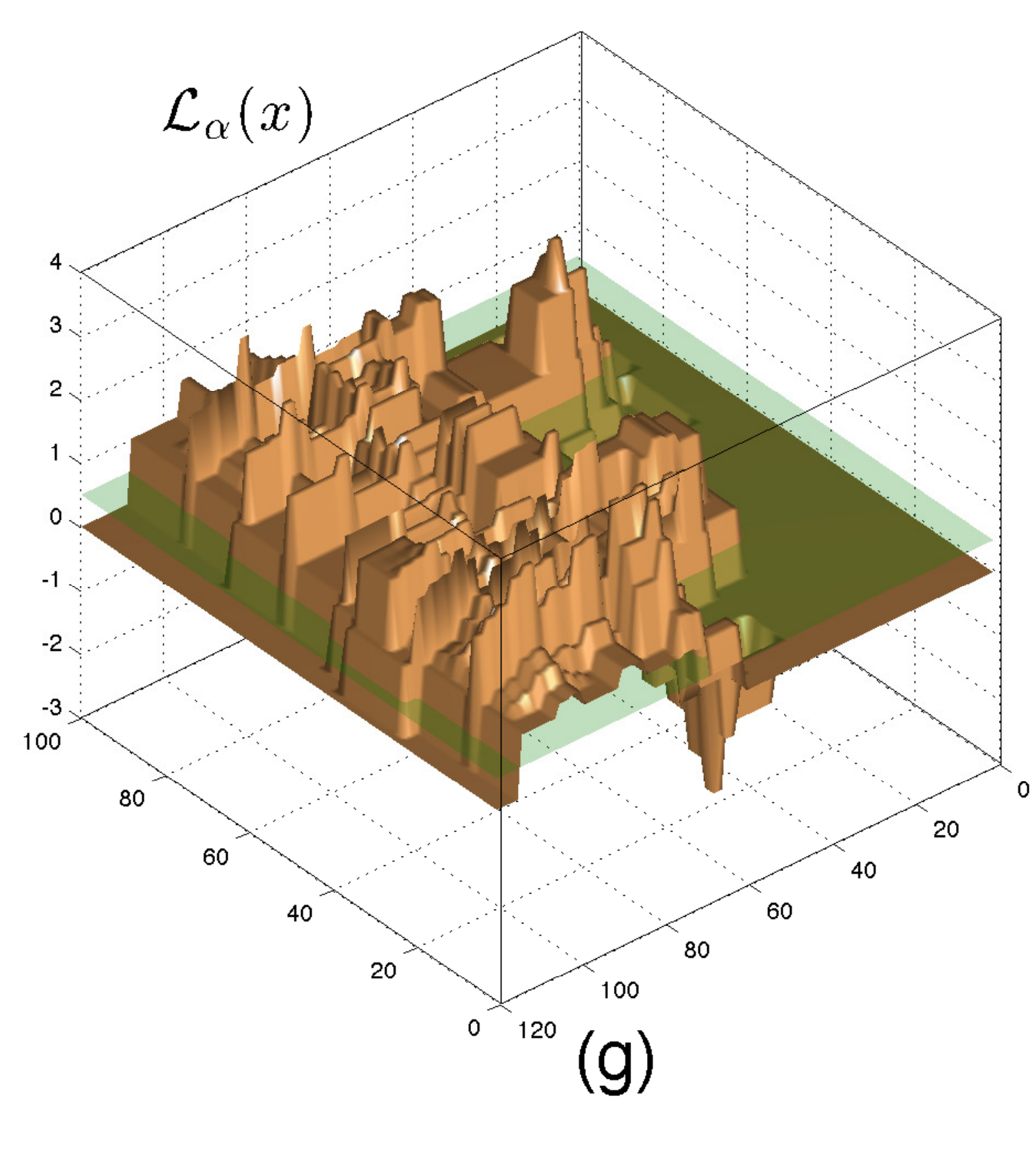} \\[-.2cm]

\includegraphics[width=42mm]{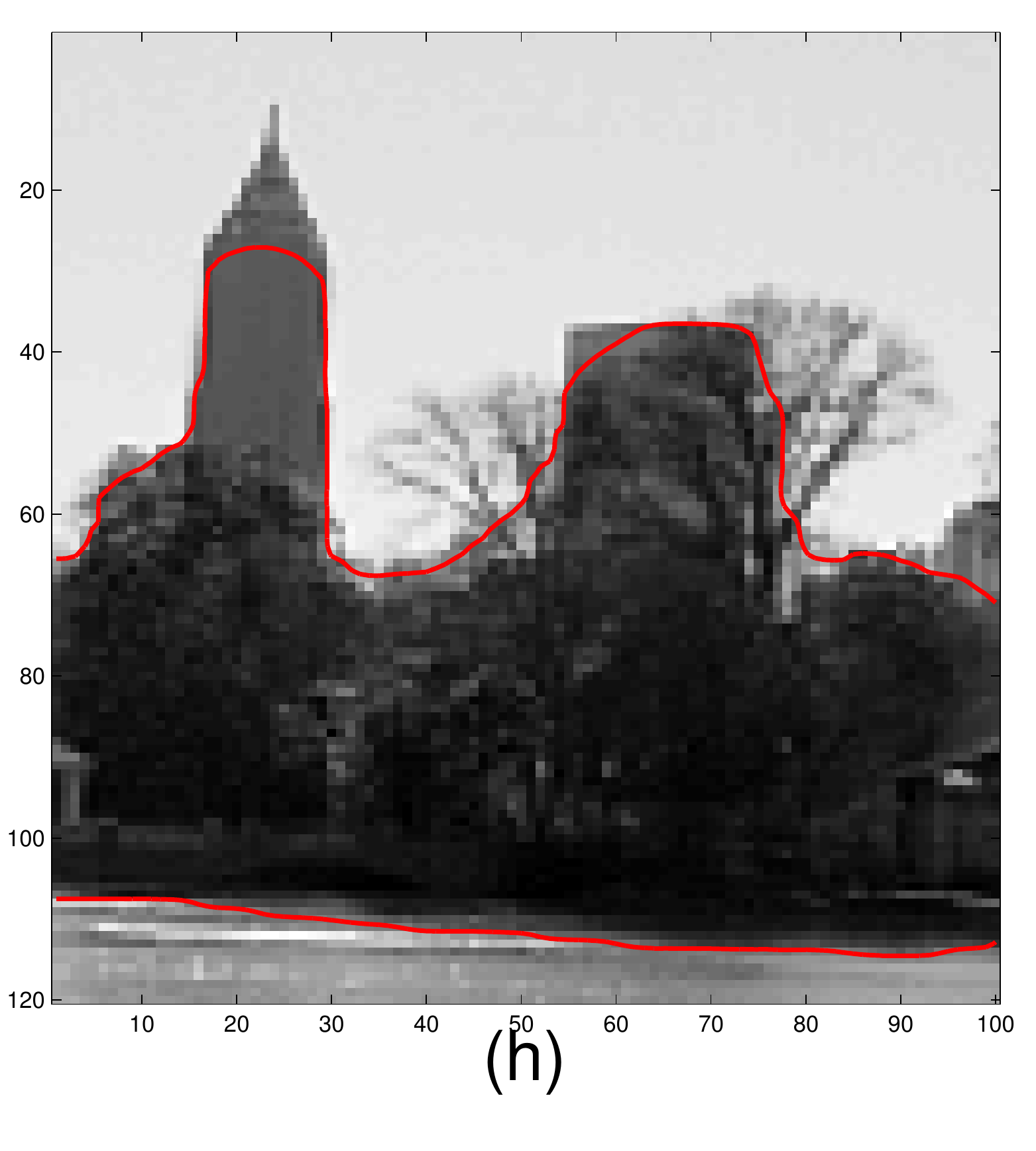}&\hspace{-.5cm}\includegraphics[width=42mm]{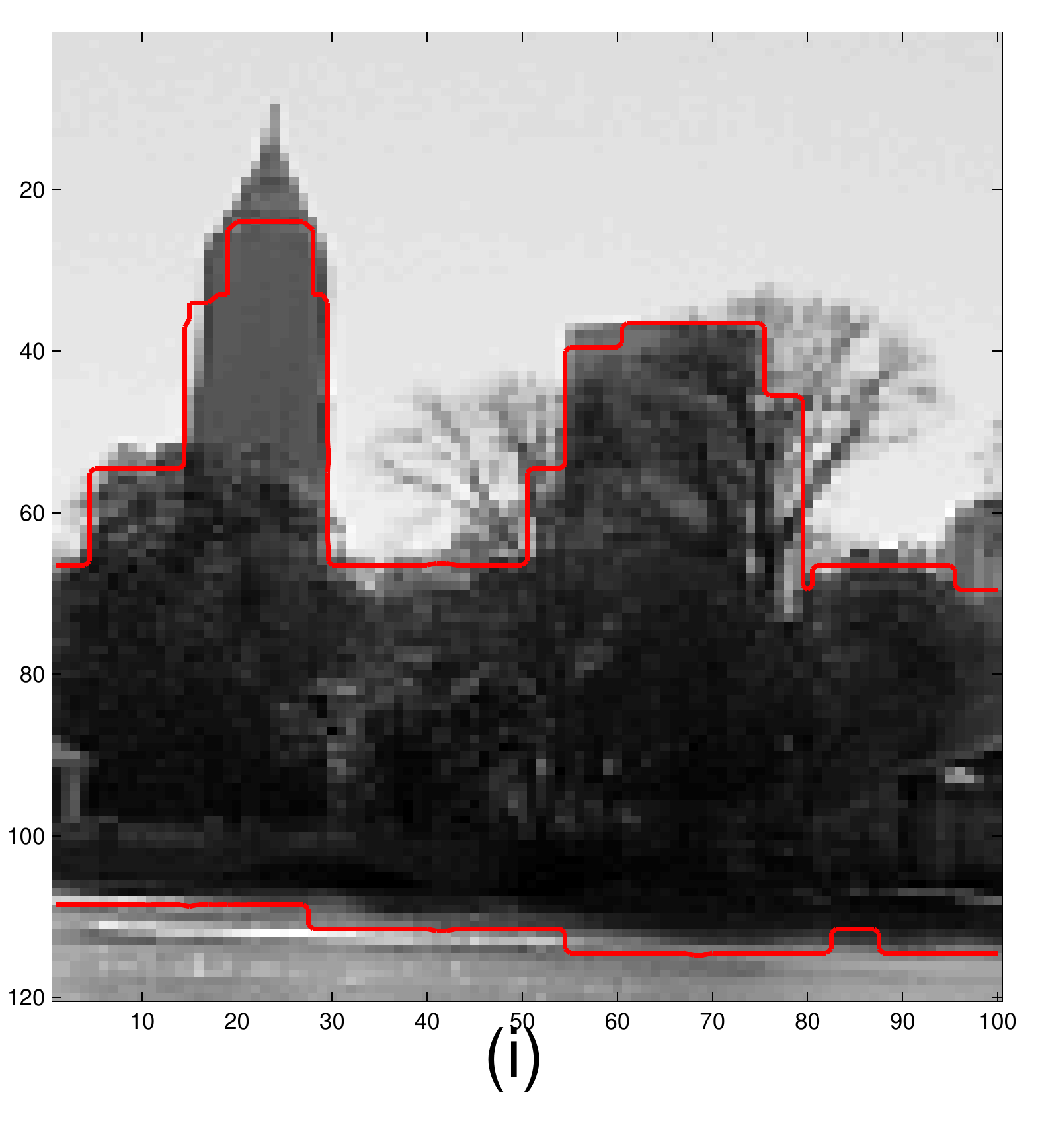}&\hspace{-.5cm}\includegraphics[width=42mm]{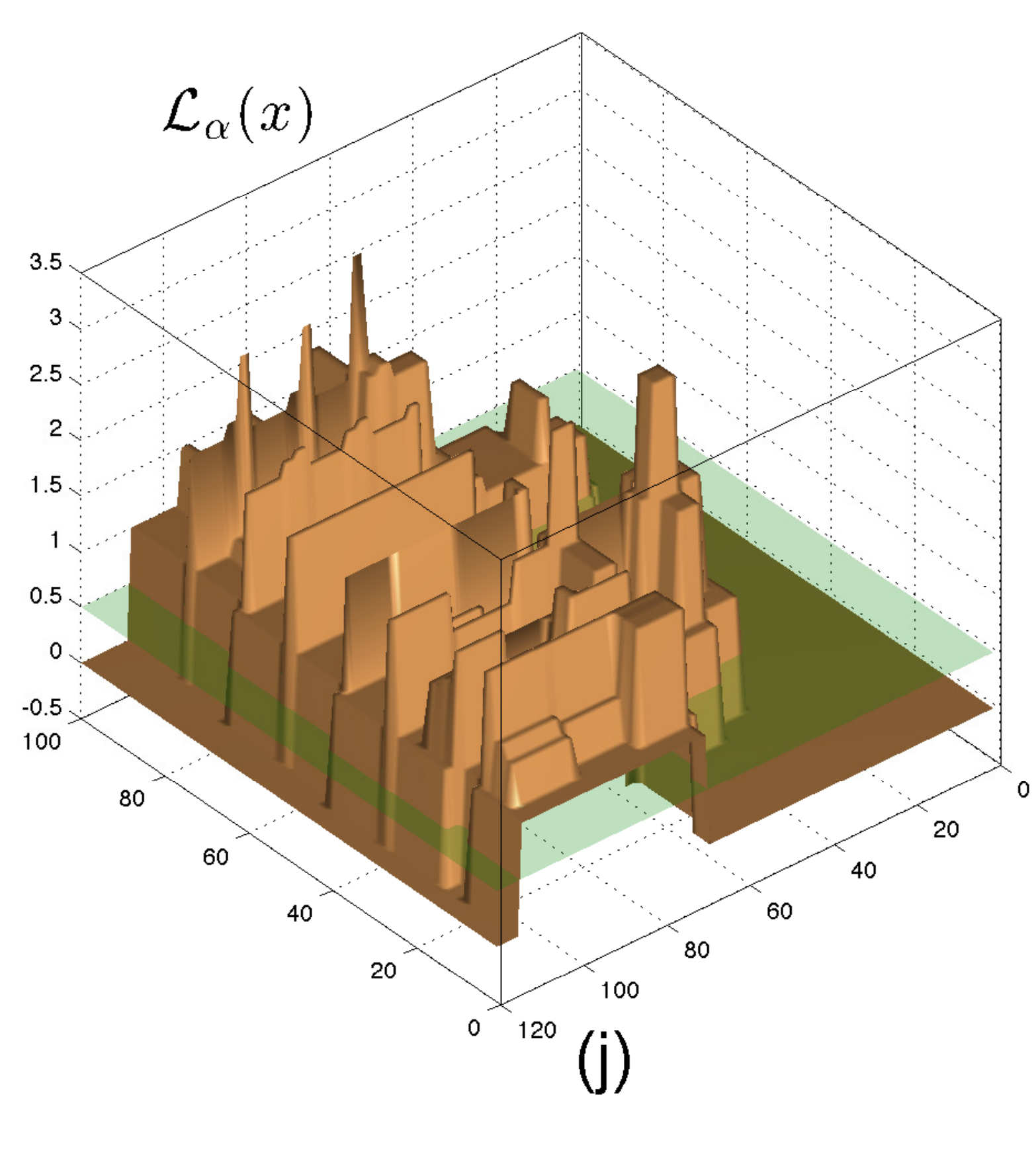}
\end{tabular}\vspace{-.3cm}
\caption{Segmentation results on the test image in panel (a); left column corresponds to the CC-CV model; middle column corresponds to the Sparse-CSC model and the right column shows $\mathpzc{L}_{\boldsymbol{\alpha}}(x)$ for each result in the middle column; (b,c) small penalty: $\lambda_s = 10^{-3}$, $\lambda=10^{-2}$; (e,f) medium valued penalty: $\lambda_s = 10^{-1}$, $\lambda=1$; (h,i) large penalty: $\lambda_s = 1$, $\lambda=10$;}\label{figsegmentation}
\end{figure*}

Since an object identification is not a primary goal of this problem and we are not seeking to match the content of the image with the dictionary elements, we simply use the regularized form of the Sparse-CSC model, parametrized by $\lambda$, as specified in (\ref{eq22yy}). The segmentation results are compared with the convex constrained Chan-Vese (CC-CV) model proposed in\cite{chan2006algorithms}, where an optimal partitioner, characterized by $\pi^*(x)$, is obtained via the minimization
\begin{equation}\label{essedcost}
\pi^*(x)= \operatorname*{arg\,min}_{0\leq \pi(x)\leq 1} \int_{D} \big(\Pi_{in}(x) - \Pi_{ex}(x)\big)\pi(x)\;\mbox{d}x + \lambda_s\int_{D} |\nabla \pi(x)|\;\mbox{d}x.
\end{equation}
The smoothing penalty parameterized by $\lambda_s$ controls the geometric complexity of the resulting partitioner. For both segmentation models the quantities $\tilde u_{in}$ and $\tilde u_{ex}$ are  set to be the 15\% and 85\% quantiles of the image histogram. To provide a qualitative comparison, the segmentations are performed for low, mid-range and high values of $\lambda$ and $\lambda_s$.

From a qualitative stand-point, a comparison between the outcomes of the CC-CV model and the Sparse-CSC problem (the left and middle columns in Figure \ref{figsegmentation}) reveals that the latter is capable of producing simpler partitioners and is less sensitive to local details. For instance comparing Figures \ref{figsegmentation}(b,e) with \ref{figsegmentation}(c,f) shows that the Sparse-CSC model does not overreact to small image details such as the tree branches. In fact, the restriction of choosing elements from a dictionary provides a natural regularizer. Specifically, in the case of square blocks (super-pixels) used as the shape elements, one expects to observe a block-wise structure in the resulting partitioners.

Another interesting observation is the outcomes of the two models when a simple geometry is sought (large values of $\lambda$ and $\lambda_s$ are considered). Increasing the length penalty in (\ref{essedcost}) causes the resulting partitioner to bypass the local details at the expense of producing smooth corners (Figure \ref{figsegmentation}(h)), while Sparse-CSC can yet track the details in a block-wise way (Figure \ref{figsegmentation}(i)).

In general, increasing the $\ell_1$ penalty in the Sparse-CSC problem controls the complexity of the resulting $\mathpzc{L}_{\boldsymbol{\alpha}}(x)$ (the right column in Figure \ref{figsegmentation}). A small penalty allows the finer details to be captured and a larger penalty provides a coarse segmentation highlighting the principal shape components. Using the CVX toolbox, each Sparse-CSC segmentation takes slightly more than 4 minutes on a desktop computer with a 3.4 GHz Intel CPU and 16 GB memory. We would like to note that implementing a specific solver using the subgradient information (e.g. \cite{alber1998projected}) can significantly reduce the computation time, however, given the current load and theoretical focus of this work, such extension is left to a future work. The remaining set of experiments take advantage of the shape identification property of the Sparse-CSC problem.

\subsection{Image Segmentation with Missing Pixels}
As a more challenging example, we consider an image segmentation problem with a significant portion of the pixel values missing. Basically, the information about the pixels is only available over a subset of the domain, $D'\subset D$, while the segmentation needs to take place, globally. Prior information about the geometry of the objects in the image would allow us to yet perform a successful segmentation, as well as identifying the principal shape components inside the image.

\begin{figure*}
\centering
\begin{tabular}{ccc}
\includegraphics[width=85mm, height = 86mm]{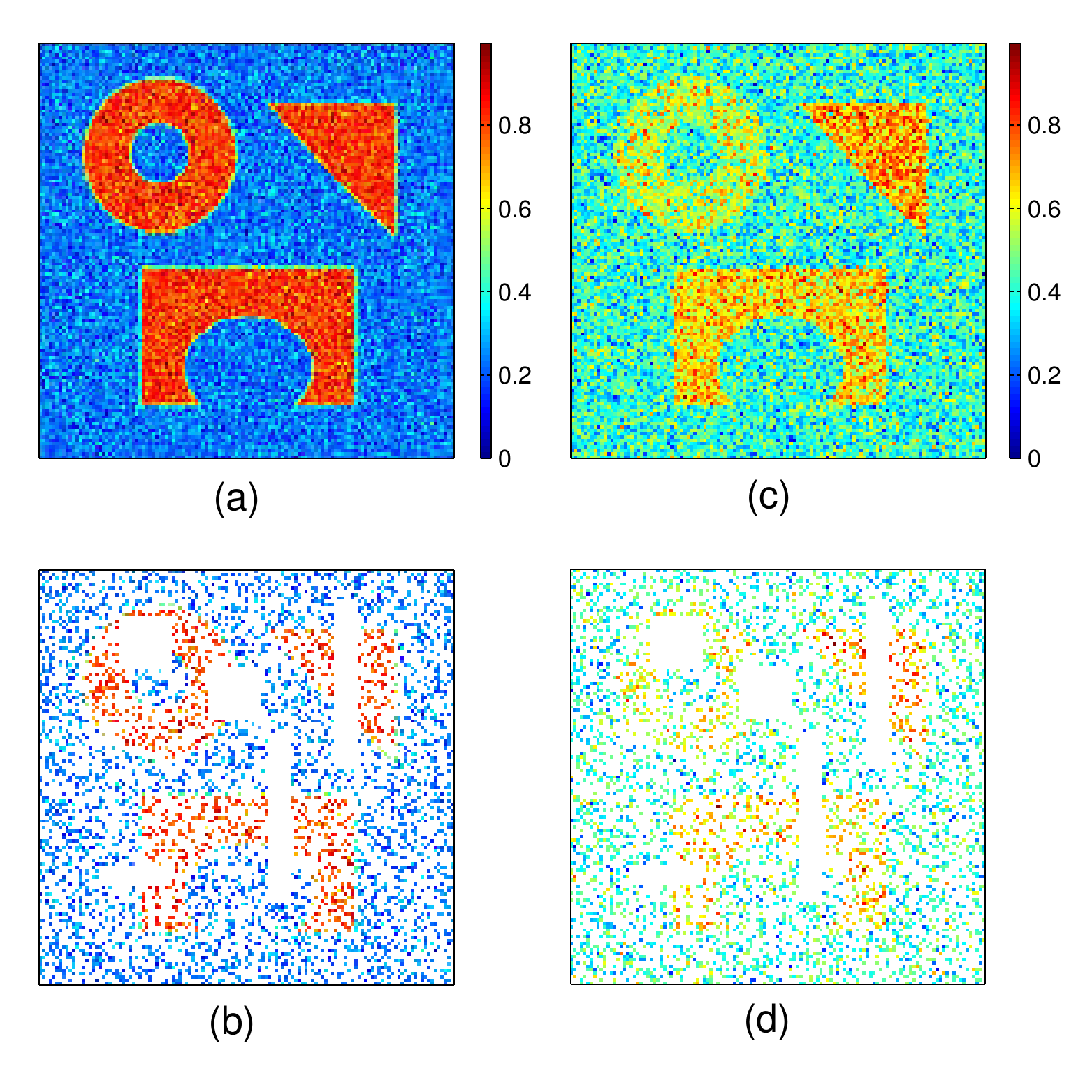}
\includegraphics[width=43mm, height = 86mm]{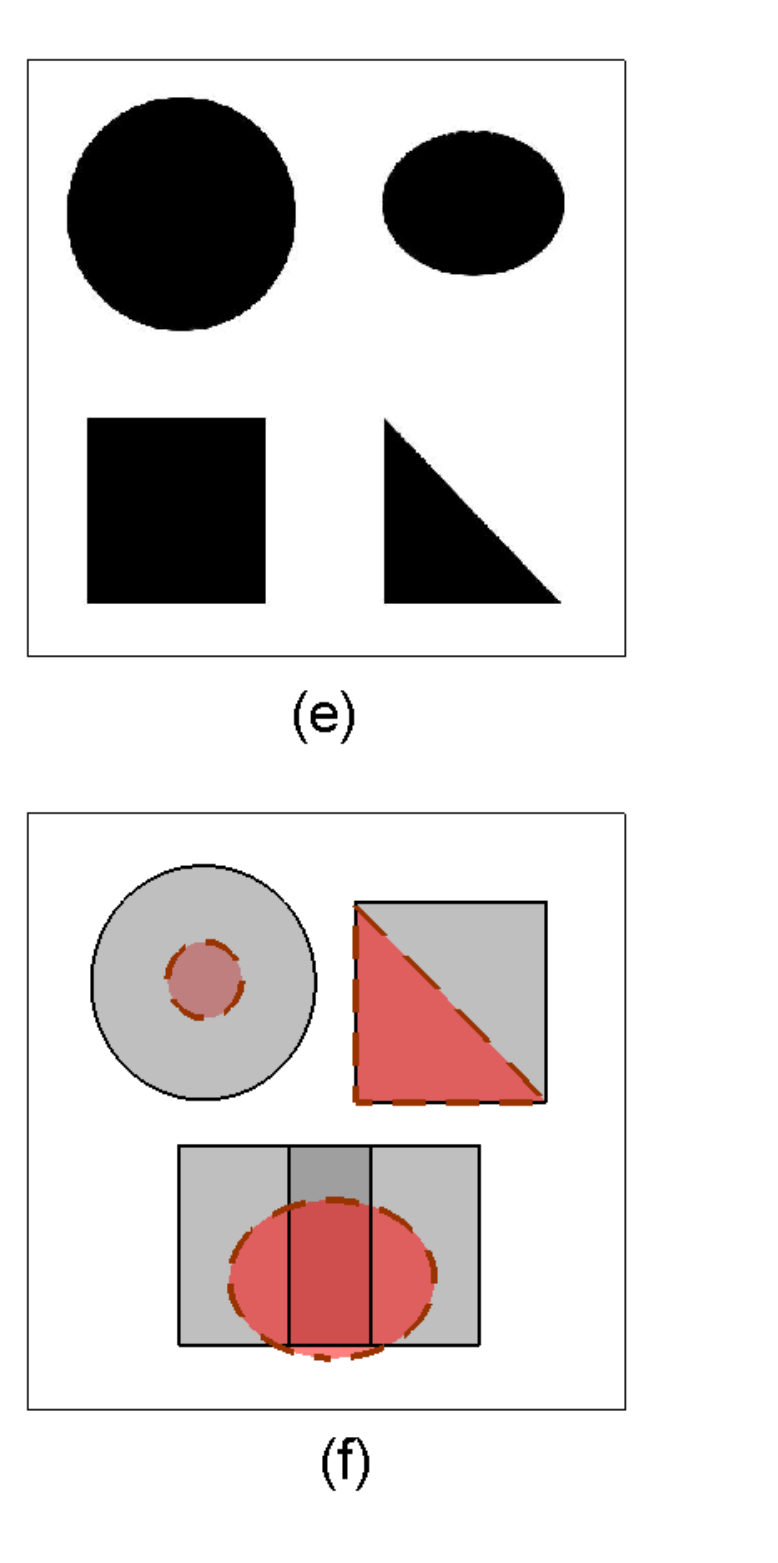}
\end{tabular}\vspace{-.5cm}
\caption{(a) A noisy grayscale image to be used as a reference; (b) The image used for the first set of simulations, with the majority of the pixels missing as patches and random samples; (c) A more noisy image to be used as a reference; (d) The image used for the second set of simulations, with a similar pattern of missing pixels; (e) The shapes used to build up the dictionary: instances of these four shapes with different sizes are placed throughout the imaging domain; (f) The way the shapes in the dictionary need to be combined to form the objects in the reference image. Taking away the red shapes (with dashed-line boundaries) from the gray ones would provide us with the target objects}\label{fig8}
\end{figure*}

Figure \ref{fig8}(a) shows a noisy image before significantly discarding the pixel information, and solely presented as a reference image. The Gaussian noise causes a relative discrepancy of approximately 24\% compared to the ideal binary image. The pixel removal applied is in the form of rectangular patches and random samples. The resulting image is shown in Figure \ref{fig8}(b), which will be used in the first set of experiments.

Figure \ref{fig8}(c) shows a second noisy image, where the mean intensities differ for the objects present in the image. The Gaussian noise is higher in this case, causing a relative discrepancy of approximately 90\% compared to the ideal piecewise-constant image. The pixel removal pattern is identical to the pattern in Figure \ref{fig8}(b). The resulting test image is presented in Figure \ref{fig8}(d), which will be used in a second set of experiments.

Our prior information about the geometry of objects in the image is reflected in the selection of the dictionary elements. To build up the dictionary we make use of four basic shapes: a circle, square, triangle and an ellipse, as shown in Figure \ref{fig8}(e). The shape dictionary consists of $n_s=999$ instances of these four shapes at various size and locations. In Figure \ref{fig8}(f), we have shown how a suitable composition of the objects in the dictionary would ideally reconstruct the objects in the reference image. It is however worth noting that, aside from the large missing pixel information, there is still a misalignment between the shapes in the dictionary and the objects appeared in the image.

Figure \ref{fig9} presents the segmentation results on the first test image following the preceding setup. To highlight the elegant performance of the method, we have shown the results for various selections of $\tau$ in the optimization. The values of $\tau$ are successively increased by integer steps until a full reconstruction is accomplished for $\tau=8$ (Figure \ref{fig9}(d)). It is striking to see how the objects are identified one after the other, as $\tau$ increases and the convenience in tuning this parameter as an integer parameter.

\begin{figure*}
\centering
\begin{tabular}{c}
\includegraphics[width=127mm]{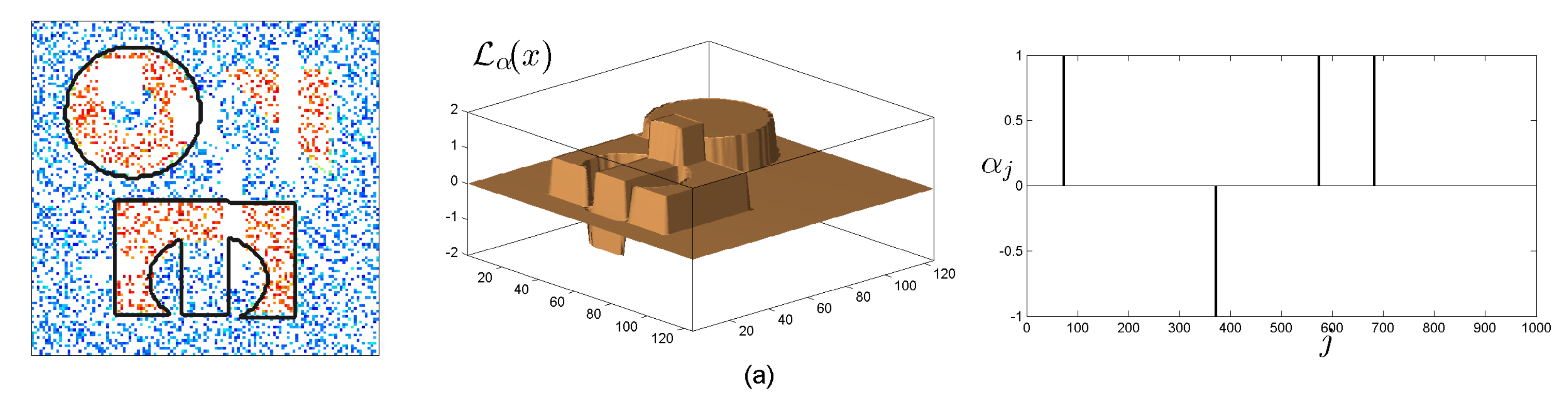} \\[-.2cm]
\includegraphics[width=127mm]{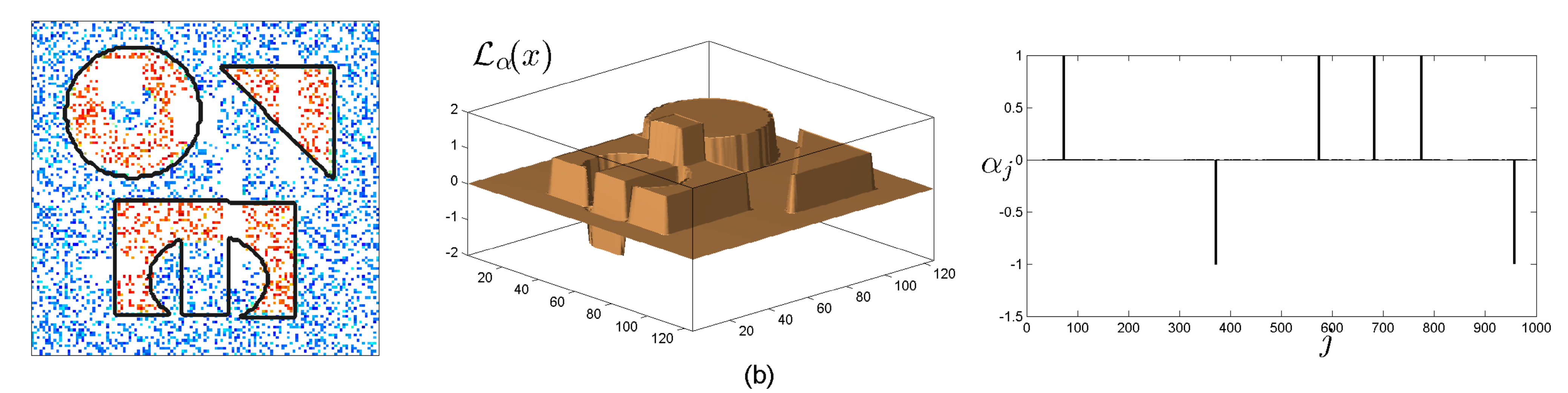} \\[-.2cm]
\includegraphics[width=127mm]{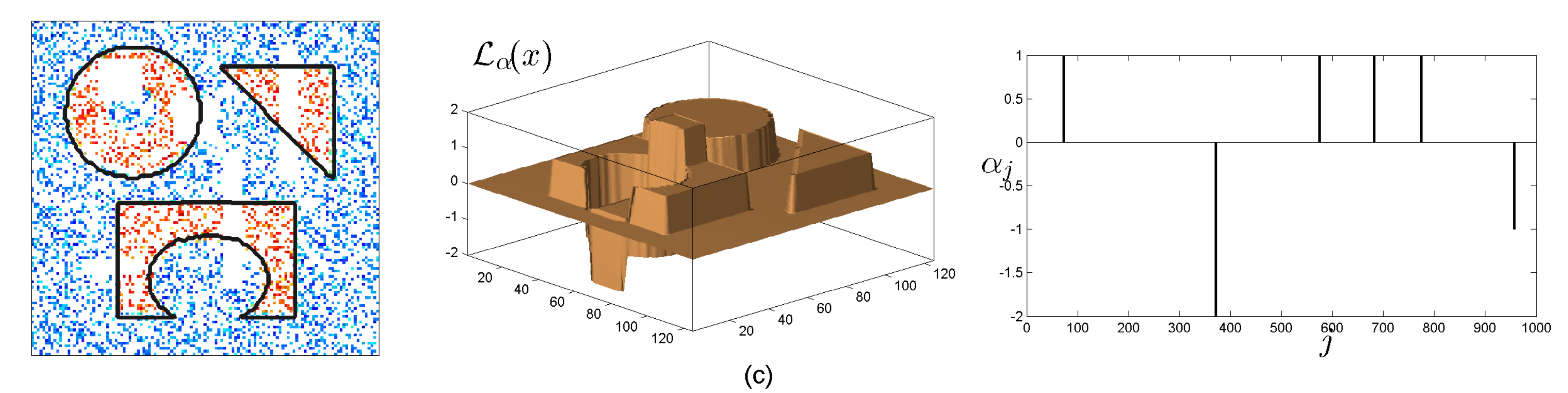} \\[-.2cm]
\includegraphics[width=127mm]{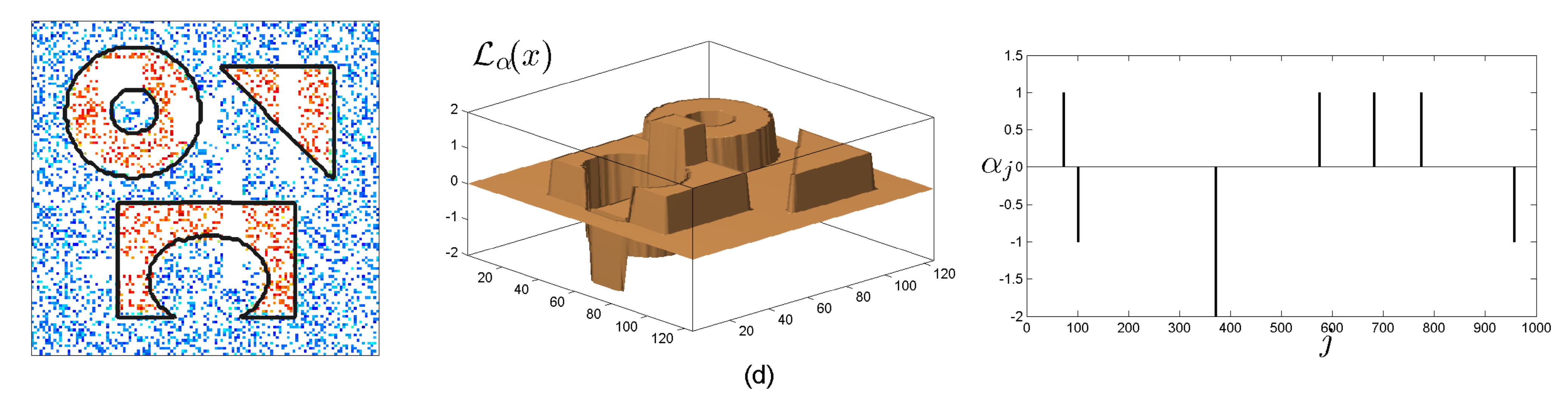} \\[-.2cm]
\includegraphics[width=127mm]{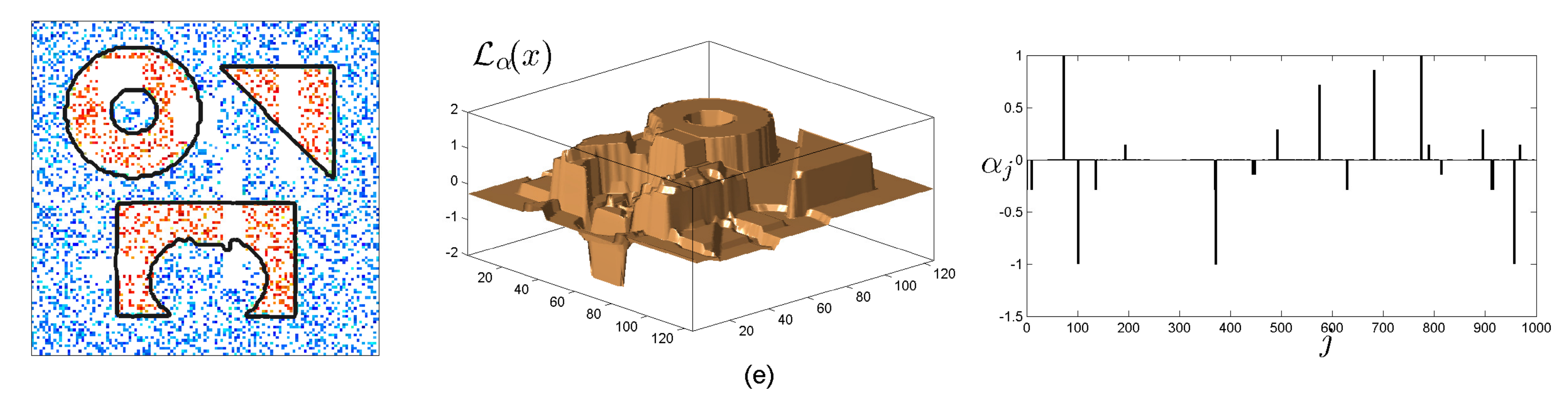} \\[-.2cm]
\includegraphics[width=127mm]{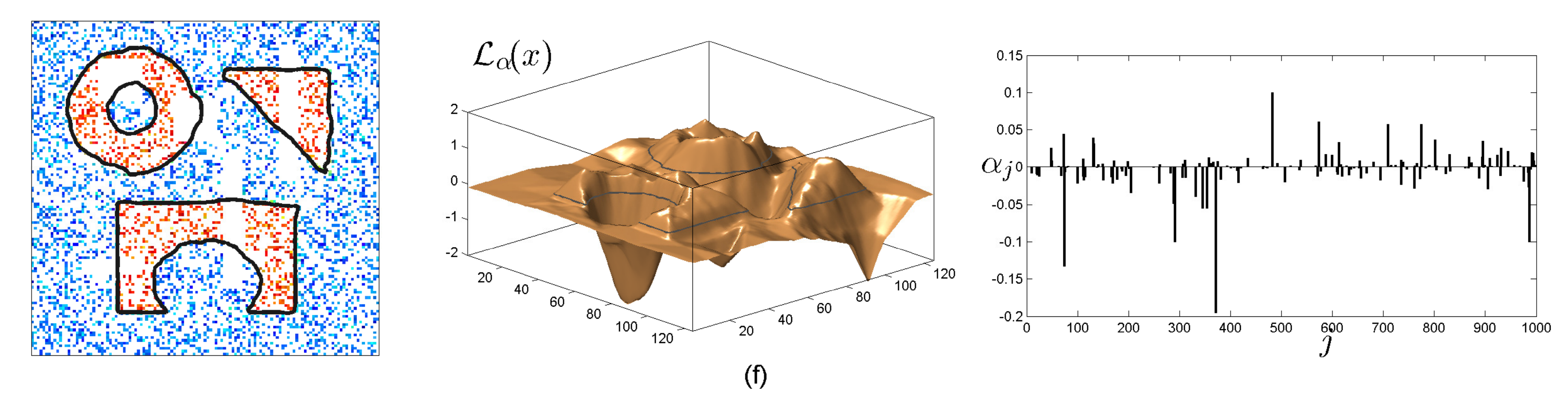} \\[-.2cm]
\end{tabular}
\caption{The segmentation results for the test image in Figure \ref{fig8}(b);
results of different settings presented on the left column, the corresponding $\mathcal{L}_{\balpha}(x)$ depicted in the middle column and the resulting $\balpha$ vector plotted at the right column. (a) $\tau = 4$; (b) $\tau = 6$; (c) $\tau = 7$; (d) $\tau = 8$; (e) $\tau = 10$; (f) Using the technique proposed in \cite{aghasi2013sparse}}\label{fig9}
\end{figure*}

From a technical standpoint, based on the composition shown in Figure \ref{fig8}(d), the value of $\|\mathcal{A}\big(\{\mathcal{S}_j\}_{j\in\Ip\cup\In};\Ip,\In\big)\|_1$ is 8. Setting $\tau$ to this value yields a set of objects that are comprised of very few geometric components and yet reasonably perform the segmentation task. While the optimal value of $\tau$ is not expected to be known a priori, the fact that in many scenarios this value is simply an integer facilitates the algorithm with a convenient tuning process.

Figure \ref{fig9}(e) depicts the segmentation result when $\tau$ steps beyond 8, yet indicating that the segmentation maintains a reasonable geometry and only few more shapes have become active to further reduce the convex cost. In Figure \ref{fig9}(f) we have employed the method presented in \cite{aghasi2013sparse}, highlighting that due to the lack of control over the sparsity level and the non-convex nature of the approach taken in \cite{aghasi2013sparse}, the results stand in a lower quality level as the ones acquired using the presented technique. Extracting the principal shape components from the outcomes of Figure \ref{fig9}(f) does not necessarily lead us to a reliable result.

Figure \ref{fig9_2} presents the segmentation results for the second test image, where the noise corruption is more severe and the mean intensities vary for different objects. For this case $\tilde u_{ex}$ and $\tilde u_{in}$ are selected to be the 15\% and 85\% quantiles of the partially observed image histogram. Similar to the first set of experiments, incrementing $\tau$ by unit steps causes the shapes to become identified one after the other.

\begin{figure*}
\centering
\begin{tabular}{c}
\includegraphics[width=127mm]{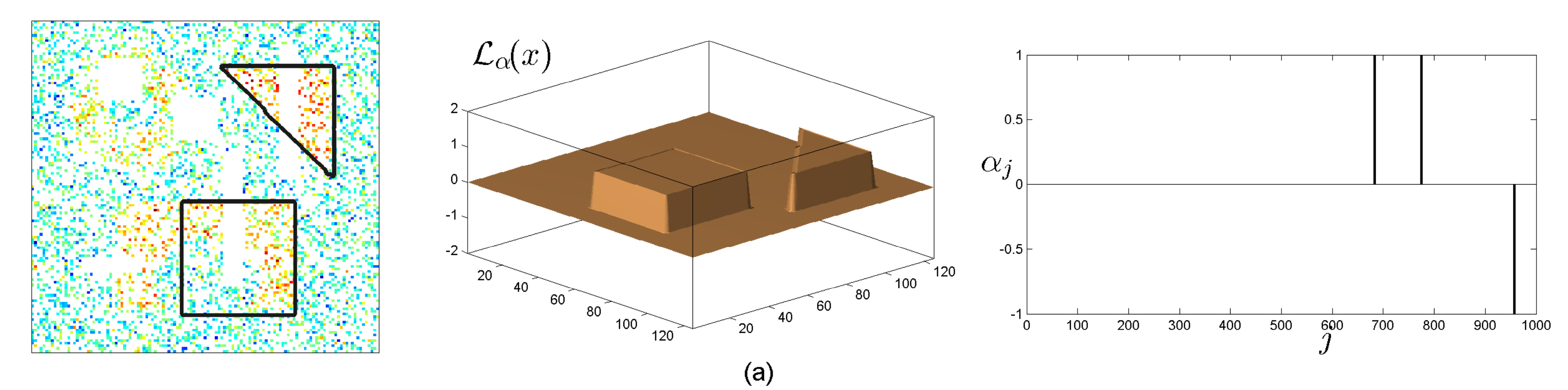} \\[-.2cm]
\includegraphics[width=127mm]{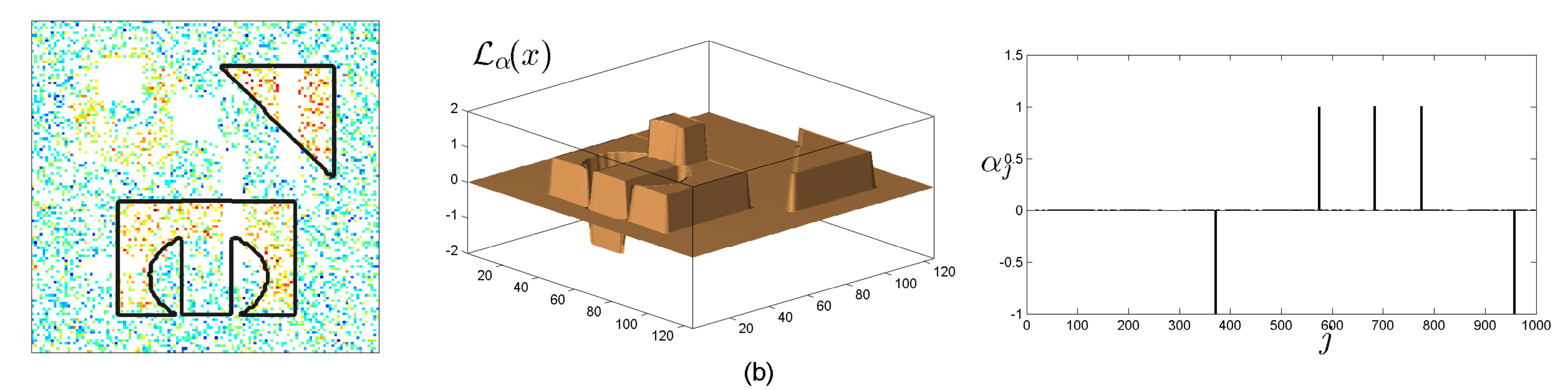} \\[-.2cm]
\includegraphics[width=127mm]{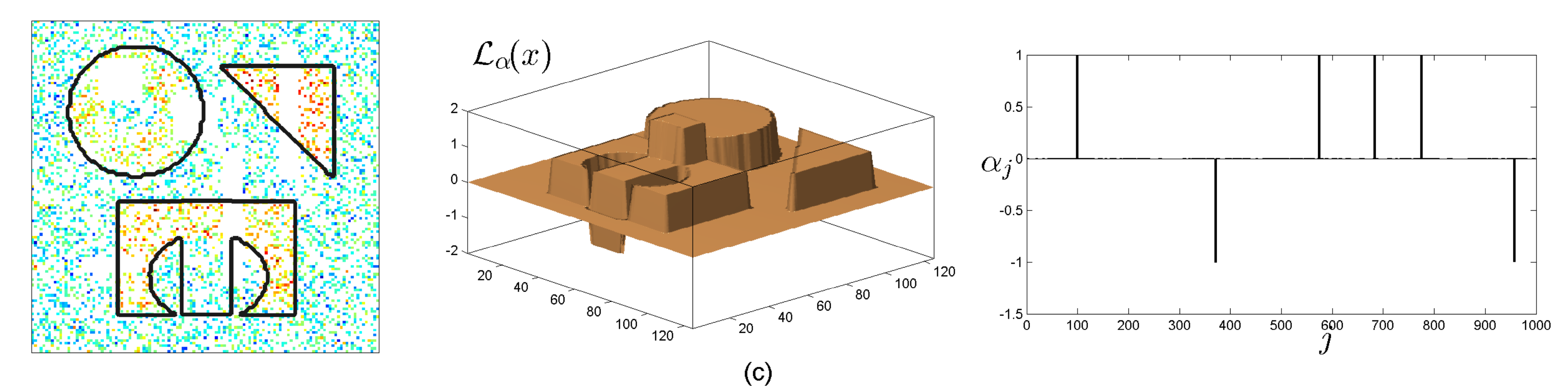} \\[-.2cm]
\includegraphics[width=127mm]{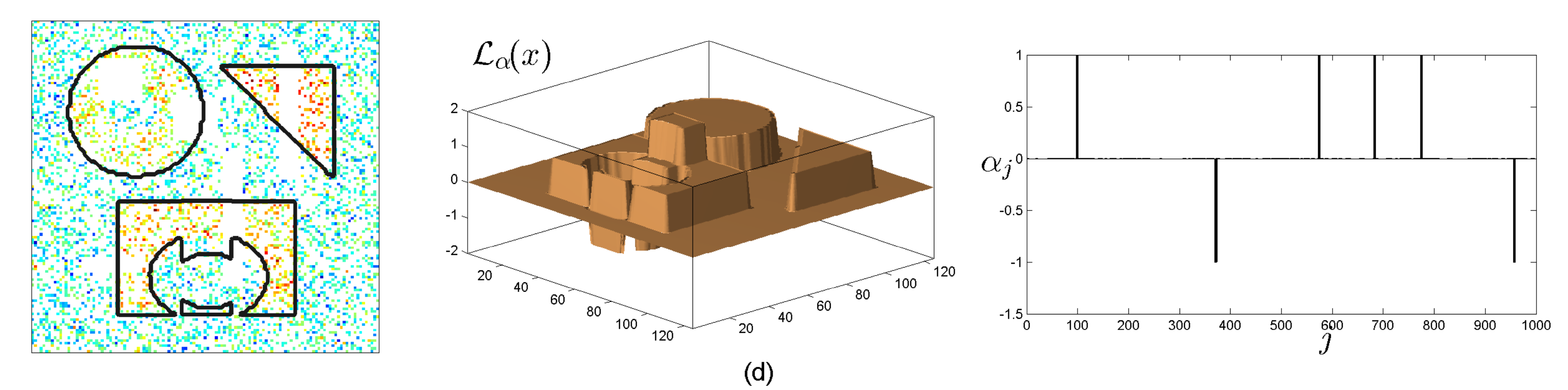} \\[-.2cm]
\includegraphics[width=127mm]{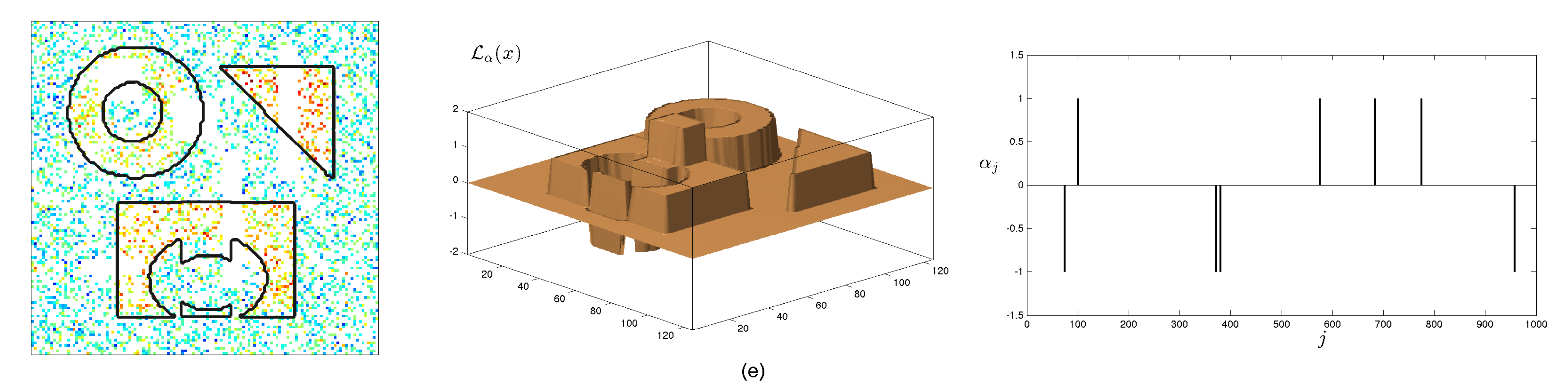} \\[-.2cm]
\includegraphics[width=127mm]{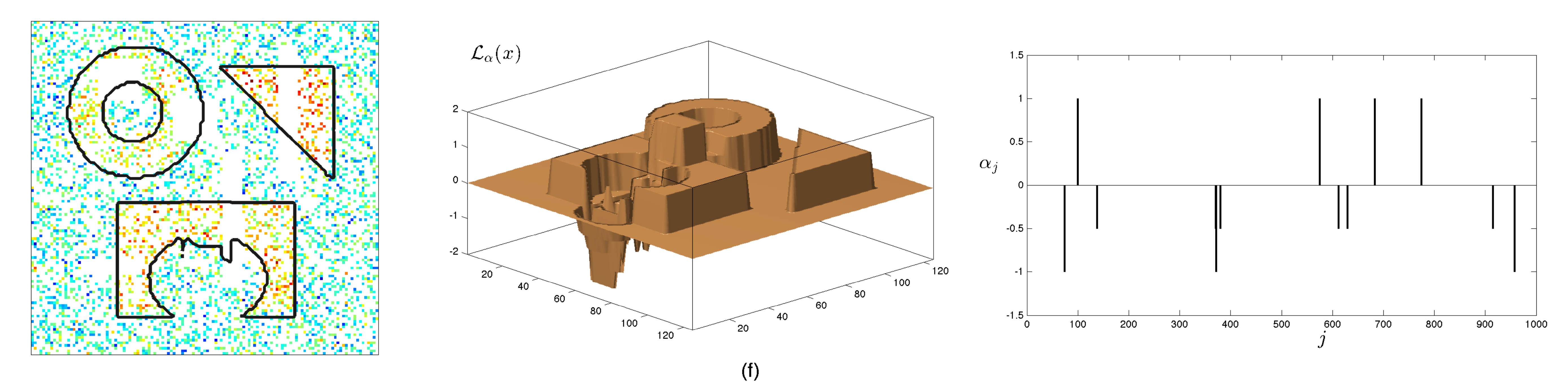} \\[-.2cm]
\end{tabular}
\caption{The segmentation results for the test image in Figure \ref{fig8}(d);
results of different settings presented on the left column, the corresponding $\mathcal{L}_{\balpha}(x)$ depicted in the middle column and the resulting $\balpha$ vector plotted at the right column. (a) $\tau = 3$; (b) $\tau = 5$; (c) $\tau = 6$; (d) $\tau = 7$; (e) $\tau = 8$; (f) $\tau = 10$}\label{fig9_2}
\end{figure*}

It is interesting to note that objects which contribute less to a decrease in the segmentation cost have the least priority in being identified (e.g., the doughnut-shaped object). Moreover, since the LOC violation is more significant in this experiment, for $\tau = 8$ we observe some level of difficulty in identifying the ellipse. More specifically, unlike the previous example in which the coefficient of the ellipse characteristic was recovered to be -2 (see the plot of $\balpha$ in Figure \ref{fig9}(d)), in this example it remains to be -1 and instead a smaller ellipse inside it becomes negatively active. Basically, due to the low image contrast, fully taking out the large ellipse causes some of the pixels highly contributing to the cost reduction to stay outside the partitioner. For $\tau = 10$ the algorithm finds enough freedom to produce a more compact partitioner and at the same time keep the highly contributing pixels inside the partitioner. The average CVX runtime for each instance of the experiments in this section was less than 2 minutes.

\subsection{Challenging OCR Scenarios}
As the proposed technique allows us to identify the principal shape elements through a segmentation task, it would fit very well into an optical character recognition (OCR) application, were the primary objective is the identification of letters appeared in an image. The proposed algorithm would be capable of handling challenging OCR problems were the constituting letters are overlapping or misaligned, and the underlying image is cluttered.

As the first basic experiment in this category, we consider the identification of the letters inside an image of size $63\times 160$ pixels (shown in Figure \ref{fig10}) which represents the word ``\emph{SampLE}''. The letters in the image undergo different case, size, rotation and level of overlap. To build up the character dictionary, yet maintaining a reasonable-sized dictionary, we use characters of a similar font, but in various uppercase/lowercase formats. The character shapes are placed throughout the image in different size and orientations. With this setup, we build up a dictionary of $n_s = 2056$ elements.

The segmentation results are shown for various values of $\tau$ in Figure \ref{fig10}, followed by the plots of the corresponding $\balpha$ vector. A simple index map relating each index in $\balpha$ to the corresponding alphabetical representation, indicates the active letters in each segmentation problem. Again, successive increments of $\tau$ from 1 to 6, allow us to identify the constituting letters one after the other. It is worth highlighting the transition of the results from $\tau=3$ to $\tau=4$, and how a slightly different letter ``m'' is identified to better match the overall shape.

In Figure \ref{fig10}(g), we have again compared the results against the method in \cite{aghasi2013sparse}, where the letters with the largest weights are highlighted. In fact, the approach taken in \cite{aghasi2013sparse} highly relies on an accurate alignment between the dictionary elements and the objects in the image. When this is not the case, a multi-stage reconstruction is proposed, where after each reconstruction, the dictionary elements with smaller weights are eliminated and are replaced with instances of the large-weighted letters at a finer resolution. This process is continued until convergence to a local minimizer is achieved. Clearly, using the proposed convex scheme allows us to perform the task more reliably and in a single step.

\begin{figure*}
\centering
\begin{tabular}{c}
\includegraphics[width=106mm]{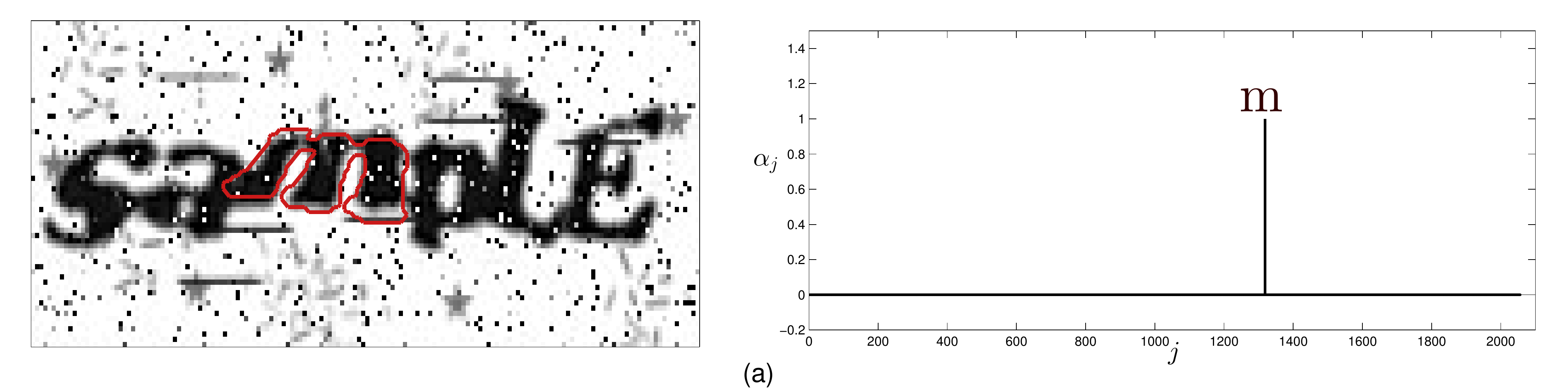} \\[-.15cm]
\includegraphics[width=106mm]{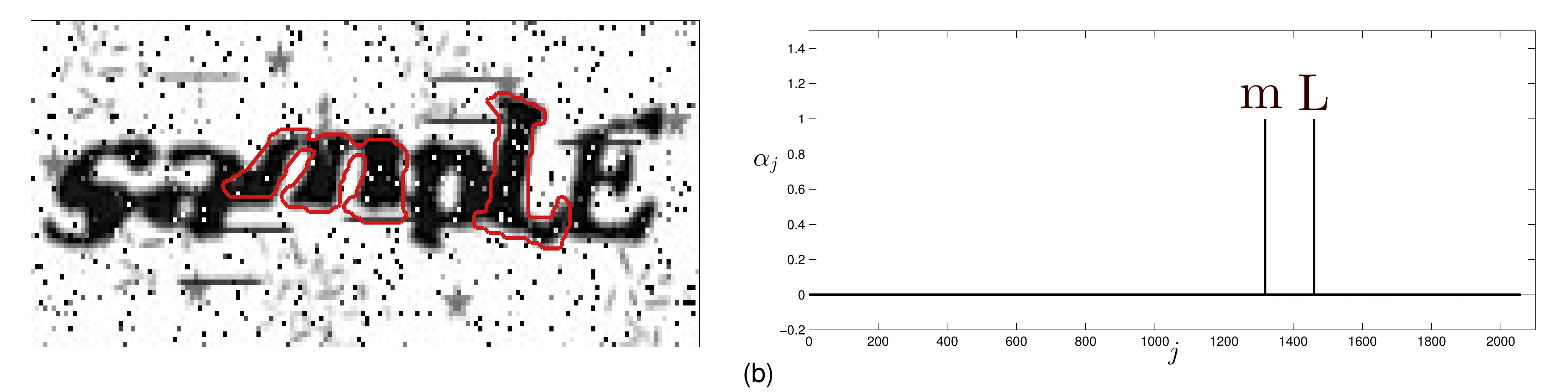} \\[-.15cm]
\includegraphics[width=106mm]{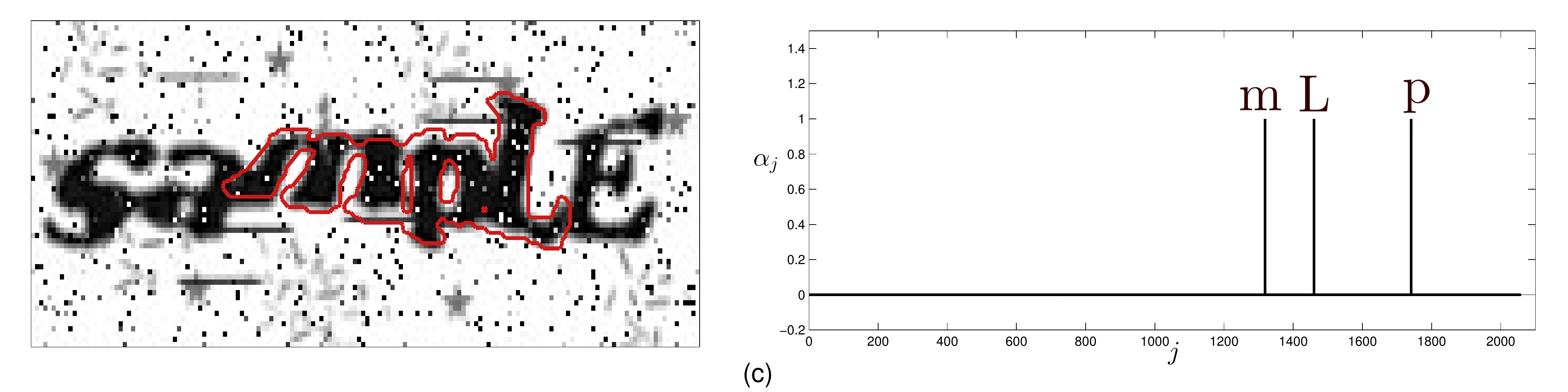} \\[-.15cm]
\includegraphics[width=106mm]{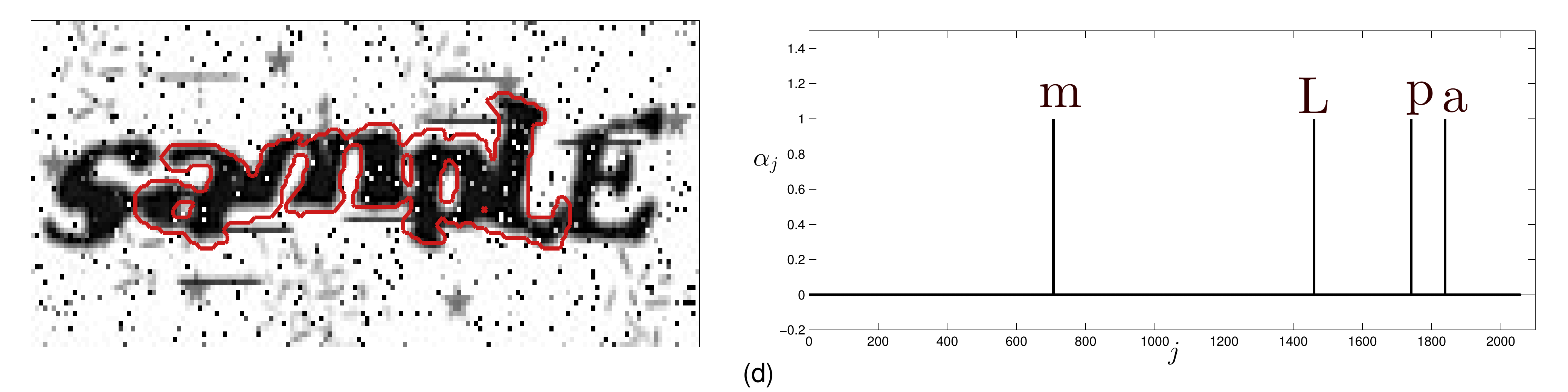} \\[-.15cm]
\includegraphics[width=106mm]{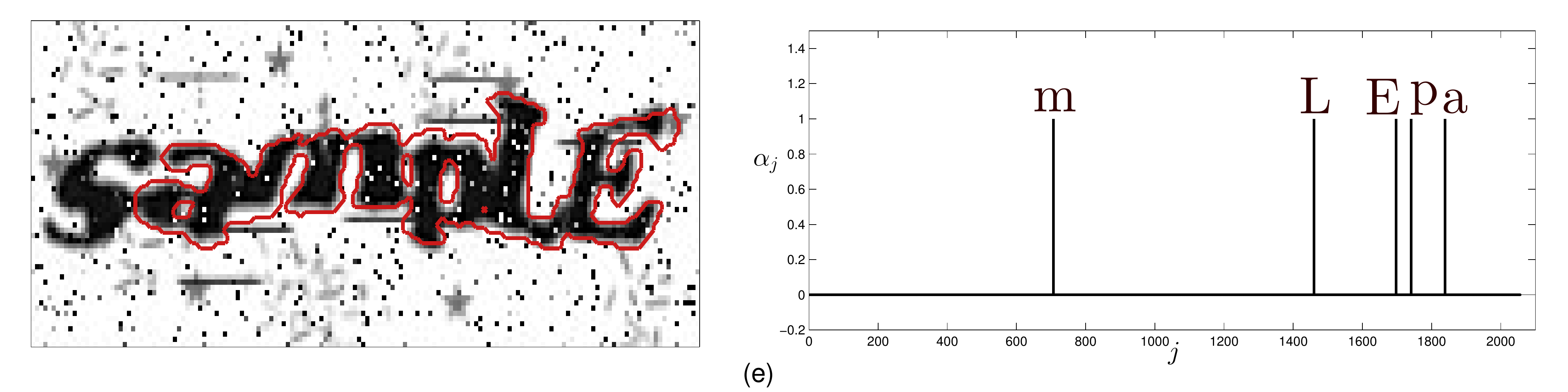} \\[-.15cm]
\includegraphics[width=106mm]{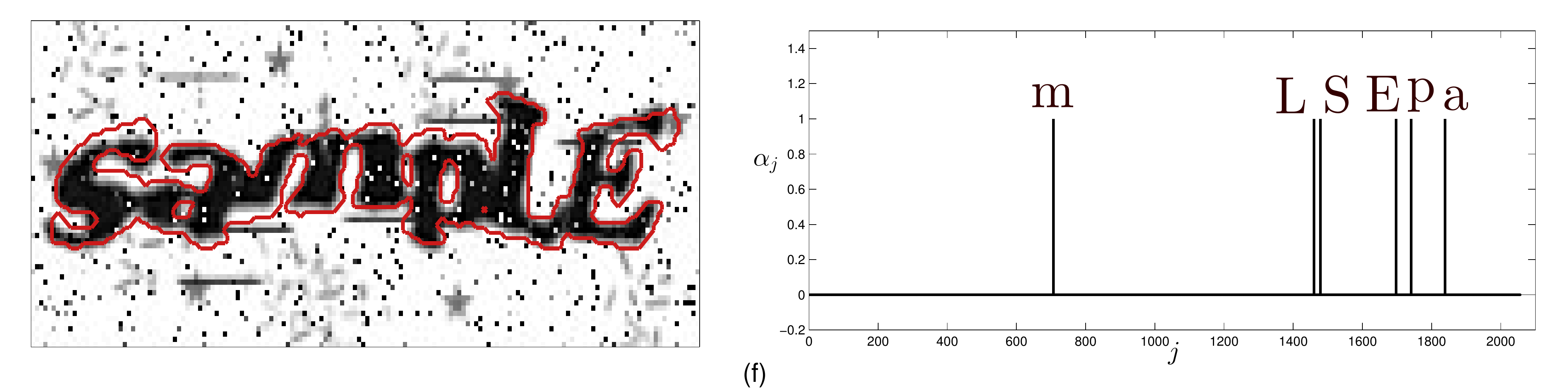} \\[-.15cm]
\includegraphics[width=106mm]{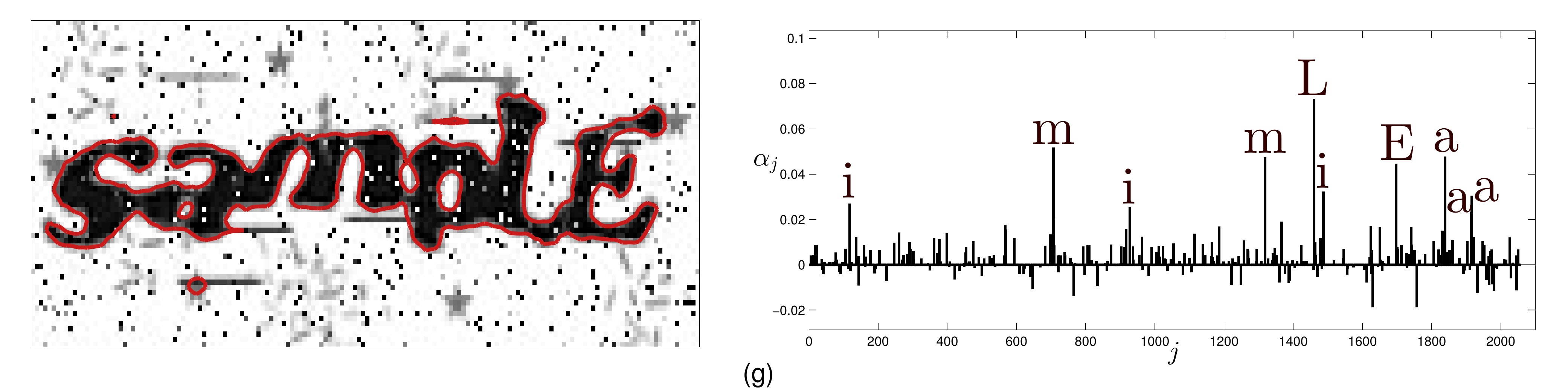}
\end{tabular}
\caption{An OCR problem addressed by segmentation: the red contour in the left column pictures shows a level-set of $\mathcal{L}_{\balpha}(x)$ at some value between zero and one. The right column corresponds to the reconstructed weights and an indication of the letter each index corresponds to; (a) $\tau = 1$; (b) $\tau = 2$; (c) $\tau = 3$; (d) $\tau = 4$; (e) $\tau = 5$; (f) $\tau = 6$; (g) Using the technique  proposed in \cite{aghasi2013sparse}}\label{fig10}
\end{figure*}

In Figure \ref{fig10-2} we have presented the OCR results for a problem of similar setup, where the test image is noisier and the overlap among the characters is more significant. The discrepancy (relative to the ideal binary image) resulted by the clutter and the Gaussian noise in the image is more than 90\%. Similar to the previous noisy cases, $\tilde u_{in}$ and $\tilde u_{ex}$ are selected to be the 15\% and 85\% quantiles of the image histogram. Again successive increments of $\tau$ from 1 to 6, allow us to identify the characters which build up the optimal partitioner. 

In this case that we have a tighter spacing among the letters, for $\tau=4$ (Figure \ref{fig10-2}(d)) we observe that Sparse-CSC identifies a letter ``d'' instead of identifying one of the letters ``a'', ``p'' or ``L''. This is clearly because such selection causes more overlap of the resulting partitioner with the dark region in the image and hence produces a lower segmentation cost. For larger values of $\tau$ the remaining characters are successively identified. Ultimately, Figure \ref{fig10-2}(g) shows the segmentation result without the $\ell_1$-constraint (i.e., $\tau=\infty$), where all the dark regions are captured, however no inference about the principal shapes can be made. The average CVX runtime for each instance of the aforementioned experiments was approximately 8 minutes.  

\begin{figure*}
\centering
\begin{tabular}{c}
\includegraphics[width=106mm]{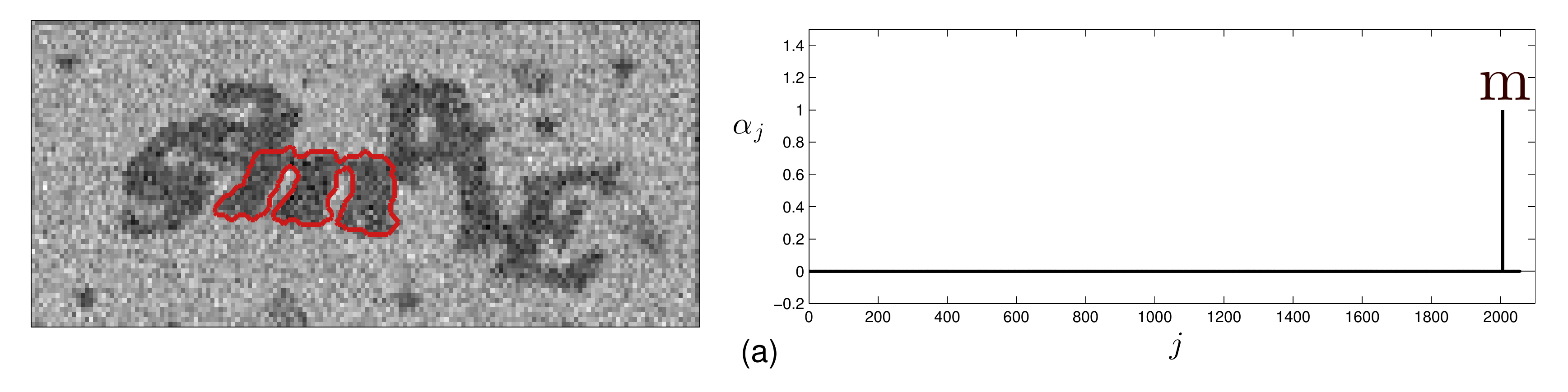} \\[-.15cm]
\includegraphics[width=106mm]{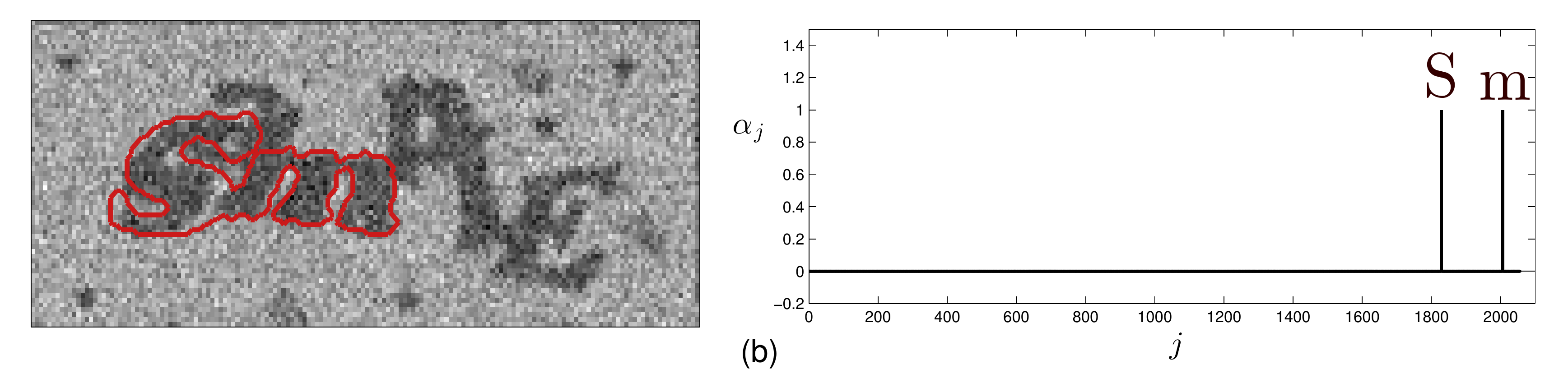} \\[-.15cm]
\includegraphics[width=106mm]{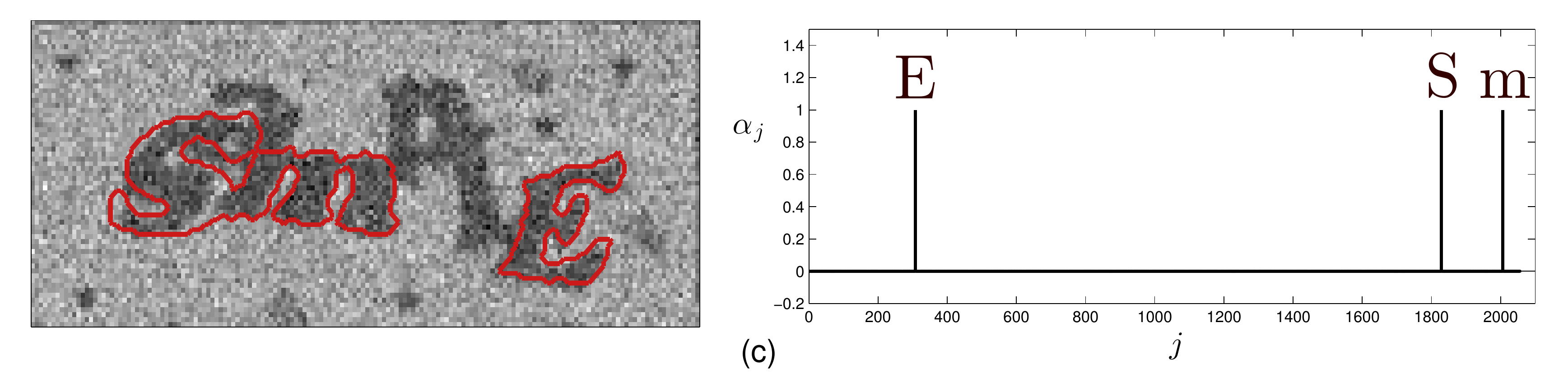} \\[-.15cm]
\includegraphics[width=106mm]{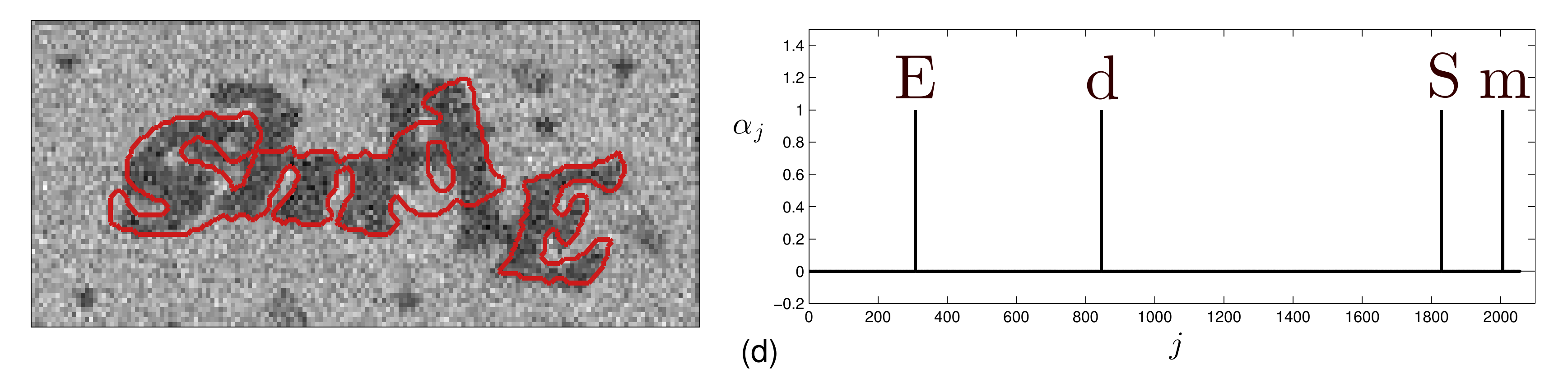} \\[-.15cm]
\includegraphics[width=106mm]{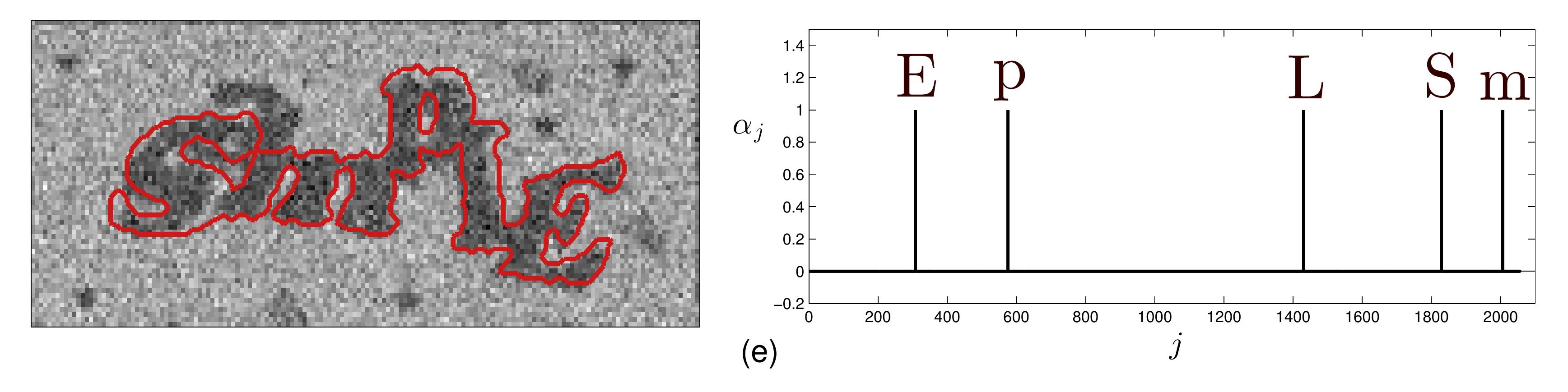} \\[-.15cm]
\includegraphics[width=106mm]{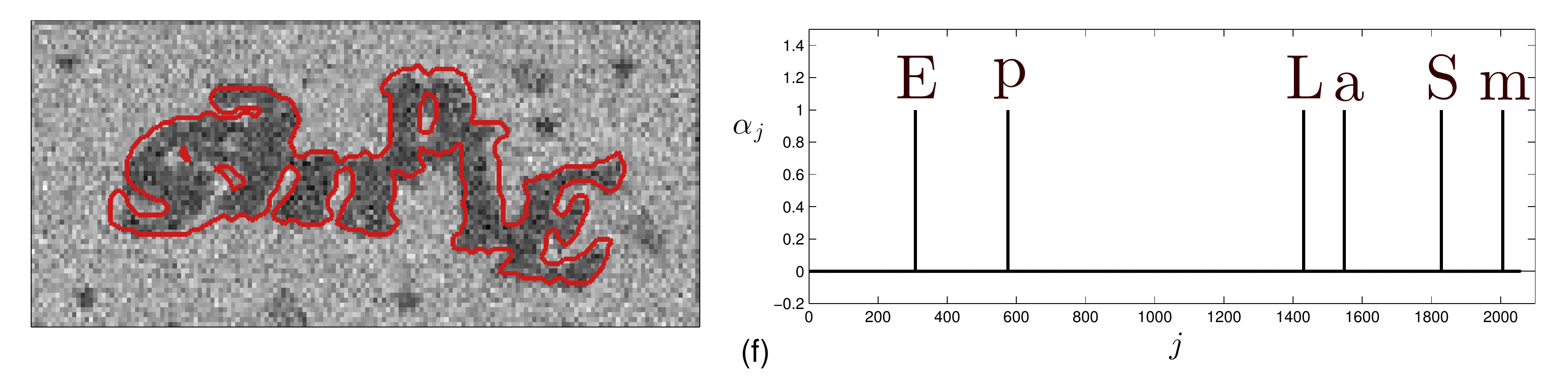} \\[-.15cm]
\includegraphics[width=106mm]{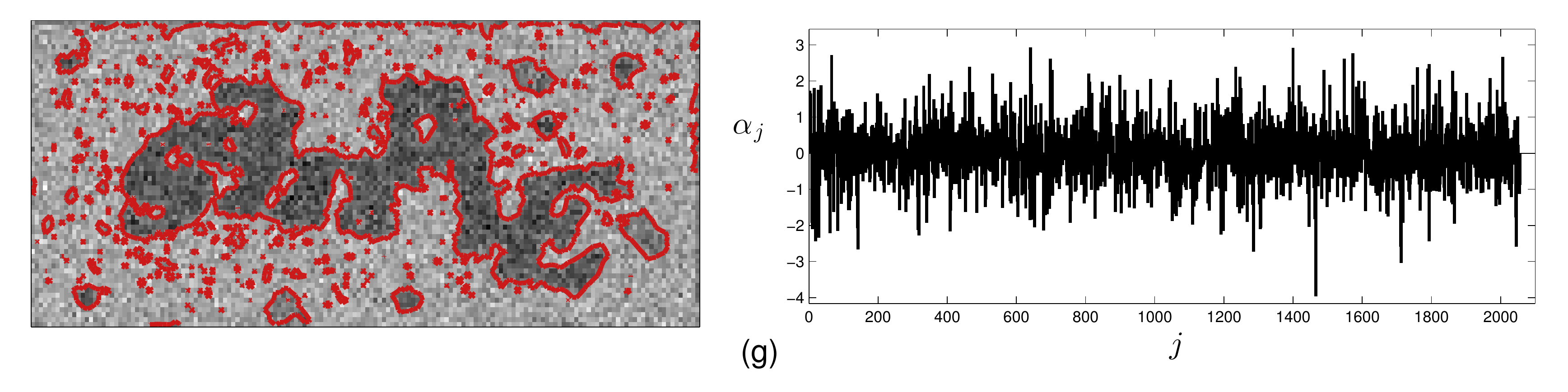}
\end{tabular}
\caption{A more challenging OCR problem addressed by segmentation: the red contour in the left column pictures shows a level-set of $\mathcal{L}_{\balpha}(x)$ at some value between zero and one. The right column corresponds to the reconstructed weights and an indication of the letter each index corresponds to; (a) $\tau = 1$; (b) $\tau = 2$; (c) $\tau = 3$; (d) $\tau = 4$; (e) $\tau = 5$; (f) $\tau = 6$; (g) $\tau = \infty$}\label{fig10-2}
\end{figure*}

As an additional OCR experiment, we consider a cluttered version of the Georgia Tech logo, where this time the letters in the dictionary take a different font than the version appeared in the image. The dictionary consists of 1068 letters of various size, rotation and placement. Still the striking performance of the method for successive values of $\tau$ is presented in Figure \ref{fig11}. We would like to note that the dictionary consists of a letter ``C'' very close to where the letter ``G'' has appeared and yet the algorithm identifies the true letter. We have also presented the segmentation results without an $\ell_1$-constraint (i.e., $\tau=\infty$) in Figure \ref{fig11}(e). One can clearly observe the regularizing effect of the $\ell_1$ constraint in the course of reconstruction. The average CVX runtime for each instance of this experiments was approximately 1 minute.

\begin{figure*}
\centering 
\begin{tabular}{c}
\includegraphics[width=131mm, height=38mm]{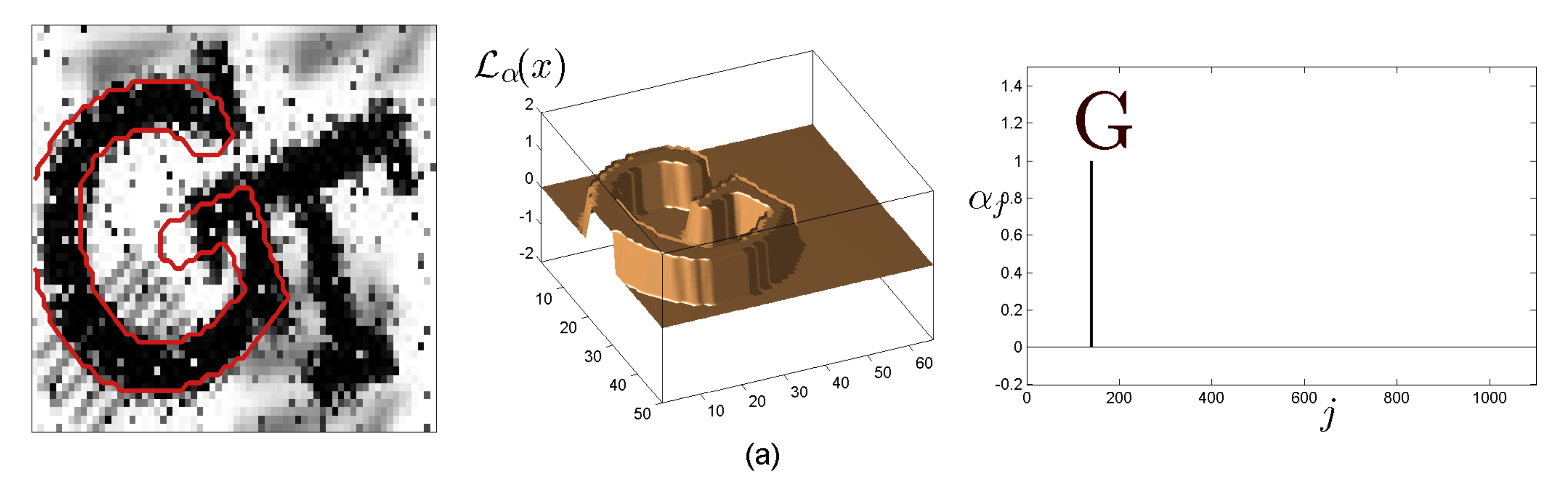} \\[-.2cm]
\includegraphics[width=131mm]{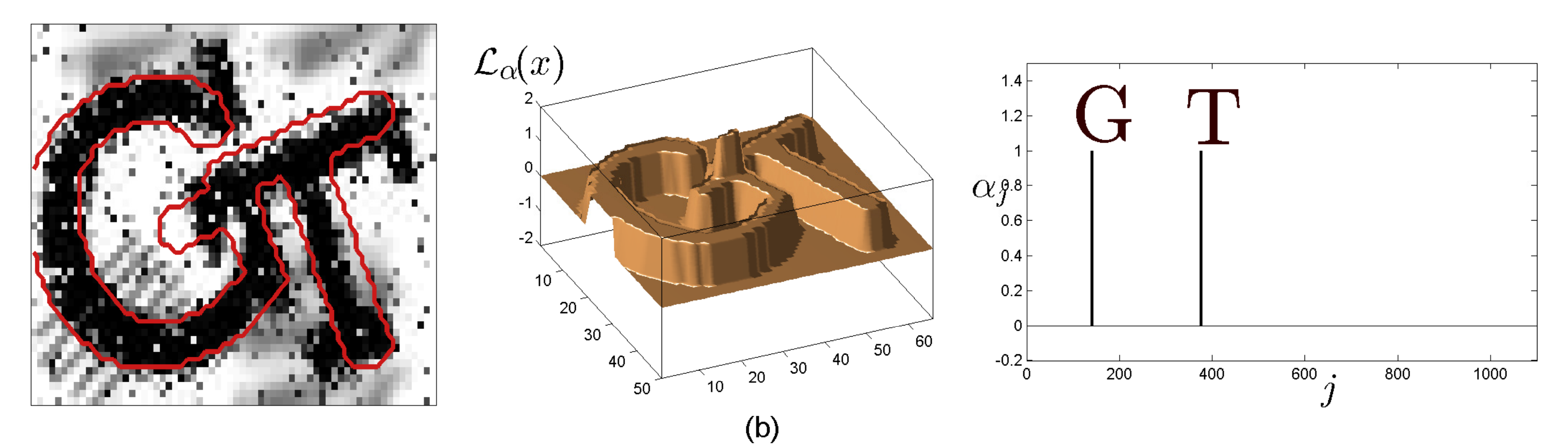} \\[-.2cm]
\includegraphics[width=131mm]{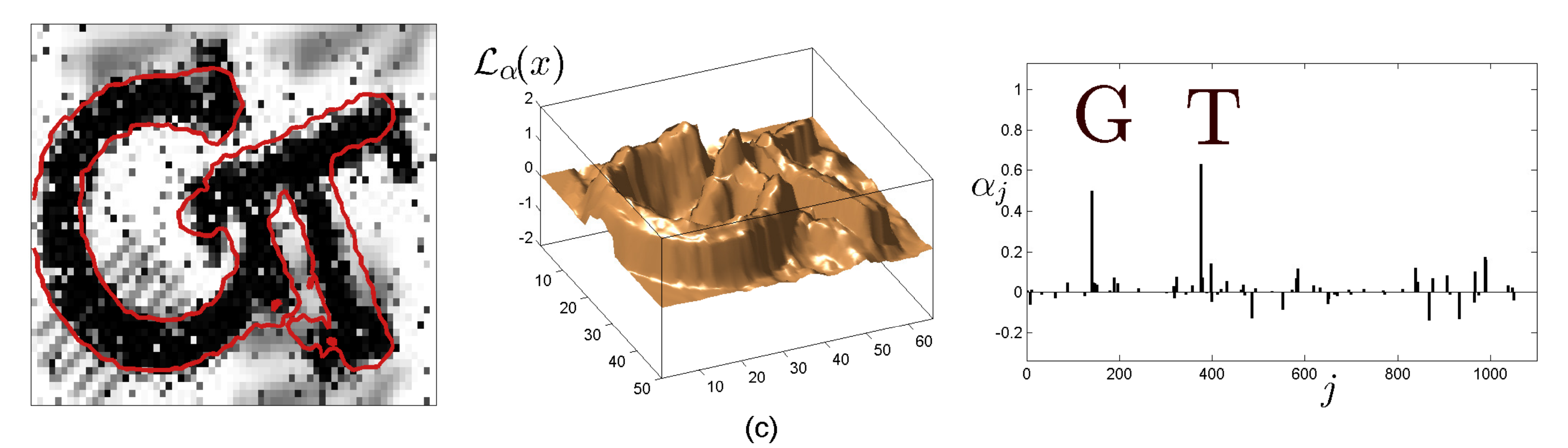} \\[-.2cm]
\includegraphics[width=131mm]{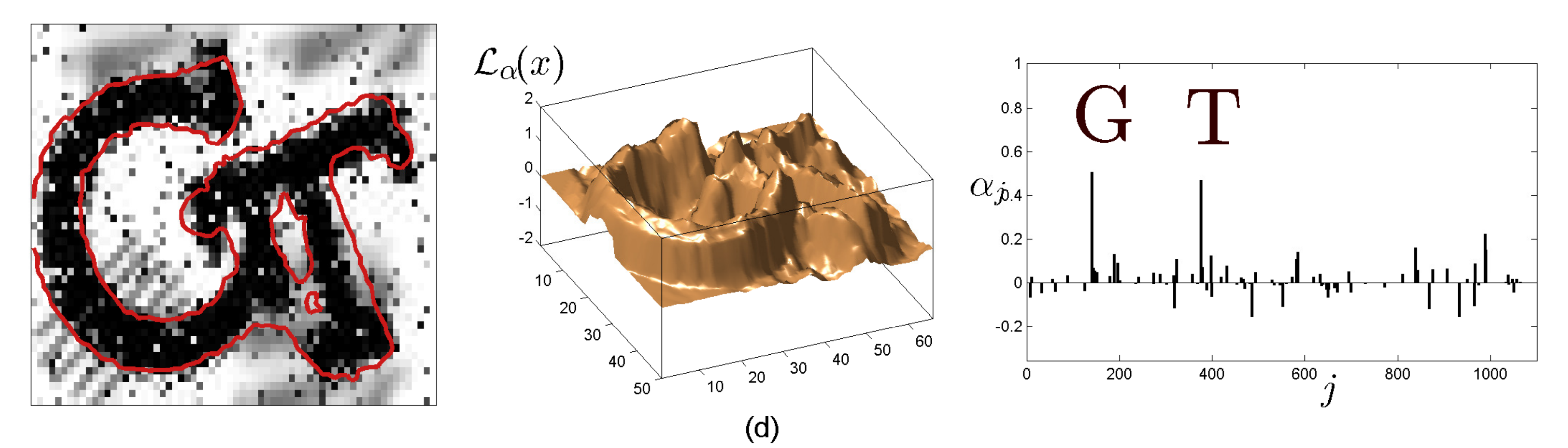} \\[-.2cm]
\includegraphics[width=131mm]{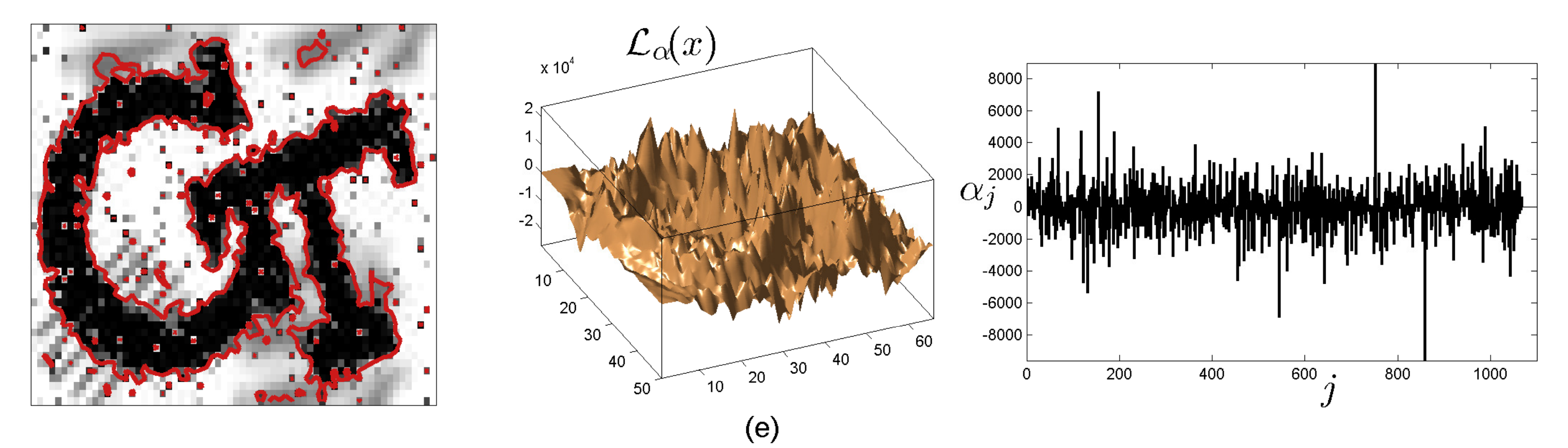} \\[-.2cm]
\end{tabular}
\caption{Another OCR problem where the image and the dictionary use different fonts: the red contour in the left column pictures shows a level-set of $\mathcal{L}_{\balpha}(x)$ at some value between zero and one. The middle column depicts $\mathcal{L}_{\balpha}(x)$ and the right column corresponds to the reconstructed weights along with an indication of the letter each index corresponds to; (a) $\tau = 1$; (b) $\tau = 2$; (c) $\tau = 4$; (d) $\tau = 5$; (e) $\tau = \infty$}\label{fig11}
\end{figure*}

\subsection{Applications in Dense Packing}

Packing problems are a class of problems concerned with the arrangement of given objects inside larger containers \cite{dowsland1992packing}. Specifically, a relevant application that might be immediately linked to the material presented in this paper is densely packing a container with the minimum number of (non-overlapping) objects in $\mathbb{R}^d$, $d=1,2,\cdots$.

The packing problems are generally hard combinatorial problems. For instance, the two-dimensional rectangular packing problem, where both the container and the objects are rectangular geometries is shown to be NP-complete \cite{fowler1981optimal}.  In higher dimensions and the case of irregular objects, the problem clearly becomes more complex and yet regarded as an NP-complete problem \cite{hopper2001review}.

Depending on the problem size and the quality of the shape dictionary, the approach presented in this paper may be employed to address instances of the packing problem, where the objects and container have irregular geometries.

As a proof of concept, we consider the basic example of solving a 12-piece jigsaw puzzle, where the elements of the puzzle are shown in Figure \ref{fig12}(a). Considering the puzzle pad to be $\Sigma\subset D$, the goal would be to tile $\Sigma$ with the puzzle elements. For a shape dictionary consisting of the puzzle elements, this problem may be cast as performing the proposed segmentation task over an image defined in $D$, with unit pixel values inside $\Sigma$ and zero outside.

\begin{figure}%
\centering
\subfigure[]{\includegraphics[width=68mm, angle =90]{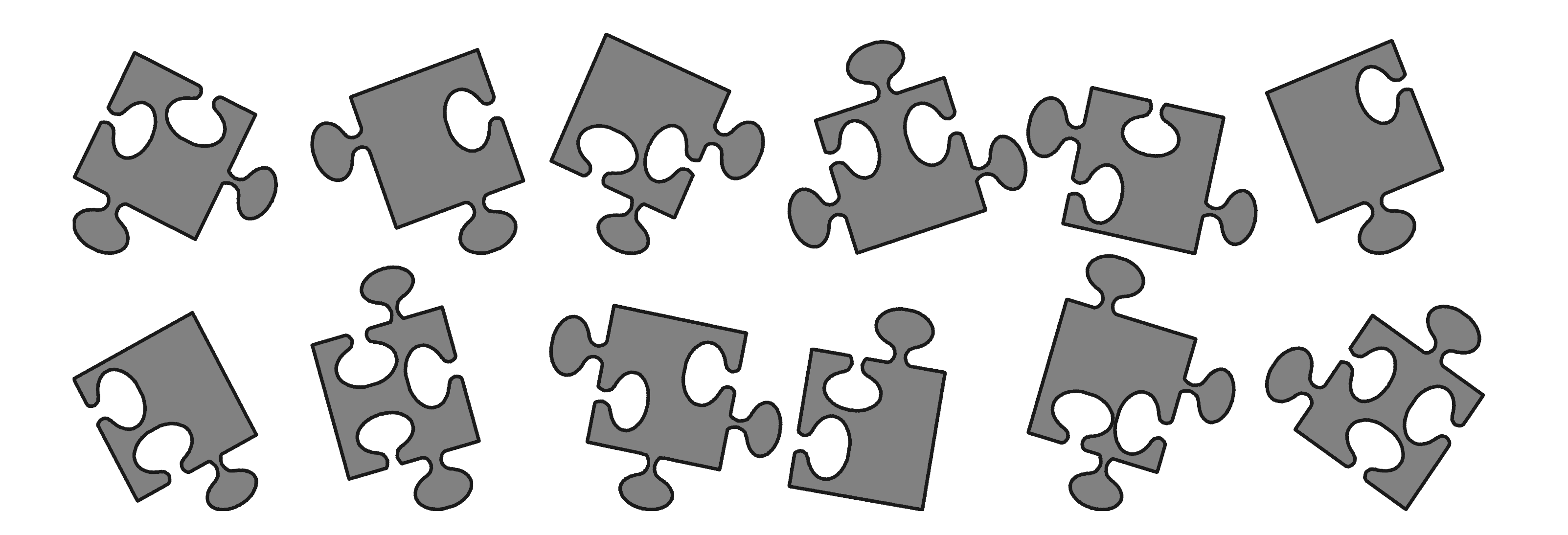} }\hspace{-.2cm}
\subfigure[]{\includegraphics[width=85mm, height=66mm]{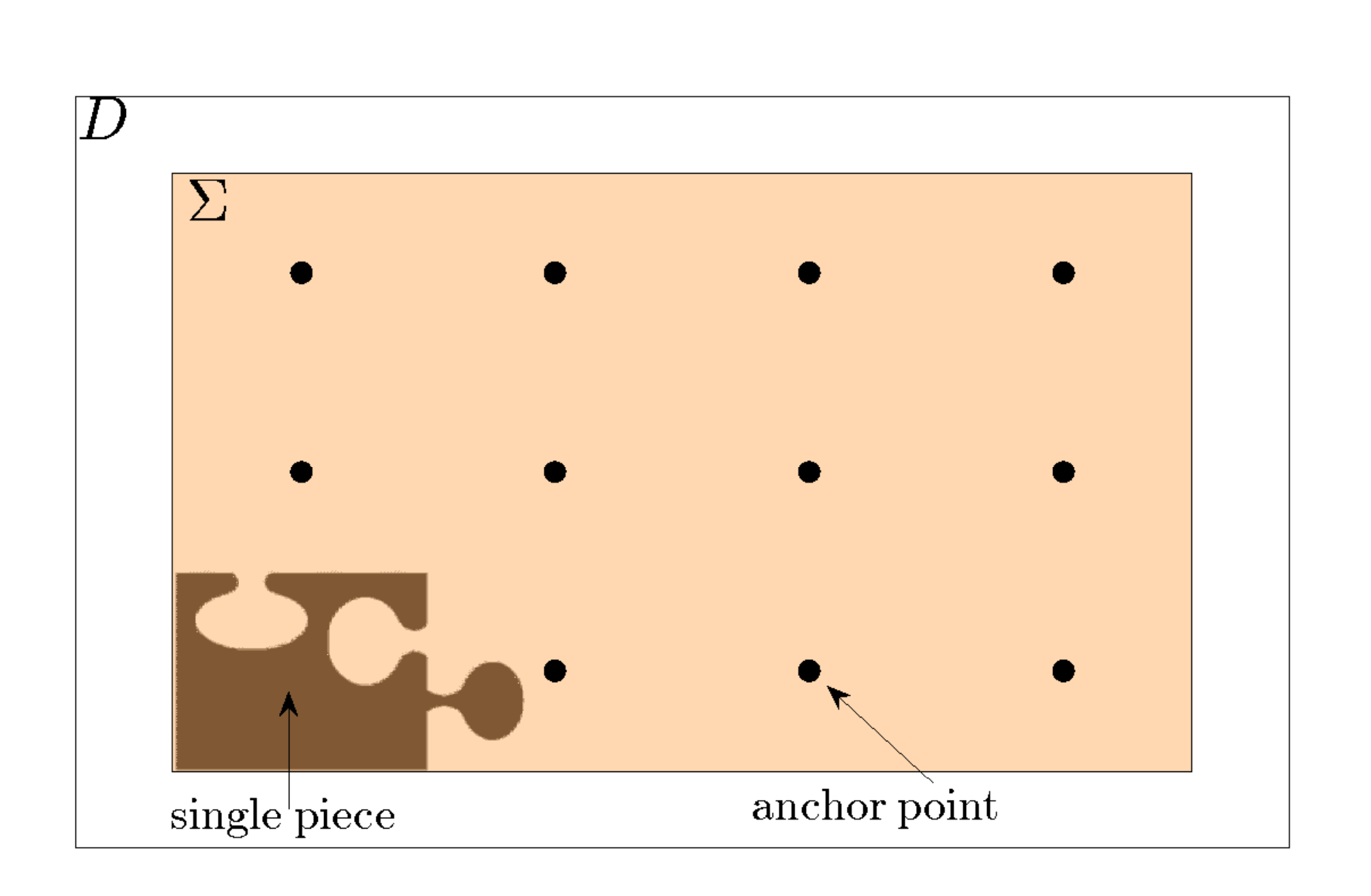} }\hspace{-.2cm}
\subfigure[]{\includegraphics[height=68mm]{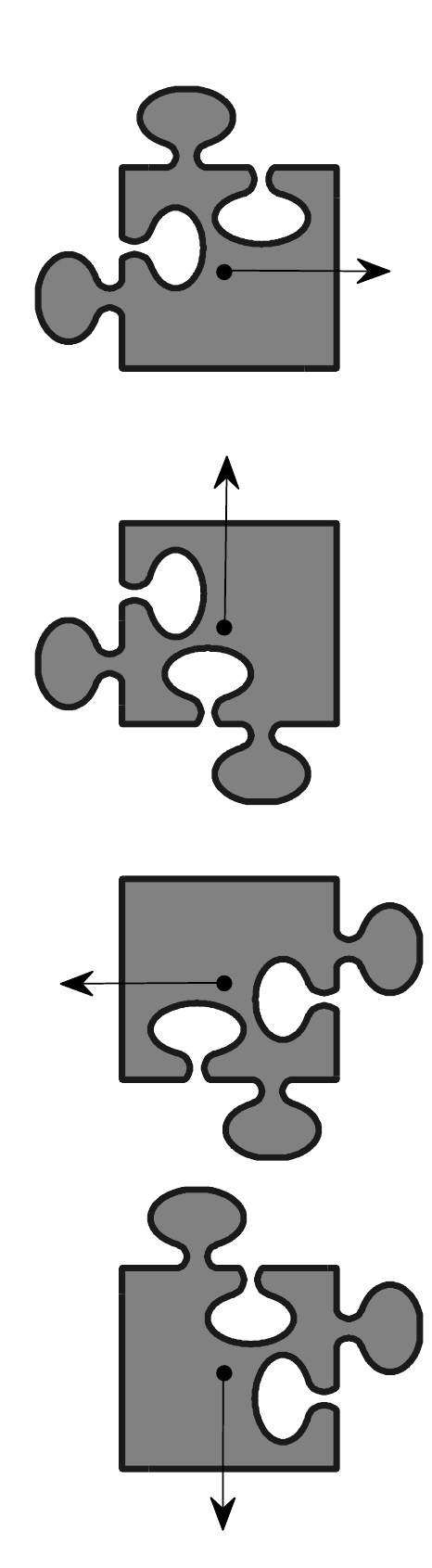} }%
\caption{(a) Twelve pieces of a puzzle to be put together; (b) Building up the shape dictionary by placing only one piece at the bottom left corner and placing rotated versions of the remaining pieces at the anchor points indicated by dots; (c) Each puzzle piece is rotated at a multiple of 90 degrees around an anchor point}
\label{fig12}
\end{figure}

To avoid a large shape dictionary, we consider twelve equispaced anchor points within $\Sigma$, centered at which a puzzle piece may be placed. Clearly, for this example the puzzle pad is symmetric and there might be multiple ways of tiling $\Sigma$. For instance if a combination of the elements tiles $\Sigma$, a 180-degrees rotated version of this combination does the same job, and we may expect a non-uniqueness issue in the Cardinal-SC problem. This phenomenon is prevented by the way the shape dictionary is generated, as follows.

As the first dictionary element, we place the shape corresponding to a true placement at the bottom-left corner anchor. This would be the sole element of the dictionary that is placed in this location, assuring that the Cardinal-SC problem yields a unique solution. For the remaining eleven anchor points, at each location we consider four rectangular rotations of the pieces of the puzzle excluding the bottom-left piece. This process is illustrated in Figure \ref{fig12}(b)-(c), which yields a dictionary of size $1+11\times 11\times 4 = 485$ elements.

Figure \ref{fig13}(a) shows the segmentation result for $\tau=12$, which indicates a successful completion of the puzzle. Panels (b) and (c) show the results for $\tau=11$ and $\tau=10$. It can be observed that in the latter case, the convex program is not able to maintain the sparsity level to the values indicated by $\tau$ (i.e., 11 and 10), and to minimize the convex cost, $\Sigma$ is covered by bringing into play more number of dictionary elements.

\begin{figure*}
\centering 
\begin{tabular}{ccc}
\includegraphics[width=45mm]{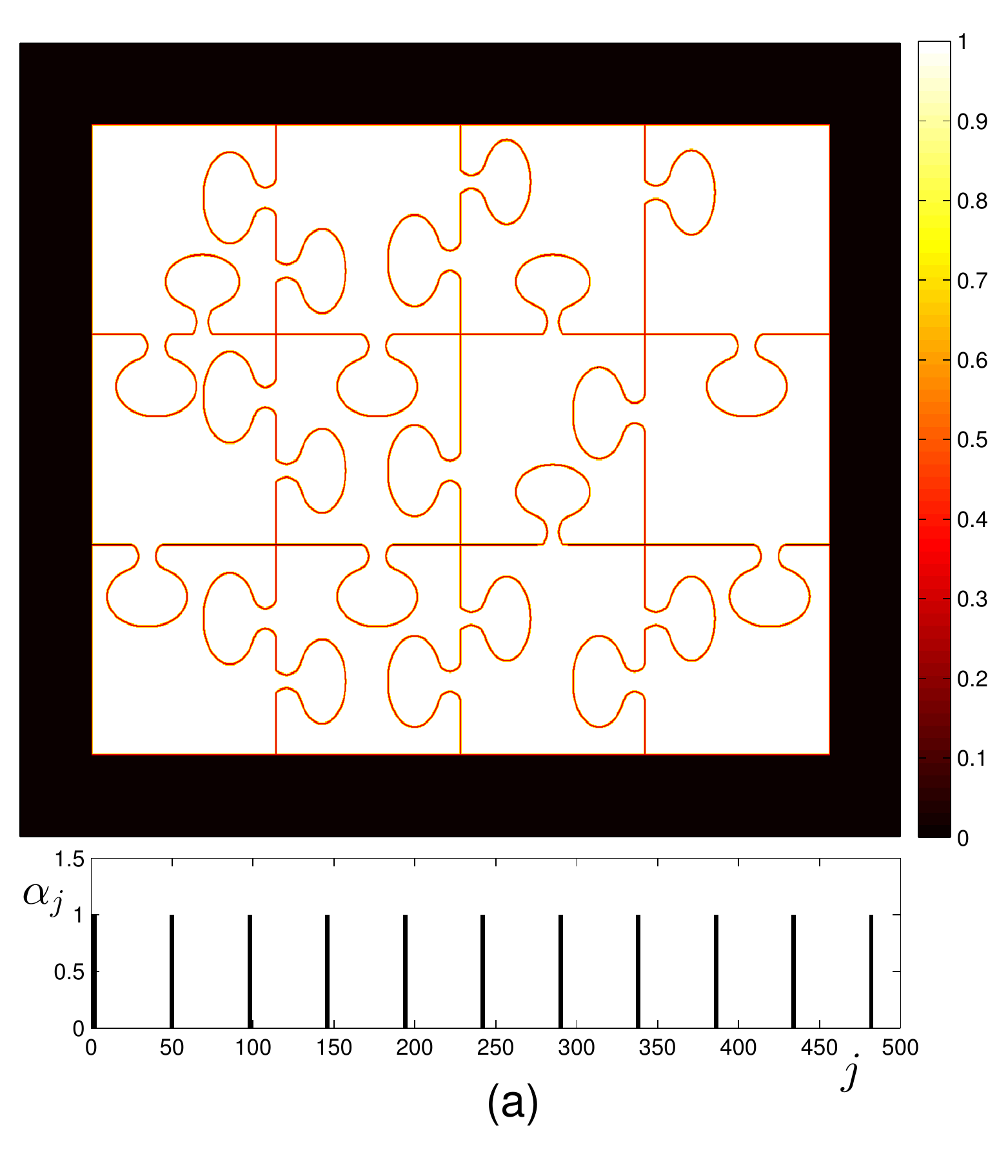} &\hspace{-.5cm}
\includegraphics[width=45mm]{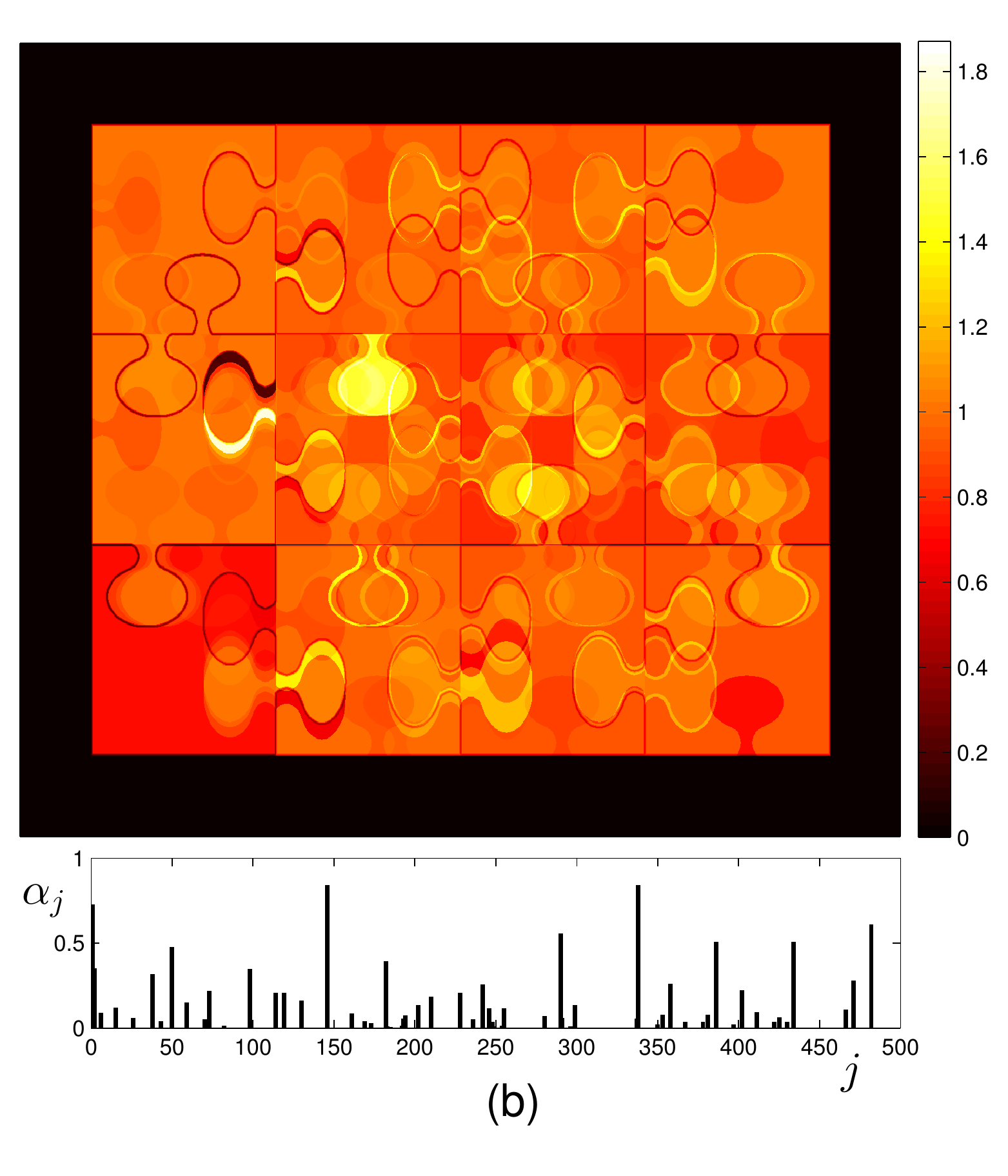} &\hspace{-.5cm}
\includegraphics[width=45mm]{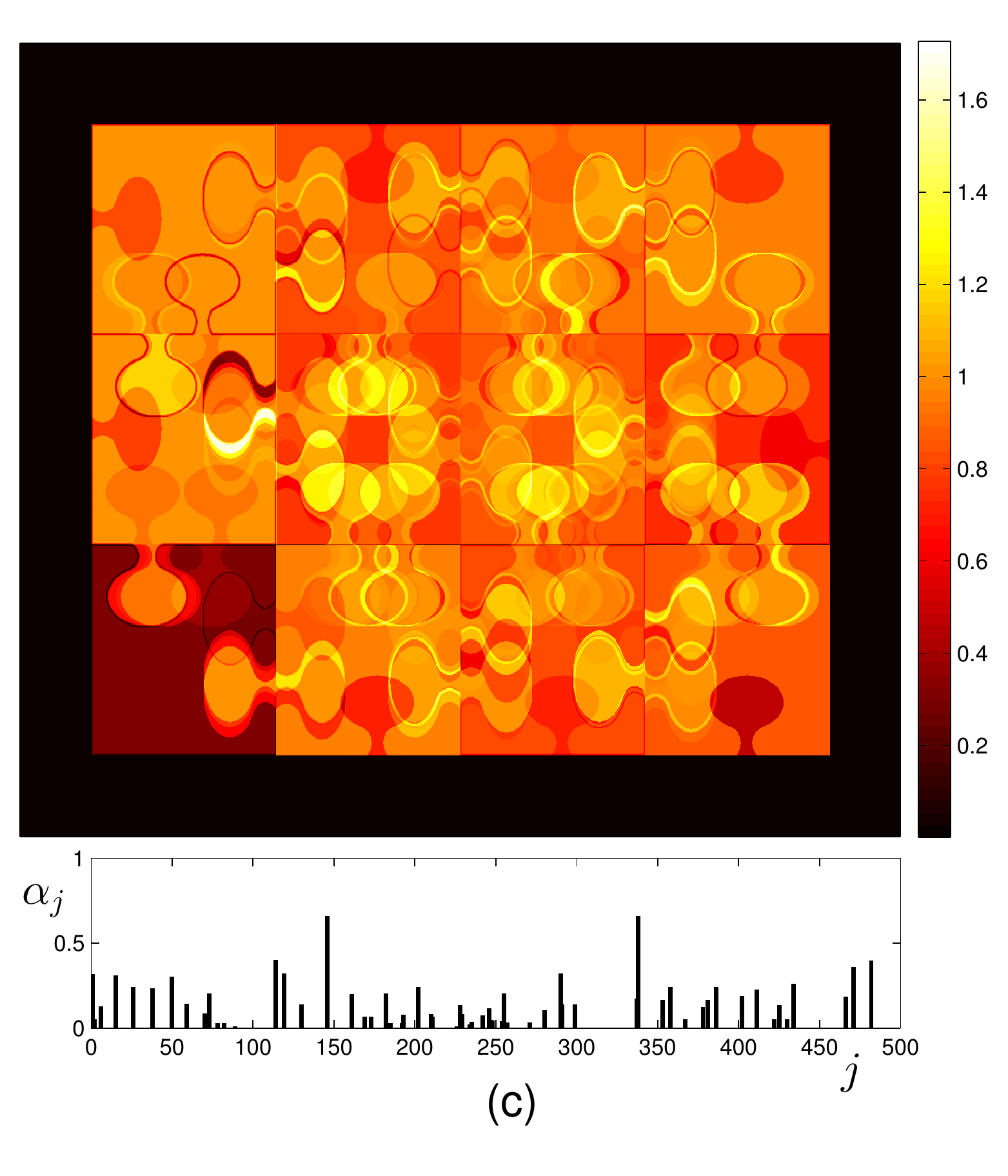}
\end{tabular}
\caption{Solving a puzzle by shape composition. The values of the resulting $\mathcal{L}_{\balpha}(x)$ shown as a colormap, and the corresponding $\balpha$ plotted below each reconstruction; (a) $\tau = 12$; (b) $\tau = 11$; (c) $\tau = 10$}\label{fig13}
\end{figure*}

One way to interpret the algorithm's success in values of $\tau$ below the true inclusion's value, is the toleration against the possible inaccuracies in the given setup. In other words, the reconstruction failure for $\tau<12$ in the preceding experiment, somehow reports a low toleration against the non-ideal setup (phenomena such as the LOC violation and inaccurate shape element alignments).

As a second example, we consider a 9-piece puzzle with a similar setup as the aforementioned example, using $1+8\times 8\times 4 = 257$ shape elements. Figure \ref{fig14} shows the successful reconstruction results for $\tau=9,\;8,\;7$. The experiment justifies the higher toleration of the proposed scheme against the possible sources of error when the true object is comprised of fewer components.

\begin{figure*}
\centering 
\begin{tabular}{ccc}
\includegraphics[width=45mm]{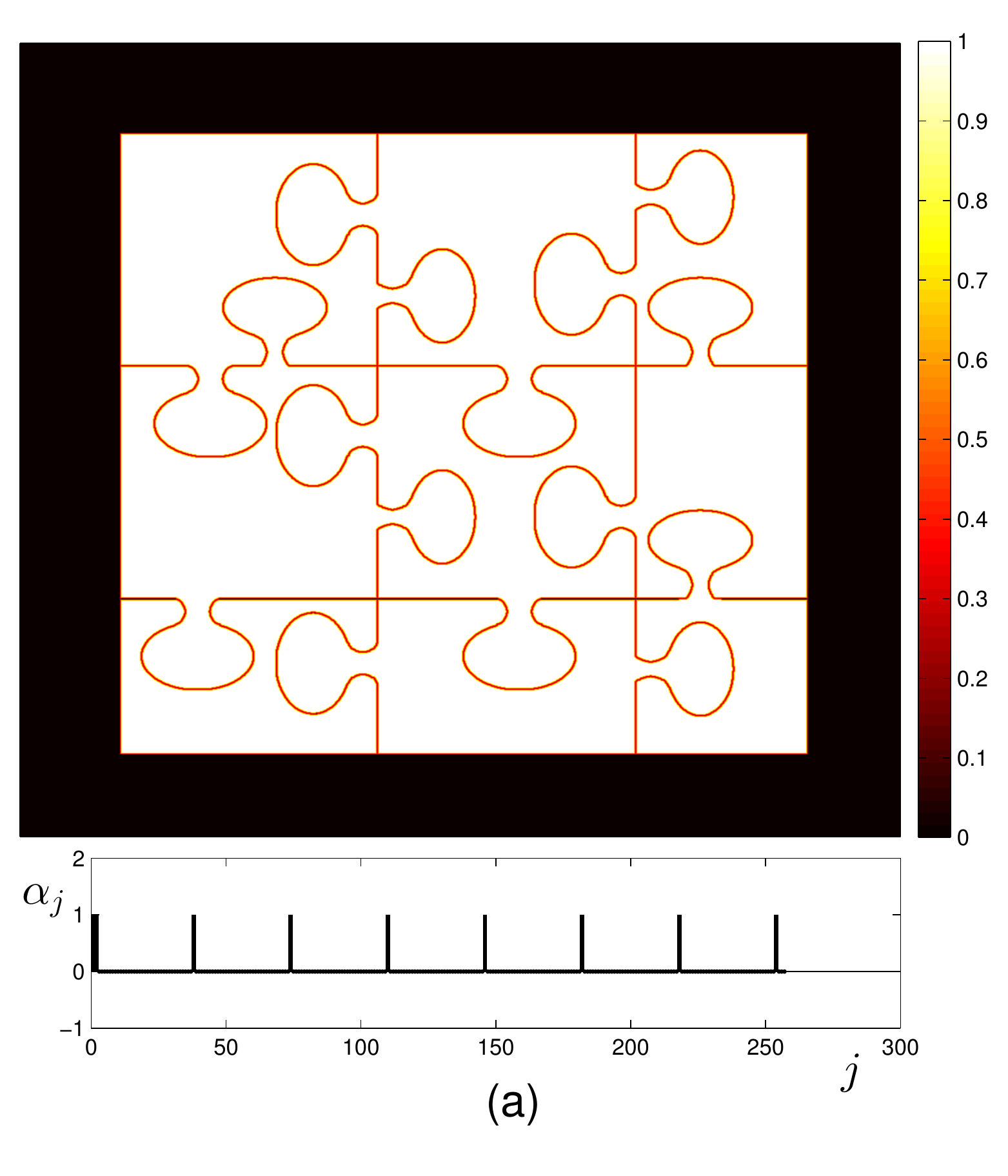} &\hspace{-.5cm}
\includegraphics[width=45mm]{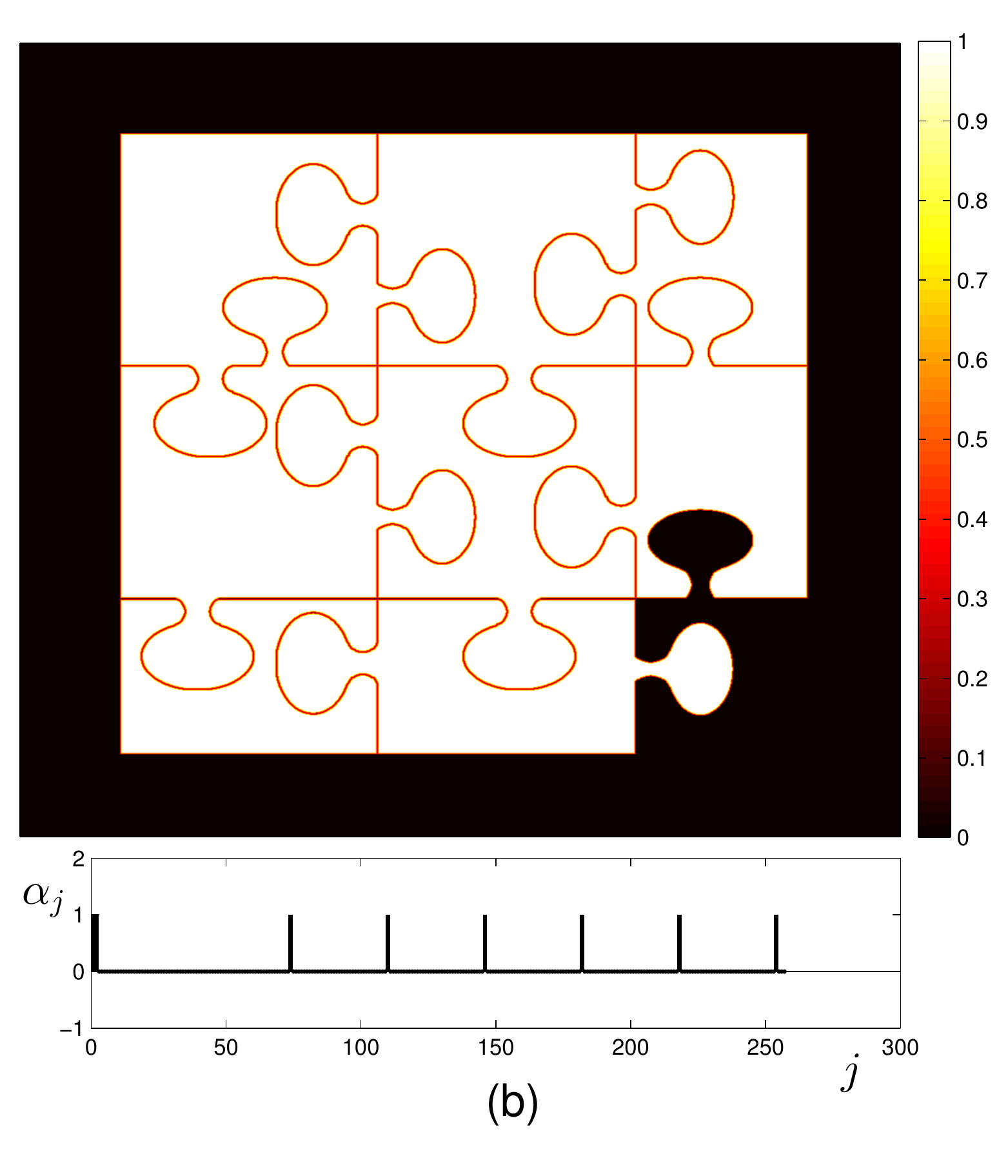} &\hspace{-.5cm}
\includegraphics[width=45mm]{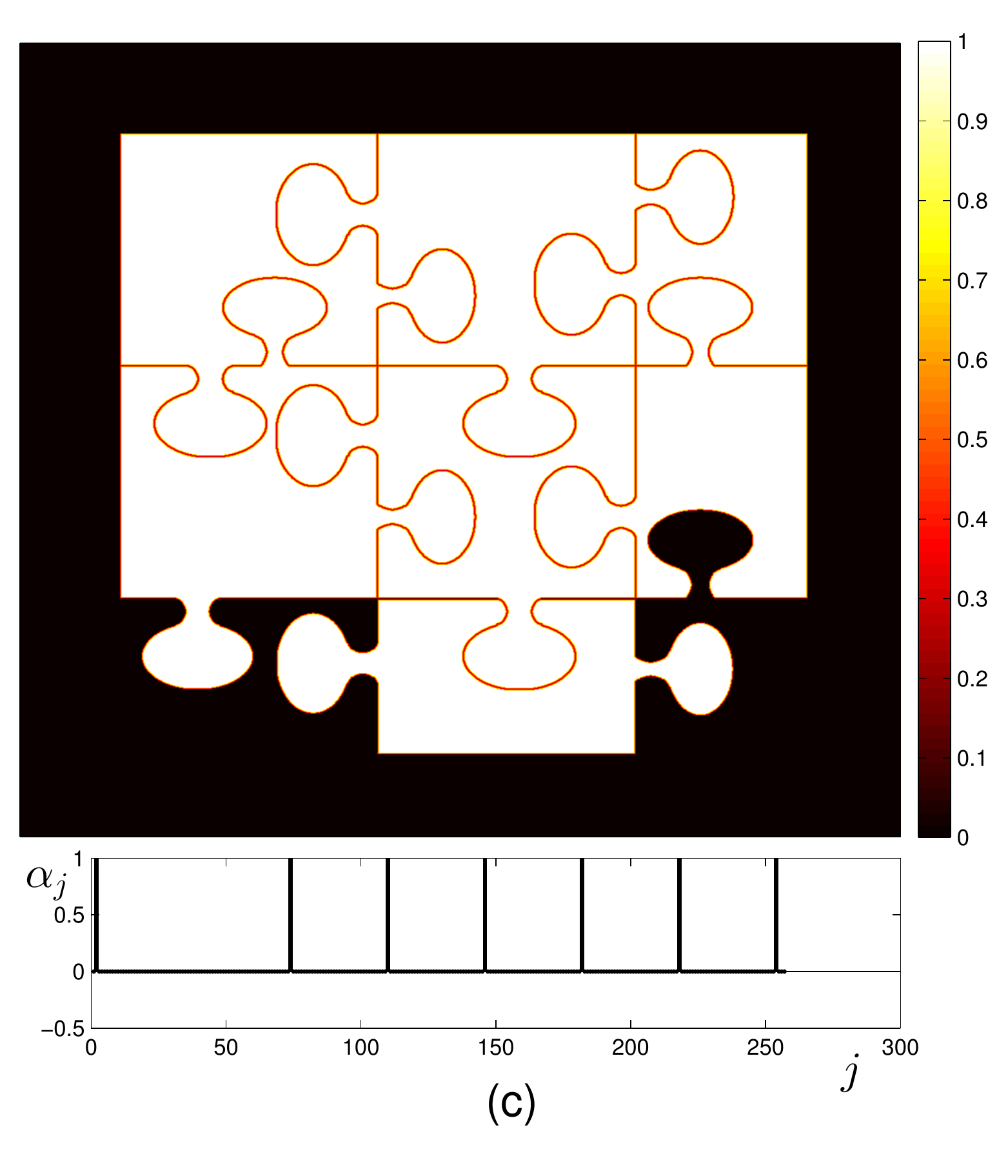}
\end{tabular}
\caption{An attempt on a smaller puzzle with successful reconstructions for consecutive values of $\tau$; (a) $\tau = 9$; (b) $\tau = 8$; (c) $\tau = 7$}\label{fig14}
\end{figure*}


\section{Future Extensions}\label{sec6}
 Our emphasis on the binary image segmentation problem is to avoid unnecessary complication in the presentation and analysis. Several ideas developed in the paper are extendible to the multi-phase  segmentation problem. Moreover, the idea of composing prototype shapes to represent more complex geometries extends beyond the context of image segmentation and can be employed in other shape-based imaging disciplines.

In this paper we specifically focused on the composition rule (\ref{eq3}), which was earlier presented in \cite{aghasi2013sparse} and provides a flexible, yet simple model to combine the shape elements. A key component of our analysis is Theorem \ref{th2x} which particularly focuses on the proposed composition rule. Depending on the application, we may consider other composition rules and inspired by the ideas developed in this paper derive new convex models. However, such extension requires stably relating the representations in the shape-domain and $\balpha$-domain as performed in Sections \ref{DSDsec} and \ref{Sec:Rel}.

In Section \ref{sec-conv} we showed the possibility of verifying the unique optimality of a given vector for the Sparse-CSC program. The optimality is cast as verifying the existence of the solution for a constrained system of equations. Although the conditions stated in Section \ref{sec-conv} are general, to maintain brevity in Section \ref{sec:accrecovery} we combined the results under the assumption that LOC holds for the target composition. These results can be extended to a more general case, where the LOC violation is sufficiently small. Such extension would provide us with quantitative results regarding the required contrast between the target shape and the background.

While the basic convex formulation of the Sparse-CSC problem allows us to employ standard solvers, developing efficient optimization algorithms to numerically handle large-scale problems remains a future avenue of research. Specific interest would be on fast, memory-efficient and parallelizable techniques which allow breaking down the problem into smaller sub-problem (e.g., combining the composition results from different sections of the image, processed in parallel). Such algorithms may be equipped with updating rules of the inhomogeneity measures in the course of shape reconstruction.

\section{Proof of the Main Results}\label{sec:proof}
The technical proofs associated with the presented results are detailed in this section.

\subsection{Proof of Proposition \ref{th1}}
For $j\in \Ip$ we define
$\mathpzc{L}_{\boldsymbol{\alpha}}^\oplus(x) = \sum_{j\in \Ip}\chi_{\mathcal{S}_j}.$
Clearly,
\[\big\{x:\mathpzc{L}_{\boldsymbol{\alpha}}^\oplus(x)\geq 1\big\} = \bigcup_{j\in \Ip}\mathcal{S}_j.
\]
We set $\eta_0=\sup_{x\in D}\mathpzc{L}_{\boldsymbol{\alpha}}^\oplus(x)$ and for an arbitrary $\eta\geq\eta_0\geq 1$ we define
\[\mathpzc{L}_{\boldsymbol{\alpha}}^\ominus(x) = \sum_{j\in \In}\eta\chi_{\mathcal{S}_j}.
\]
In this case we have
\[\big\{x:\mathpzc{L}_{\boldsymbol{\alpha}}^\ominus(x)\geq \eta\big\} = \bigcup_{j\in \In}\mathcal{S}_j.
\]
We now show $\mathpzc{L}_{\boldsymbol{\alpha}}(x)= \mathpzc{L}_{\boldsymbol{\alpha}}^\oplus(x) - \mathpzc{L}_{\boldsymbol{\alpha}}^\ominus(x)$ meets the requirements of the problem. Clearly,
\begin{equation} \label{et1}
\mathpzc{L}_{\boldsymbol{\alpha}}^\oplus(x) - \mathpzc{L}_{\boldsymbol{\alpha}}^\ominus(x)\left\{
     \begin{array}{cc}
       =\mathpzc{L}_{\boldsymbol{\alpha}}^\oplus(x)\geq 1 & x\in (\bigcup_{j\in \Ip}\mathcal{S}_j)\setminus (\bigcup_{j\in \In}\mathcal{S}_j)\\\leq 0 & x\in(\bigcup_{j\in \Ip}\mathcal{S}_j)\cap(\bigcup_{j\in \In}\mathcal{S}_j)\\ = - \mathpzc{L}_{\boldsymbol{\alpha}}^\ominus(x)\leq 0&x\in (\bigcup_{j\in \In}\mathcal{S}_j)\setminus (\bigcup_{j\in \Ip}\mathcal{S}_j)\\\leq 0 & x\in D\setminus \bigcup_{j\in \Ip}\mathcal{S}_j
     \end{array}.
   \right.
\end{equation}
For arbitrary sets $A$, and $B\subset C$ we have
\begin{equation*}
C\setminus(A\setminus B) = (B\setminus A)\cup (B\cap A)\cup(C\setminus A),
\end{equation*}
and as a result
\begin{align}\nonumber
D\setminus \Sigma &= D\setminus \Big((\bigcup_{j\in \Ip}\mathcal{S}_j) \setminus (\bigcup_{j\in \In}\mathcal{S}_j)\Big)\\ &= \Big((\bigcup_{j\in \In}\mathcal{S}_j) \setminus (\bigcup_{j\in \Ip}\mathcal{S}_j)\Big) \cup \Big((\bigcup_{j\in \Ip}\mathcal{S}_j)\cap (\bigcup_{j\in \In}\mathcal{S}_j)\Big)\cup \Big(D\setminus \bigcup_{j\in \Ip}\mathcal{S}_j\Big).\label{et2}
\end{align}
Combination of (\ref{et1}) and (\ref{et2}) yields
\begin{equation*}
\mathpzc{L}_{\boldsymbol{\alpha}}^\oplus(x) - \mathpzc{L}_{\boldsymbol{\alpha}}^\ominus(x)\left\{
     \begin{array}{cc}
       \geq 1 & x\in \Sigma\\
       \leq 0 & x\in D\setminus \Sigma
     \end{array},
   \right.
\end{equation*}
which completes the proof. \qed

\subsection{Proof of Proposition \ref{th3}}
We know $\Ip^\Sigma$ and $\In^\Sigma$ are the only non-redundant index sets that $\mathpzc{R}_{\;\Ip^\Sigma,\In^\Sigma}\seq\Sigma$. Therefore for any other non-redundant representation $\Sigma' = \mathpzc{R}_{\;\Ip',\In'}$, we have $\Sigma'\nseq\Sigma$ or basically $\Sigma'\cap(D\setminus \Sigma)\nseq \emptyset$. By evaluating the SC cost for $\Sigma'$ under the LOC we get
\begin{align*}
\int_{\Sigma'}\big(\Pi_{in}(x)- \Pi_{ex}(x)\big)\;\mbox{d}x&= \underset{\Sigma'\cap\Sigma}{\int}\big(\Pi_{in}(x) - \Pi_{ex}(x)\big)\;\mbox{d}x + \underset{\Sigma'\cap(D\setminus \Sigma)}{\int}\!\!\!\big(\Pi_{in}(x)- \Pi_{ex}(x)\big)\;\mbox{d}x\\&>\underset{\Sigma'\cap\Sigma}{\int}\big(\Pi_{in}(x) - \Pi_{ex}(x)\big)\;\mbox{d}x + \underset{\Sigma\cap(D\setminus \Sigma')}{\int}\!\!\!\big(\Pi_{in}(x) - \Pi_{ex}(x)\big)\;\mbox{d}x\\&= \int_{\Sigma}\big(\Pi_{in}(x)- \Pi_{ex}(x)\big)\;\mbox{d}x,
\end{align*}
where the inequality is trivially true due to the sign of each term. Therefore, any $\Sigma'\nseq\Sigma$ generates a cost strictly greater than the cost evaluated for $\Sigma$. Since $\Sigma$ has a unique representation, $\{\Ip^\Sigma,\In^\Sigma\}$ is the strict minimizer of the SC cost.\qed

\subsection{Proof of Proposition \ref{prop4}}
We define
\begin{equation*}
\left\{
\begin{array}{l}
\Sigma_0 \triangleq \{x: \mathpzc{L}_{\balpha}(x)<0\}\\
\Sigma_1 \triangleq \{x: 0\leq \mathpzc{L}_{\balpha}(x)<1\}\\
\Sigma_2 \triangleq \{x: \mathpzc{L}_{\balpha}(x)\geq 1\}
\end{array}
\right.
\end{equation*}
For the sake of convenience we use the following notations throughout the proof:
\begin{equation}\label{et3x}
\Delta_{in}(x) = \big(\Pi_{in}(x)-\Pi_{ex}(x)\big)^+ \quad \mbox{and} \quad \Delta_{ex}(x) = \big(\Pi_{ex}(x)-\Pi_{in}(x)\big)^+.
\end{equation}
Clearly,
\begin{align}\label{et3}
\underset{D}{\int}\Delta_{in}(x) \max\big(\mathpzc{L}_{\boldsymbol{\alpha}}(x),0\big)\;\mbox{d}x = \underset{\Sigma_1\cup\Sigma_2}{\int}\Delta_{in}(x) \mathpzc{L}_{\boldsymbol{\alpha}}(x)\;\mbox{d}x= \!\!\!\!\underset{(\Sigma_1\cup\Sigma_2)\cap(D\setminus\Sigma)}{\int} \!\!\!\!\Delta_{in}(x) \mathpzc{L}_{\boldsymbol{\alpha}}(x)\;\mbox{d}x
\end{align}
where the last equality is thanks to the LOC. In a similar fashion
\begin{align}\nonumber
\underset{D}{\int}\Delta_{ex}(x)\min\big(\mathpzc{L}_{\boldsymbol{\alpha}}(x),1\big)\;\mbox{d}x &= \underset{\Sigma}{\int}\Delta_{ex}(x)\min\big(\mathpzc{L}_{\boldsymbol{\alpha}}(x),1\big)\;\mbox{d}x\\ &= \underset{\Sigma\cap\Sigma_2}{\int}\Delta_{ex}(x)\;\mbox{d}x + \underset{\Sigma\cap(\Sigma_1\cup\Sigma_0)}{\int}\Delta_{ex}(x)\mathpzc{L}_{\boldsymbol{\alpha}}(x)\;\mbox{d}x.
\label{et4}
\end{align}
Also
\begin{equation*}
\underset{\Sigma}{\int}\Delta_{ex}(x)\;\mbox{d}x = \underset{\Sigma\cap\Sigma_2}{\int}\Delta_{ex}(x)\;\mbox{d}x + \underset{\Sigma\cap(\Sigma_1\cup\Sigma_0)}{\int}\Delta_{ex}(x)\;\mbox{d}x,
\end{equation*}
combining which with (\ref{et4}) yields
\begin{equation}\label{et5}
\underset{D}{\int}\Delta_{ex}(x)\min\big(\mathpzc{L}_{\boldsymbol{\alpha}}(x),1\big)\;\mbox{d}x = \underset{\Sigma}{\int}\Delta_{ex}(x)\;\mbox{d}x - \underset{\Sigma\cap(\Sigma_1\cup\Sigma_0)}{\int}\Delta_{ex}(x)\;\mbox{d}x.
\end{equation}
From (\ref{et3}) and (\ref{et5}) we arrive at
\begin{align*}
&\underset{D}{\int}\Delta_{in}(x)\max\big(\mathpzc{L}_{\boldsymbol{\alpha}}(x),0\big)\;\mbox{d}x-\underset{D}{\int}\Delta_{ex}(x)\min\big(\mathpzc{L}_{\boldsymbol{\alpha}}(x),1\big)\;\mbox{d}x\\=&\underset{(\Sigma_1\cup\Sigma_2)\cap(D\setminus\Sigma)}{\int} \!\!\!\!\Delta_{in}(x)\mathpzc{L}_{\boldsymbol{\alpha}}(x)\;\mbox{d}x   + \underset{\Sigma\cap(\Sigma_1\cup\Sigma_0)}{\int}\Delta_{ex}(x)\;\mbox{d}x-\underset{\Sigma}{\int}\Delta_{ex}(x)\;\mbox{d}x\\\geq& -\underset{\Sigma}{\int}\Delta_{ex}(x)\;\mbox{d}x.\qquad\qed
\end{align*}

\subsection{Proof of Theorem \ref{th5}} We continue to use the notations (\ref{et3x}).

We first show if (\ref{eq32}) holds, $\balpha^*$ minimizes (\ref{eq31}). For this purpose we evaluate each component of the convex cost assuming that (\ref{eq32}) holds:
\begin{align}\nonumber \label{et6}
\underset{D}{\int}\Delta_{in}(x)\max\big(\mathpzc{L}_{\boldsymbol{\alpha}^*}(x),0\big)\;\mbox{d}x &= \underset{\Sigma}{\int}\Delta_{in}(x)\max\big(\mathpzc{L}_{\boldsymbol{\alpha}^*}(x),0\big)\;\mbox{d}x \\&+ \underset{D\setminus \Sigma}{\int}\Delta_{in}(x)\max\big(\mathpzc{L}_{\boldsymbol{\alpha}^*}(x),0\big)\;\mbox{d}x\nonumber\\  &=0,
\end{align}
where the first term vanishes by the LOC and the second term becomes zero using (\ref{eq32}). Moreover,
\begin{align}\nonumber
\underset{D}{\int}\Delta_{ex}(x)\min\big(\mathpzc{L}_{\boldsymbol{\alpha}^*}(x),1\big)\;\mbox{d}x &= \underset{\Sigma}{\int}\Delta_{ex}(x)\min\big(\mathpzc{L}_{\boldsymbol{\alpha}^*}(x),1\big)\;\mbox{d}x \\\nonumber &+ \underset{D\setminus \Sigma}{\int}\Delta_{ex}(x)\min\big(\mathpzc{L}_{\boldsymbol{\alpha}^*}(x),1\big)\;\mbox{d}x\\ &= \underset{\Sigma}{\int}\Delta_{ex}(x)\;\mbox{d}x.\label{et7}
\end{align}
Comparing (\ref{et6}) and (\ref{et7}) with the result from Proposition \ref{prop4} we can see that for $\balpha^*$ the convex cost (\ref{eq31}) reaches its lower bound and therefore $\balpha^*$ is a minimizer.

We now show if $\balpha^*$ minimizes (\ref{eq31}) then (\ref{eq32}) should hold. From Proposition \ref{th1} we know there exists $\tilde\balpha\in \mathbb{R}^{n_s}$ such that
\begin{equation}\label{et8}
\left\{\begin{array}{lc}
\mathpzc{L}_{\boldsymbol{\tilde\alpha}}(x)\geq 1&x\in\Sigma\\
\mathpzc{L}_{\boldsymbol{\tilde\alpha}}(x)\leq 0&x\in D\setminus \Sigma
\end{array}.
\right.
\end{equation}
Knowing (\ref{et8}) holds, a similar argument as (\ref{et6}) and (\ref{et7}) results in
\begin{equation}\label{et9}
\underset{D}{\int}\Delta_{in}(x)\max\big(\mathpzc{L}_{\boldsymbol{\tilde\alpha}}(x),0\big)\;\mbox{d}x-\underset{D}{\int}\Delta_{ex}(x)\min\big(\mathpzc{L}_{\boldsymbol{\tilde\alpha}}(x),1\big)\;\mbox{d}x = -\underset{\Sigma}{\int}\Delta_{ex}(x)\;\mbox{d}x.
\end{equation}
From Proposition \ref{prop4} we know if $\balpha^*$ is a minimizer, then
\begin{equation}\label{et10}
\underset{D}{\int}\Delta_{in}(x)\max\big(\mathpzc{L}_{\boldsymbol{\alpha}^*}(x),0\big)\;\mbox{d}x-\underset{D}{\int}\Delta_{ex}(x)\min\big(\mathpzc{L}_{\boldsymbol{\alpha}^*}(x),1\big)\;\mbox{d}x = -\underset{\Sigma}{\int}\Delta_{ex}(x)\;\mbox{d}x.
\end{equation}
This is due to the fact that $-\int_\Sigma\Delta_{ex}(x)\;\mbox{d}x$ is not only a lower bound and according to (\ref{et9}) it is also attainable. Thanks to the LOC we have
\begin{equation}\label{et11}
\underset{D}{\int}\Delta_{in}(x)\max\big(\mathpzc{L}_{\boldsymbol{\alpha}^*}(x),0\big)\;\mbox{d}x = \underset{D\setminus \Sigma}{\int}\Delta_{in}(x)\max\big(\mathpzc{L}_{\boldsymbol{\alpha}^*}(x),0\big)\;\mbox{d}x
\end{equation}
and
\begin{equation}\label{et12}
\underset{D}{\int}\Delta_{ex}(x)\min\big(\mathpzc{L}_{\boldsymbol{\alpha}^*}(x),1\big)\;\mbox{d}x = \underset{\Sigma}{\int}\Delta_{ex}(x)\min\big(\mathpzc{L}_{\boldsymbol{\alpha}^*}(x),1\big)\;\mbox{d}x.
\end{equation}
Using (\ref{et11}) and (\ref{et12}) in (\ref{et10}) yields
\begin{equation}\label{et13}
\underset{D\setminus \Sigma}{\int}\Delta_{in}(x)\max\big(\mathpzc{L}_{\boldsymbol{\alpha}^*}(x),0\big)\;\mbox{d}x+ \underset{\Sigma}{\int}\Delta_{ex}(x)\Big(1-\min\big(\mathpzc{L}_{\boldsymbol{\alpha}^*}(x),1\big)\Big)\;\mbox{d}x=0.
\end{equation}
Since the integrants in \ref{et13} are non-negative and $\Delta_{in}(x)>0$ for $x\in D\setminus \Sigma$, and $\Delta_{ex}(x)>0$ for $x\in \Sigma$, we must have
\[\left \{\begin{array}{lc}
\min\big(\mathpzc{L}_{\boldsymbol{\alpha}^*}(x),1\big) = 1 & x\in \mbox{int}(\Sigma)\\
\max\big(\mathpzc{L}_{\boldsymbol{\alpha}^*}(x),0\big) = 0 & x\in \mbox{int}(D\setminus \Sigma)
\end{array},
\right.
\]
which leads us to (\ref{eq32}).\qed

\subsection{Proof of Proposition \ref{pr2}}
(a) The proof of this part is straightforward:
\begin{align*}
&i_1\neq i_2 \Rightarrow \exists j\in\{1,2,\cdots,n\}\quad s.t.\quad \boldsymbol{J}_j^{(i_1)}\neq\boldsymbol{J}_j^{(i_2)}\\
\therefore \quad&\Theta_{\boldsymbol{J}_j^{(i_1)}}(\mathcal{S}_j) = \Big(\Theta_{\boldsymbol{J}_j^{(i_2)}}(\mathcal{S}_j)\Big)^c \Rightarrow \Omega_{i_1}\cap\Omega_{i_2}\seq\emptyset.
\end{align*}

(b) We prove this part for the case $n_\Omega = 2^n-1$ for which we use induction. The argument certainly holds for $n=2$. We now assume the argument holds for $n-1$, i.e.,
\[\bigcup_{i=1}^{2^{n-1}-1}\Omega_i =\bigcup_{j=1}^{n-1}\mathcal{S}_j.
\]
We add $\mathcal{S}_n$ to the collection and perform a DSD to obtain $\Omega'$ shapelets. Based on (\ref{eq10}) by using an appropriate indexing we can formulate the $\Omega'$ shapelets for the worst case $n_{\Omega'}=2^n-1$, as
\begin{equation*}
\Omega_i' = \left \{
\begin{array}{lr}
\Omega_i\cap\mathcal{S}_n& i =1,2,\cdots,2^{n-1}-1\\
\bigcap_{j=1}^{n-1}\mbox{cl}(\mathcal{S}_i^c)\cap \mathcal{S}_n & i=2^{n-1}\\
\Omega_{i-2^{n-1}}\cap\mbox{cl}(\mathcal{S}_n^c) & i =2^{n-1}+1,\cdots,2^n-1
\end{array}.
\right .
\end{equation*}
Therefore
\begin{align*}
\bigcup_{i=1}^{2^n-1} \Omega_i' &=
 \Big(\bigcup_{i=1}^{2^{n-1}-1}\big( \Omega_i\cap\mathcal{S}_n \big) \Big)\cup \Big(\bigcup_{i=1}^{2^{n-1}-1}\big( \Omega_i\cap\mbox{cl}(\mathcal{S}_n^c)\big) \Big)\cup \Big( \bigcap_{j=1}^{n-1}\mbox{cl}(\mathcal{S}_i^c)\cap \mathcal{S}_n \Big)\\&=
 \bigg(\bigcup_{i=1}^{2^{n-1}-1}\Big( \Omega_i\cap\big(\mathcal{S}_n\cup\mbox{cl}(\mathcal{S}_n^c)\big) \Big) \bigg)\cup \Big( \bigcap_{j=1}^{n-1}\mbox{cl}(\mathcal{S}_i^c)\cap \mathcal{S}_n \Big)\\&=
 \Big( \bigcup_{i=1}^{2^{n-1}-1}(\Omega_i\cap D)\Big) \cup \Big( \bigcap_{j=1}^{n-1}\mbox{cl}(\mathcal{S}_i^c)\cap \mathcal{S}_n \Big)\\& =
 \Big( \bigcup_{i=1}^{2^{n-1}-1}\Omega_i\Big) \cup \Big( \bigcap_{j=1}^{n-1}\mbox{cl}(\mathcal{S}_i^c)\cap \mathcal{S}_n \Big)\\& =
 \Big( \bigcup_{j=1}^{n-1}\mathcal{S}_j\Big) \cup \bigg( \mbox{cl}\Big(\big(\bigcup_{j=1}^{n-1}\mathcal{S}_i\big)^c\Big)\cap \mathcal{S}_n \bigg)\\&=  D \cap (\bigcup_{j=1}^{n}\mathcal{S}_j)\\&= \bigcup_{j=1}^{n}\mathcal{S}_j
\end{align*}
Which completes the inductive proof. When $n_\Omega < 2^n-1$ the argument still holds as neglecting the empty shapelets does not affect the proof.

(c) We prove the claim for an arbitrary $j=j_0$. We exclude $\mathcal{S}_{j_0}$ from the collection and perform a DSD on $\{\mathcal{S}_1, \mathcal{S}_2, \cdots , \mathcal{S}_n\}\setminus \{\mathcal{S}_{j_0}\}$ to obtain $\Omega'$ shapelets. We again consider the worst case that $n_{\Omega'}=2^{n-1}-1$. Using the result from part (b) we certainly have
\[\bigcup_{i=1}^{2^{n-1}-1}\Omega_i' = \bigcup_{\begin{subarray}{l}
        1\leq j\leq n\\ j\neq j_0
      \end{subarray}}\mathcal{S}_j.
\]
Consider momentarily that $j_0=n$. For $i\in \mathcal{I}_n$, the constructor vectors associated with the $\Omega$ shapelets are either in the form of $[\boldsymbol{0}_{1\times (n-1)}, 1]^T$ or $[{\boldsymbol{J}'}^T, 1]^T$, where $\boldsymbol{J}'$ is a constructor vector associated with the $\Omega'$ shapelets. Based on this argument and the definition of a shapelet in (\ref{eq10}), we have
\begin{align*}
\bigcup_{i\in \mathcal{I}_{j_0}} \Omega_i &= \Big(\bigcup_{i=1}^{2^{n-1}-1}\Omega_i'\cap\mathcal{S}_{j_0}\Big)\cup \Big(\mathcal{S}_{j_0}\cap\big(\bigcap_{\begin{subarray}{l}
        1\leq j\leq n\\ j\neq j_0
      \end{subarray}}\mbox{cl}(\mathcal{S}_j^c)\big) \Big)\\&=\Big(\mathcal{S}_{j_0}\cap \big(\bigcup_{i=1}^{2^{n-1}-1}\Omega_i'\big)\Big)\cup \bigg(\mathcal{S}_{j_0}\cap\mbox{cl}\Big(\big(\bigcup_{\begin{subarray}{l}
        1\leq j\leq n\\ j\neq j_0
      \end{subarray}}\mathcal{S}_j\big)^c\Big) \bigg)\\&=\Big(\mathcal{S}_{j_0}\cap \big(\bigcup_{\begin{subarray}{l}
        1\leq j\leq n\\ j\neq j_0
      \end{subarray}}\mathcal{S}_j\big)\Big)\cup \bigg(\mathcal{S}_{j_0}\cap\mbox{cl}\Big(\big(\bigcup_{\begin{subarray}{l}
        1\leq j\leq n\\ j\neq j_0
      \end{subarray}}\mathcal{S}_j\big)^c \Big)\bigg)\\&=\mathcal{S}_{j_0}\cap D\\&= \mathcal{S}_{j_0},
\end{align*}
which completes the proof. \qed

\subsection{Proof of Theorem \ref{th2x}}
Let $\Sigma = (\bigcup_{j\in \Ip}\mathcal{S}_j) \setminus (\bigcup_{j\in \In}\mathcal{S}_j)$. We start by proving part (a).

Consider $j_0\in \Ip$. To avoid long expressions we use the notations
\[\mathcal{S}_+ = \bigcup_{\begin{subarray}{l}
        j\in \Ip\\ j\neq j_0
      \end{subarray}}\mathcal{S}_j \quad\mbox{and} \quad \mathcal{S}_- = \bigcup_{j\in \In}\mathcal{S}_j.
\]
The non-redundancy of $\Sigma$ requires that excluding $\mathcal{S}_{j_0}$ from the composition should cause a change in the measure of $\Sigma$. In other words, since in general
\[(\mathcal{S}_+\setminus \mathcal{S}_- )   \subseteq    \big ( (\mathcal{S}_{j_0}\cup \mathcal{S}_+)\setminus \mathcal{S}_-\big),
\]
we must have $\big ( (\mathcal{S}_{j_0}\cup \mathcal{S}_+)\setminus \mathcal{S}_-\big) \setminus  (\mathcal{S}_+\setminus \mathcal{S}_- )\nseq \emptyset$, viz.,
\begin{align*}
\exists \tilde\Omega \quad s.t. \quad \mbox{int}(\tilde\Omega)\neq \emptyset \quad \mbox{and}\quad \tilde\Omega &=\big ( (\mathcal{S}_{j_0}\cup \mathcal{S}_+)\setminus \mathcal{S}_-\big) \setminus  (\mathcal{S}_+\setminus \mathcal{S}_- )\\& =  \big ( (\mathcal{S}_{j_0}\setminus  \mathcal{S}_-)\cup  (\mathcal{S}_+\setminus \mathcal{S}_- )\big) \setminus  (\mathcal{S}_+\setminus \mathcal{S}_- )\\ &=  (\mathcal{S}_{j_0}\setminus  \mathcal{S}_-) \setminus  (\mathcal{S}_+\setminus \mathcal{S}_- )\\& = (\mathcal{S}_{j_0}\cap  \mathcal{S}_-^c) \cap (\mathcal{S}_+^c\cup \mathcal{S}_- )\\& = (\mathcal{S}_{j_0}\cap  \mathcal{S}_-^c\cap \mathcal{S}_+^c) \cup (\mathcal{S}_{j_0}\cap  \mathcal{S}_-^c \cap \mathcal{S}_- )\\& = \mathcal{S}_{j_0}\cap  \mathcal{S}_-^c\cap \mathcal{S}_+^c.
\end{align*}
Another way of interpreting this result is there exists a region $\tilde \Omega \subset \mathcal{S}_{j_0}$ that
\begin{equation}\label{et2x}
\mbox{int}(\tilde\Omega)\neq \emptyset \quad \mbox{and}\quad \forall j\in (\Ip\setminus \{j_0\})\cup\In: \quad \tilde\Omega\cap \mathcal{S}_j=\emptyset.
\end{equation}
Based on Proposition \ref{pr2}, we know the DSD process performs a partitioning over the shapes (neglecting the overlaps on the boundaries which are null sets). In other words, the outcome of set operations $\cup$, $\cap$ and $\setminus$ amongst the $\mathcal{S}_j$, $j\in \Ip\cup\In$, can always be written as the union of some shapelets, up to a null set. Therefore (\ref{et2x}) implies that there exists a shapelet $\Omega\subset\mathcal{S}_{j_0}$ that satisfies
\begin{equation}\label{et2xx}
\forall j\in (\Ip\cup \In) \; \mbox{and}\; j\neq j_0 : \quad \Omega \cap \mathcal{S}_{j}\seq \emptyset.
\end{equation}
It is straightforward to see that if
\[\{\Omega_i,\boldsymbol{J}^{(i)}\}_{i=1}^{n_\Omega} = \mbox{DSD}\Big( \{\mathcal{S}_j\}_{j\in \Ip \cup \In}\Big),
\]
the only shapelet $\Omega\subset\mathcal{S}_{j_0}$ that satisfies (\ref{et2xx}) is the one with a constructor vector that is zeros everywhere except the $j_0$-th element.

To prove part (b), we take a similar strategy. Consider this time $j_0\in \In$. With a slight reuse of notation, this time we define
\[\mathcal{S}_+ = \bigcup_{j\in \Ip}\mathcal{S}_j \quad \mbox{and} \quad \mathcal{S}_- = \bigcup_{\begin{subarray}{l}
        j\in \In\\ j\neq j_0
      \end{subarray}}\mathcal{S}_j.
\]
Excluding $\mathcal{S}_{j_0}$ from the composition should cause a change in the measure of $\Sigma$. Since in general
\[ \big ( \mathcal{S}_+  \setminus  (\mathcal{S}_-\cup \mathcal{S}_{j_0}) \big)    \subseteq   (\mathcal{S}_+\setminus \mathcal{S}_- ) ,
\]
the non-redundancy of $\Sigma$ requires
\begin{align*}
\exists \tilde\Omega \quad s.t. \quad \mbox{int}(\tilde\Omega)\neq \emptyset \quad \mbox{and}\quad     \tilde\Omega &=  (\mathcal{S}_+\setminus \mathcal{S}_- ) \setminus \big ( \mathcal{S}_{+}\setminus (\mathcal{S}_-\cup \mathcal{S}_{j_0})   \big) \\& =  (\mathcal{S}_+\setminus \mathcal{S}_- ) \cap \big ( \mathcal{S}_{+}\setminus (\mathcal{S}_-\cup \mathcal{S}_{j_0})   \big)^c  \\&= (\mathcal{S}_+\setminus \mathcal{S}_- ) \cap  ( \mathcal{S}_{+}^c\cup \mathcal{S}_-\cup \mathcal{S}_{j_0}) \\ & = (\mathcal{S}_+\setminus \mathcal{S}_- ) \cap  \big(( \mathcal{S}_{+}\setminus \mathcal{S}_-)^c\cup \mathcal{S}_{j_0}\big)\\& = \mathcal{S}_{j_0}\cap \mathcal{S}_+\cap \mathcal{S}_-^c
\end{align*}
This result basically means that there exists a region $\tilde \Omega \subset (\mathcal{S}_{j_0}\cap\mathcal{S}_+)$ such that
\[\mbox{int}(\tilde\Omega)\neq \emptyset \quad \mbox{and}\quad \forall j\in \In\setminus \{j_0\}: \quad \tilde\Omega \cap \mathcal{S}_j=\emptyset.
\]
A similar argument as part (a) justifies the existence of some shapelet $\Omega \subset \mathcal{S}_{j_0} \cap\mathcal{S}_+$ such that
\[\forall j\in \In \; \mbox{and}\; j\neq j_0 : \quad \Omega \cap \mathcal{S}_{j}\seq \emptyset.\qed
\]

\subsection{Proof of Proposition \ref{th2xx}}
We consider a DSD as
\begin{equation*}
\{\Omega_i,\boldsymbol{J}^{(i)}\}_{i=1}^{n_\Omega} = \mbox{DSD}\Big( \{\mathcal{S}_j\}_{j\in \Ip \cup \In}\Big).
\end{equation*}
The proof to each part is presented as follows:

(a) From Theorem \ref{th2x}(b) we have
\[\forall j\in \In, \; \exists i \in \mathcal{I}_j, \quad s.t. \quad \Omega_i \subset \mathcal{S}_j\cap (\bigcup_{\hat{j}\in \Ip}\mathcal{S}_{\hat{j}})
\]
and for $j'\in \In$, $j'\neq j:$
\[\Omega_i \cap \mathcal{S}_{j'} \seq \emptyset.
\]
As a result, the constraint (\ref{eq20}) corresponding to the shapelet $\Omega_i$ simplifies to
\[\alpha_j \leq -|\Kp^{(i)}|.
\]
Since $\Omega_i\subset (\bigcup_{\hat{j}\in \Ip}\mathcal{S}_{\hat{j}})$, we have $|\Kp^{(i)}|\geq 1$. In other words, any solution of the convex program (\ref{eq21}) needs to satisfy
\[\alpha_j^\dagger \leq -|\Kp^{(i)}|\leq -1.
\]

(b) For each $j\in \Ip$, a result of Theorem \ref{th2x}(a) is the existence of a unique shapelet $\Omega_i\subset \mathcal{S}_j$ that does not overlap with any other shapes $\mathcal{S}_{j'}$, $j'\in ({\Ip\setminus\{j\})\cup \In}$. Since for all $j\in \Ip$ we set $\alpha_j^\dagger$ to one, we must have $\beta_i^\dagger=1$.

(c) For a given $j\in \In$, by construction, the set of constraints (\ref{eq19}) assure that for every $i\in \mathcal{I}_j$, $\beta_i^\dagger\leq 0$ (this is another way of saying $\mathcal{L}_{\balpha^\dagger}(x)\leq 0$ over $\mbox{int}(\Omega_i)$). We only need to show that for a given $j\in \In$, it is not possible to have $\beta_i^\dagger<0$ for every $i\in \mathcal{I}_j$.

Suppose this is the case and for a $j\in \In$, $\forall i\in\mathcal{I}_j$, $\beta_i^\dagger<0$. This means none of the constraints in (\ref{eq21}), that $\alpha_j$ contributes in, are active (attain an equality). But if this is the case, we can increase $\alpha_j$ away from $\alpha_j^\dagger$ until one of the inequalities becomes active. In other words, concerning the linear program (\ref{eq21}), we can find a point that produces a larger cost than $\balpha_{\In}^\dagger$, which contradicts the theorem's supposition.

(d) To facilitate the proof, we prescribe the index sets to be
\[\Ip = \{1,2,\cdots,n_\oplus\}\quad \mbox{and}\quad \In = \{n_\oplus+1,\cdots,n_\oplus+n_\ominus\},
\]
where $n_\oplus,n_\ominus\in\mathbb{N}$. We form the corresponding bearing matrix $\boldsymbol{B}^\mathpzc{R}\in \{0,1\}^{n_\Omega\times(n_\oplus+n_\ominus)}$, through which $\Beta^\dagger = \boldsymbol{B}^\mathpzc{R}\balpha^\dagger$. To prove the claim, it suffices to show that $\boldsymbol{B}^\mathpzc{R}$ is full column rank.

By applying suitable permutation on the rows of $\boldsymbol{B}^\mathpzc{R}$ we can reshape it as

  \[
\boldsymbol{P}\boldsymbol{B}^\mathpzc{R}=
\begin{array}{c@{}c}
\xleftrightarrow{\makebox[.63cm]{$n_\oplus$}}\xleftrightarrow{\makebox[.43cm]{$n_\ominus$}}\\[-.5em]\left[
  \begin{BMAT}{c}{c:c}
      \begin{BMAT}[4pt]{c:c}{cc}
       \boldsymbol{I} & \;\boldsymbol{0} \\
       \boldsymbol{B_1} & \;\boldsymbol{I}
      \end{BMAT}
      \\
      \begin{BMAT}[4pt]{c}{c}
        \boldsymbol{B}_2
      \end{BMAT}
    \end{BMAT}
\right]
\end{array}.
\]
Here, $\boldsymbol{P}$ is the underlying permutation matrix, $\boldsymbol{I}$ represents the identity matrix and the blocks $\boldsymbol{B}_1$ and $\boldsymbol{B}_2$ are some binary $0-1$ matrices. The structure of the top $n_\oplus\times(n_\oplus+n_\ominus)$ block of $\boldsymbol{P}\boldsymbol{B}^\mathpzc{R}$ is a direct result of Theorem \ref{th2x}(a) that guarantees to have $n_\oplus$ independent constructor vectors that are all zeros, except exactly one unit-valued element that occurs between the indices 1 and $n_\oplus$. The middle block structure $[\boldsymbol{B}_1\; \boldsymbol{I}]$ is thanks to Theorem \ref{th2x}(b) that guarantees to have at least $n_\ominus$ constructor vectors that are all zeros on the index set $\In$, except exactly one active index the location of which differs for every $j\in\In$ (Figure \ref{fig3} and the example at the end of Section \ref{DSDsec} can be helpful here).

Demonstrating the aforementioned structure for $\boldsymbol{P}\boldsymbol{B}^\mathpzc{R}$ proves that it is a full column rank matrix, regardless of what the binary matrices $\boldsymbol{B}_1$ and $\boldsymbol{B}_2$ are.\qed

\subsection{Proof of Theorem \ref{thbasic}}
The claim is a corollary of the following lemma.
\begin{lemma}
Suppose $\balpha^\dagger\in\mathbb{R}^n$ is a maximizer of the linear program
\begin{equation}\label{eqxx2}
       \underset{\balpha}{\max} \;\; \boldsymbol{1}^T\balpha\qquad s.t.\quad \boldsymbol{A} \balpha \preceq \boldsymbol{b}
\end{equation}
and $\Gamma = \{i: \boldsymbol{A}_{i,:}\balpha^\dagger = \boldsymbol{b}_i \}$.\\
(a) There exists a vector $\boldsymbol{w}\succeq \boldsymbol{0}$ such that
\begin{equation}\label{eqxx4}
(\boldsymbol{A}_{\Gamma,:})^T\boldsymbol{w}=\boldsymbol{1}.
\end{equation}
(b) If $\rank(\boldsymbol{A}_{\Gamma,:}) = |\Gamma|=n$ and the essentially non-negative solution of (\ref{eqxx4}) is strictly positive, then $\balpha^\dagger$ is the unique maximizer of (\ref{eqxx2}).\\
(c) If $|\Gamma|=n$ and $\balpha^\dagger$ is the unique maximizer of (\ref{eqxx2}), then $\rank(\boldsymbol{A}_{\Gamma,:}) = n$ and the solution of (\ref{eqxx4}) is strictly positive.
\end{lemma}
\begin{proof}
Parts of the proof follow similar lines of argument as \cite{mangasarian1979uniqueness}, which are appropriately modified and included here to keep the paper self-contained.

The dual linear program associated with (\ref{eqxx2}) is
\begin{equation*}
\left\{
\begin{array}{lc}
       \underset{\boldsymbol{z}}{\min} & \boldsymbol{b}^T\boldsymbol{z} \\
       s.t.& \boldsymbol{A}^T \boldsymbol{z} =\boldsymbol{1}\\& \boldsymbol{z}\succeq \boldsymbol{0}
     \end{array}.
   \right.
\end{equation*}
Considering $\boldsymbol{z}^\dagger$ to be a minimizer of the dual problem, the optimality of $\balpha^\dagger$ requires $\boldsymbol{z}^\dagger_{\Gamma^c}=\boldsymbol{0}$, and subsequently $(\boldsymbol{A}_{\Gamma,:})^T\boldsymbol{z}_\Gamma^\dagger =\boldsymbol{1}$. To establish part (a), it suffices to set $\boldsymbol{w}$ to $\boldsymbol{z}_\Gamma^\dagger$.

To prove the uniqueness of $\balpha^\dagger$, suppose there exists another maximizer $\hat\balpha\neq \balpha^\dagger$. Based on the linearity of the problem, any convex combination of the maximizers, such as $(\balpha^\dagger + \hat\balpha)/2$, is also a maximizer. More specifically, $\boldsymbol{1}^T \balpha^\dagger = \boldsymbol{1}^T (\balpha^\dagger + \hat\balpha)/2$ or
\begin{equation}
\boldsymbol{1}^T(\balpha^\dagger - \hat\balpha) = \boldsymbol{0}.
\end{equation}
The feasibility of the maximizer also requires
\begin{equation*}
\frac{1}{2}\boldsymbol{A}_{\Gamma,:}\balpha^\dagger + \frac{1}{2}\boldsymbol{A}_{\Gamma,:} \hat\balpha\preceq \boldsymbol{b}_{\Gamma} = \boldsymbol{A}_{\Gamma,:}\balpha^\dagger,
\end{equation*}
or equivalently
\begin{equation}
\boldsymbol{A}_{\Gamma,:}(\balpha^\dagger - \hat\balpha)\succeq \boldsymbol{0}.
\end{equation}
Suppose there exists $j\in \Gamma$ such that
\[\boldsymbol{A}_{j,:}(\balpha^\dagger - \hat\balpha)>0.
\]
Based on the lemma's assumptions we have $\boldsymbol{w}=\boldsymbol{z}^\dagger_\Gamma\succ \boldsymbol{0}$ and subsequently
\begin{align*}
{0}&< {\boldsymbol{z}^\dagger_\Gamma}^T \boldsymbol{A}_{\Gamma,:}(\balpha^\dagger - \hat\balpha)\\& = {\boldsymbol{z}^\dagger}^T \boldsymbol{A}(\balpha^\dagger - \hat\balpha)\\& = \boldsymbol{1}^T(\balpha^\dagger - \hat\balpha)\\ &={0},
\end{align*}
which is not possible. As a result, $\boldsymbol{A}_{\Gamma,:}(\balpha^\dagger - \hat\balpha)=\boldsymbol{0}$, that given the full rank property of $\boldsymbol{A}_{\Gamma,:}$, is only possible when $\hat\balpha = \balpha^\dagger$. This completes the proof for part (b).

To prove the necessary condition (c), we will make use of the following theorem which is a variant of Theorem 1 proved in \cite{mangasarian1979uniqueness}:
\begin{theorem}\label{thwq}
A solution $\balpha^\dagger$ of the linear program (\ref{eqxx2}) is unique, if and only if for any $\bdelta\in \mathbb{R}^n$ there exists a real positive number $\epsilon$ such that $\balpha^\dagger$ remains a solution of the perturbed linear program
\begin{equation}\label{wq1}
\underset{\balpha}{\max} \;\; (\boldsymbol{1}+\epsilon \bdelta)^T\balpha\qquad s.t.\quad \boldsymbol{A} \balpha \preceq \boldsymbol{b}.
\end{equation}
\end{theorem}
Suppose $\boldsymbol{y}^\dagger$ is a solution to the dual program of (\ref{wq1}), which is cast as
\begin{equation*}
\left\{
\begin{array}{lc}
       \underset{\boldsymbol{y}}{\min} & \boldsymbol{b}^T\boldsymbol{y} \\
       s.t.& \boldsymbol{A}^T \boldsymbol{y} =\boldsymbol{1}+\epsilon \bdelta\\& \boldsymbol{y}\succeq \boldsymbol{0}
     \end{array}.
   \right.
\end{equation*}
We first assume $\boldsymbol{A}_{\Gamma,:}$ is rank deficient and the solution to (\ref{eqxx2}) is unique, and exhibit a contradiction. Based on this assumption $\rank(\boldsymbol{A}_{\Gamma,:})<n$ and we can find a vector $\bdelta_0\neq \boldsymbol{0}$ such that
\begin{equation*}
\boldsymbol{A}_{\Gamma,:}\bdelta_0 = \boldsymbol{0}.
\end{equation*}
On the other hand, for any $(\epsilon,\bdelta)$ pair used in (\ref{wq1}) we have
\begin{equation}\label{wq2}
\epsilon \bdelta = (\boldsymbol{A}_{\Gamma,:})^T(\boldsymbol{y}_\Gamma^\dagger - \boldsymbol{z}_\Gamma^\dagger),
\end{equation}
which holds in view of $\boldsymbol{y}^\dagger_{\Gamma^c}=\boldsymbol{0}$ and
\begin{align*}\nonumber
(\boldsymbol{A}_{\Gamma,:})^T\boldsymbol{y}_\Gamma^\dagger &= \boldsymbol{A}^T\boldsymbol{y}^\dagger \\ \nonumber &= \boldsymbol{1}+\epsilon \bdelta \\ &= (\boldsymbol{A}_{\Gamma,:})^T\boldsymbol{z}_\Gamma^\dagger + \epsilon \bdelta.
\end{align*}
If we set $\bdelta = \bdelta_0$, the uniqueness of the solution to (\ref{eqxx2}) warrants the existence of $\epsilon_0>0$ for which (\ref{wq2}) holds. A pre-multiplication of both sides by $\bdelta_0^T$ yields
\begin{equation*}
0<\epsilon_0 \bdelta_0^T \bdelta_0 = \bdelta_0^T(\boldsymbol{A}_{\Gamma,:})^T(\boldsymbol{y}_\Gamma^\dagger - \boldsymbol{z}_\Gamma^\dagger)=0,
\end{equation*}
which is a contradiction.

As the other possible case, assume that $\rank(\boldsymbol{A}_{\Gamma,:})=n$ but the solution of (\ref{eqxx4}) is not strictly positive. In this case we have $\Gamma = \Gamma' \cup \Gamma''$, where
\begin{equation*}
\Gamma' = \{i: \boldsymbol{A}_{i,:}\balpha^\dagger = \boldsymbol{b}_i , \boldsymbol{z}_i^\dagger > 0\}, \quad \mbox{and}\quad \Gamma'' = \{i: \boldsymbol{A}_{i,:}\balpha^\dagger = \boldsymbol{b}_i , \boldsymbol{z}_i^\dagger = 0\}.
\end{equation*}
Consider a vector $\boldsymbol{\xi}\in\mathbb{R}^n$ such that
\begin{equation*}
\boldsymbol{\xi}_{\Gamma'}=\boldsymbol{0},
\qquad \mbox{and}\qquad \boldsymbol{\xi}_{\Gamma''}=-\boldsymbol{1}.
\end{equation*}
Since $\rank(\boldsymbol{A}_{\Gamma,:})=n$, the equation
\begin{equation*}
\boldsymbol{A}_{\Gamma,:}\tilde \bdelta_0 = \boldsymbol{\xi}
\end{equation*}
has a non-zero solution $\tilde \bdelta_0$. Again, appealing to (\ref{wq2}) and taking into account that $\Gamma = \Gamma' \cup \Gamma''$ yields
\begin{align}\nonumber
\epsilon \bdelta &= (\boldsymbol{A}_{\Gamma',:})^T(\boldsymbol{y}_{\Gamma'}^\dagger - \boldsymbol{z}_{\Gamma'}^\dagger) + (\boldsymbol{A}_{\Gamma'',:})^T(\boldsymbol{y}_{\Gamma''}^\dagger - \boldsymbol{z}_{\Gamma''}^\dagger)\\& = (\boldsymbol{A}_{\Gamma',:})^T(\boldsymbol{y}_{\Gamma'}^\dagger - \boldsymbol{z}_{\Gamma'}^\dagger) + (\boldsymbol{A}_{\Gamma'',:})^T\boldsymbol{y}_{\Gamma''}^\dagger. \label{wq3}
\end{align}
If the solution to (\ref{eqxx2}) is unique, using Theorem \ref{thwq} we can set $\bdelta$ to $\tilde\bdelta_0$ and expect to have $\tilde\epsilon_0>0$ for which (\ref{wq3}) holds. Again a pre-multiplication of both sides by $\tilde\bdelta_0^T$ yields a contradiction as
\begin{align*}\nonumber
0<\tilde\epsilon_0 \tilde\bdelta_0^T \tilde\bdelta_0 &= \tilde\bdelta_0^T(\boldsymbol{A}_{\Gamma',:})^T(\boldsymbol{y}_{\Gamma'}^\dagger - \boldsymbol{z}_{\Gamma'}^\dagger) + \tilde\bdelta_0^T(\boldsymbol{A}_{\Gamma'',:})^T\boldsymbol{y}_{\Gamma''}^\dagger \\ \nonumber & = 0 - \boldsymbol{1}^T \boldsymbol{y}_{\Gamma''}^\dagger\\& \leq 0.
\end{align*}
In other words, if $\rank(\boldsymbol{A}_{\Gamma,:}) < n$ or the solution of (\ref{eqxx4}) is not strictly positive, the solution of (\ref{eqxx2})  cannot be unique.
\end{proof}

\subsection{Proof of Proposition \ref{te1}} We first show the following lemma holds.
\begin{lemma}\label{leme1}
Under the conditions stated in Proposition \ref{te1}, for any $\balpha\in\mathbb{R}^{n_s}$, if $\|\balpha\|_1\leq \tau$, then $\boldsymbol{c}^T\balpha\leq \tau$.
\end{lemma}
\begin{proof}
Suppose the claim is not true and there exists $\hat\balpha$ such that
\begin{align}\label{ee1}
\|\hat\balpha\|_1\leq \tau
\end{align}
and
\begin{align}\label{ee2}
 \boldsymbol{c}^T\balpha> \tau.
\end{align}
Since $\boldsymbol{c}_{\Gamma} \in \{1,-1\}^{|\Gamma|}$ we have
\begin{align*}
\sum_{j\in\Gamma}c_j\hat\alpha_j\leq \sum_{j\in\Gamma}|\hat\alpha_j|,
\end{align*}
using which in (\ref{ee2}) gives
\begin{align}\label{ee4}
\sum_{j\in\Gamma}|\hat\alpha_j|+\sum_{j\in\Gamma^c}c_j\hat\alpha_j>\tau.
\end{align}
Consequently,
\begin{align}
\sum_{j\in\Gamma^c}|\hat\alpha_j|<\sum_{j\in\Gamma^c}c_j\hat\alpha_j\leq |\sum_{j\in\Gamma^c}c_j\hat\alpha_j|\leq \|\boldsymbol{c}_{\Gamma^c}\|_\infty \sum_{j\in\Gamma^c}|\hat\alpha_j|,
\end{align}
where the first inequality is a direct result of combining (\ref{ee4}) and (\ref{ee1}), and the last relation is basically the H\"{o}lder's inequality. Ultimately,
we must have
\[(1-\|\boldsymbol{c}_{\Gamma^c}\|_\infty)\sum_{j\in\Gamma^c}|\hat\alpha_j|<0
\]
which is only possible when $\|\boldsymbol{c}_{\Gamma^c}\|_\infty>1$. This would be a contradiction with the assumptions made about $\boldsymbol{c}$.
\end{proof}

We now proceed by showing that $\balpha^*$ is the unique minimizer of (\ref{ee-2}). Suppose this is not true and there exists $\hat\balpha\neq \balpha^*$ such that $G(\hat\balpha)\leq G(\balpha^*)$ (i.e., $\hat\balpha\in \mathcal{C}$) and $\|\hat\balpha\|\leq \tau$. By Lemma \ref{leme1}, we must have
\begin{align}\label{ee5}
\boldsymbol{c}^T\hat\balpha \leq \tau.
\end{align}
On the other hand, by definition of the tangent cone
\begin{align*}
\exists \epsilon>0 \quad s.t.\quad \balpha^* + \epsilon \boldsymbol{\delta} \in \mathcal{C}.
\end{align*}
Setting $\tilde \balpha = \balpha^* + \epsilon \boldsymbol{\delta}$, we must have $\tilde \balpha \in \mathcal{C}$ and
\begin{align}\label{ee6}
\boldsymbol{c}^T\tilde\balpha = \boldsymbol{c}^T\balpha^* + \epsilon \boldsymbol{c}^T\boldsymbol{\delta}>\tau.
\end{align}
Consider now the line segment connecting $\hat\balpha$ and $\tilde\balpha$ represented by
\begin{align*}
L=\big\{t\hat\balpha + (1-t)\tilde\balpha:t\in[0,1]\big\}.
\end{align*}
By the convexity of $\mathcal{C}$ we have $L\subset\mathcal{C}$. Also, the continuity of the line segment along with (\ref{ee5}) and (\ref{ee6}) require us to have a $\boldsymbol{z}\in L$ such that $\boldsymbol{c}^T\boldsymbol{z}=\tau$. The point $\boldsymbol{z}$ does not need to be the same as $\balpha^*$, because to have that, $\tilde\balpha$ needs to lie on the line passing through $\balpha^*$ and $\hat\balpha$. Since $\mathcal{C}$ has a nonempty interior, we can always choose $\tilde \balpha$ somehow that this is not the case. In other words, $\boldsymbol{z}$ is another minimizer of (\ref{ee-1}) and this is a contradiction with condition (I).\qed

\subsection{Proof of Theorem \ref{th7}}
As discussed at the beginning of Section \ref{sec-conv}, using the variable change $\Beta = \boldsymbol{B}\balpha$ makes the objective function in (\ref{ee-1}) separable and the corresponding convex program is cast as
\begin{equation}\label{ew1}
\min_{\Beta} \;\; \mathcal{G}(\Beta)\quad s.t. \quad \boldsymbol{K}\begin{bmatrix}\Beta\\ \balpha\end{bmatrix} = \begin{bmatrix}\boldsymbol{0}\\ \tau\end{bmatrix},
\end{equation}
where
\begin{align*}
\boldsymbol{K} = \begin{bmatrix} -\boldsymbol{I} & \boldsymbol{B}\\ \boldsymbol{0} & \boldsymbol{c}^T\end{bmatrix}.
\end{align*}
We start the proof by the following lemma, which provides a set of sufficient conditions for optimality and uniqueness, while treating the linear constraint in a rather general way. We will then discuss how these condition translate into the ones stated in the theorem.

\begin{lemma}\label{lm7}
Consider $\balpha^*\in\mathbb{R}^{n_s}$ and $\Beta^* = \boldsymbol{B}\balpha^* \in \mathbb{R}^{n_\Omega}$, which together are feasible for the convex program (\ref{ew1}), and all the entries of $\Beta^*$ lie outside the interval $(0,1)$. If there exists a vector $\Blambda\in \mathbb{R}^{n_\Omega +n_s}$ in the range of $\boldsymbol{K}^T$ such that
\begin{align}\label{ee7}
\left\{\begin{array}{ll}
\Blambda_\ell \in (-q_\ell,p_\ell-q_\ell) & \ell\in \Gamma_0\\
\Blambda_\ell \in (p_\ell-q_\ell,p_\ell) & \ell \in \Gamma_1
\end{array}
\right. and\;
\left\{\begin{array}{ll}
\Blambda_\ell =p_\ell & \ell\in \Gamma_{1^+}\\
\Blambda_\ell =-q_\ell & \ell \in \Gamma_{0^-}\\
\Blambda_\ell = 0 & \ell \in \big\{n_\Omega +1, \cdots ,n_\Omega +n_s\big\}
\end{array}
\right. ,
\end{align}
and for
\begin{equation}\label{ee7.5}\Gamma' = \Gamma_{0^-}\cup \Gamma_{1^+}\cup \big\{n_\Omega +1, \cdots ,n_\Omega +n_s\big\}
\end{equation}
the matrix $\boldsymbol{K}_{:,\Gamma'}$ is full column rank, then $\Beta^*$ is the unique minimizer of (\ref{ew1}) (and accordingly, $\balpha^*$ is the unique minimizer of (\ref{ee-1})).
\end{lemma}

\begin{proof}
We start by defining an objective function $\hat{\mathcal{G}}(.):\mathbb{R}^{n_\Omega+n_s}\to \mathbb{R}$:
\[\hat{\mathcal{G}}\Big(\begin{bmatrix}\Beta\\ \balpha\end{bmatrix}\Big) = \mathcal{G}(\Beta).
\]
Clearly, an equivalent formulation of the convex program (\ref{ew1}) is
\begin{equation}\label{ee8}
\min_{\Beta,\balpha} \;\; \hat{\mathcal{G}}\Big(\begin{bmatrix}\Beta\\ \balpha\end{bmatrix}\Big) \quad s.t. \quad \boldsymbol{K}\begin{bmatrix}\Beta\\ \balpha\end{bmatrix} = \begin{bmatrix}\boldsymbol{0}\\ \tau\end{bmatrix}.
\end{equation}
Using basic subgradient calculus \cite{boyd2004convex}, it can be verified that a vector $\Bz\in\mathbb{R}^{n_\Omega+n_s}$ is a subgradient of $\hat{\mathcal{G}}(.)$ at $[\Beta^T,\balpha^T]^T$ (denoted as $\Bz \in \partial \hat{\mathcal{G}}([\Beta^T,\balpha^T]^T)$) when $\Bz_\ell = 0$ for $\ell \in \big\{n_\Omega +1, \cdots ,n_\Omega +n_s\big\}$ and
\begin{align*}
\left\{\begin{array}{lll}
\Bz_\ell \in [-q_\ell,p_\ell-q_\ell] & \mbox{:}\!\!\!&\Beta_\ell = 0\\
\Bz_\ell \in [p_\ell-q_\ell,p_\ell] & \mbox{:}\!\!\!& \Beta_\ell = 1
\end{array}
\right. and\;
\left\{\begin{array}{llc}
\Bz_\ell =p_\ell & \mbox{:}\!\!\!&\Beta_\ell>1\\
\Bz_\ell =p_\ell - q_\ell & \mbox{:}\!\!\! & 0<\Beta_\ell<1\\
\Bz_\ell =-q_\ell & \mbox{:}\!\!\! &\Beta_\ell<0\\
\end{array}
\right. .
\end{align*}
Since $\Beta^*_\ell \notin (0,1)$ for all $\ell\in\{1,\cdots,n_\Omega\}$, (\ref{ee7}) implies that $\Blambda\in \partial \hat{\mathcal{G}}([{\Beta^*}^T,{\balpha^*}^T]^T)$.

To prove that $[{\Beta^*}^T\!\!,{\balpha^*}^T ]$ is the unique minimizer of (\ref{ee8}), we need to show that any other feasible vector $[\hat\Beta^T\!\!,\hat\balpha^T ] \neq [{\Beta^*}^T\!\!,{\balpha^*}^T ]$ generates a greater cost, i.e.,
\[\hat{\mathcal{G}}\Big(\begin{bmatrix}\hat\Beta\\ \hat\balpha\end{bmatrix}\Big)>\hat{\mathcal{G}}\Big(\begin{bmatrix}\Beta^*\\ \balpha^*\end{bmatrix}\Big).
\]
We proceed by defining a difference vector
\[\Bh=\begin{bmatrix}\hat\Beta\\ \hat\balpha\end{bmatrix} - \begin{bmatrix}\Beta^*\\ \balpha^*\end{bmatrix}.
\]
From the basic properties of the subgradient we have\cite{boyd2004convex}
\begin{equation}\label{ee10}
\forall \Bz\in\partial \hat{\mathcal{G}}\Big(\begin{bmatrix}\Beta^*\\ \balpha^*\end{bmatrix}\Big): \qquad \hat{\mathcal{G}}\Big(\begin{bmatrix}\hat\Beta\\ \hat\balpha\end{bmatrix}\Big)\geq \hat{\mathcal{G}}\Big(\begin{bmatrix}\Beta^*\\ \balpha^*\end{bmatrix}\Big) + \boldsymbol{h}^T \Bz.
\end{equation}
As $\Beta^*_\ell \notin (0,1)$ for all $\ell\in\{1,\cdots,n_\Omega\}$,  for the index set $\Gamma'$ in (\ref{ee7.5}) we can consider the complement set
\[{\Gamma'}^c = \Gamma_0\cup\Gamma_1,
\]
and decompose $\Blambda$ and $\Bz\in \partial \hat{\mathcal{G}}([{\Beta^*}^T,{\balpha^*}^T]^T)$
 as $\Blambda = \mathcal{P}_{\Gamma'} (\Blambda) + \mathcal{P}_{{\Gamma'}^c} (\Blambda)$ and $\Bz = \mathcal{P}_{\Gamma'} (\Bz) + \mathcal{P}_{{\Gamma'}^c} (\Bz)$, where $\mathcal{P}_A(.)$ is the projection operator onto the index set $A$:
\[ [\mathcal{P}_A(\boldsymbol{z})]_\ell = \left\{\begin{array}{lc}
\boldsymbol{z}_\ell & \ell\in A\\
0 & \ell\notin A
\end{array}
\right..
\]
As $\mathcal{P}_{\Gamma'}(\Bz) = \mathcal{P}_{\Gamma'}(\Blambda)$, (\ref{ee10}) can be rewritten as
\begin{align}\nonumber
\hat{\mathcal{G}}\big([\hat\Beta^T\!\!,\hat\balpha^T]^T\big)  & \geq  \hat{\mathcal{G}}\big([{\Beta^*}^T\!\!,{\balpha^*}^T]^T\big)  + \boldsymbol{h}^T \Bz \\ \nonumber & = \hat{\mathcal{G}}\big([{\Beta^*}^T\!\!,{\balpha^*}^T]^T\big) + \boldsymbol{h}^T \big(\Blambda - \mathcal{P}_{{\Gamma'}^c}(\Blambda) + \mathcal{P}_{{\Gamma'}^c} (\Bz)\big)\\ & = \hat{\mathcal{G}}\big([{\Beta^*}^T\!\!,{\balpha^*}^T]^T\big) + \boldsymbol{h}^T   \mathcal{P}_{{\Gamma'}^c}(\Bz) - \boldsymbol{h}^T   \mathcal{P}_{{\Gamma'}^c}(\Blambda),\label{ee11}
\end{align}
where, to establish the last equality we used the fact that $\boldsymbol{K}\boldsymbol{h}=\boldsymbol{0}$ and $\Blambda$ is in the range of $\boldsymbol{K}^T$.

Based on the sign of the entries in $\boldsymbol{h}$, we can partition ${\Gamma'}^c$ into disjoint index sets denoted as follows:
\begin{align*}
{\Gamma'}^c =& \underbrace{\big( \Gamma_0 \cap \{\ell:\boldsymbol{h}_\ell\geq 0\}\big)}_{\mathcal{I}_{+}^{0}} \;\cup\; \underbrace{\big( \Gamma_0 \cap \{\ell:\boldsymbol{h}_\ell< 0\}\big)}_{\mathcal{I}_{-}^{0}}  \\& \cup \; \underbrace{\big( \Gamma_1 \cap \{\ell:\boldsymbol{h}_\ell\geq 0\}\big)}_{\mathcal{I}_{+}^{1}} \;\cup\; \underbrace{\big( \Gamma_1 \cap \{\ell:\boldsymbol{h}_\ell< 0\}\big)}_{\mathcal{I}_{-}^{1}}.
\end{align*}
Since inequality (\ref{ee11}) needs to hold for any $\Bz\in \partial \mathcal{G}([{\Beta^*}^T\!\!,{\balpha^*}^T]^T)$, we can narrow our choice to a specific subgradient vector $\By$ such that $\By_{\Gamma'} =\Blambda_{\Gamma'}$ and
\begin{equation}\By_\ell = \left\{\begin{array}{cl}p_\ell - q_\ell & \ell\in \mathcal{I}_{+}^{0}\cup \mathcal{I}_{-}^{1}\\
- q_\ell & \ell\in \mathcal{I}_{-}^{0}\\
p_\ell & \ell \in \mathcal{I}_{+}^{1}
\end{array}
\right..\label{ee12}
\end{equation}
As every entry of $\By$ is set to an extreme possible value in ${\Gamma'}^c$, (\ref{ee12}) and (\ref{ee7}) trivially reveal that
\[\left\{\begin{array}{cc}\By_\ell - \Blambda_\ell>0 & \ell\in  \mathcal{I}_{+}^{0}\cup \mathcal{I}_{+}^{1}\\
\By_\ell - \Blambda_\ell<0  & \ell\in  \mathcal{I}_{-}^{0}\cup \mathcal{I}_{-}^{1}\end{array}
\right..
\]
Based on this observation, we set
\begin{align*}\varepsilon = \min\Big\{ &\{\By_\ell - \Blambda_\ell: \ell\in  \mathcal{I}_{+}^{0}\cup \mathcal{I}_{+}^{1}\}\cup \{ \Blambda_\ell - \By_\ell: \ell\in  \mathcal{I}_{-}^{0}\cup \mathcal{I}_{-}^{1}\}\Big\},
\end{align*}
which, based on the properties stated for $\By$ and $\Blambda$, is a strictly positive quantity. Accordingly,
\begin{align*}\nonumber
\boldsymbol{h}^T \big(  \mathcal{P}_{{\Gamma'}^c}(\By) -    \mathcal{P}_{{\Gamma'}^c}(\Blambda)\big)  &= \sum_{\ell\in {\Gamma'}^c}\boldsymbol{h}_\ell(\By_\ell  - \Blambda_\ell )\\& = \!\!\!\sum_{\ell\in  \mathcal{I}_{+}^{0}\cup \mathcal{I}_{+}^{1}}\!\!\!\boldsymbol{h}_\ell(\By_\ell  - \Blambda_\ell ) - \!\!\!\sum_{\ell\in  \mathcal{I}_{-}^{0}\cup \mathcal{I}_{-}^{1}}\!\!\!\boldsymbol{h}_\ell(  \Blambda_\ell - \By_\ell  ) \nonumber\\&\nonumber = \sum_{\ell\in {\Gamma'}^c}|\boldsymbol{h}_\ell||\By_\ell  - \Blambda_\ell |\\ &\geq \varepsilon \sum_{\ell\in {\Gamma'}^c}|\boldsymbol{h}_\ell|,
\end{align*}
using which in (\ref{ee11}) gives
\begin{equation}\label{ee14}
\hat{\mathcal{G}}\Big(\begin{bmatrix}\hat\Beta\\ \hat\balpha\end{bmatrix}\Big)\geq \hat{\mathcal{G}}\Big(\begin{bmatrix}\Beta^*\\ \balpha^*\end{bmatrix}\Big) + \varepsilon \|\boldsymbol{h}_{{\Gamma'}^c}\|_1.
\end{equation}
Since $\varepsilon > 0$, we always have $\hat{\mathcal{G}}\big([\hat\Beta^T\!\!,\hat\balpha^T]^T\big)  >  \hat{\mathcal{G}}\big([{\Beta^*}^T\!\!,{\balpha^*}^T]^T\big)$ except when $\boldsymbol{h}_{{\Gamma'}^c}=\boldsymbol{0}$ that an equality can happen. The possibility of having $\boldsymbol{h}\neq\boldsymbol{0}$ when $\boldsymbol{h}_{{\Gamma'}^c}=\boldsymbol{0}$ is abolished, as $\boldsymbol{K}_{:,\Gamma'}$ is full rank and the linear equation $\boldsymbol{K}\boldsymbol{h}=\boldsymbol{0}$ subject to $\boldsymbol{h}_{{\Gamma'}^c}=\boldsymbol{0}$ has the unique solution $\boldsymbol{h}=\boldsymbol{0}$. In other words, the equality in (\ref{ee14}) only happens when
\begin{equation*}
\begin{bmatrix}\hat\Beta\\ \hat\balpha\end{bmatrix} = \begin{bmatrix}\Beta^*\\ \balpha^*\end{bmatrix}.
\end{equation*}
Since for any pair of vectors $\Beta$ and $\balpha$ related by $\Beta = \boldsymbol{B}\balpha$, we have $\hat{\mathcal{G}}(\begin{bmatrix}\Beta^T, \balpha^T\end{bmatrix}^T) = \mathcal{G}(\Beta) = G(\balpha)$, the preceding line of argument reveals that $\Beta^*$ is the unique minimizer of (\ref{ew1}) and accordingly $\balpha^*$ is the unique minimizer of (\ref{ee-1}).
\end{proof}

We can use the result of Lemma \ref{lm7} to complete the theorem's proof. To show the full column rank property of $\boldsymbol{K}_{:,\Gamma'}$ we have
\[\boldsymbol{K}_{:,\Gamma'} = \begin{bmatrix}-\boldsymbol{I}_{:,\Gamma_{0^-}\cup\Gamma_{1^+}}& \boldsymbol{B}\\ \boldsymbol{0}&\boldsymbol{c^T}
  \end{bmatrix}.
\]
By applying basic rank preserving operations we can reform $\boldsymbol{K}_{:,\Gamma'}$ as
\[\boldsymbol{K}_{:,\Gamma'} \rightarrow \begin{bmatrix}-\boldsymbol{I} & \boldsymbol{0}\\ \boldsymbol{0}&\boldsymbol{B}_{\Gamma_0\cup\Gamma_1,:}\\ \boldsymbol{0}&\boldsymbol{c^T}
  \end{bmatrix},
\]
where the identity block has a width $|\Gamma_{0^-}\cup\Gamma_{1^+}|$. Clearly, if $\begin{bmatrix}\big(\boldsymbol{B}_{\Gamma_0\cup\Gamma_1,:}\big)^T\;,\; \boldsymbol{c} \end{bmatrix}$ has a full row rank, $\boldsymbol{K}_{:,\Gamma'}$ will be full column rank.

Further, suppose there exist $\boldsymbol{\eta}$ and $\eta_c$ that satisfy (\ref{ee-11}), (\ref{ee-12}). We introduce a vector $\tilde{\boldsymbol{\eta}}$, the entries of which are set to be
\begin{equation}\label{ew2}
\tilde{\boldsymbol{\eta}}_\ell  = \left \{\begin{array}{ll}   q_\ell & \ell\in \Gamma_{0^-} \\  -p_\ell & \ell\in \Gamma_{1^+} \\   \boldsymbol{\eta}_{m(\ell)} & \ell\in \Gamma_{0} \cup \Gamma_{1}\end{array} \right. .
\end{equation}
We show the vector $\boldsymbol{K}^T\begin{bmatrix}\tilde{\boldsymbol{\eta}}^T, \eta_c \end{bmatrix}^T$ satisfies conditions (\ref{ee7}) of Lemma \ref{lm7}. Showing this will complete the theorem's proof.

Clearly,
\begin{equation*}
\boldsymbol{K}^T\begin{bmatrix}\tilde{\boldsymbol{\eta}}\\ \eta_c \end{bmatrix} = \begin{bmatrix} -\boldsymbol{I}&\boldsymbol{0}\\\boldsymbol{B}^T & \boldsymbol{c}\end{bmatrix}\begin{bmatrix}\tilde{\boldsymbol{\eta}}\\ \eta_c \end{bmatrix} = \begin{bmatrix} - \tilde{\boldsymbol{\eta}} \\ \boldsymbol{B}^T \tilde{\boldsymbol{\eta}} + \eta_c\boldsymbol{c} \end{bmatrix}.
\end{equation*}
By the construction (\ref{ew2}), it is straightforward to verify that $- \tilde{\boldsymbol{\eta}}_\ell$ satisfies the conditions specified in (\ref{ee7}) for $\ell=1,2,\cdots,n_\Omega$, and we only need to show that $\boldsymbol{B}^T \tilde{\boldsymbol{\eta}} + \eta_c\boldsymbol{c} = \boldsymbol{0}$. For this purpose we have
\begin{align*}\nonumber
\eta_c\boldsymbol{c} + \boldsymbol{B}^T \tilde{\boldsymbol{\eta}}  &=\eta_c\boldsymbol{c} +  \big(\boldsymbol{B}_{\Gamma_1\cup\Gamma_0,:}\big)^T \tilde{\boldsymbol{\eta}}_{\Gamma_1\cup\Gamma_0} + \big(\boldsymbol{B}_{\Gamma_{1^+}\cup\Gamma_{0^-},:}\big)^T\tilde{\boldsymbol{\eta}}_{\Gamma_{1^+}\cup\Gamma_{0^-}} \\\nonumber  & = \eta_c\boldsymbol{c} + \big(\boldsymbol{B}_{\Gamma_1\cup\Gamma_0,:}\big)^T \boldsymbol{\eta} + \sum_{\ell\in \Gamma_{1^+}\cup\Gamma_{0^-}}\tilde{\boldsymbol{\eta}}_\ell \boldsymbol{J}^{(\ell)} \\[-.25 cm]&= \boldsymbol{e} + \sum_{\ell\in \Gamma_{1^+}\cup\Gamma_{0^-}}\tilde{\boldsymbol{\eta}}_\ell \boldsymbol{J}^{(\ell)}.
\end{align*}
However, every element of the last expression vanishes since
\begin{align*}\nonumber
\Big(\sum_{\ell\in \Gamma_{1^+}\cup\Gamma_{0^-}}\tilde{\boldsymbol{\eta}}_\ell \boldsymbol{J}^{(\ell)}\Big)_j &= \sum_{\ell\in \Gamma_{1^+}\cup\Gamma_{0^-}}\tilde{\boldsymbol{\eta}}_\ell \boldsymbol{J}^{(\ell)}_j \\ \nonumber& = \sum_{\ell\in (\Gamma_{1^+}\cup\Gamma_{0^-})\cap\mathcal{I}_j}\tilde{\boldsymbol{\eta}}_\ell \\ \nonumber& = \sum_{\ell\in \mathcal{I}_j\cap\Gamma_{1^+}}\tilde{\boldsymbol{\eta}}_\ell + \sum_{\ell\in \mathcal{I}_j\cap \Gamma_{0^-}}\tilde{\boldsymbol{\eta}}_\ell \\\nonumber & = -\sum_{\ell\in \mathcal{I}_j\cap \Gamma_{1^+}}p_\ell + \sum_{\ell\in \mathcal{I}_j\cap \Gamma_{0^-}} q_\ell\\&=-\boldsymbol{e}_j.\qquad \qed
\end{align*}

\subsection{Proof of Proposition \ref{lemcons}}
From (\ref{eqbr}), it is straightforward to verify that
\begin{equation}\label{eqbrt}
\boldsymbol{B}^\mathpzc{R}_{\Gamma_0^\mathpzc{R}\cup\Gamma_1^\mathpzc{R},:} =\begin{bmatrix}
 \boldsymbol{P} & \boldsymbol{0}\\[.2cm] \boldsymbol{B}^{(2,1)}_{\Gamma_0^\mathpzc{R},:} & \boldsymbol{\Delta}_\mathpzc{R}
 \end{bmatrix},
\end{equation}
where $\boldsymbol{P}\in \{0,1\}^{n_\oplus\times n_\oplus}$ is a permutation matrix. The appearance of the permutation block is thanks to
 to Theorem \ref{th2x}(a).

The invertibility of $\boldsymbol{\Delta}_\mathpzc{R}$ for a basic composition warrants the invertibility of the matrix in (\ref{eqbrt}), and hence the existence of a unique solution for (\ref{eqbw}). Accordingly, (\ref{eqbw}) may be cast as
\begin{equation}\label{eqg1}
\left\{\begin{array}{l}\boldsymbol{\Delta}_\mathpzc{R}^T \boldsymbol{w}_{\Gamma_0^\mathpzc{R}} = \boldsymbol{1} \\ \boldsymbol{P}^T \boldsymbol{w}_{\Gamma_1^\mathpzc{R}} + \tilde{\boldsymbol{B}} \boldsymbol{w}_{\Gamma_0^\mathpzc{R}} = - \boldsymbol{1}
\end{array}\right.,
\end{equation}
where $\tilde{\boldsymbol{B}} = {{\boldsymbol{B}^{(2,1)}_{\Gamma_0^\mathpzc{R},:}}}^T$\!\!.

The first equation in (\ref{eqg1}) is identical to (\ref{eqxx3}), the strictly positive solution of which is a characteristic of basic compositions. Since $\boldsymbol{\Delta}_\mathpzc{R}$ is a binary matrix and $\boldsymbol{w}_{\Gamma_0^\mathpzc{R}}$ is a strictly positive vector, its entries cannot exceed 1 and we must have $\boldsymbol{0} \prec \boldsymbol{w}_{\Gamma_0^\mathpzc{R}} \preceq \boldsymbol{1}$.

The solution to the second equation in (\ref{eqg1}) is simply $\vspace{-1.3mm}\boldsymbol{w}_{\Gamma_1^\mathpzc{R}} = - \boldsymbol{1} - \boldsymbol{P}\tilde{\boldsymbol{B}} \boldsymbol{w}_{\Gamma_0^\mathpzc{R}}$. As $\vspace{-1.mm}\boldsymbol{P}\tilde{\boldsymbol{B}}\in \{0,1\}^{n_\oplus \times n_\ominus}$ and $\boldsymbol{0} \prec \boldsymbol{w}_{\Gamma_0^\mathpzc{R}} \preceq \boldsymbol{1}$, we immediately come to the conclusion that $\vspace{-1.3mm} \boldsymbol{0} \preceq \boldsymbol{P}\tilde{\boldsymbol{B}}\boldsymbol{w}_{\Gamma_0^\mathpzc{R}}\preceq n_\ominus \boldsymbol{1}$. This would validate the claimed bounds on $\boldsymbol{w}_{\Gamma_1^\mathpzc{R}}$.\qed

\subsection{Proof of Lemma \ref{lemloc}}

We show the possibility of finding a point $\hat\balpha$ in the interior of $\mathcal{C}$, such that $\boldsymbol{c}^T\hat\balpha > \boldsymbol{c}^T \balpha^* = \|\balpha^*\|_1$.

Consider $\balpha'\in \mathbb{R}^{n_s}$, where $\balpha_{\Ip\cup\In}' = \balpha^*_{\Ip\cup\In}$ and $\balpha_{(\Ip\cup\In)^c}' = -\epsilon \boldsymbol{1}$. The strictly positive scalar $\epsilon\ll 1$ is taken to be sufficiently small so that $\supp^+(\mathcal{L}_{\balpha'}(x))=\Sigma$.

Further, consider $\balpha'' = k\balpha^*$, where $k>1$. Thanks to the LOC, $\balpha''$ is a minimizer of $G(.)$ as it satisfies (\ref{eq32}). We set $\hat\balpha = \balpha'+\balpha''$. For a similar reason, $\hat\balpha$ is a minimizer of $G(.)$, which means $G(\hat\balpha) = G(\balpha^*)$, or simply $\hat\balpha\in \mathcal{C}$.

Finally,
\[\boldsymbol{c}^T\hat\balpha - \|\balpha^*\|_1  = k\|\balpha^*\|_1 - \epsilon \sum_{j\in (\Ip\cup\In)^c}\boldsymbol{c}_j,
\]
where the right-hand side expression is capable of becoming strictly positive by choosing $k$ sufficiently large.\qed

\subsection{Proof of Theorem \ref{thglast}}
Since LOC holds for $\mbox{cl}(\mathpzc{R}_{\;\Ip,\In})$, using Theorem \ref{th7}, we need to verify the possibility of finding a vector $\boldsymbol{\eta}\in\mathbb{R}^{|\Gamma_0\cup\Gamma_1|}$ and a scalar $\eta_c$ such that
\begin{equation}\label{eq-g1}\begin{bmatrix}\boldsymbol{B}'\\ \boldsymbol{B}''\end{bmatrix}\boldsymbol{\eta} + \eta_c\boldsymbol{c}=\boldsymbol{0},
\end{equation}
where $\boldsymbol{c}_{\Ip} = \boldsymbol{1}$, $\boldsymbol{c}_{\In} = -\boldsymbol{1}$ and $\|\boldsymbol{c}_{(\Ip\cup\In)^c}\|_\infty \leq 1$. Taking LOC into account, to satisfy (\ref{ee-12}), the entries of $\boldsymbol{\eta}$ need to satisfy
\begin{equation*}
\left\{\begin{array}{lc} -p_i = q_i-p_i <\boldsymbol{\eta}_i <q_i = 0 & i\in \Gamma_0\\ 0 = -p_i <\boldsymbol{\eta}_i <q_i-p_i = q_i& i\in \Gamma_1
\end{array}\right..
\end{equation*}
The cost function in (\ref{eq43}) can be scaled by multiplying a constant, without changing the problem. Therefore, the entries of $\boldsymbol{\eta}$ only need to satisfy
\begin{equation}\label{eq-g2}
\left\{\begin{array}{lc}\boldsymbol{\eta}_i < 0 & i\in \Gamma_0\\ \boldsymbol{\eta}_i> 0& i\in \Gamma_1
\end{array}\right..
\end{equation}
We will show that when $\eta_c<0$ is chosen in a way that $|\eta_c|$ is sufficiently large, a vector with the following entries would satisfy the aforementioned requirements:
\begin{equation}\label{eq-g3}
 \boldsymbol{\eta}_i = \left\{\begin{array}{cl}
 \frac{\eta_c}{|\mathcal{J}_\ell|}\boldsymbol{w}_\ell & i\in \mathcal{J}_\ell , \quad \ell \in \Gamma_0^\mathpzc{R}, \Gamma_1^\mathpzc{R}\\[.1cm]
-1 & i\in T
\end{array}
\right. .
\end{equation}
Verifying that (\ref{eq-g3}) is in agreement with (\ref{eq-g2}) is straightforward. To demonstrate that (\ref{eq-g1}) holds, for $j\in \Ip$ we have
\begin{align}\label{eq-g4}
\boldsymbol{B}_{j,:}'\boldsymbol{\eta} = \sum_{i\in \Gamma_0\cup\Gamma_1} 1_{\{\Omega_i\subset\mathcal{S}_j\}}\boldsymbol{\eta}_i = \sum_{i\in\mathcal{J}_\ell}   \sum_{\ell\in \Gamma_0^\mathpzc{R}\cup \Gamma_1^\mathpzc{R}}  1_{\{\Omega_i\subset\mathcal{S}_j\}} \frac{\eta_c}{|\mathcal{J}_\ell|}\boldsymbol{w}_\ell - \sum_{i\in T} 1_{\{\Omega_i\subset\mathcal{S}_j\}}.
\end{align}
The second term on the right hand side of (\ref{eq-g4}) is simply zero. Using the fact that $1_{\{\Omega_i\subset\mathcal{S}_j\}} = 1_{\{\Omega_i\subset\Omega_\ell^\mathpzc{R}\}} 1_{\{\Omega_\ell^\mathpzc{R}\subset\mathcal{S}_j\}}$ we get
\begin{align*}
\boldsymbol{B}_{j,:}'\boldsymbol{\eta} &=  \eta_c \sum_{\ell\in \Gamma_0^\mathpzc{R}\cup \Gamma_1^\mathpzc{R}} \frac{1}{|\mathcal{J}_\ell|}    1_{\{\Omega_\ell^\mathpzc{R}\subset\mathcal{S}_j\}} \boldsymbol{w}_\ell   \sum_{i\in\mathcal{J}_\ell}   1_{\{\Omega_i\subset\Omega_\ell^\mathpzc{R}\}}\\&= \eta_c \sum_{\ell\in \Gamma_0^\mathpzc{R}\cup \Gamma_1^\mathpzc{R}} 1_{\{\Omega_\ell^\mathpzc{R}\subset\mathcal{S}_j\}} \boldsymbol{w}_\ell\\ &=   \eta_c \boldsymbol{w}^T \boldsymbol{B}_{\Gamma_0^\mathpzc{R}\cup \Gamma_1^\mathpzc{R},j}^\mathpzc{R}  \\& =-\eta_c.
\end{align*}
A similar line of argument shows that $\boldsymbol{B}_{j,:}'\boldsymbol{\eta} = \eta_c$ for $j\in \In$. We are only left to show the possibility of finding $\boldsymbol{c}_j \in [-1,1]$ such that
\begin{align*}
\boldsymbol{B}_{j-n_\oplus-n_\ominus,:}''\boldsymbol{\eta} + \eta_c \boldsymbol{c}_j=  0, \qquad j=n_\oplus+n_\ominus +1, \cdots, n_s.
\end{align*}
By setting $\boldsymbol{c}_j = - (\eta_c)^{-1}\boldsymbol{B}_{j-n_\oplus-n_\ominus,:}''\boldsymbol{\eta}$, we have
\begin{align*}
| \boldsymbol{c}_j| &= \frac{1}{|\eta_c|}\Big|-\sum_{i\in T} 1_{\{\Omega_i\subset\mathcal{S}_j\}}     +    \eta_c \sum_{\ell\in \Gamma_0^\mathpzc{R}\cup \Gamma_1^\mathpzc{R}} \frac{1}{|\mathcal{J}_\ell|}\boldsymbol{w}_\ell    \sum_{i\in\mathcal{J}_\ell}   1_{\{\Omega_i\subset\mathcal{S}_j\}}  \Big| \\ &=  \frac{1}{|\eta_c|}\Big|-\sum_{i\in T} 1_{\{\Omega_i\subset\mathcal{S}_j\}}     +    \eta_c \sum_{\ell\in \Gamma_0^\mathpzc{R}\cup \Gamma_1^\mathpzc{R}} \gamma_{\ell,j}\boldsymbol{w}_\ell  \Big|\\ &\leq  \frac{1}{|\eta_c|}\Big|\sum_{i\in T} 1_{\{\Omega_i\subset\mathcal{S}_j\}} \Big|    +  \Big|  \sum_{\ell\in \Gamma_0^\mathpzc{R}\cup \Gamma_1^\mathpzc{R}} \gamma_{\ell,j}\boldsymbol{w}_\ell  \Big|.
\end{align*}
When $|  \sum_{\ell\in \Gamma_0^\mathpzc{R}\cup \Gamma_1^\mathpzc{R}} \gamma_{\ell,j}\boldsymbol{w}_\ell|<1$, we can take $|\eta_c|$ to be sufficiently large so that $|\boldsymbol{c}_j|\leq1$ for all $j\in\{n_\oplus+n_\ominus+1, \cdots, n_s\}$.\qed

\section*{Appendix} The following table provides a quick reference to some important notations:\\

\hspace{-1cm}
\begin{tabular}[h]{c|c||c|c}
     \!\! \small\bf{Notation}\!\! & \!\!\small \bf{Description} \!\!& \small \!\!\bf{Notation}\!\! & \!\!\small \bf{Description}\!\! \\ \hline
      $D$ & \small imaging domain&  $\Omega$ & \small a shapelet / cell\\

      $u(x)$ & \small image pixel value at $x$        &$n_\Omega$ & \small number of shapelets \\

      $\Pi_{in/ex}(x)$& \small inhomogeneity measures&  $\mathcal{I}_j$ & \small see (\ref{eq11})\\

      $\tilde u_{in}$, $\tilde u_{ex}$&\small Chan-Vese mean pixels values & $\boldsymbol{B}$, $\boldsymbol{B}^\mathpzc{R}$ & \small dictionary / bearing matrices \\

      $n_s$& \small number of dictionary elements&  $\xi^\pm(\balpha)$ &\small see (\ref{eq13})\\

      $\mathcal{S}_j$&\small a shape in the dictionary& $\Kp$, $\Kn$ &\small index sets, see (\ref{eq18})\\

      $\mathpzc{R}_{\Ip,\In}$&\small  shape composition, see (\ref{eq3})&  $\boldsymbol{\beta}$, $\beta_i$ &\small shapelet coefficients, see (\ref{eq21x}) \\

      $\Ip,\In$ & \small index sets, see (\ref{eq3})& $\mathcal{A}(.)$ & \small linkage operation \\

            $|\;.\;|$ & \small set cardinality / magnitude & $\boldsymbol{\Delta}_\mathpzc{R}$ &\small discriminant matrix, see (\ref{dismat}) \\

      $\chi_\mathcal{S}$&\small see (\ref{eq9xyz})&$G(\balpha)$ &\small convex cost, see (\ref{econv}) \\

      $\mathpzc{L}_{\boldsymbol{\alpha}}$&\small see (\ref{eq9xyz}) &$\mathcal{G}(\Beta)$ &\small convex cost, see (\ref{eq43}) \\

      $\balpha$, $\alpha_j$&\small shape coefficients, see (\ref{eq9xyz}) &$p_i$, $q_i$ &\small see (\ref{eqpiqi}) \\

$\supp^+(.)$&\small positive support, see (\ref{eqsupp}) &$\Gamma_{0^-}$, $\Gamma_{1^+}$ &\small see (\ref{ee-7})  \\

        $\mbox{int}(.)$, $\mbox{cl}(.)$&\small interior and closure of a set &$\Gamma_0$, $\Gamma_1$ &\small see (\ref{ee-8})\\

 $\seq$, $\nseq$&\small see Section \ref{secdef} & $\boldsymbol{e}$ &\small LOC violation, see (\ref{ee-10}) \\

  $\boldsymbol{J}$&\small see (\ref{eq10}) &
$\gamma_{\ell,j}$ &\small see (\ref{gammalj})\\

$\Theta_{m}$&\small see (\ref{eq10xyz})&
$C_j$ &\small shape coherence, see (\ref{eqcoherence})
      \end{tabular}

\bibliographystyle{siam}

%
\end{document}